\documentclass{article}
\pdfoutput=1
\usepackage{fullpage}
\usepackage[breaklinks,hidelinks]{hyperref}
\hypersetup{
  pdftitle={Inductive Grounded Arithmetic: Goodstein, the Hydra, and Gamma-0},
  pdfauthor={Bryan Ford}}

\usepackage{amsmath}
\usepackage{amssymb}
\usepackage{amsthm}
\usepackage{booktabs}
\usepackage{array}
\usepackage{longtable}

\newtheorem{thm}{Theorem}[section]

\newtheorem{lem}[thm]{Lemma}
\newtheorem{defi}[thm]{Definition}

\usepackage[T1]{fontenc}	%

\usepackage{times}
\usepackage{xspace}
\usepackage{cite}		%
\usepackage{bussproofs}
\usepackage{stmaryrd}
\usepackage[noabbrev,capitalize]{cleveref}
\usepackage{thmtools}
\usepackage{xstring}
\usepackage{graphicx}
\usepackage{subcaption}
\usepackage{nicematrix}		%
\usepackage{longtable}
\usepackage{pifont}

\usepackage{isabelle}
\usepackage{isabellesym}
\usepackage{pdfpages}
\usepackage{ragged2e}
\newcommand{\SNIP}[2]{\expandafter\newcommand\csname snippet--#1\endcsname{#2}}
\IfFileExists{snips.tex}{\input{snips}}{}

\newcommand{\GetSnip}[1]{%
    \ifcsname snippet--#1\endcsname%
        \csname snippet--#1\endcsname%
    \else%
        \PackageWarning{snips}{Snippet ``#1'' is undefined.}%
        \emph{Warning: Snippet ``#1'' is undefined.}%
    \fi%
}

\newcommand{\Snippet}[1]{{%
  \newcount\i
  \i=0
  \loop
    \GetSnip{#1-\the\i}%
    \advance \i 1
  \ifcsname snippet--#1-\the\i\endcsname
  \repeat
}}

\newcommand{\SnippetPart}[3]{{%
  \newcount\i
  \i=#1
  \loop
    \ifnum \i=#2
      \renewcommand{\isanewline}{}%
    \fi
    \GetSnip{#3-\the\i}%
    \advance \i 1
    \ifnum \i>#2 {}
    \else \repeat
}}

\hypersetup{breaklinks=true}

\newcommand{\com}[1]{}

\long\def\note#1{}	%

\newcommand{\rga}{RGA\xspace}		%

\com{	%

}%

\makeatletter
\newlength{\doublefracgap}
\DeclareRobustCommand{\doublefrac}[2]{%
  \mathinner{\mathpalette\doublefrac@{{#1}{#2}}}%
}
\newcommand{\doublefrac@}[2]{\doublefrac@@#1#2}
\newcommand{\doublefrac@@}[3]{%
  \ooalign{%
    \raisebox{\doublefracgap}{$\m@th#1\frac{#2}{\phantom{#3}}$}\cr
    \raisebox{-\doublefracgap}{$\m@th#1\frac{\phantom{#2}}{#3}$}\cr
  }%
}
\newcommand{\ddoublefrac}[2]{{\displaystyle\doublefrac{#1}{#2}}}

\makeatother

\newcommand{\irl}[1]{\ensuremath{\mathit{#1}}}		%

\newcommand{\infrule}[3][]{\cfrac{#2}{#3}\IfStrEq{#1}{}{}{\ \irl{#1}}}
\newcommand{\infeqv}[3][]{\ddoublefrac{#2}{#3}\IfStrEq{#1}{}{}{\ \irl{#1}}}
\newcommand{\infceqv}[4][]{\cfrac{#2\qquad}{}\ddoublefrac{#3}{#4}
				\IfStrEq{#1}{}{}{\ \irl{#1}}}

\newcommand{\quo}[1]{\ulcorner{#1}\urcorner}

\crefname{thm}{Theorem}{Theorems}
\crefname{cor}{Corollary}{Corollaries}
\crefname{lem}{Lemma}{Lemmas}
\crefname{defi}{Definition}{Definitions}
\crefname{rem}{Remark}{Remarks}

\providecommand{\oga}{OGA\xspace}
\providecommand{\rga}{RGA\xspace}
\providecommand{\iga}{IGA\xspace}
\providecommand{\oiga}{OIGA\xspace}	%
\newcommand{\artifacttag}{\texttt{iga-}\allowbreak\texttt{2026-11-11}\xspace}	%
\newcommand{\cA}{\mathcal A}             %
\newcommand{\fix}{\mathrm{fix}}          %
\newcommand{\Acc}{\mathrm{Acc}}
\newcommand{\WF}{\mathrm{WF}}
\newcommand{\Ord}{\mathrm{Ord}}
\newcommand{\Tr}{\mathrm{Tr}}
\newcommand{\Trc}{\mathrm{Tr}^{\circ}}  %
\newcommand{\lfp}{\mathrm{lfp}}
\newcommand{\isn}[1]{\textsf{#1}}        %
\newcommand{\ConA}[1]{\mathit{Con}_{\mathcal O(#1)}}  %
\newcommand{\ent}{\vdash}
\newcommand{\eps}{\varepsilon_0}
\newcommand{\dotminus}{\mathbin{\dot{-}}}   %

\title{Inductive Grounded Arithmetic: 
	Goodstein, the Hydra, and $\Gamma_0$}
\author{Bryan Ford\\EPFL}
\date{}					%

\begin{document}
\maketitle

\begin{abstract}
What does a recursive definition mean in a grounded logic, where a
sentence is true only if a terminating process makes it so?  This
paper takes Kripke's answer for truth, the least fixed point, as a
rule.  Reflective Grounded Arithmetic (\rga) writes any recursive
definition as one self-referential term; Inductive Grounded Arithmetic
(\iga) adds the elimination rule for every positive one, making each an
inductive definition.  \iga
is omega-sound and proof-grounded: sound and complete for a semantics
whose certificates are its own proofs.  \iga proves Goodstein's
theorem, the Kirby--Paris hydra game for primitive recursive
strategies, and, along Gentzen's route, the consistency of Peano
Arithmetic.
\iga certifies the Veblen notation below every $\varphi(n,0)$:
its accessibility ordinal there is
$\varphi(\omega,0)$, that of the
paracomplete truth theory PKF.  Every theorem is machine-checked in
Isabelle/HOL; order types and PKF's ordinal are cited.

\iga proves the Veblen
notation well-founded exactly when it admits bivalence of
accessibility for it, which it does not, and that bivalence is sound
exactly when the order is well-founded. So $\Gamma_0$ lies exactly one
instance of excluded middle away, and that instance, the well-foundedness sentence and one
membership claim in an internal form of Kleene's $\mathcal O$
are interderivable over \iga.  Bivalence at the
Bachmann--Howard order keeps proof-groundedness. One such
claim buys well-foundedness of the countable part of Buchholz's
notation at every finite level, of type $\psi_0(\Omega_\omega)$, and
$k$ stratified bivalence rules give a proof-grounded system proving
the countable part up to level $k$ well-founded; the systems of $k$
iterated closings-off of truth, the Kripke universes, do too, sound only for a
closed-off model.  At a constant-free presentation of the truth term
the Kripke universes form a consistency tower: each proves the
consistency of the one below and not its own.  Unrestricted conditional proof is
inconsistent (Curry's paradox); the purchase is
implication at a hypothesis not provably decided.
\end{abstract}

\newpage
\setcounter{tocdepth}{2}
\tableofcontents
\newpage

\section{Introduction}
\label{sec:intro}

Consider a game on the natural numbers.  From a position $x$ a player
may move to any $y$ with $\mathit{move}(x,y)$; the opponent replies
with any $z$ with $\mathit{reply}(y,z)$; and a position is winning
when some move leaves every reply at a winning position.  Written as a
definition,
\[
  W(x) \;\equiv\; \exists y.\ \mathit{move}(x,y) \wedge
                 \forall z.\ \mathit{reply}(y,z) \rightarrow W(z).
\]
Nobody mistakes this for a paradox.  It is the kind of recursive
definition that working mathematicians and programmers write every
day: a predicate defined in terms of itself, meant to pick out the positions that are winning because the game below them is.  A
proof assistant accepts it as an inductive definition, checks that $W$
occurs positively, and hands back an induction principle.  Classical
first-order arithmetic cannot even state it as a single sentence; it
needs a schema, a second-order quantifier, or a truth predicate.

This paper is about what such definitions mean, and what they buy, in
a \emph{grounded} arithmetic: one in which a sentence is true only
when a terminating computation makes it so, and in which a sentence
that no computation settles is not false but has no value.  The
starting point is Reflective Grounded Arithmetic
(\rga)~\cite{rga-paper}, a paracomplete arithmetic, one in which a
sentence may be neither true nor false, descended from Kripke's
theory of truth~\cite{kripke75outline} and grounded
deduction~\cite{ford24reasoning,ford25have}, in which every term is a
computation and self-reference is native: a fixed-point former turns
any body into a term that unfolds to itself.  \rga can write the
definition of $W$ as one term.  It can also write the Liar as one
term, and prove that the Liar has no value.  What \rga cannot do is use the fact that $W$ is the least such predicate.  It can
unfold $W$, but it has no principle that says a property closed under
the game's rules holds at every winning position.

Inductive Grounded Arithmetic (\iga) adds that principle, and nothing
else.  The rule, called ground induction here, is the elimination
rule for a positive self-referential term: from the fixed point at a
point $x$, and a derivation that a motive is closed under one
unfolding, conclude the motive at $x$.  Semantically it says that the
grounded meaning of a positive recursive definition is its least
fixed point, which is Kripke's construction read as a rule.
Classically it is Park's fixed-point induction~\cite{park69fixpoint},
the induction schema of Feferman's theories of inductive
definitions~\cite{feferman70formal,buchholz81iterated}, and Cantini's
minimality principle for Kripke--Feferman
truth~\cite{cantini89notes}.  The thesis of the paper is that the
objects were always there: \rga already writes accessibility,
winning positions, and its own truth predicate as single terms.  What
the rule adds is the ability to reason about them, and the paper
measures what that ability buys and where it stops, with every result
checked in Isabelle/HOL.

\paragraph{Contributions.}
\begin{itemize}
\item \emph{The rule and its semantics} (\S\ref{sec:positive},
  \S\ref{sec:induction}).  A syntactic grammar of positive bodies with parameters, the rule,
  soundness for the omega model (the infinitary reading in which a
  universal is true when all its instances are), consistency, and the
  characterization theorem: the truth set of a positive fixed point is
  the least fixed point of its operator.  A positive self-referential term is an inductive definition, in the sense a proof
  assistant's inductive package gives the words.  \iga is also sound
  and complete for the certificate-grounded semantics with its own
  proofs as the certificates, so it has the disjunction property.
\item \emph{Truth as a fixed point} (\S\ref{sec:truth}).  Kripke truth
  for \rga's own language is one positive fixed point that agrees
  with the omega model exactly, with its induction principle, its clause laws and their inversions,
  the internal refutation of truth at $0=1$, and determinacy of its
  values, proved by ground induction.
\item \emph{Well-foundedness as a sentence} (\S\ref{sec:wf}).
  Accessibility along a relation is one self-referential term, and
  ``$R$ is well-founded'' is one closed sentence, never refutable and
  sound into HOL.  \rga alone proves it for $\omega$ and every finite
  tower; \iga proves it for $\eps$, as a closed order on
  Cantor-normal-form codes, and for every $\varphi(n,0)$ of the Veblen
  notation, and certifies the notations as ordinals on canonical
  codes.  HOL's well-foundedness of the standard $\eps$ notation
  system follows from the internal theorem through soundness.
\item \emph{Goodstein and the hydra} (\S\ref{sec:goodstein}).  \iga
  proves Goodstein's theorem outright, every step of the arithmetic a
  grounded derivation, in a logic that has no implication at an
  undecided hypothesis; and the Kirby--Paris hydra game, for every primitive recursive
  strategy, outright as well, by the same method.
\item \emph{The consistency of Peano Arithmetic} (\S\ref{sec:conpa}).
  \iga derives, from no hypotheses, the consistency sentence of a
  checker that, for each theorem of Peano Arithmetic, \rga proves to
  accept some proof of it,
  along Gentzen's route through Buchholz's notations for derivations
  with the omega rule~\cite{buchholz91notation,buchholz97explaining}.
  The reduction steps, their descent and the ordinal bound below $\eps$
  are \rga theorems, and \iga supplies only transfinite induction.  So
  an inconsistency in Peano Arithmetic gives one in \iga, with no
  appeal to a soundness theorem of \iga (\cref{thm:pa}).
\item \emph{The price of the classical world} (\S\ref{sec:gamma0}).
  \iga proves the whole Veblen notation well-founded if and only if it
  admits bivalence of accessibility for that order, and bivalence of
  accessibility for an order is sound if and only if the order is
  well-founded.  \iga does not prove the Veblen notation well-founded: a stage-indexed refinement of the omega model shows that every closed theorem of \iga is true at a stage below $\varphi(\omega,0)$, so \iga proves a notation accessible exactly when the notation lies below $\varphi(\omega,0)$, its ordinal on the notation is $\varphi(\omega,0)$, and bivalence of accessibility is not admissible.  The commitment itself is sound for that order, so $\Gamma_0$ lies exactly one instance of excluded middle away, at the accessibility
  predicate of the notation: added to \iga as a rule it gives a sound
  and consistent system that proves it, and over \iga the rule, the
  sentence and one
  membership claim in a primitive-recursive form of Kleene's $\mathcal
  O$, defined inside \iga as a positive fixed point, are
  interderivable.  Conditional proof at arbitrary hypotheses is inconsistent, by
  Curry, so \rga's deduction theorem at decided hypotheses is forced;
  the bivalence schema explodes on the Liar. The internal
  soundness argument for \rga's consistency, in the form the
  companion paper on internal reasoning in \rga~%
  \cite{ford-conrga}
  analyzes, is not available to any omega-sound extension of \rga,
  \iga included.
\item \emph{Rungs above $\Gamma_0$} (\S\ref{sec:one-rule},
  \S\ref{sec:pricing}, \S\ref{sec:measure}, \S\ref{sec:tower},
  \S\ref{sec:contower}).  Bivalence of accessibility at an order with a
  classical warrant, decided at numerals, whose well-foundedness the
  extended system proves, keeps \iga proof-grounded, at the Veblen order and
  at Buchholz's countable order for the Bachmann--Howard ordinal
  (\cref{thm:rung}).  The purchase climbs further, to Buchholz's
  notation system with collapsing functions at every finite
  level~\cite{buchholz86new}.  One membership claim in $\mathcal O$
  buys \iga the well-foundedness of the notation's countable part,
  and the claim and the well-foundedness sentence are interderivable
  (\cref{thm:otnu-claim}).  For every $k \geq 1$, \iga with bivalence
  of accessibility at $k$ stratified orders is proof-grounded and
  proves the countable part of the levels up to $k$ well-founded
  (\cref{thm:otnu-rule}).  The order types, $\psi_0(\Omega_{k+1})$ and
  $\psi_0(\Omega_\omega)$, come from the cited ordinal analyses and are
  not mechanized; a build-time dependency check confirms that the rule
  route uses nothing of the claim route.  Closing off the truth
  predicate reaches the same notations by another route: \iga with
  bivalence of the truth atom proves Buchholz's countable notation
  well-founded, through a decided proxy for accessibility, and the
  tower of its iterations, whose systems the paper calls \emph{Kripke
  universes} (after Kripke's closing-off of truth, and unrelated to
  Grothendieck universes), proves the countable part of
  $\mathrm{OT}(k)$ well-founded at level $k$, every Kripke universe
  sound for one closed-off model and none from level $1$ on for the
  omega model
  (\cref{thm:itb-bh,thm:ladder-k}; for the omega model,
  \cref{thm:itb-nostage,thm:one-model} and \S\ref{sec:tower}).  At a
  constant-free presentation of the truth term, which replaces the
  constant $0$ inside its body by the negated self-equation of the
  body's own point, the Kripke universes form a consistency tower:
  each proves the consistency of the one below, a grounded
  claim, and not its own, a classical one, by G\"odel's second theorem
  at its own proof checker (\cref{thm:cf-sep}).
\end{itemize}

\paragraph{Scope.}
\iga's ordinal on the Veblen notation is $\varphi(\omega,0)$, mechanized in both directions, and it coincides with the ordinal of the paracomplete Kripke--Feferman theory PKF; the measure is accessibility, equivalently well-foundedness as one sentence, and transfinite induction at decided motives is not bounded by the argument.  The same method prices the other rules of the grounded family (\S\ref{sec:pricing}): \oga's decidedness rule leaves the ordinal at $\varphi(\omega,0)$, one reflection over \iga with the decidedness rule is worth at most one step, to $\varphi(\omega,1)$, and in calibration devices whose reflection premise is checked in HOL one reflection is worth exactly that step and the autonomous progression of reflection closes at $\varphi(\omega+1,0)$.  Closing off every positive body at once by $\mu$-bivalence derives $0=1$; restricted to the closed, leaf-decided bodies that use their recursive variable it is consistent and sound for a closed-off model, but neither omega-sound nor proof-grounded (\cref{thm:mubiv}; Appendix~\ref{app:stage}, step 58).  Mechanized order-type claims on the Veblen and $\eps$ notations are on canonical codes: the internal
orders do not identify two codes of one notation; the order types of Buchholz's notation are cited.    
The identification of the fixed-point truth sets with the $\Pi^1_1$
sets is by citation.  \iga also has a certificate-grounded
semantics of its own: with \iga's proofs as the certificates, \iga is
sound and complete for that semantics, and so has the disjunction
property (\cref{thm:pg}).  Appendix~\ref{sec:table} lists every
theorem with its Isabelle name and its viewpoint.

\paragraph{Reading the metatheory.}
Two kinds of claim appear below, and the paper labels them.  A
\emph{grounded} claim is a derivation in \rga, \iga or \oiga: something the
system itself has.  A \emph{classical} claim is a HOL theorem about the system, about its derivability relation or about
the omega model: something true of the system from outside, by
classical reasoning.  An equivalence between two derivability
statements is a classical claim, even though each side is grounded.
Soundness connects the two kinds in one direction only.  The
distinction matters most in \S\ref{sec:gamma0}, where the paper's
point is exactly which claims are of which kind.

\paragraph{AI disclosure.}
The machine-checked proofs and the first drafts of this paper's text
were produced with substantial assistance from Claude models from
Anthropic, working under the direction and review of the author
(the artifact paragraph of \S\ref{sec:concl}).  The definitions that a
reader must trust are the rows marked ``def.''\ in
Appendix~\ref{sec:table}, which lists every result with its Isabelle
name; the author reviewed these definitions and theorems.

\paragraph{Road map.}
\S\ref{sec:bg} recalls \rga.  \S\ref{sec:positive} defines positive
self-reference and \S\ref{sec:induction} the rule, with a worked
example.  \S\ref{sec:truth}--\S\ref{sec:conpa} are the applications:
truth (\S\ref{sec:truth}), ordinals (\S\ref{sec:wf}), Goodstein and the
hydra (\S\ref{sec:goodstein}), and the consistency of Peano Arithmetic
(\S\ref{sec:conpa}, \cref{thm:pa}).  \S\ref{sec:gamma0} states the limits and prices what lies beyond them: the rule and the sentence at $\Gamma_0$ (the two subsections before \S\ref{sec:one-rule}), the claim (\S\ref{sec:one-rule}), the ordinal (\S\ref{sec:strength}), the other rules and the rungs up to Buchholz's notations (\S\ref{sec:one-rule}, \S\ref{sec:pricing}), the closings-off of the truth term and their tower (\S\ref{sec:measure}, \S\ref{sec:tower}), Curry's paradox (\S\ref{sec:classical}), and the consistency tower (\S\ref{sec:contower}).  \S\ref{sec:related} places the system among its neighbors
and \S\ref{sec:concl} lists what remains open.  Theorem headers carry the Isabelle names of their main statements.
Appendices~\ref{app:gfix} to \ref{app:stage} give the eight headline proofs step by step, for a reader who wants to follow or re-mechanize
them; the numbered steps of Appendix~\ref{app:stage} also give the statements of \S\ref{sec:pricing}, of \cref{thm:pa} (step 61) and of the rungs above $\Gamma_0$ (steps 62 and 63), and Appendix~\ref{sec:table} lists every theorem with its name.

\section{Background: Reflective Grounded Arithmetic}
\label{sec:bg}

This section recalls the parts of \rga that the rest of the paper
rests on.  The full system and its semantics are in the \rga
paper~\cite{rga-paper}, its internal reasoning and consistency in
a companion paper~%
\cite{ford-conrga}, called the internal-reasoning paper below, and
G\"odel's theorems for \rga and its extensions in a further companion
paper~%
\cite{goedel-paper}; nothing here is new.  Throughout,
$\bar n$ is the numeral for the number $n$, and $\lceil t \rceil$ is
the numeral that codes the term $t$.

\paragraph{Terms as computations.}
\rga has one syntactic category: terms.  A term is built from
numerals, variables (de~Bruijn indices, so that a bound variable is
named by its distance from its binder), successor, a recursor, a
single binder $\lambda$, application under call-by-name (an argument
is passed unevaluated and computed only when used), equality,
negation, disjunction, and one more constant, the universal
combinator $\cA$.  There are no separate formulas.  A term computes by
head reduction; a closed term may reach a numeral, in which case that
numeral is its value, or it may run forever.  The logical constants are
read on values: $0$ and $1$ serve as the truth values, so an equality
$a = b$ has value $1$ when both sides reach the same numeral and
$0$ when both reach different ones, and no value when either side
diverges.  Negation flips the value; disjunction has value $1$ when
either disjunct does and $0$ when both have value $0$.  A universal
$\cA \cdot u$ has value $1$ when every instance $u \cdot \bar n$ has
value $1$ and value $0$ when some instance has value $0$.  A term with
no value is neither true nor false: it has no value at all.  This is
the paracomplete reading.  The Liar, the term that unfolds to its own
negation, is the standard example of a term with no value.

\paragraph{Judgments and rules.}
A judgment is a sequent $\Gamma \ent c$ with a finite set of
hypothesis terms on the left.  A term is \emph{well-formed} when its
binders and applications are built correctly from the constructors
with de~Bruijn indices in range; the mechanization carries
well-formedness as a side condition, and every term in this paper is
well-formed.  Three judgment forms recur below.
$\Gamma \ent c$ says $c$ is true; $\Gamma \ent c\,\mathbf N$ says $c$
is a number, which is to say it reaches a numeral; and $\Gamma \ent
c\,\mathbf B$ says $c$ is \emph{decided}, which is to say it reaches
$0$ or $1$.  The proof system has forty-two rules, all of them familiar
except for two features.  First, universal introduction is schematic:
to prove $\cA \cdot u$ one proves $u \cdot v$ for a fresh variable
$v$ under the hypothesis that $v$ is a number.  A universal is thus
earned by one argument that covers every instance, never by checking
the instances.  Second, conversion is a rule in both directions: if
$c$ head-reduces to $c'$ then $\Gamma \ent c$ and $\Gamma \ent c'$ are
interderivable.  Primitive recursive functions enter through a
verified compiler.  An \emph{index tree} is a syntax tree naming a
primitive recursive function by its constructors (zero, successor,
projection, composition, primitive recursion), so that every such
function has an explicit, finite description.  For every index tree
$f$ there is a closed term $\mathrm{cmp}\,f$ that computes the
function, and \rga proves the compiled term total.

\paragraph{Self-reference.}
The fixed-point former is a term, not a rule:
$\fix\, f$ is $(\lambda x.\, f (x\,x))(\lambda x.\, f (x\,x))$, and it
head-reduces to $f (\fix f)$ in one step.  The Liar is $\fix(\lambda
x.\, \neg x)$.  \rga proves the Liar neither true nor false, and more:
a well-formed sentence with no value is independent, neither provable
nor refutable~\cite{rga-paper}.  The same former will build every
recursive definition in this paper.

\paragraph{Two semantics.}
\rga has two semantics, and the difference between them shapes the rule's two soundness proofs in \S\ref{sec:induction}.  The \emph{certificate-grounded}
semantics is the one in which a universal is true only when every instance is true and \rga has a schematic proof of the open
instance.  Under it, \rga is sound and complete: a sentence is
provable exactly when it is grounded-true, which is the sense in which
what \rga proves is exactly what is true in its own world.  The \emph{omega model} drops the proof: a universal is true when
every instance is true, and nothing else changes.  Formally the omega
model is a pair of mutually inductive relations, ``the term $t$
reaches the head form $w$'' (an abstraction, the recursor with some
of its arguments, or the universal combinator) and ``the term $t$
has the value $r$'', with twenty-one rules, among them the
infinitary rules for the universal; a term has at most one value.  \rga is sound for the
omega model~\cite{rga-paper}.  Soundness for the omega model is a
plain rule induction with no reference to proofs, and it is the
semantics against which this paper first proves its new rule sound;
\S\ref{sec:induction} then proves the rule sound for the
certificate-grounded semantics as well (\cref{thm:pg}).  The
two semantics differ in character as well.  The certificate-grounded
one is finitary: every truth has a finite witness, a computation or
a proof.  The omega model is not: a universal is true when all its
instances are, and the sets it makes true include sets far beyond
the computable ones (\S\ref{sec:induction}).  When this paper says
that every \iga term is a computation and every \iga proof a finite
object, that is a fact about the system; the omega model is the
classical yardstick against which the system is measured.

\paragraph{What \rga lacks, exactly.}
Two absences shape everything below.  \rga has no deduction theorem at
an undecided hypothesis: from $a, \Gamma \ent b$ one may conclude
$\Gamma \ent \neg a \vee b$ only when $a$ is decided.  The
internal-reasoning paper proves this exact~%
\cite{ford-conrga}: the
deduction theorem holds at a hypothesis $a$ for every conclusion if
and only if $a$ is provably decided, and the two-line proof transfers
to every extension of \rga considered here.
\S\ref{sec:classical}, on conditional proof, shows why the restriction cannot be lifted.
And \rga has no induction principle for a self-referential term beyond
unfolding it.  It can prove that the game's winning predicate unfolds
to its definition; it cannot conclude anything about all winning
positions from that.  Supplying that principle is the whole of \iga.

\paragraph{Notation.}
A term with one free index is written as the body of a binder;
$\lambda\, b$ binds it.  A sibling extension of \rga, Omega Grounded
Arithmetic (\oga)~\cite{oga-paper}, adds a rule that makes a universal
over a decided matrix decided; it appears in the ordinal results of \S\ref{sec:wf} and in the related work of \S\ref{sec:related}.  Isabelle names appear as \isn{name}.

\section{Positive Self-Reference}
\label{sec:positive}

A recursive definition has a body with a hole for the thing being
defined.  Ground induction applies only when the hole occurs positively.  The
reason is not monotonicity.  The operator that the rule of \S\ref{sec:induction} attaches to a body reads a recursive atom as true at a member of the
set and as having no value otherwise, so it only ever moves an atom
from ``no value'' to ``true'', and it is monotone whatever the body:
for the Liar body $\neg(a\cdot x)$ it is constantly empty, and its
least fixed point, the empty set, happens to be the term's truth set.
The reason is that only for a positive body does the operator
describe the term: the truth of the body with the fixed point plugged
in then depends on the fixed point's instances only through their
truth, never through their falsity or their having no value, whereas
plugging the term into the Liar body lets an atom be false.  That
dependence is the hinge of the soundness proof.  This section fixes the
grammar.

\paragraph{Bodies.}
A body is a term $\Phi$ with two distinguished free indices: index $1$
is the recursive variable, written $a$ below, and index $0$ is the
point, written $x$.  The fixed point of a body is the closed term
\[
  \fix\,\Phi \;\equiv\; \fix(\lambda a.\, \lambda x.\, \Phi),
\]
and its unfolding at a point is three head steps:
$(\fix\,\Phi)\cdot x$ reduces to $\Phi$ with $\fix\,\Phi$ substituted
for $a$ and $x$ for the point.  By the conversion rule, the fixed
point at $x$ and its one-step unfolding are interderivable.  A body
may have further free indices, its \emph{parameters}: a definition
indexed by a parameter may carry it in the point, as a code, or leave
it free in the body, where the rule of \S\ref{sec:induction} treats it
as a term of the context.  The one restriction on parameters is in the
grammar below.

\paragraph{The grammar.}
Positivity is a mutual inductive predicate on bodies, $\mathit{pos}$
and $\mathit{neg}$, indexed by the recursive variable:
\begin{itemize}
\item a term in which the recursive variable does not occur is both
  positive and negative;
\item $a \cdot u$ is positive when $u$ is a \emph{numeral
  expression} in the bound variables: the point, a variable bound by a
  quantifier inside the body, zero, or the successor of a numeral
  expression; a parameter may not appear in $u$;
\item $\neg b$ is positive when $b$ is negative, and negative when $b$
  is positive;
\item $b \vee c$ is positive when both disjuncts are, and negative
  when both are;
\item $\cA \cdot (\lambda\, b)$ is positive when $b$ is, with the
  indices shifted under the binder, and likewise for negative.
\end{itemize}
Conjunction is defined in \rga as $\neg(\neg b \vee \neg c)$, so the
clauses for it follow.  The grammar is checked by simplification on any concrete body.  At a
closed body the restriction on parameters is empty and the grammar
is the plain one, with ``a variable other than $a$'' in the atom
clause; the development keeps both, and the parametric grammar is the
rule's side condition.  Since the recursive atom may sit under two
negations, this is positivity in the sense of Isabelle's inductive
package, not the strict positivity of Rocq or Agda, which forbids the
recursive variable under any negation or to the left of any
implication.  Strict positivity is not needed here because there is no implication
to be on the left of; a double negation only wraps a disjunction, and
the reading of the body defined below sees through it.  Three examples fix the idea.
The accessibility body of \S\ref{sec:wf}, $\cA \cdot \lambda y.\,
\neg R(y,x) \vee a \cdot y$, is positive.  The winning-position body
of the introduction, with the existential written as a negated
universal, puts $a \cdot z$ under an even number of negations and is positive.  The
Liar's body $\neg a$ is not positive, and neither is the unary Liar's
$\neg (a \cdot x)$: the rule of \S\ref{sec:induction} never applies to
it, which is as it should be.

\paragraph{What the grammar forces.}
Two consequences of the grammar run through every later section.

The first consequence is the restriction to numeral expressions in the argument of
a recursive atom.  One would like to write $a \cdot (f\,u)$ for a
computed code $f\,u$, as a universe of codes or a truth predicate on
codes must.  The grammar does not admit it, for one reason that is an
obstruction and one that is a feature of the present soundness proof.
The obstruction: an arbitrary argument need not have a value, and the
fixed point applied to a divergent term can be true while the closure
premise of the rule speaks only of numerals, so a rule admitting it
would be unsound as it stands.  The feature: even when the argument
reaches a numeral, the soundness proof's induction runs on the omega
model's evaluation, and the step that replaces the argument by its
value builds a fresh evaluation rather than a sub-evaluation, out of
the induction hypothesis's reach.  Whether a different proof removes
the second reason is open.  The grammar therefore admits a computed
argument only by routing it through a quantifier with a decided
equality guard, in either of two forms,
\[
  \cA \cdot \lambda m.\, \neg (m = f\,u) \vee a \cdot m
  \qquad\text{or}\qquad
  \neg\, \cA \cdot \lambda m.\, \neg\bigl( m = f\,u \wedge a \cdot m \bigr) ,
\]
a universal and an existential.  When $f$ is a compiled primitive
recursive function and $u$ a number, the guard $m = f\,u$ is decided,
so the disjunction has a value, and $a$ is applied to the bound
variable $m$.  Every computed argument in this paper is routed this
way; the truth predicate of \S\ref{sec:truth} uses the existential
form.  A parameter is excluded from the argument of a recursive atom
for the first of the two reasons: the rule may be used at any term in
place of the parameter, so after substitution the argument is an
arbitrary term, possibly without a value, and an atom at it is not
something the oracle reading can read.  Routed through the guard, a
parameter is as usable as a computed argument.  The cost is binders: the truth predicate of \S\ref{sec:truth} needs
up to nine of them in a single clause.

The second consequence is what positivity buys.  Read the body with
the recursive atoms interpreted as membership in a set $S$, an atom being true when its argument lies in $S$ and having no value
otherwise: this is the \emph{oracle reading}, used throughout
\S\ref{sec:induction}.  For a positive body, truth under the oracle
reading at the fixed point's truth set coincides with truth of the
body with the fixed point plugged in, because truth of a positive
body never consults the falsity of a recursive atom, and falsity of a
negative body never consults its truth.  An atom under an even number
of negations is true exactly when the recursive term's instance is
true.  This is the only place positivity enters the semantics.

\paragraph{Substituting into a body.}
Two substitutions recur.  Writing a closed term $Q$ for the recursive
variable, $\Phi[Q/a]$, yields a body in the point alone; writing both
$Q$ for $a$ and a term $z$ for the point, $\Phi[Q/a][z/x]$, yields a
closed term when $Q$ and $z$ are closed.  Positivity is stable under
substituting numeral expressions for the other indices and under
lifting, but not under substituting arbitrary terms for them, which is
the same restriction seen from the other side.  The parametric grammar is stable under substituting any term for a
parameter, which is what lets the rule be used at open parameters.

\section{Ground Induction}
\label{sec:induction}

\iga is \rga with one more rule.  This section states the rule, works
one example, proves the rule sound and the system consistent, and then
proves the theorem that gives the rule its meaning: a positive fixed
point is an inductive definition, and its truth set is the least fixed
point of its operator.

\subsection{The rule}

Let $\Phi$ be a well-formed positive body, possibly with free
parameters beyond $a$ and $x$; $c$ a well-formed motive (a term with
the point as its free index); $t$ a term, the point at which the rule
is used; and $i$ a variable fresh for the hypotheses, for the motive,
and for the parameters of the body.  (The mechanization uses
de~Bruijn indices; ``fresh'' means that the index does not occur
free.)  Writing $c[u/x]$ for the motive with $u$ in place of the
point, the rule reads:
\[
  \frac{
    \Gamma \ent (\fix\,\Phi)\cdot t
    \qquad
    i\,\mathbf N,\ \Phi[\lambda c/a][i/x],\ \Gamma \ent c[i/x]
    \qquad
    \Gamma \ent t\,\mathbf N
  }{
    \Gamma \ent c[t/x]
  }
\]  In words: if $t$ lies in the fixed point, and the motive holds at a
fresh point whenever the motive holds throughout one unfolding of the
body at that point, then the motive holds at $t$.  The second premise is the closure condition,
stated as a hypothetical derivation with an eigenvariable, a fresh
variable standing for an arbitrary point: the body with the
motive-as-predicate $\lambda c$ in place of the recursive variable, as
a hypothesis.  Where the body applies $a$ to a term, the hypothesis
applies $\lambda c$ to it, one head step from the motive at that
term.

Three features of the rule matter for a reader who knows inductive
definitions from a proof assistant.  First, no implication appears inside any term.  The closure
condition is a sequent, a hypothetical derivation, and the hypothesis
it carries may have no value, which the logic tolerates in a
hypothesis and would not tolerate in a formula.  A premise may be a sequent; a motive must be
a term.  That asymmetry is the whole of \S\ref{sec:gamma0}, the price of the classical world.  Second,
the rule is the weak form of fixed-point induction, Park's form: the
closure hypothesis offers the motive at the unfolded points, not the
motive together with membership in the fixed point.  The strong form, with the fixed point conjoined to the motive, is derived from the weak one by replacing the motive under the positive body, at the cost of one more premise, that the context is well formed (Appendix~\ref{sec:table}); the development uses the weak form throughout.  Third, the rule is for a
unary predicate on numbers.  A definition indexed by a parameter may
carry it in the point, as a code, or leave it as a free variable of
the body; the only restriction is the grammar's, that a parameter
never sits inside the argument of a recursive atom.  The closed form of the rule, with $\lambda a.\,\lambda x.\,\Phi$
closed, is an instance of the present one, and the applications in
this paper have closed bodies, except the induction along an open
order of \S\ref{sec:wf}.

The side condition that $t$ is a number is not decorative.  Without
it the rule would be unsound: take the body $0=0$, whose fixed point
is true at every argument, a divergent term $\Omega$ as the point,
and the motive $x = x$; the first two premises are derivable and the
conclusion $\Omega = \Omega$ has no value.  (This is checked: any entailment relation that extends \iga's and is closed under the rule without its number premise derives $\Omega = \Omega$, and so is not sound for the omega model; Appendix~\ref{sec:table}.)  Decidedness of anything else is not assumed.

The derived form that the applications use is the induction schema: from $\Gamma \ent (\fix\,\Phi)\cdot x$ and $\Gamma
\ent x\,\mathbf N$, and a closure argument given uniformly over every
context extending $\Gamma$ and every point $z$ (from $\Delta \ent z\,
\mathbf N$ and $\Delta \ent \Phi[\lambda c/a][z/x]$ derive $\Delta
\ent c[z/x]$), conclude $\Gamma \ent c[x/x]$.

\subsection{A worked example: the subtraction game}

Return to the game of the introduction, with the simplest rules: from
$x$ a player may move to $y$ exactly when $y+1 = x$, and the opponent
replies by the same rule.  The winning-position body is
\[
  \neg\, \cA \cdot \lambda y.\, \neg\bigl( \mathit{move}(x,y) \wedge
     \cA \cdot \lambda z.\, \neg\,\mathit{reply}(y,z) \vee a \cdot z \bigr),
\]
the existential written as a negated universal, and it is positive.
Its fixed point $W$ unfolds, by conversion, to the definition.  Note
the asymmetry the grammar imposes: a move must be true, since it sits under conjunction, while a reply need only be not refuted, since it sits as the negated disjunct of a universal.  For
the decided relations here the two coincide.

Two things are provable, of different kinds.  That the odd positions are winning is a grounded claim: a number
induction with unfolding at each step, and it needs nothing from
\iga.  Ground induction enters with the first property of all winning
positions: every provably winning position is provably nonzero.  The
motive is $\neg (x = 0)$.  Its closure under one unfolding: at a point
$i$ with the unfolded body as hypothesis, the equation $i = 0$ is
decided, so split on it; in the case $i = 0$ the hypothesis, a
negated universal, is refuted by unfolding it at $0$, where no $y$
has $y + 1 = 0$; the other case is the motive.  The
rule then gives, from $W \cdot x$ and $x\,\mathbf N$, that $x \ne
0$.  \rga can unfold $W$ at any point, but it has no way to reach a
conclusion about $x$ from $W \cdot x$ for an open $x$.  Finally, that
\iga proves $W \cdot \bar n$ for no even $n$ is a classical claim:
it follows from soundness and the denotation of $W$ in the omega
model, so that \iga proves $W \cdot \bar n$ exactly when $n$ is odd.  The three statements are typical of the
paper: a grounded derivation, a grounded induction, and a classical
theorem about what the system proves.

\subsection{Soundness and consistency}

\iga's proof system restates \rga's forty-two rules verbatim and adds
the one above.  Every \rga theorem is an \iga theorem.

Soundness is proved for both semantics of \S\ref{sec:bg}.  Under the
certificate-grounded semantics, a universal's truth carries a proof of
the extended system, so the rule's case must supply a proof, not only
the truth, that the fixed point lies below the motive; \cref{thm:pg}
below does this by an induction on derivations.  The proof given
first, for \cref{thm:sound}, is for the omega model, in which a
universal is true when every instance is.  The omega model has no certificates, and \rga's soundness for it
is a rule induction whose forty-two cases replay unchanged.  The one
new case is the content of the rule.

\begin{lem}[\isn{gfix\_omv\_ind\_set}, \isn{gfix\_omv\_ind}]
\label{lem:gfix-ind}
Let $\Phi$ be positive and $\fix\,\Phi$ closed.  (Set form.)  Let $S$
be a set of numbers closed under the body: whenever $\Phi$ at the
point $\bar n$ is true with its recursive atoms read as membership in
$S$, $n \in S$.  Then every $n$ with $(\fix\,\Phi)\cdot\bar n$ true in
the omega model lies in $S$.  (Term form.)  Let $P$ be any term such
that, for every $n$, if $\Phi[P/a][\bar n/x]$ is true in the omega
model then so is $P \cdot \bar n$.  Then for every $n$, if $(\fix\,
\Phi)\cdot \bar n$ is true in the omega model, so is $P \cdot \bar n$.
\end{lem}

The set form is a single rule induction on the omega model's
evaluation relation, with a five-part motive, over the oracle reading
of the body (\S\ref{sec:positive}); the term form is the set form at
the truth set of $P$, with one further lemma that turns the oracle
reading at that set into the term reading.  Appendix~\ref{app:gfix}
gives the induction case by case.  For a body with parameters both
forms are proved again with no closedness hypothesis at all: the
closed proof used closedness at exactly one step, the environment of
the unfolding, and the grammar's restriction on parameters is what
that step needs instead.  The oracle reading is also what makes the
characterization theorem below possible.

\begin{thm}[\isn{omega\_sound}, \isn{gi\_consistent}]
\label{thm:sound}
Every \iga derivation $\Gamma \ent c$ from a well-formed context with
omega-true hypotheses has an omega-true conclusion.  In particular \iga does not prove $0=1$.
\end{thm}

\begin{thm}[\isn{gi\_gwosr\_open\_iff}, \isn{gi\_disj\_prop}]
\label{thm:pg}
Read the certificate-grounded semantics of \S\ref{sec:bg} with \iga's
proofs as the certificates: a universal is true when \iga proves its
open instance and every numeral instance is true.  The conclusion of
every \iga derivation from a well-formed context is true at every
well-formed substitution that makes the derivation's hypotheses true,
and a well-formed term is a theorem of \iga exactly when it is true in
this semantics.
Consequently \iga is consistent; if $\ent a \vee b$ then $\ent a$ or
$\ent b$; and if $\ent \neg(\cA\cdot u)$ then $\ent \neg(u\cdot\bar n)$
for some $n$.
\end{thm}

The proof runs \rga's soundness proof once, for every certificate
system whose rules contain \rga's, so that only the new rule's case is
new.  That case is proved once for every system whose steps contain
\iga's.  Like \cref{lem:gfix-ind}, the case is one rule induction over
the evaluation, and it differs at the two clauses of a universal.
There, a certificate for an instance of the fixed point must become
one for the motive, and that needs a proof, not only the truth, that
the fixed point lies below the motive.  The rule's own closure
premise, taken at a fresh variable, supplies the proof.  No omega
model of \iga is used: a build-time check confirms that the proof reaches none
of the omega-model soundness proofs of \iga.  A system that is sound
and complete for the certificate-grounded semantics with its own
proofs as the certificates is \emph{proof-grounded}.  \rga is
proof-grounded~\cite{rga-paper}, and by \cref{thm:pg} so is \iga.
Read from the grounded viewpoint, a sentence is true exactly when \iga
proves it.  Appendix~\ref{app:gfix} gives the rule's case.

\subsection{A positive fixed point is an inductive definition}

For a positive body $\Phi$ define the operator on sets of naturals
\[
  F_\Phi(S) \;=\; \{\, n \mid \text{the oracle reading of } \Phi
  \text{ at } n \text{ with recursive atoms read in } S \text{ is true} \,\}.
\]
As \S\ref{sec:positive} said, the operator is monotone for every
body; positivity is what makes the oracle reading agree with the term
reading.

\begin{thm}[\isn{gfix\_omv\_eq\_lfp}]
\label{thm:lfp}
For every positive body $\Phi$ with $\fix\,\Phi$ closed,
\[
  \{\, n \mid (\fix\,\Phi)\cdot\bar n \text{ is true in the omega model} \,\}
  \;=\; \lfp\,F_\Phi .
\]
\end{thm}

The left-to-right inclusion is the set form of \cref{lem:gfix-ind}
with $S$ the least fixed point.  The right-to-left inclusion is
least-fixed-point induction with the truth set, which needs the
transport in the other direction: from the oracle reading at the
truth set to the term reading, one head step after substituting the
fixed point for the oracle; that direction needs neither positivity
nor closedness, and holds for every body.  So a positive self-referential term in \iga denotes
the least fixed point of its operator, which is what an inductive
definition denotes, and ground induction is its induction principle.
This is a classical claim: it identifies a truth set of the omega
model with a least fixed point in HOL.

Two instances anchor the theorem.  The accessibility term of \S\ref{sec:wf} denotes HOL's accessible
part of the relation it is built from, and the winning-position term
of the example above denotes the inductively defined set of winning
positions, so that for games the term is complete as well as sound.

\paragraph{The class of truth sets.}
It is natural to ask which sets the positive fixed points define.
The answer has a sting.  The truth sets of positive fixed points are
the least fixed points of the operators above, and the same class is
the class of truth sets of arbitrary closed terms, because any closed term $P$ is
the fixed point of the body $P \cdot x$, which never mentions the
recursive variable.  Positivity adds nothing to the class of
definable sets.  What it adds is the induction principle.  The class
is closed under union, intersection, preimage along a definable
function, and bounded and unbounded existential and universal
quantification over a decided relation; it contains every recursively enumerable set and, through
accessibility, the accessible part of every decided relation, which
classically includes a $\Pi^1_1$-complete set.  Classically, Kripke's
least fixed point over arithmetic is
$\Pi^1_1$-complete~\cite{burgess86truth}, and the least fixed points
of positive arithmetic operators are exactly the $\Pi^1_1$ sets, the
sets definable by a universal quantifier over all sets of numbers in
front of an arithmetic formula, by the Kleene--Spector
theorem~\cite{kleene55forms,spector61inductively}.  Since the omega
model is itself the least fixed point of a positive arithmetical
operator, the class above is exactly the $\Pi^1_1$ sets by a classical
argument from the cited theorems; the development does not formalize
that argument, and the open questions of \S\ref{sec:concl} name what a mechanization would take.

\section{Truth as a Fixed Point}
\label{sec:truth}

Kripke's theory of truth~\cite{kripke75outline} defines the truth
predicate for a language containing it as the least fixed point of a
jump operator.  Feferman's Kripke--Feferman axioms~\cite{feferman91reflecting}
turn that construction into a classical theory, and Cantini's
minimality principle~\cite{cantini89notes} adds the statement that
the fixed point is the least one.  In \iga the construction needs no
separate axioms.  Truth for \rga's own language is one positive
fixed point, and ground induction is its minimality
principle.

\subsection{The omega model internalized}

What the fixed point internalizes is not a true/false predicate but
the omega model itself.  A point of the truth term is a triple, a tag,
the code of a term, and a result: with tag $0$ the point says ``the
coded term has this value'', with tag $1$ ``the coded term has this
head form''.  Truth and falsity are the value-$1$ and value-$0$
points, and negation swaps them.  The body is the omega model's twenty-one evaluation rules
transcribed onto codes as twenty-two clauses, the equality rule split
in two; Appendix~\ref{app:truth} lists them.  The body consults no decider and
restricts itself to no class of codes; the equality clause is the
omega model's own, split into an equal-values case and a decided
disequality case.  The body is positive by the grammar of
\S\ref{sec:positive} and here the grammar's cost is
paid in full: every recursive atom sits at a bound variable fed by a
decided equality guard under an existential, up to nine binders in a
clause.

\begin{thm}[\isn{tr\_omv\_iff}, \isn{tr\_omv\_iff\_hd}]
\label{thm:tr-exact}
For every term $t$ and result $r$, with no hypothesis on $t$, the
truth term at the point $\langle 0, \langle \lceil t \rceil,
r\rangle\rangle$, written $\Tr(\lceil t\rceil, r)$ below, is true in
the omega model exactly when $t$ has value $r$ in the omega model;
and at the tag-$1$ point exactly when $t$ has the given head form.
\end{thm}

This is a classical claim.  Soundness is the set form of \cref{lem:gfix-ind} with the set of
points matching the omega model, over twenty-two per-clause
decodings and the determinacy of values in the omega model;
completeness is an induction on the omega model's evaluation, with
the oracle-to-term transport (Appendix~\ref{app:truth}).  By \cref{thm:lfp} the truth term's
extension is the least fixed point of its operator, which is the
formal content of ``Kripke's fixed point''.  The truth term is
type-free: its points are codes of arbitrary terms, its own
applications included, and no hierarchy of languages appears.  At the false atom $0=1$ the truth term is not true.  Two classical
consequences for \iga: whatever \iga proves true of a code is true of
the term, and \iga does not prove $0=1$ true.

\subsection{Clause laws, inversions, and the refutation}

The twenty-two clauses are derived rules of \rga at open codes, and
of \iga under the same names: for instance, from the truth of $p$ at
value $1$, the truth of $p \vee q$ at value $1$, and from the truth
of every instance, for an eigenvariable, the truth of a universal in
head form.  Ground induction specialized to the truth term is its induction
principle, and it is the one place in this section where \iga adds
to \rga: the clause laws, the inversions and the refutation below are
\rga theorems, embedded, and determinacy of values, at the end of
the section, is the one result that spends it.

The inversions, the elimination direction of each clause, need
injectivity and distinctness of the code builders at open arguments,
which the development's internal theory of Cantor pairing supplies,
and an existential elimination with a fresh eigenvariable, a
variable standing for an arbitrary witness.  All twenty-two inversions are derived rules, packaged as a
twenty-two-way case split.  With them:

\begin{thm}[\isn{tr\_bot\_refuted}, \isn{gi\_tr\_bot\_refuted}]
\label{thm:tr-bot}
\rga, and hence \iga, proves $\neg\,\Tr(\lceil 0 = 1 \rceil, 1)$: the
truth term is refuted at the false atom.
\end{thm}

This is a grounded claim, and the one the section is after: the
system itself says that $0=1$ is not true, and ground induction is
not needed for it.

The proof is not by the case split.  \rga and \iga have no negation
introduction from a hypothetical derivation, so the refutation
unfolds the fixed point negatively, through the negative-unfolding
rule of the truth term, and refutes every clause instance at the
false atom; only the equality clause can match, and its two value
premises fall to two small facts, that the code of $0$ has no nonzero
value and that the code of $1$ never has value $0$.  Determinacy of values inside \iga, that the truth term gives a code
at most one value, is the one result of the section that needs ground
induction; \cref{thm:tr-bot} does not depend on it.

\begin{thm}[\isn{gi\_tr\_det}]
\label{thm:tr-det}
\iga proves that the truth term's values are determinate: at numbers
$c$, $r$ and $r'$, from $\Tr(c, r)$ and $\Tr(c, r')$ it derives
$r' = r$.
\end{thm}

The proof is ground induction at a value-carrying motive: at a point
with tag $g$, code $c$ and result $r$, the motive says that every
result of the same tag and code is $r$.  The closure premise of the
rule is one obligation per clause of the truth body, twenty-two in
all, and each is an \rga derivation.  The clause gives the motive,
not the truth term, at its sub-points, and the motive at the point
follows once every variant point, with the same tag and code and
another result, is refuted by the negative unfolding of the truth
term.  Of the twenty-two clauses that the negative unfolding offers
at a variant point, those with a different head fall to distinctness
of the code builders, which the inversions already needed.  A clause
with the same head is refuted from the induction hypothesis at the
sub-points: determinacy of the function code's head form separates
the application and recursor clauses, determinacy of the scrutinee's
value separates the two recursor branches, and the hypothesis at the
reduct fixes the result.  The universal clause is the one that needs
the whole instance family: when the universal is true, the hypothesis
at every instance point refutes a false instance at its witness, and
when it is false, the hypothesis at the false instance supplies the
witness that refutes the true clause.

\subsection{What the truth term does and does not give}

The truth term gives \iga the ability to reason about the semantics of
its own base language, uniformly, by induction over grounding.  It
does not give \iga a proof of \rga's consistency by the obvious route,
and the no-go result of \S\ref{sec:classical} says why.  The classical comparison is
exact in its shape: the Kripke--Feferman axioms plus Cantini's
minimality are the classical twin of ``the omega model as a positive
fixed point plus ground induction''.  The strength comparison is
not, and \S\ref{sec:strength}, on where the lower bound sits, sets it out: the paracomplete sibling
of Kripke--Feferman is known to be weaker than Kripke--Feferman
itself, and the cost has been traced to induction being unusable on
indeterminate formulas~\cite{halbach06axiomatizing,halbach18costs}.
The same cost is visible here in the form of \S\ref{sec:gamma0}: the
induction principle is available, and what is unavailable is the
implication at an undecided antecedent that the strongest uses of it
require.  The Liar returns in \S\ref{sec:gamma0}: \rga's, against bivalence
and with Curry's paradox, and the truth term's, at $\mu$-bivalence and
the closed-off model.

\section{Well-Foundedness as a Sentence}
\label{sec:wf}

The second application is the one that makes ordinals available.
Classically, ``the relation $R$ is well-founded'' is a $\Pi^1_1$
statement, one that quantifies over all sets of numbers (every
non-empty set has an $R$-minimal element): no first-order sentence of
arithmetic says it.  In \rga it is one closed sentence, because the accessible part of a
relation is a positive recursive definition and self-reference is
native.

The ordinals that appear below are the standard landmarks, named here
once.  $\omega$ is the order type of the naturals; $\omega^\omega$ is
reached by the finite towers $\omega, \omega^2, \omega^3, \ldots$ of
lexicographic products; $\eps$ is the limit of $\omega, \omega^\omega,
\omega^{\omega^\omega}, \ldots$, and every ordinal below it has a
unique Cantor normal form $\omega^{a_1} + \cdots + \omega^{a_k}$ with
descending exponents, themselves normal forms, which is what makes
$\eps$ a notation system; the Veblen functions $\varphi(a, b)$ extend
it, $\varphi(0, b) = \omega^b$ and $\varphi(a+1, \cdot)$ enumerating
the fixed points of $\varphi(a, \cdot)$, up to $\Gamma_0$, the first
$\gamma$ with $\varphi(\gamma, 0) = \gamma$; and Kleene's $\mathcal
O$ is the set of codes of computable ordinals, in which a code is zero,
a successor of a code, or the limit of a computable sequence of codes.

\subsection{Accessibility and the certificate}

Let $R$ be a closed term computing a binary relation on numbers.  The
accessibility term is the fixed point of the body
\[
  \Acc_R \;\equiv\; \fix\,\bigl(\lambda a.\,\lambda x.\,
     \cA \cdot \lambda y.\, \neg R(y,x) \vee a \cdot y\bigr),
\]
so that $\Acc_R \cdot x$ unfolds to ``every $R$-predecessor of $x$ is
accessible''.  Well-foundedness of $R$ on a domain $D$ is the sentence
\[
  \WF(D,R) \;\equiv\; \cA \cdot \lambda x.\, \neg D(x) \vee \Acc_R\cdot x ,
\]
and the full certificate that $(D,R)$ behaves as an ordinal,
$\Ord(D,R)$, conjoins decidedness of $D$ and of $R$ on $D$,
irreflexivity, transitivity, trichotomy on $D$, and $\WF(D,R)$.
These are terms of \rga; nothing from \iga is needed to write them.
Throughout this section $R$ and $D$ are closed and well-formed, which
every theorem assumes.

Three facts about the accessibility sentence are classical claims,
theorems of HOL about the omega model rather than derivations in the
system, and they set the terms of everything later.

\begin{thm}[\isn{acc\_omv\_iff}, \isn{acc\_omv\_never\_false}]
\label{thm:acc-sem}
For every $n$, $\Acc_R\cdot\bar n$ is true in the omega model exactly
when $n$ lies in HOL's accessible part of the relation $\{(y,x) \mid
R(y,x) \text{ is not false}\}$; and $\Acc_R \cdot t$ is never false in
the omega model, for any term $t$.
\end{thm}

The relation that appears is the not-refuted one, not the true one.
Where $R(y,x)$ has no value, the disjunct $\neg R(y,x)$ has none
either, and the universal can be true only if $\Acc_R \cdot y$ is; so
an undecided pair counts as a predecessor.  For a decided $R$ the two
relations coincide.  Never-false is the fact with consequences.
Falsity of accessibility at $x$ would need a refuted predecessor, and
there is no base case for falsity.  Hence $\WF(D,R)$ is never
refutable, in \rga and so in \iga: asserting it is unfalsifiable, and sound to adopt exactly when it is
true.  This is the property that will make
bivalence of accessibility a commitment rather than an axiom in \S\ref{sec:gamma0}, on the price of the classical world.

\begin{thm}[\isn{gi\_wf\_sound}, \isn{gi\_wf\_sound\_true}]
\label{thm:wf-sound}
If \iga proves $\WF(D,R)$, then the not-refuted relation of $R$,
restricted to the points where $D$ is true, is well-founded in HOL,
and so is the true relation.
\end{thm}

The converse fails, as it must: a true well-foundedness sentence need
not be provable.  The gap between the two is the subject of \S\ref{sec:gamma0}, the price of the classical world.

\subsection{What \rga alone can prove}

Proving $\WF(D,R)$ needs induction along $R$, and \rga has only
induction on the natural numbers.  The order $<$ on the naturals goes
through: the motive ``every $y \le n$ is accessible'' is proved by
number induction, unfolding accessibility at $n+1$ and splitting the
decided guard $y \le n$ . The full certificate follows, with irreflexivity, transitivity, and trichotomy each by
number induction, and the order term denotes HOL's $<$.  As a sanity
check on the soundness bridge, HOL's well-foundedness of $<$ follows
from the \rga theorem.

Lexicographic products go through by nesting, and so does every finite tower $\omega^{k+1}$, the $k$-fold
lexicographic product of the naturals.  The mechanism is a
HOL-level schema, relative transfinite induction: ``\rga proves
induction along $R$ for every motive.''  It is closed under
lexicographic products, and it gives well-foundedness for each tower.
It stops at $\omega^\omega$, and the obstruction is precise: the
induction step at the limit must use the hypothesis ``the tower at
height $h$ is well-founded'', for an open $h$, as an induction
principle, and \rga can only unfold accessibility one step.  The
schema holds for each closed height; what is missing is the ability
to use it inside the system.  That ability is ground induction.

\subsection{What \iga proves}

With the rule, relative transfinite induction becomes a theorem of the
system rather than a schema about it: from a proved $\WF(D,R)$,
induction along $R$ on $D$ for any motive.  Since the rule takes
parameters, the order may itself be a variable of the context: from
the hypothesis that an open order term is well-founded, induction
along it for any motive.  This is what lets a statement about all
well-orders be one hypothetical derivation at eigenvariables rather
than a schema; in the union of \iga with \oga, for instance, the
comparability of two well-orderings is one such sentence.  The orders
below are built over pairs of numbers, and the theorems are
parametric in two closed, well-formed, provably total terms $\mathit{hi}$
and $\mathit{lo}$ projecting a pair to its components; the development
instantiates them with the compiled Cantor projections.  Then:

\begin{thm}[\isn{gi\_WF\_omega\_omega}, \isn{gi\_WF\_cnf}]
\label{thm:iga-wf}
\iga proves $\WF$ for the order on pairs $\langle \mathrm{level},
\mathrm{code}\rangle$ that stopped \rga: a pair is below another when
its level is smaller, or when the levels are equal and its code is
below in the tower order of that level, the family of tower orders
itself one self-referential term.  And it proves $\WF$ for a closed
self-referential order on Cantor-normal-form codes.  Between the two
it proves the transfer recipe: from the hypothesis $\WF(R)$ on a
domain, the sentence $\WF(\omega^R)$ for the order on finite
descending sequences over $R$, as a hypothetical derivation.
\end{thm}

The first order is the union of the towers, of type $\omega^\omega$
when the projections are a bijection; that type is not
machine-checked, and the theorem is stated on the order as a term.  The second order is the one whose type is established in \S\ref{sec:cnf} below, where the notation is certified.  The transfer recipe deserves a
word.  Classically ``$R$ well-founded implies $\omega^R$
well-founded'' is an implication between two $\Pi^1_1$ statements.
Here it is a derivation from a hypothesis, which is the only form
available in a paracomplete logic, and it is also the natural one: a
recipe that turns one certificate into another.  The shape recurs wherever a motive would need implication at an
antecedent that is not provably decided: the relative object replaces
the implication, and the $\mathcal O$ term of \S\ref{sec:one-rule}
uses the same device.  The
Cantor-normal-form result is a nest of four inductions: a strong
induction on rank outermost, then ground induction on the
accessibility of the head exponent, a second strong induction on
rank, and a ground induction on the accessibility of the tail for the
jump (Appendix~\ref{app:cnf}).  Every case split is on a decided matrix, and both
the jump and the limit are eliminations in sequent form.  The cost one might expect a paracomplete logic to charge at the
limit, where \rga stopped for lack of an induction principle, is not
charged here.

\subsection{The notation certified, and HOL's $\eps$ as a corollary}
\label{sec:cnf}

The order of \cref{thm:iga-wf} is a term over codes.  To say what it is the order of, the development defines the standard notation
system in HOL, the datatype of Cantor normal forms $\omega^a + b$ with
its validity predicate and strict order, a decoding of codes into it,
and the canonical code of each normal form.  The internal order is sound on all codes and exact on canonical
codes, and every valid normal form has a canonical code; a code
carries its rank explicitly, and the rank is the measure every
induction of this section runs on.  The internal order does not identify two codes of one normal form: they are
incomparable, so trichotomy fails on codes, and the order on all
codes is a well-founded strict partial order whose restriction to
canonical codes is $\eps$.  Every order-type claim in this paper is
of this kind: on canonical codes.  With that stated, the corollary
that the paper is after:

\begin{thm}[\isn{epsilon0\_wf}]
\label{thm:eps-wf}
The strict order of the standard $\eps$ notation system, on valid
normal forms, is well-founded in HOL.
\end{thm}

This is a classical claim, and the route to it is the point.  The
theorem is not proved by a direct HOL argument about normal forms.
It is pulled back from the internal theorem through
\cref{thm:wf-sound} and the canonical coding.  The bridge is
omega-soundness (\cref{thm:sound}), which is itself a HOL argument of
strength well beyond $\eps$, so the corollary is not a reduction in
proof-theoretic strength.  What it shows is that the grounded system
proves its ordinal sentence outright, and that the sentence means
what it should.

The same program runs for the Veblen notation below $\Gamma_0$.
The Veblen functions $\varphi(a,b)$ enumerate, for each $a$, the
common fixed points of all $\varphi(a',\cdot)$ with $a' < a$, starting
from $\varphi(0,b) = \omega^b$; so $\varphi(1,0) = \eps$, and
$\varphi(2,0)$ is the first fixed point of $\alpha \mapsto
\varepsilon_\alpha$.  Every ordinal below $\Gamma_0$, the first
ordinal $\gamma$ with $\varphi(\gamma,0) = \gamma$, is a finite sum of
principal values $\varphi(a,b)$ in a normal form due to Sch\"utte.
The development codes this as a closed order with two sorts of code,
principal values and sums.

\begin{thm}[\isn{gi\_acc\_phi\_fin}, \isn{vOrd\_canon\_ordinal}]
\label{thm:veblen}
For every finite $n$, \iga proves $\varphi(n,0)$ accessible, and on
canonical codes below $\varphi(n,0)$ the notation satisfies the full
certificate $\Ord$: decidedness, irreflexivity, transitivity,
trichotomy, and well-foundedness.  HOL's well-foundedness of the
segment below $\varphi(n,0)$ follows.
\end{thm}

Trichotomy on arbitrary codes fails for the Veblen notation as it did
for $\eps$; the certificate is stated on the canonical codes, which a
closed recursive test recognizes.  Transitivity is three nested rank
inductions, and irreflexivity follows from it.  So the lower bound
on what \iga certifies as an ordinal is $\varphi(\omega,0)$, the
supremum of the $\varphi(n,0)$.  An upper bound, read off the soundness proof, is given in \S\ref{sec:strength}, on where \iga's strength sits.  What happens at the next step, the price of the classical world, is the subject of \S\ref{sec:gamma0}.

\section{Goodstein's Theorem and the Hydra Game}
\label{sec:goodstein}

Goodstein's theorem~\cite{goodstein44restricted} and the Kirby--Paris
hydra game~\cite{kirby82accessible} are the standard examples of
natural arithmetic statements that are true but unprovable in Peano
Arithmetic.  Both are proved by assigning ordinals below $\eps$ and
showing each step decreases the ordinal, and both are therefore
within reach of any system that proves $\eps$ well-founded: Peano
Arithmetic with transfinite induction to $\eps$ added proves them, and
so does its intuitionistic counterpart, since the theorems are
$\Pi^0_2$.  What this section shows is that \iga proves Goodstein's
theorem outright, from one native rule rather than an added schema,
in a logic that has not even intuitionistic implication at an
undecided hypothesis, with every step of the arithmetic a grounded
derivation; and the hydra game for every primitive recursive strategy, by the
same method.

\subsection{Goodstein's theorem}

The Goodstein sequence from $m$ writes $m$ in hereditary base $2$,
replaces every $2$ by $3$, subtracts one, writes the result in
hereditary base $3$, and so on.  The theorem is that every such
sequence reaches zero.  The proof reads each term as an ordinal
below $\eps$ by replacing the base by $\omega$; the base change leaves
the ordinal fixed and the subtraction lowers it.

The formalization takes a state to be a pair of the step index and the
current value, and builds, as explicit primitive recursive index
trees, the hereditary base representation, the base-bump-and-decrement
step, the run, and the ordinal map to canonical Cantor-normal-form codes, with a small
course-of-values combinator for the tree recursion, which stores the
whole history of smaller values and reads the two it needs, the
leading exponent and the tail.
Termination of the run from a state is the sentence ``some step
reaches zero'', a $\Sigma^0_1$ sentence over a decided matrix.
Goodstein's theorem is the universal closure over starting values.

The proof in \iga is ground induction along the $\eps$ order of
\cref{thm:iga-wf}, with a motive indexed by the run: ``for every starting state $p$ and
step $i$, if the state reached after $i$ steps from $p$ has ordinal
$\alpha$, the run from $p$ terminates''; the antecedent is a decided
equation, so the motive is the disjunction $\neg A \vee B$ with $A$
decided, which the logic allows.  The closure of the motive needs one arithmetic fact, the \emph{decrease
lemma}: at any state with a nonzero value, the step strictly lowers
the ordinal.  With the universal closure of the decrease lemma as a hypothesis
the theorem is a hypothetical derivation; and the lemma is itself
derived:

\begin{thm}[\isn{gDec\_derivable}, \isn{iga\_goodstein}]
\label{thm:goodstein}
\rga proves the decrease lemma schematically, at an open state.
Hence \iga proves Goodstein's theorem: the closed sentence ``every
Goodstein sequence reaches zero'' is derivable from no hypotheses.
\end{thm}

Both halves are grounded claims.

The schematic decrease lemma is where the work is.  Its proof is an
internal theory of hereditary base representations over the compiled
trees: laws for powers and for the leading exponent, the
course-of-values tree unfolded at open arguments, a law for the
digits of a representation, and on them four strong inductions with
a base parameter, that the base bump is strictly increasing in the
value, that the base change preserves the ordinal, and that the
representation is valid and monotone, together with an upper bound
on the bump that needs no induction; the decrease lemma also asserts
that both ordinal codes are valid.  Appendix~\ref{app:goodstein}
gives the proof step by step.  Every step of the classical argument
has an internal counterpart.  The HOL statement, that every Goodstein
sequence reaches zero, follows from the \iga theorem through
soundness, a classical claim.

\subsection{The hydra game}

A hydra is a finite rooted tree.  Hercules chops off a head, a leaf;
if the head's parent is not the root, the hydra regrows copies of the
subtree at the head's parent, as many as the current stage number,
attached at the grandparent.  The theorem is that every battle ends
with the bare root, whatever heads Hercules chooses.  The ordinal of a
hydra is its reading as a Cantor normal form, $\omega$ to the power
of each child's ordinal, summed in descending order, and every chop
lowers it.

The grounded statement differs from the classical one in one place,
and the difference is deliberate.  Classically ``whatever heads
Hercules chooses'' quantifies over all strategies, functions from
positions to heads.  A grounded system quantifies over numbers, so
the strategies it speaks of are given by codes: here, any index tree
of a primitive recursive function choosing a head from the current
hydra and the stage, out-of-range choices reading as the last head.
``Hercules wins against every primitive recursive strategy'' is the
recipe form of the theorem, and the form a computable logic states;
extending it to all computable strategies would need a universal
over partial computations with a totality guard, which the present
coding does not attempt.  The classical form, for every function, is
stated in HOL alongside it.

The formalization codes a hydra as a Cantor-normal-form term read as
a tree, with a bijective code, and builds the chop-and-regrow step,
the ordinal map by insertion into the descending list of exponents
(the natural sum), and the run as explicit index trees over the same
course-of-values combinator.  The Goodstein argument is restated as a
generic termination engine over any step tree and ordinal tree, and
the hydra theorem is its instance.

\begin{thm}[\isn{hyd\_hDec\_derivable}, \isn{iga\_hydra}]
\label{thm:hydra}
For every primitive recursive strategy $s$, \rga proves the decrease
lemma for $s$ schematically, at an open state: at every state with a
nonzero hydra, the chop chosen by $s$ lowers the ordinal.  Hence \iga
proves outright that every battle under $s$ terminates.
\end{thm}

Both halves are grounded claims, as for Goodstein.  The decrease
lemma's derivation is the same kind of work as Goodstein's: an
internal theory of the hydra's trees over the compiled projections,
in seven layers.  Projection bounds at open arguments for the pairing
trees come from a bridge to the development's list coder; the six
unfolding laws of the hydra trees follow as case laws; the internal
order theory of the comparison tree against the Cantor-normal-form
order on canonical codes reuses the ordinal kit of the $\mathcal O$
development; the insertion lemmas and the validity of the ordinal map
are nine strong inductions; and the main induction mirrors the HOL
decrease lemma with the position universally quantified.  The three
HOL readings, through soundness, directly from $\eps$, and for every
strategy as a function, are theorems as well.

\section{The Consistency of Peano Arithmetic}
\label{sec:conpa}

The last application is the consistency of Peano Arithmetic.  The
goal is the consistency sentence itself, not an ordinal analysis: Cantor
normal forms written in the Veblen codes of \S\ref{sec:wf}, and the
Veblen order, serve as tools, and Gentzen's
route is carried out in full in \iga.

\begin{thm}[\isn{gi\_ConPA}, \isn{pa\_ChkPAr\_complete}, \isn{gi\_ConPA\_relative}]
\label{thm:pa}
Let $\mathrm{Chk}$ be the guarded checker for Peano Arithmetic of Appendix~\ref{app:stage}, step 61.
\begin{enumerate}
\item \iga derives, from no hypotheses, the consistency sentence $\forall p.\,\neg\mathrm{Chk}(p, \quo{0=1})$, with $\quo{0=1}$ the numeral of the code of $0=1$.
\item If Peano Arithmetic derives $\varphi$ from no hypotheses, then for some number $L$, \rga derives $\mathrm{Chk}(L, \quo{\varphi})$.
\item If Peano Arithmetic derives $0=1$, then \iga derives $0=1$.
\end{enumerate}
\end{thm}

\noindent
Part~1 is a grounded claim, a derivation in \iga.
Parts~2 and~3 are classical claims, HOL theorems about the derivability relations; part~2 makes the sentence of part~1 a consistency sentence for Peano Arithmetic itself.
The derivation of part~1 follows Gentzen's route through Buchholz's notations for derivations with the omega rule~\cite{buchholz91notation,buchholz97explaining}.
A proof of $0=1$ that $\mathrm{Chk}$ accepts embeds, with Buchholz's cut-elimination operator applied until its cut degree is $0$, as a notation for an infinitary derivation of the false equation, and Buchholz's reduction steps then make its ordinal descend.
The notations are coded in \rga, and their ordinals are Cantor normal forms below $\eps = \varphi(1,0)$, written in the Veblen codes of \S\ref{sec:wf} (\cref{thm:veblen}) and compared inside the Veblen order.
That every reduction step descends is proved inside \rga: the reduction law is one course-of-values induction for the local correctness of every notation, and the iteration of the reduction, the totality of the ordinal assignment, its comparison and its bound below $\varphi(1,0)$ are \rga theorems as well.
\iga supplies only transfinite induction along accessible notations, at a bound far below its own ordinal $\varphi(\omega,0)$ (\cref{thm:ordinal}).
The conditional of step 60 derived the consistency sentence of a liberal checker from a reduction, its iteration, an ordinal assignment with a comparison, and a bound taken as premises, with the safety of Gentzen's terminal fragment proved in \rga.
\cref{thm:pa} discharges those premises at the guarded checker and does not use the terminal-fragment safety: the reduction removes every cut, so the argument for the terminal form sits inside the reduction law: an accepted notation of cut degree $0$ ends in a repetition, which the reduction removes.
Part~3 follows from parts~1 and~2: the consistency sentence, instantiated at the code of a proof of $0=1$ in Peano Arithmetic, meets the acceptance that part~2 derives, and explosion gives $0=1$ in \iga.
So the consistency of \iga implies the consistency of Peano Arithmetic, with no appeal to a soundness theorem of \iga; the link concerns consistency strength and does not depend on \cref{thm:pg}.
The theorem does not state that Peano Arithmetic does not derive $0=1$, which would take \iga's soundness, a different theorem.

\section{The Price of the Classical World}
\label{sec:gamma0}

The applications stop at $\varphi(\omega,0)$.  This section says what
the next step would cost, and then what the classical world has that
\iga does not: in every case an implication whose antecedent is not provably
decided, bought in one of three forms, bivalence of accessibility,
conditional proof, and the motive of an internal soundness proof.
Each theorem below is a classical claim, a HOL theorem about what
\iga derives, unless it says otherwise.

\subsection{Well-foundedness of the Veblen notation, exactly}

The Veblen notation below $\Gamma_0$ is where predicativity ends.  A
predicative definition builds a new object only from objects already
accepted, and a predicative proof of well-foundedness may use
induction only along orders already proved well-founded; the
autonomous progressions of Feferman and Sch\"utte, which iterate that
step as far as it will go, stop exactly at
$\Gamma_0$~\cite{feferman64systems,schutte65predikative,kreisel58ordinal}.
The classical reason is well known: to prove the whole notation
well-founded one must prove, for each first argument $a$, that
$\varphi(a,\cdot)$ carries accessible arguments to accessible values,
and the induction on $a$ that does this has a motive with
accessibility in the antecedent of an implication.  The theories of
inductive definitions allow that motive, and they allow it with an
intuitionistic implication: intuitionistic $\mathrm{ID}_1$, even
restricted to accessibility predicates, reaches the Bachmann--Howard
ordinal, far beyond $\Gamma_0$~\cite{buchholz81iterated}.  So the
price of $\Gamma_0$ is not classical logic.  It is an implication
whose antecedent may be undecided, which \iga has for no hypothesis
that is not provably decided (\S\ref{sec:bg}), and accessibility is
not decided.  What is new here is that the purchase can be stated as an
equivalence.  The equivalence locates the grounded boundary, bivalence
of an inductive predicate, and not the boundary of predicativity: the
same one instance, at the accessibility predicate of the notation,
buys each notation its well-foundedness, Buchholz's notation for the
Bachmann--Howard ordinal included (\S\ref{sec:one-rule}), and it is
sound exactly when the notation is well-founded.  Higher up, one
membership claim buys the well-foundedness of the countable part of
Buchholz's notation with collapsing functions at every finite level,
whose order type is $\psi_0(\Omega_\omega)$, and $k$ instances of
bivalence, at a
stratified family of orders, buy the countable part of its levels up
to $k$, each purchase with its classical warrant (\S\ref{sec:one-rule}
and \S\ref{sec:pricing}).  $\Gamma_0$ is the
worked case because it is where the unconditional development stops,
every $\varphi(n,0)$ proved and the whole notation open
(\S\ref{sec:strength}).  And the content of the equivalence lies in
the descent from bivalence, which is the classical proof with excluded
middle used at one predicate only, and in the exactness of the
soundness theorem; the step from the rule to the sentence is the exact
deduction theorem.

Call \emph{bivalence of accessibility} for an order $R$ the following
rule, for every well-formed context $\Delta$, every $d$ with $\Delta
\ent d\,\mathbf N$, and every well-formed $X$:
\[
  \frac{\Acc_R\cdot d,\ \Delta \ent X}{\Delta \ent \neg(\Acc_R \cdot d) \vee X} .
\]
It is the deduction theorem at the hypothesis $\Acc_R \cdot d$.  Say that \iga \emph{admits} the rule for $R$ when every instance of it
holds of \iga's own derivability: whenever the premise is derivable
in \iga, so is the conclusion.  Admissibility is a statement about
the system, not a rule added to it.

\begin{thm}[\isn{gi\_WF\_gamma0\_iff\_bivalence}]
\label{thm:gamma0}
\iga proves $\WF$ for the Veblen order on all codes if and only if
\iga admits bivalence of accessibility for that order.
\end{thm}

\iga does not prove the Veblen notation well-founded (\cref{thm:ordinal} in \S\ref{sec:strength}), so by the theorem it does not admit the rule either; what the theorem fixes is that the two are one question, and \S\ref{sec:strength} answers it.  The rule itself has a
classical warrant.  The Veblen notation is well-founded in HOL, by an accessible-part
induction on the notation's clauses, so the not-refuted relation of
the Veblen order is well-founded, and by \cref{thm:biv-wf} below every
instance of bivalence of accessibility for that order is sound for
the omega model.  The paper calls this HOL fact the order's
\emph{classical warrant}.  The commitment is sound to take, and \iga has not taken it.  The
left-to-right direction is short: from $\WF$ for the whole order,
every accessibility claim $\Acc \cdot d$ at a number is derivable
outright, and the disjunction follows by weakening and cut.  The right-to-left direction carries the content.  It is the
inductive-definitions proof of well-foundedness below $\Gamma_0$
carried out in \iga with the admitted rule supplying, at each first
argument, the implication from accessibility of the argument to
accessibility of the value (Appendix~\ref{app:gamma0}).  Without the
rule, the induction hypothesis for a smaller first argument has to be
applied at a second argument whose accessibility is established only
inside the derivation, and no term expresses that conditional
accessibility: a premise may be a sequent, but a motive must be a
term (\S\ref{sec:induction}).  By the exactness of the deduction
theorem (\S\ref{sec:bg}), admitting the rule is the same as deriving,
at every number in every context, the instance $\Acc\cdot d \vee
\neg(\Acc\cdot d)$ of excluded middle, so the theorem can also be read
at the level of sentences: \iga proves the Veblen notation
well-founded exactly when it proves excluded middle for its
accessibility predicate.

\subsection{What bivalence costs}

The equivalence would be idle if the bivalence rule were a harmless
addition.  It is not, and its cost can be stated.

\begin{thm}[\isn{accbiv\_omega\_sound\_iff\_wf}]
\label{thm:biv-wf}
For a closed, well-formed order $R$, every instance of the bivalence
rule for $\Acc_R$ is sound for the omega model if and only if the
not-refuted relation of $R$ is well-founded in HOL.
\end{thm}

For the forward direction, argue by contraposition with the
smallest possible instance.  Take
an inaccessible point $n$, the empty context, and $X$ the
accessibility claim itself.  The premise is then a hypothesis, and the
conclusion is excluded middle for accessibility at $n$.
\cref{thm:acc-sem} refutes it: accessibility is never false, and at
$n$ it is not true.  The backward direction is immediate: when every
point is accessible, every accessibility claim is true and the right
disjunct carries the conclusion.

Read together, the two theorems say: \iga proves the Veblen notation
well-founded exactly when it admits bivalence of accessibility for
it, and taking bivalence of accessibility as a rule is sound exactly
when the order is well-founded.  \iga never refutes the commitment (\cref{thm:acc-sem}); it is sound
exactly when it is true; and the same shape is inconsistent at the
Liar (\cref{thm:biv-liar} below).  The classical
logician has the commitment built into the logic, as excluded middle
for every sentence.  The grounded logician can take it, order by
order, with the price written on it.

\subsection{One rule, or one claim}
\label{sec:one-rule}

Two further forms of the purchase turn the admissibility statement
into theorems about systems.

The first adds the rule.  Let $\iga+\mathrm{AccBiv}_\varphi$ be \iga
with bivalence of accessibility for the Veblen order as a primitive
rule.

\begin{thm}[\isn{giv\_WF\_gamma0\_price}]
\label{thm:price}
$\iga+\mathrm{AccBiv}_\varphi$ proves the Veblen notation
well-founded; it is sound for the omega model; and it is consistent.
\end{thm}

So $\Gamma_0$ lies exactly one rule away, and the rule is sound to add by the classical warrant above.  The system remains proof-grounded as well (\S\ref{sec:pricing}).  The theorem is proved in an abstract
calculus that takes the rule as an assumption and carries the
ground-induction development of \S\ref{sec:wf} to any system in which
the rule is present (Appendix~\ref{app:gamma0}).  The accessibility
ordinals of the combined systems that add \oga's decidedness rule and
reflection to \iga are the subject of \S\ref{sec:pricing}.

The second replaces the rule by a claim.  The development defines a
form of Kleene's $\mathcal O$, the constructive ordinal notations, as
one positive fixed point of \iga on tagged codes (zero, successor,
and limit of a primitive recursive sequence, where Kleene's own
notations allow any computable one), so that induction along
$\mathcal O$ is an instance of ground induction; that development is
part of the artifact (the Kleene-$\mathcal O$ theories of session
\isn{IGA}, Appendix~\ref{app:gamma0}), and only one of its theorems
is needed here.  Let $\hat\gamma$ be the code in $\mathcal O$ of the
limit of the chain
$g_0 = 0$, $g_{n+1} = g_n + \omega^{\gamma_n}$, where $\gamma_n$ runs
through the iterates $0,\ 1,\ \eps,\ \varphi(\eps,0),\ \ldots$ of
$\varphi(\cdot,0)$; the supremum of the $\varphi(n,0)$ is only
$\varphi(\omega,0)$, so the iterates, not the $\varphi(n,0)$, are the
right sequence.

\begin{thm}[\isn{o\_wf\_of\_gamma0}, \isn{gamma0\_triangle}]
\label{thm:triangle}
From the hypothesis that $\hat\gamma$ is a notation, \iga derives $\WF$
for the Veblen order.  Consequently \iga proves the Veblen order
well-founded if and only if it admits bivalence of accessibility for
that order, if and only if it proves $\hat\gamma$ a notation.
\end{thm}

The first sentence is a grounded claim, a hypothetical derivation.
Its inductions are induction along $\mathcal O$ at the claimed code
and ground induction on the numbers, and nothing else; a build-time
check over the transitive dependencies of the theorem, by names and
statement shapes, gives evidence that no induction along the Veblen
order, and no well-foundedness, transfinite induction or bivalence
fact about it, is used, and a reading of the proofs confirms it
(Appendix~\ref{app:gamma0}).  The other direction of the second
equivalence derives the claim from $\WF$ by transfinite induction
along the order.  The second sentence is classical.  As extensions of
\iga, adopting the sentence $\WF$ for the Veblen order and adopting
the claim are interderivable, and the adoption of the claim is sound
for the omega model and consistent, the classical warrant entering at
exactly one place, the omega-truth of the well-foundedness sentence.
So the commitment can be made in three forms, a sentence, a rule, or
a claim, and the three are interderivable, with the same classical
warrant behind each.

The same three-way equivalence holds one landmark higher, for
Buchholz's notation for the Bachmann--Howard ordinal, with its own
classical warrant, and \iga with the rule at that order is
proof-grounded as well (\cref{thm:rung,thm:bh} in
\S\ref{sec:pricing}).  The triangle is why the opening of this section
separates the boundary found here from the boundary of predicativity.
The combined systems with \oga's decidedness rule and with reflection
are priced in \S\ref{sec:pricing}, and the stage analysis of this
paper stays below $\Gamma_0$.

The claim form climbs further, to Buchholz's notation system with
collapsing functions at every finite level~\cite{buchholz86new}.  Its
codes are $0$ and sums $D_u a + t$ of a \emph{principal} code $D_u a$,
with a natural number $u$ as its \emph{level} and a code $a$ as its
argument, followed by a tail $t$.  The order compares levels first,
then arguments, then tails, and the \emph{valid} codes are those that
satisfy Buchholz's side conditions: the summands of a sum do not
increase, and the argument $a$ of each $D_u a$ lies above the argument
$b$ of every $D_v b$ with $v \geq u$ that $a$ contains outside a
principal code of level below $u$.  The code $\Omega_{u+1} = D_{u+1}0$ stands for the
$(u+1)$-th uncountable cardinal, and the \emph{countable} codes, $0$
and the codes whose first summand has level $0$, are the codes below
$\Omega_1$.  $\mathrm{OT}(k)$ is the system of valid codes whose levels
are all at most $k$.  The system of all valid codes, written
$\mathrm{OT}(\omega)$ here, is Buchholz's system without the symbol
$D_\omega$, its levels being the natural numbers.  $\mathrm{OT}(1)$ is
a second coding of Buchholz's notation for the Bachmann--Howard
ordinal.  The countable part of $\mathrm{OT}(k)$ has order type
$\psi_0(\Omega_{k+1})$, the proof-theoretic ordinal of the theory of
$k$ iterated inductive definitions, and the countable part of
$\mathrm{OT}(\omega)$ has order type $\psi_0(\Omega_\omega)$, the
supremum of the levels' types and the proof-theoretic ordinal of the
subsystem $\Pi^1_1$-$\mathrm{CA}_0$ of second-order arithmetic.  These
order types come from the ordinal analyses of iterated inductive
definitions and of
$\Pi^1_1$-$\mathrm{CA}_0$~\cite{buchholz81iterated,buchholz88proof},
through the identification of the system with Buchholz's collapsing
functions~\cite{buchholz86new}; none of this is mechanized here.  The
fundamental sequences of the system are compiled recursions on the
codes.  The notation assigns each code a \emph{class}, its cofinality
type: $0$, a successor, cofinality $\omega$, or $\Omega_{w+1}$.  At a
principal code $D_u b$ whose argument $b$ has the class
$\Omega_{w+1}$, with $u \leq w$, the $n$-th element is
$D_u(b[\gamma_n])$, where $\gamma_0 = D_w 0$, $\gamma_{j+1} =
D_w(b[\gamma_j])$, and $b[\xi]$ replaces by $\xi$ the code
$\Omega_{w+1}$ at the end of the \emph{rightmost path} of $b$, the path
that descends from $b$ through last summands and arguments.  Buchholz defines this case by a single element that
does not depend on $n$~\cite{buchholz86new}; the cofinal chain is the
development's construction.  The validity of the elements, their
descent, their cofinality and the bounds are \rga derivations at open
codes.  Let $c_n = D_0\Omega_{n+1}$, the countable code that bounds
$\mathrm{OT}(n)$: classically, a valid countable code lies in
$\mathrm{OT}(n)$ exactly when it lies below $c_n$.  Let $\hat c$ be the
limit code of the chain $h_0 = 0$, $h_{n+1} = h_n + \omega^{c_n}$,
built in the form of the limit codes of $\mathcal O$, and $\hat c_k$
the limit code of the same chain over the fundamental sequence of
$c_k$; the claim is that the code is a notation, a member of
$\mathcal O$.

\begin{thm}[\isn{o\_wf\_of\_comega}, \isn{ot\_iff\_claim}]
\label{thm:otnu-claim}
From the hypothesis that $\hat c$ is a notation, \iga derives $\WF$
for the countable part of $\mathrm{OT}(\omega)$, and for each $k$, from
the hypothesis that $\hat c_k$ is a notation, \iga derives $\WF$ for
the countable part of $\mathrm{OT}(k)$.  Moreover \iga proves the
countable part of $\mathrm{OT}(\omega)$ well-founded if and only if it
proves $\hat c$ a notation, and adopting the claim that $\hat c$ is a
notation is sound for the omega model and consistent.
\end{thm}

The first sentence is a grounded claim, a hypothetical derivation, and
its finite-level half is \isn{o\_wf\_of\_ck\_K}.  The inductions of
\isn{o\_wf\_of\_comega} and of \isn{o\_wf\_of\_ck}, at the countable
codes below $c_k$, are induction along $\mathcal O$ at the claimed code
and inductions on numbers, as for \cref{thm:triangle}; for the
countable part of $\mathrm{OT}(k)$ there is also transfinite induction
along the order below $c_k$, which the claim has just made
well-founded.  No canonical form of codes is used anywhere.  The comparison of codes is transitive on raw codes,
and a code that the comparison places at most an element of a
sequence is reduced to the codes strictly below it.  The derivation
is generic.  A \emph{fundamental-sequence kit} is the sequence, its
\emph{guard}, the predicate that the sequence's elements satisfy, and
its comparisons, below and at most, with the congruence that a code
below one at most another is below the other.  The development proves
once, for every kit with a decided order term, cofinality and a bound,
whose approximants pass the guard, that one membership claim in
$\mathcal O$ buys $\WF$ for the kit's order (\isn{wf\_of\_claim});
the countable part of $\mathrm{OT}(\omega)$ and, for each $k$, the
countable codes below $c_k$ are its instances, and the countable part
of $\mathrm{OT}(k)$ follows because its codes lie below $c_k$.  A build-time check over the
transitive dependencies, by names and statement shapes, gives evidence
that no soundness theorem and no well-foundedness fact about the
notation's order, neither its HOL well-foundedness nor the omega-truth
of the sentence, is used (Appendix~\ref{app:stage}, step 62).  The
second sentence is classical.  The other direction of the equivalence
derives the claim from $\WF$ by transfinite induction along the order;
at a finite level, the claim $\hat c_k$ and $\WF$ for the countable
codes below $c_k$ are interderivable in the same way
(\isn{ot\_iff\_claim\_k}).  The classical warrant is the HOL
well-foundedness of the countable part, a consequence of the HOL
well-foundedness of the whole notation, every finite level at once
(\isn{ot\_wf\_all}), which an accessible-part induction stratified by
level proves.  It enters at one place, the omega-truth of the
well-foundedness sentence, which makes the adoption of the claim
omega-sound and consistent (\isn{otC\_adopt\_sound},
\isn{otC\_adopt\_consistent}).  So reaching $\psi_0(\Omega_\omega)$ costs one claim over \iga, as
reaching $\Gamma_0$ does, while the order types themselves are
cited.  The rule
form at these notations is the stratified rung of \S\ref{sec:pricing}
(\cref{thm:otnu-rule}).

\subsection{What an adopted claim gives, and what tests it}
\label{sec:claims}

The claim form above adopts one membership claim in $\mathcal O$ as a
hypothesis.  For a code $a$, write $\iga+\mathcal O(a)$ for \iga with
the claim that $a$ is a notation as a hypothesis in every context
(\isn{adoptDer}); the development proves this presentation equal to
\iga with one more rule, an axiom rule that concludes the claim
(\isn{adopt\_ext\_iff}).  Say that $a$ lies in $\mathcal O$, a HOL
statement, when $a$ lies in the least fixed point of the operator of
$\mathcal O$'s body.  Classically, $\iga+\mathcal O(a)$ is sound for
the omega model exactly when $a$ lies in $\mathcal O$, and it is then
consistent (\isn{adopt\_sound\_iff\_lfp}, \isn{adopt\_consistent}); call
such an extension a \emph{sound adoption}.  The claim is a sentence of
\iga, while the fact that $a$ lies in $\mathcal O$ is the classical
warrant, as for $\hat\gamma$ and $\hat c$.  This subsection collects
what adoptions of bare claims give, what tests them, and what
G\"odel's second theorem says about them; the G\"odel paper states the
same results with its checkers~\cite{goedel-paper}.

\begin{thm}[\isn{fuAut\_eq\_iga}, \isn{fuU\_pi01\_iff}]
\label{thm:claim-autonomy}
Admit the code of the finite notation $0$, and admit a code $b$
whenever $\iga+\mathcal O(a)$ proves $b$ a notation for an admitted
code $a$.  Then $b$ is admitted if and only if \iga proves $b$ a
notation, and a sequent is derivable in $\iga+\mathcal O(a)$ for some
admitted $a$ if and only if it is derivable in \iga.  On the other
side, \iga with the claims at finitely many codes in $\mathcal O$ as
hypotheses derives the universal $\forall v.\,(t\cdot v = 0)$ of a
compiled primitive-recursive test $t$ exactly when the test has the
value $0$ at every number.
\end{thm}

Both sentences are classical.  The first sentence is Kreisel's autonomy
condition~\cite{kreisel58ordinal} applied to bare claims
(\isn{fuAdm\_iff}), and its proof is one cut: a claim that \iga
proves adds nothing, since every derivation from the claim becomes a
derivation of \iga.  So the autonomous closure of \iga under its own
notation claims is \iga itself.  A claim already carries transfinite
induction at every closed motive along the claimed code, so the license that
admits a claim, a proof of the claim, already grants what the claim
would add.  By \cref{thm:triangle,thm:ordinal}, \iga does not prove
$\hat\gamma$ a notation, so the step to $\Gamma_0$ is not an autonomous
step from \iga.  The bivalence form collapses in the same way: where
\iga proves $\WF$ for an order, it admits bivalence of accessibility
for that order (\isn{wf\_gives\_accbiv}).  A growing progression needs
a step that a bare claim does not supply: consistency, as in
\cref{thm:claim-g2} below; reflection, as in the autonomous progression
of \S\ref{sec:pricing}; or truth, as in the Kripke universes of
\S\ref{sec:tower}.  The second sentence is the other extreme.  The
union of the sound adoptions, chosen with the classical knowledge of
which codes lie in $\mathcal O$, is complete and sound for the
universals of primitive-recursive tests; it is sound for the omega
model and consistent (\isn{fuU\_omega\_sound}, \isn{fuU\_consistent}).
For each test the development builds a limit code whose claim and the
test's universal are interderivable in \rga, a code that climbs through
the finite notations while the test has the value $0$
(\isn{pwcT\_rule\_I}, \isn{pwcT\_rule\_E}, \isn{fu\_pi01\_iff}).  The
code is built from the test itself, so the theorem gives no way of
finding the right claim.

\begin{thm}[\isn{gi\_never\_refutes\_WF}, \isn{dirAC\_inconsistent}, \isn{fu\_AC\_inconsistent}, \isn{fu\_falsified}]
\label{thm:claim-tests}
\leavevmode
\begin{enumerate}
\item For every closed, well-formed order $R$, \iga does not refute
  $\WF$ for $R$; and \iga derives $0 = 1$ from $\WF$ for the
  successor-reversed order, in which $n+1$ lies below $n$.  So \iga with
  a single well-foundedness sentence that \iga does not refute can
  derive $0 = 1$.
\item Let a system contain \iga, be sound for the omega model, and
  have a proof checker at a table of rule detectors that meets the
  G\"odel paper's structural conditions, accepts a proof of every
  closed theorem and is faithful at closed codes.  Then some closed
  order, even one whose instances \iga decides, has a $\WF$ sentence
  that the system can add without deriving $0 = 1$ and does not
  derive.  \iga
  is such a system, and so is every sound adoption.
\item If $\iga+\mathcal O(a)$ derives $0 = 1$, or a sentence false in
  the omega model, or a sentence whose negation \iga derives, then $a$
  does not lie in $\mathcal O$.
\item There is a limit code $\mathit{qn}$ whose sequence has the value
  $\mathit{qn}$ at every argument such that \iga does not refute the
  claim that $\mathit{qn}$ is a notation, the claim has no value in the
  omega model, $\mathit{qn}$ does not lie in $\mathcal O$, and
  $\iga+\mathcal O(\mathit{qn})$ derives $0 = 1$.
\end{enumerate}
\end{thm}

Part~1 combines a classical fact, from \cref{thm:acc-sem}: accessibility
is never false in the omega model, so neither is a $\WF$ sentence, and
omega-soundness excludes a refutation; and a grounded derivation,
\isn{wf\_stepD\_explodes}: the identity sequence descends along the
successor-reversed order, and $\WF$ makes its start accessible, which
the descending-sequence theorem (\S\ref{sec:concl}; Appendix~\ref{app:stage},
step~54) turns into $0 = 1$.  So the direct adequacy rule, ``adopt
$\WF$ for every order that \iga does not refute,'' is inconsistent, and
no self-reference is needed to see it; in rule form, every set of
sentences closed under \iga consequence and under that rule contains
$0 = 1$ (\isn{dirAC\_rule\_inconsistent}).  Part~2 is classical, and
its witness is a Liar order: $x+1$ lies below $x$ exactly when some
$p \leq x$ is accepted by the checker as a proof of the order's own
$\WF$ sentence, the self-reference being by the G\"odel paper's diagonal
function~\cite{goedel-paper}.  If the system derives that sentence, an
accepted code exists and \iga derives $0 = 1$ from the sentence and
the acceptance of $p$, along the descent $j \mapsto j + p$ (\isn{fuRL\_WF\_explodes\_seq}, grounded);
if not, the order is empty in the omega model and the sentence is
consistent to add, so the rule would make the system derive it.  The
instances are \isn{gi\_AC\_inconsistent} and
\isn{fu\_AC\_inconsistent\_fuel}.  A Rosser variant of the order
refutes the general form of the rule with consistency in place of
soundness, for a system that is the entailment of a step relation
whose table of rule detectors fires at a node exactly when the node
is a step (\isn{fu\_ACR\_inconsistent}, \isn{gi\_ACR\_inconsistent},
\isn{fu\_ACR\_inconsistent\_fuel}, every consistent adoption); the
form restricted to orders whose instances \iga decides is not stated
there.  Its
clause cites proofs of the sequent ``the sentence derives $0 = 1$,''
not proofs of the sentence's negation, because the step from the first
to the second is not a rule of \iga: \iga derives $0 = 1$ from $\WF$
for the successor-reversed order and does not refute that sentence
(\isn{fu\_WF\_no\_neg\_intro}).  Part~3 is classical
(\isn{fu\_falsified}, \isn{fu\_falsified\_by\_test},
\isn{fu\_falsified\_by\_refuted}; also \isn{fu\_falsified\_by\_instance},
for a derived universal with an instance that \iga refutes), and in
the refuted cases the adoption derives $0 = 1$ outright, a grounded
derivation (\isn{fu\_refuted\_explodes}).  Part~4 is classical except
for its last clause, a grounded derivation in $\iga$ from the claim
(\isn{fuQ\_falsified\_by\_consequence}, \isn{fuQ\_omv\_gap};
the derivation is \isn{adopt\_fuQ\_explodes}).  The code $\mathit{qn}$ is a quine at the
level of index trees, and the derivation is one induction along
$\mathcal O$ at the claimed code with the decided motive ``$x$ is not
$\mathit{qn}$'': a limit code equal to $\mathit{qn}$ would have a
member of its sequence different from $\mathit{qn}$, while every member
is $\mathit{qn}$.  \iga refutes a claim only at a code with a
malformed or missing element somewhere in its structure
(\isn{gi\_refutes\_only\_bad}), and $\mathit{qn}$ has none
(\isn{fuQ\_honest}), so part~4 shows that a claim \iga never refutes
can still explode when adopted; the inference from the explosion to the
falsity of the claim is modus tollens in the metatheory, not a rule of
\iga.  The same code witnesses that adopting every claim \iga does not
refute is inconsistent (\isn{dirAC\_code\_inconsistent}).  Read
together, unrefutedness is no closure condition on adopted claims or
on $\WF$ sentences, and consistency with the system is none on $\WF$
sentences; that it is none on adopted claims either is a prediction,
expected through a Liar code built like $\tau$ below, and the
development does not state it.  An adoption is a supplied hypothesis,
tested only through its consequences.

The last theorem is G\"odel's second theorem at every claim.  The proof
checker of $\iga+\mathcal O(a)$ is \iga's checker of the G\"odel
paper~\cite{goedel-paper} with one \emph{axiom cell}, a detector that
accepts a node exactly when its conclusion is the claim, as the level
checkers of \S\ref{sec:contower} add one tag rule for each element of
a list; write
$\ConA{a}$ for its consistency sentence (\isn{chkF},
\isn{ConF}; the G\"odel paper writes $\mathit{Con}_{[a]}$, and the
bracket is reserved here for the level lists of \S\ref{sec:contower}).  \emph{Turing's code} $\tau$ is the limit code built, as
in the second sentence of \cref{thm:claim-autonomy}, from the test
``$n$ is an accepted \iga proof of $0 = 1$,'' whose universal says
that no $n$ is; $\tau_a$ is built in the same way from the checker of
$\iga+\mathcal O(a)$ (\isn{tc}, \isn{tcF}).  Write
$\mathit{Con}_{\mathrm{IGA}}$ for the G\"odel paper's consistency
sentence of \iga~\cite{goedel-paper}, the universal of that test.

\begin{thm}[\isn{fu\_tc\_strict}, \isn{fu\_goedel2}, \isn{fu\_goedel2\_internal}, \isn{fu\_needs\_larger\_fuel}]
\label{thm:claim-g2}
\leavevmode
\begin{enumerate}
\item The claim that $\tau$ is a notation and
  $\mathit{Con}_{\mathrm{IGA}}$ derive each other in \rga.  $\tau$ lies
  in $\mathcal O$; $\iga+\mathcal O(\tau)$ derives
  $\mathit{Con}_{\mathrm{IGA}}$ and does not derive $0 = 1$; and \iga
  derives neither $\mathit{Con}_{\mathrm{IGA}}$ nor the claim that
  $\tau$ is a notation.
\item For every code $a$, $\iga+\mathcal O(a)$ derives, from the
  hypothesis $\ConA{a}$, that $\ConA{a}$ has no
  proof accepted by its checker.
\item For every $a$ in $\mathcal O$, $\iga+\mathcal O(a)$ does not
  derive $\ConA{a}$; $\tau_a$ lies in $\mathcal O$;
  $\iga+\mathcal O(\tau_a)$ derives $\ConA{a}$; and
  $\iga+\mathcal O(a)$ does not prove $\tau_a$ a notation.
\end{enumerate}
\end{thm}

The interderivability of part~1 is grounded, two \rga derivations
under one hypothesis each (\isn{tc\_ConIGA\_iff}); the rest of part~1
is classical, with G\"odel's second theorem at \iga as its
non-derivability premise.  Part~2 is a grounded statement, a derivation
in $\iga+\mathcal O(a)$ at every code, sound or not, whose proof uses
no soundness or consistency fact (a build-time check).  Part~3 is
classical: the second theorem of the G\"odel paper, stated once for
every system that meets its conditions, holds at every sound adoption
with the extended checker, by the G\"odel paper's table-generic laws
unchanged (\isn{fu\_g2\_ent}, \isn{fu\_g2\_con}).  So ``G\"odel's
second theorem says when a strictly larger claim is needed'' is exact,
at \iga and at every sound adoption: the adoption does not derive its
own consistency sentence, and the sound claim built from its proof
checker does.  The strictness is non-containment, since the adoption
does not derive the larger claim (\isn{fu\_not\_tcF}).  The caveat
that Turing made himself, of ``small ordinals expressed in a very
complex and artificial way''~\cite[p.~202]{turing39ordinals}, and
that Feferman's analysis of progressions
sharpened~\cite{feferman62transfinite}, is built in: $\tau$ and $\tau_a$ read a proof
checker, so the step gains nothing that was not put into the code.
Read classically, and not mechanized, $\tau$ is a notation for $\omega$,
like the standard code for $\omega$ whose claim \iga already derives
(\isn{adopt\_omega\_free}); the extra strength lies in the code, not in
the ordinal.

\subsection{The ordinal}
\label{sec:strength}

Below the equivalence, the development's route reaches
$\varphi(n,0)$ for every finite $n$ and no further
(\cref{thm:veblen}).  The limit at $\omega$ in the first argument is
the first place where the closure of $\varphi(a,\cdot)$ is needed
uniformly over infinitely many smaller first arguments, and it meets
the obligation above.  The development has located that obligation
exactly.

\begin{thm}[\isn{LADDER\_of\_ObRC}, \isn{ObRC\_iff\_AccBiv}]
\label{thm:ladder}
Call the \emph{ladder} the sentence that $\varphi(c,0)$ is accessible
for every code $c$ below $\omega$, and call the \emph{closure-shaped
instances} of bivalence the instances of the rule of
\cref{thm:gamma0} whose conclusion is the closure property of
$\varphi(s,\cdot)$ at $d$, for codes $s$ and $d$.  The ladder is
derivable from the closure-shaped instances, and \iga admits the
closure-shaped instances if and only if it admits bivalence, hence if
and only if it proves the whole Veblen order well-founded.
\end{thm}

The derivation of the ladder from the instances is ground induction on
the first argument with the uniform closure property as motive.  At
an eigenvariable $d$ of the motive, the step derives the closure
property from the hypothesis that $d$ is accessible, and turning that
derivation into the sentence ``$d$ is not accessible, or the closure
property holds at $d$'' is one closure-shaped instance: implication
introduction at an undecided hypothesis.  Every other route the
development tried, number induction at the accessibility of
$\varphi(n,0)$, ground induction on the second argument, the rule
form of the conclusion, and a lexicographic fixed point on pairs,
meets the same instance.  Whether the instances at finite first arguments alone are weaker than bivalence is open.  The lower bound is sharp.
\begin{thm}[\isn{gi\_acc\_iff\_below\_pw}, \isn{gi\_not\_WF\_gamma0}]
\label{thm:ordinal}
For every valid notation $\alpha$ of the Veblen system, \iga proves
the code of $\alpha$ accessible along the Veblen order if and only if
$\alpha < \varphi(\omega,0)$.  Moreover, \iga does not prove the
Veblen notation well-founded, nor the well-foundedness of any closed
domain true at the code of a valid notation at or above
$\varphi(\omega,0)$; it does not admit bivalence of accessibility for
the Veblen order, nor the closure-shaped instances; and it does not
derive the ladder sentence, although it derives, for each $n$, that
$\varphi(n,0)$ is accessible.
\end{thm}

The proof refines the soundness proof of \S\ref{sec:induction}.  The
omega model is one positive inductive definition, so it has stages,
and the development indexes them by the Veblen notations themselves,
below $\Gamma_0$; truth and falsity are both generated positively, so
true and false facts persist from a stage to every later one.  Every rule of \iga
reads its premises for truth only, so each rule has a stage cost that
is a function of the stage at which its hypotheses hold: a constant,
a fixed finite number of successor stages, for the rules that add or
remove clauses of the model; a uniform bound for universal
introduction, since the eigenvariable derivation, applied at every
numeral at one hypothesis stage, puts every instance at the same
stage, and one more stage closes the universal; an $\omega$-fold iterate for number induction, and, for
ground induction, the transfinite iterate of the closure derivation's
cost along the stages at which points enter the fixed point.
Iterating a function at Veblen level $n$ stays below level $n+1$, and
the stage judgment's bounds are all closed notations of the form
$\varphi_n^k(\max(\alpha, \omega)+1)$.  The closure lemmas need only the
first clause of the notation's order and a head comparison; the other bound facts rest on
the successor property, that $x < y$ implies $x+1 \leq y$, which comes
from the cofinality of the fundamental sequences.  So every closed
theorem of \iga is true at a stage below some $\varphi(n+1,0)$, hence
below $\varphi(\omega,0)$.  Read off the proof rather than stated in
HOL, $n$ can be taken as the depth to which the derivation's number
and ground inductions nest through their step and closure premises.  The stage at which a notation enters
accessibility is at least the notation itself, because each
predecessor enters strictly earlier; so a provably accessible notation
lies below $\varphi(\omega,0)$.  The converse is \cref{thm:veblen}.
The three consequences follow, as five theorems of the development:
well-foundedness of a domain true at the code of a notation at or
above $\varphi(\omega,0)$, and of the Veblen notation in particular,
would put that notation at a stage below itself; bivalence and its
closure-shaped instances fail by \cref{thm:gamma0,thm:ladder}; and the
ladder, true below some $\varphi(n+1,0)$, would put $\varphi(n+1,0)$ at
a stage below itself.  As a reading of the proof rather than a
theorem, bivalence is a rule whose cost is not a function of its
hypotheses' stage but the full height of the order.  That cost is the
content of the limit: implication introduction at an undecided
accessibility hypothesis.  \S\ref{sec:pricing} prices the other rules
of the grounded family by the same method.  Classically, the
well-foundedness and ladder sentences are true in the omega model, and
bivalence and its closure-shaped instances are sound for it as rules;
the omega model needs stages beyond the valid notations for these
sentences and rules,
while the mechanized stages are valid notations, all below $\Gamma_0$,
and every closed theorem of \iga holds at a mechanized stage below
$\varphi(\omega,0)$.  Appendix~\ref{app:stage} gives the proof in
detail.  The measure is accessibility, equivalently well-foundedness
as one sentence; transfinite induction at decided motives is not
bounded by the argument.  With \oga's decidedness rule added, such
induction below $\varphi(\omega,0)$ is derivable as one sentence
(\cref{thm:oiga-ti}), and its exact ordinal is open.

The classical landmarks put the ordinal in perspective.
Kripke--Feferman truth has the ordinal $\varphi(\eps,0)$ of ramified
analysis below $\eps$; its paracomplete sibling PKF proves
transfinite induction exactly below $\varphi(\omega,0)$, the ordinal
of ramified analysis below
$\omega^\omega$~\cite{halbach06axiomatizing,halbach18costs}; and
Kripke--Feferman with Cantini's minimality principle, of which ground
induction is the grounded form, reaches the strength of one inductive
definition~\cite{cantini89notes,burgess14friedman}.  \iga's ordinal
coincides with PKF's: in a paracomplete truth theory, minimality for
every positive definition adds nothing to the accessibility ordinal,
while the same principle in classical logic takes Kripke--Feferman
truth from $\varphi(\eps,0)$ to the Bachmann--Howard ordinal.
Everything that separates \iga from its classical twin is excluded
middle at the inductive predicates.  For the other measure, the
soundness proof is an induction on the clauses of the omega model,
one positive inductive definition, with every other step arithmetic,
so it formalizes in Feferman's theory $\mathrm{ID}_1$ of one
inductive definition~\cite{feferman70formal,buchholz81iterated}, and
\iga's consistency strength is at most that of one inductive
definition; this reading is classical and not mechanized as a
formalization.  Whether that strength equals PKF's, and whether PKF
interprets into \iga, are questions for ordinal analysis; the two
questions are among the open questions of \S\ref{sec:concl}.

\subsection{Pricing the other rules}
\label{sec:pricing}

The stage method of \S\ref{sec:strength} is not tied to ground
induction.  A rule enters the argument only through its case of the
step lemma, the bound that the rule puts on its conclusion's stage
given its premises' stages, and the rank lemma and the bound family
do the rest.  This subsection uses the method to price two other
rules of the grounded family, \oga's decidedness rule and reflection,
and an autonomous progression of reflection.  The answers are three.
The decidedness rule leaves the accessibility ordinal at
$\varphi(\omega,0)$.  One reflection over \iga with the decidedness rule costs at most one
step, from $\varphi(\omega,0)$ to $\varphi(\omega,1)$, and exactly one
step when its warrant is granted.  The autonomous progression of
reflection, with its warrants granted, closes at $\varphi(\omega+1,0)$,
the first fixed point of $\varphi(\omega,\cdot)$, far below $\Gamma_0$.
The table at the end of the subsection, \emph{the price list},
collects the systems and their ordinals.  Every theorem below
is a classical claim, a HOL theorem about derivability and the stage
model, unless it says otherwise.

Some of the systems are not candidates for grounded systems.  Call a
system a \emph{calibration device} when one of its rules has a side
condition that is a HOL fact about derivability, checked in HOL rather
than derived in the system; such a system's side condition is not
recursively enumerable, and the system has no proof checker.  A calibration device measures
what a rule is worth when the rule's warrant is granted.  A
calibration device's theorems say nothing about what a grounded system
derives.  The other
systems below are systems of the grounded family: their rules are
rules of derivation, and they are sound for the omega model.
Membership of the family is omega-soundness, not a
certificate-grounded semantics in the sense of \S\ref{sec:bg}.

\paragraph{Decidedness.}
Let \oiga be \iga together with \oga's decidedness
rule~\cite{oga-paper}: from a derivation of the one sentence that
every numeral instance of a term $u$ is decided, infer that the
universal $\cA\cdot u$ is decided.  A proof in \oiga may mix the rules
of the two systems, so \oiga contains \iga and \oga, and through \oga
it contains Peano Arithmetic under the translation described in
\S\ref{sec:related}.  As a classical claim, \oiga has no
certificate-grounded semantics: no semantics in which a true universal
needs an \oiga derivation of its open instance, false universals,
negations and disjunctions have their usual clauses, and no term is
both true and false, makes every theorem of \oiga true
(Appendix~\ref{app:stage}, step 51).  Recursive enumerability of
\oiga's derivations is the hypothesis that does the work, since the
omega model meets every other hypothesis of the general argument.
\iga itself is proof-grounded (\cref{thm:pg}), and the same argument
turns into a non-derivability result for \iga.  The \emph{halting
family} is the family of universals $\cA\cdot\mathit{ofam}(x)$, one for
each code $x$, saying that program $x$ on input $x$ never halts.
\iga's derivations are recursively enumerable, so by \iga's
consistency, disjunction property and witness property, some universal
of the halting family at a non-halting point is one whose decidedness
\iga does not derive.  Nor does \iga derive the decidedness of the
family uniformly over codes, stated through the value record of the
internal-reasoning paper~%
\cite{ford-conrga}.  \oga, and so \oiga, derives every
one of these decidedness sentences.  No omega
model of \iga is used (Appendix~\ref{app:stage}, step 55).

\begin{thm}[\isn{oiga\_acc\_iff\_below\_pw}, \isn{oiga\_not\_dec\_acc}]
\label{thm:oiga}
For every valid notation $\alpha$ of the Veblen system, \oiga proves
the code of $\alpha$ accessible along the Veblen order if and only if
$\alpha < \varphi(\omega,0)$.  Moreover, \oiga does not prove the
Veblen notation well-founded, nor the well-foundedness of any closed
domain true at the code of a valid notation at or above
$\varphi(\omega,0)$; it does not admit bivalence of accessibility for
the Veblen order; it does not derive the ladder sentence; and it does
not decide accessibility along the Veblen order, neither as a rule,
deriving $(\Acc\cdot d)\,\mathbf B$ in every well-formed context
$\Delta$ with $\Delta \ent d\,\mathbf N$, nor as the one sentence
that the accessibility of every number is decided.
\end{thm}

The decidedness rule costs one step of the bound family and never
raises the level.  The rule's premise is one universal sentence, true
at some stage $s$.  The universal-true clause of the stage model puts
the value of every instance strictly below $s$, so the rule only
re-reads those values as ``all true, or one false'': the universal's
clause one stage up, then negation and disjunction, at most three
successor stages, which one step absorbs.  No supremum is formed and
no persistence argument is needed.  The rule therefore joins universal
introduction in the row $(M, K+1)$ of the cost table of
Appendix~\ref{app:stage}.  Ground induction raises the level because
its closure premise is consumed at every stage of a transfinite
iteration; the decidedness rule consumes nothing iteratively.  \oiga's
step lemma is then a composition, with nothing restated: an \oga step
is an \rga step or a decidedness step, an \rga step is an \iga step,
and \iga's forty-three cases apply.  So every closed theorem of \oiga
is true at a stage below some $\varphi(n+1,0)$, and the rank lemma
gives the upper half of the ordinal; the lower half is \iga's
(\cref{thm:ordinal}), through the embedding.  The independence results
follow as for \iga.  The last independence result reduces to bivalence: a uniform
decision of accessibility would give the rule of \cref{thm:gamma0},
by a case split on the decided disjunction $\Acc\cdot d \vee
\neg(\Acc\cdot d)$ that uses the premise's derivation in one branch
and the hypothesis in the other.  Read as a statement about the
deduction theorem, this non-derivability result says that the decidedness rule feeds
conditional proof only at decided hypotheses, and accessibility is not
one of them.

What the decidedness rule does buy is classical arithmetic, and
transfinite induction as a sentence.

\begin{thm}[\isn{oiga\_TI\_decided\_below\_pw}]
\label{thm:oiga-ti}
Let $x$ be a valid notation below $\varphi(\omega,0)$, and let $M$ be
a closed term decided at numbers in every context: whenever $\Delta
\ent a\,\mathbf N$, also $\Delta \ent (M\cdot a)\,\mathbf B$.  Then
\oiga derives the sentence $\mathrm{Prog}(M) \to M\cdot\lceil
x\rceil$, where $\to$ is the definable implication $\neg a \vee b$ and
the progressiveness of $M$ along the Veblen order $\prec$ is the one
universal sentence
\[
  \mathrm{Prog}(M) \;=\; \cA\cdot\lambda z.\,
  \neg\bigl(\cA\cdot\lambda y.\,\neg(y \prec z) \vee M\cdot y\bigr)
  \vee M\cdot z .
\]
\end{thm}

The conclusion is a grounded claim, a derivation in \oiga, under a
hypothesis about $M$ that is a HOL schema over all contexts.  The
derivation has four parts: \oiga proves the code of $x$ accessible
(\cref{thm:oiga}); \iga's transfinite induction (\S\ref{sec:wf})
applies in the context that holds the progressiveness sentence; the
decidedness rule, applied twice, decides that sentence, since the
order is decided on numbers; and the deduction theorem at the
now-decided hypothesis gives the implication.  The theorem covers one
notation at a time; the uniform form, over all notations below $x$ at
once, is not mechanized.  Whether transfinite induction at decided
motives in \oiga reaches beyond $\varphi(\omega,0)$ is open, and the
stage method does not bound it, because a true arithmetic sentence is
true at a small stage.  The theorems equate only the accessibility
ordinals of \oiga and \iga.  The two systems differ on other
sentences: \oiga derives the decidedness of every universal of the
halting family, and \iga leaves one undecided (paragraph
\emph{Decidedness} above), so \iga contains neither \oga nor \oiga
(\isn{iga\_not\_contains\_oga}, \isn{iga\_not\_contains\_oiga}).
Whether \oga contains \iga is open (\S\ref{sec:related}).  Excluded middle at arithmetic sentences, which
the decidedness rule supplies, leaves the accessibility ordinal where
it is.  Excluded middle at the inductive predicate moves it
(\S\ref{sec:one-rule}).

\paragraph{Reflection.}
Reflection enters the grounded family as \emph{adoption}, a rule
rather than an axiom: from a derivation of a \emph{warrant}, a
sentence saying that a search through a base system's proofs finds a
proof of every numeral instance of a closed term $u$, infer
$\cA\cdot u$.  Let $\oiga+\mathrm{Adopt}$ be \oiga with adoption over
\oiga.  Its warrants speak of one fixed search for \oiga's proofs, a
primitive recursive index whose values at numerals a HOL law fixes.
For closed $u$, the warrant of $u$ is true in the omega model exactly
when every numeral instance of $u$ is a theorem of \oiga, so in the
omega model adoption is the uniform reflection rule over \oiga.  Three
differences from classical reflection matter below.  Adoption is a
rule, not the reflection principle that whatever the base proves is
true, an implication whose antecedent, provability, is not decided.
The warrant reflects \oiga, not $\oiga+\mathrm{Adopt}$ itself: one
level of reflection, not self-reflection.  And the search is known
only through its values at numerals, so the system may fail to derive
warrants whose classical counterparts are trivial.  Every result
below about $\oiga+\mathrm{Adopt}$ holds for every search that
satisfies the law.  Like \oiga, $\oiga+\mathrm{Adopt}$ is a system of
the grounded family, sound for the omega model and consistent.  Let
$\oiga+\mathrm{Refl}$ be \oiga with adoption of $\cA\cdot u$, for
closed well-formed $u$, at every context and with no premise, whenever
every numeral instance of $u$ is a theorem of \oiga as a HOL fact.
$\oiga+\mathrm{Refl}$ is a calibration device: it measures what one
level of reflection over \oiga is worth when the warrant is granted.

\begin{thm}[\isn{oiaga\_stage\_sound\_closed}, \isn{chol\_acc\_iff\_below\_pw1}]
\label{thm:adopt}
Every closed theorem of $\oiga+\mathrm{Adopt}$ is true at a stage
below $\varphi(\omega,1)$.  So for every valid notation $\alpha$, if
$\oiga+\mathrm{Adopt}$ proves the code of $\alpha$ accessible, then
$\alpha < \varphi(\omega,1)$; and $\oiga+\mathrm{Adopt}$ neither
proves the Veblen notation well-founded nor admits bivalence of
accessibility for the Veblen order.  For every valid notation
$\alpha$, $\oiga+\mathrm{Refl}$ proves the code of $\alpha$
accessible if and only if $\alpha < \varphi(\omega,1)$, and
$\oiga+\mathrm{Refl}$ is consistent.
\end{thm}

Adoption costs a constant.  If every numeral instance of $u$ is a
theorem of \oiga, each instance is true below $\varphi(\omega,0)$ by
\oiga's stage soundness, so $\cA\cdot u$ is true one stage above
$\varphi(\omega,0)$, whatever the stage of the warrant: the base's
stage soundness bounds the instances, and in $\oiga+\mathrm{Adopt}$
the warrant's derivation is used only for the warrant's omega-truth.
To fit the constant into the bound family, every context of the proof
carries a \emph{marker}, the true closed sentence that the code of
$\varphi(\omega,0)$ is accessible.  By the rank lemma, the marker's
truth at a stage $\alpha$ forces $\varphi(\omega,0) \leq \alpha$, and
the adoption then lands at the base bound $(0,0)$ of the family.  The
marker is itself true below $\varphi(\omega,1)$: \iga derives it from
the universal that every notation below $\varphi(\omega,0)$ is
accessible, which is true at $\varphi(\omega,0)+1$, and
$\varphi(\omega,1)$ is closed under $\varphi(n,\cdot)$ for every finite
$n$.  Marking every context is needed.  Treating adopted sentences as
global hypotheses would fail, since an adoption inside a branch with
unsatisfiable hypotheses can adopt a false sentence.  The bound holds
whatever warrants $\oiga+\mathrm{Adopt}$ derives, and reflection never
raises the level index $n$ of the bound family.

The lower bound of $\oiga+\mathrm{Refl}$ uses two adoptions only, and
their side conditions are theorems of \iga instance by instance.

\begin{thm}[\isn{adst\_acc\_below\_pw1\_ctx}, \isn{adst\_uA\_inst}]
\label{thm:adopt-lb}
There are closed terms $u_A$ and $u_C$ such that every numeral
instance of each is a theorem of \iga, and, for every valid notation
$\alpha$ below $\varphi(\omega,1)$, \iga derives the accessibility of
the code of $\alpha$ from the two hypotheses $\cA\cdot u_A$ and
$\cA\cdot u_C$.
\end{thm}

The second part is a grounded claim, one derivation in \iga for each
$\alpha$.  The term $u_A$ says that a notation below
$\varphi(\omega,0)$ is accessible, $\lambda y.\,\neg(y \prec
\lceil\varphi(\omega,0)\rceil) \vee \Acc\cdot y$; the term $u_C$ says
that the canonical form of a notation below $\varphi(\omega,0)+1$ is
at most $\varphi(\omega,0)$.  The instances of both are \iga theorems,
by evaluation outside the guard and, inside it, by \cref{thm:ordinal}
and \iga's omega-soundness.  From the two
universals, \iga derives accessibility at $\varphi(\omega,0)+1$, then
at every iterate $\varphi_k^j(\varphi(\omega,0)+1)$ by the closure of
$\varphi(k,\cdot)$, which \iga derives for each finite $k$ separately,
and so at every valid notation below
$\varphi(\omega,1)$, by the cofinality of the fundamental sequence of
$\varphi(\omega,1)$.  So the non-recursively-enumerable side condition
of $\oiga+\mathrm{Refl}$ matters only for the upper bound, as a
hypothesis on what may be adopted.  The second universal supplies the step
from $\varphi(\omega,0)$ to $\varphi(\omega,0)+1$, which from $u_A$
alone would need internal reasoning about successors and canonical
forms.

For $\oiga+\mathrm{Adopt}$ itself the exact ordinal is conditional.

\begin{thm}[\isn{oiaga\_acc\_iff\_below\_pw1\_if}, \isn{oiaga\_warrant\_uAC\_true}]
\label{thm:adopt-if}
The conjunction $u_{AC}$ of the bodies of $u_A$ and $u_C$ has an
omega-true warrant.  If $\oiga+\mathrm{Adopt}$ derives that warrant,
then for every valid notation $\alpha$, $\oiga+\mathrm{Adopt}$ proves
the code of $\alpha$ accessible if and only if $\alpha <
\varphi(\omega,1)$.
\end{thm}

Whether $\oiga+\mathrm{Adopt}$ derives the warrant of $u_{AC}$ is
open.  The search is known only through its values at numerals, so no
reasoning about the search at a free variable is available to the
system.  Unconditionally, $\oiga+\mathrm{Adopt}$ proves the code of every
notation below $\varphi(\omega,0)$ accessible, through the embedding
of \oiga, and of none at or above $\varphi(\omega,1)$.

\paragraph{The autonomous progression.}
The \emph{progression} is a family of calibration devices over \oiga,
one for each valid notation $a$, its \emph{level}.  A proof list is
valid at level $a$ when each step is an \oiga step or an adoption of
$\cA\cdot u$, for closed well-formed $u$, at a notation $b < a$.  An
adoption at $b$ has two conditions.  The side condition, checked in
HOL, says that every numeral instance of $u$ is a closed theorem at
level $b$ of the same family, so the progression's lower levels feed its higher ones.
The \emph{license} says that the list already contains a derivation
of the accessibility of the code of $b$, a derivation in the system
at the current level.  In the \emph{empty-context form} the license
is derived in the empty context.  In the \emph{in-context form} the license may
be derived under any sub-context of the conclusion's hypotheses, the
empty one included, so the empty-context form is contained in the
in-context form.  The \emph{autonomous progression} is the union of
all the levels with the license; the \emph{control} is the same family
without the license.  The level is bookkeeping that makes the side
condition well-founded; autonomy is the license.

\begin{thm}[\isn{cauto\_acc\_iff\_level}, \isn{cctl\_acc\_iff\_level}]
\label{thm:levels}
Let $a$ and $\alpha$ be valid notations.  If $a < \varphi(\omega+1,0)$,
level $a$ of the progression with the license, in either form, proves the code of
$\alpha$ accessible if and only if $\alpha < \varphi(\omega,a)$.  For
every valid $a$, level $a$ of the control proves the code of $\alpha$
accessible if and only if $\alpha < \varphi(\omega,a)$, where
$\varphi(\omega,a)$ is read as $a$ itself when $a$ is a fixed point of
$\varphi(\omega,\cdot)$.  Level $\alpha+1$ of the control proves the
code of $\alpha$ accessible.
\end{thm}

\begin{thm}[\isn{cauto\_acc\_iff\_below\_pw\_omega}, \isn{cauto2\_acc\_iff\_below\_pw\_omega}]
\label{thm:auto}
In both forms, the autonomous progression proves the code of a valid
notation $\alpha$ accessible if and only if $\alpha <
\varphi(\omega+1,0)$.  Every closed theorem of either form is true at
a stage below $\varphi(\omega+1,0)$, and both forms are consistent.
\end{thm}

The lower bounds go by well-founded induction on the level $e$.
Below $\varphi(\omega,0)$, \oiga proves the accessibility.  Otherwise
there is a smaller level $d$ with $\alpha < \varphi(\omega,d+1)$: $d =
\alpha$ when $e$ is a fixed point of $\varphi(\omega,\cdot)$, and
otherwise $d$ is a term of the fundamental sequence of $e$, since the
fundamental sequence of $\varphi(\omega,e)$ is cofinal.  Level $e$ then
adopts the two bodies of \cref{thm:adopt-lb} at the point
$\varphi(\omega,d)$ in place of $\varphi(\omega,0)$, whose instances
are level-$d$ theorems, and climbs as in that theorem.  With the
license, the accessibility of the code of $d$ is a level-$d$ theorem
because $d < \varphi(\omega,d)$ whenever $d < \varphi(\omega+1,0)$.
The upper bounds go by well-founded induction on $a$.  A proof list
valid at level $a$ is valid at level $0$ or at some level $t+1$ with
$t < a$.  The marker of the adoption case becomes the accessibility of
the code of $\varphi(\omega,t)$, true below $\varphi(\omega,t+1)$, and
$\varphi(\omega,t+1) \leq \varphi(\omega,a)$.

The license is what bounds the autonomous progression.  In the
empty-context form, each license is a closed theorem of some level $L
< \varphi(\omega+1,0)$, so by the rank lemma the licensed notation
lies below $\varphi(\omega,L) < \varphi(\omega+1,0)$, and every
autonomous proof list is valid at some level below
$\varphi(\omega+1,0)$.  In the in-context form, a license bounds its
notation only under the license's hypotheses, so nothing excludes proof lists at
levels at or above $\varphi(\omega+1,0)$.  The stage judgment is then
read only at hypothesis stages below $\varphi(\omega+1,0)$.  There an
adoption at a notation $b \geq \varphi(\omega+1,0)$ is vacuous, since
its license would make the accessibility of the code of $b$ true below
$\varphi(\omega+1,0)$, and an adoption below costs a constant under
the marker.  A closed theorem has no hypotheses, so the bound applies.
The rule with the license derived in exactly the conclusion's context
is a derived rule of the in-context form, so its closed theorems obey
the same upper bound; its lower bound is not mechanized separately.
Read across the two proofs, the common reason is that whatever is
derivable in the empty context has its stage below $\varphi(\omega+1,0)$, and
$\varphi(\omega+1,0)$ is closed under $\varphi(\omega,\cdot)$.  On the
same reading, with the license as a premise the rule's cost is
no longer a constant: the license acts as an internal marker, the cost
becomes $s \mapsto \varphi(\omega,s)+1$, and the bound must close under
$\varphi(\omega,\cdot)$.

Read together, \cref{thm:levels,thm:auto} say that autonomy, the
accessibility of the level index as a theorem of the system, is
exactly what holds the progression at the first fixed point of
$\varphi(\omega,\cdot)$.  In the empty-context form each license is a
closed theorem, so the licensed notation lies below
$\varphi(\omega,L)$ for a level $L$ already reached, and
$\varphi(\omega+1,0)$ is closed under $L \mapsto \varphi(\omega,L+1)$.
In the in-context form the same bound holds wherever the license's
hypotheses are true below $\varphi(\omega+1,0)$, which covers every
closed theorem.
Without the license nothing ties the level to what has been derived,
and the control reaches every valid notation of the Veblen system,
whose height is $\Gamma_0$ by the notation's design; that
identification is not a theorem of the development.  Reflection never
raises the first Veblen argument at any level.  The first argument
rises from $\omega$ to $\omega+1$ only as the supremum of the
autonomous closure, and $\Gamma_0$ would need levels that raise it
further.  The ``exactly'' of this reading is the contrast of the two
theorems, not one theorem.  A progression whose adoptions are keyed to
a proof checker inside the system, rather than to a side condition
checked in HOL, is not mechanized, and its ordinal is open.

The stage method refines omega-soundness, since a stage fact is an
omega fact.  So the method cannot price a rule that makes a system
prove an omega-false $\Pi_2$ sentence: no stage makes that sentence
true.

\paragraph{The classical analog.}
Classically, uniform reflection over Peano Arithmetic climbs the
$\varepsilon$-numbers.  Kreisel and L\'evy showed that one uniform
reflection over PA is equivalent to transfinite induction up to
$\eps$~\cite{kreisel68reflection}.  Schmerl generalized Kreisel and
L\'evy's theorem.  The $\alpha$-fold iteration of full uniform
reflection over PA is equivalent to primitive recursive arithmetic with
transfinite induction below $\varepsilon_\alpha$ for every arithmetic
formula, and for each $n$ the iteration proves the same
$\Pi^0_{n+1}$ sentences as the $\varepsilon_\alpha$-fold iteration of
$\Pi^0_{n+1}$ reflection over primitive recursive
arithmetic~\cite[\S1]{schmerl82iterated}.  Pakhomov and Walsh restate
the second equivalence over a weak base theory with a set
parameter~\cite[Thm.~4.2]{pakhomov21reflection}.  The iteration's
ordinal is therefore $\varepsilon_\alpha$: $\eps$ for PA,
$\varepsilon_1$ after one reflection.  The companion result for partial reflection is Schmerl's
formula~\cite{schmerl79fine}~\cite[Thm.~4.1]{pakhomov21reflection}.  The license above is a form of Kreisel's autonomy condition: an
ordinal may index a level only once an earlier stage of the
construction has proved it an
ordinal~\cite{kreisel58ordinal,feferman64systems}\cite[\S4.3.1]{horsten21truth}.  Here the earlier
stage is an earlier line of the same proof list, and ``proved an
ordinal'' means one derived accessibility sentence.  Horsten and Zicchetti
report, as an apparently folklore result, that the autonomous
progression of uniform reflection over PA has length exactly
$\varphi(2,0)$, the first fixed point of $\alpha \mapsto
\varepsilon_\alpha$, and coincides with the first-order part of
ramified analysis below level $\omega$; they trace the claim to
Feferman's 1964 paper~\cite[Thm.~4.3.3 and note~6]{horsten21truth}.
The grounded and classical
pictures run in parallel.  In both pictures, one reflection adds the base's
ordinal as one point, the accessibility of $\varphi(\omega,0)$ as
reflection over PA adds transfinite induction to $\eps$, and the
base's closure carries the point to the next value.  The closure
operation is what differs: $x \mapsto \omega^x$ for PA, giving
$\varepsilon_1$, and $\varphi(n,\cdot)$ for every finite $n$ for \iga
and \oiga, giving $\varphi(\omega,1)$.  The grounded restriction enters
elsewhere.  The ladder sentence is not derivable
(\cref{thm:ordinal,thm:oiga}), so the base never closes under
$\varphi(\omega,\cdot)$, and only autonomy adds closure under
$\varphi(\omega,\cdot)$, which stops at its first fixed point.  In
both pictures autonomy closes the progression at the first fixed point
of the per-level step: $\varphi(2,0)$ for the classical step
$\varphi(1,\cdot)$, which is $\alpha \mapsto \varepsilon_\alpha$, and
$\varphi(\omega+1,0)$ for the grounded step $\varphi(\omega,\cdot)$.
This comparison is a reading: the development mechanizes no
classical progression, and the grounded progression is not Schmerl's
or Feferman's.

\paragraph{Proof-grounded and omega-grounded.}
The stage method prices a rule's strength.  A second question asks
what a rule's theorems rest on.  Call a system of the grounded family
\emph{omega-grounded} when it is sound for the omega model but has no
certificate-grounded semantics; a system is proof-grounded, as in
\S\ref{sec:induction}, when it is sound and complete for the
certificate-grounded semantics with its own proofs as the
certificates.  Both words name properties of a whole system.  Neither
word is the viewpoint label ``grounded'' of \S\ref{sec:intro}, which
marks a claim as a derivation in the system, nor the ``grounded
family'' of the price list below, which is omega-soundness.  The
decidedness rule has no certificate-grounded semantics in any calculus
that has the rule and is $\Pi_1$-sound and recursively enumerable at
the halting family (Appendix~\ref{app:stage}, step 51).  So \oiga is omega-grounded, and
so is $\oiga+\mathrm{Adopt}$ (Appendix~\ref{app:stage}, step 55).

Two kinds of rule keep proof-groundedness.  One is ground induction
(\cref{thm:pg}).  The other is bivalence of accessibility at a
suitable order, first the Veblen order, where the system proves every numeral accessible, so
that the rule's disjunction always has a derivable disjunct.
$\iga+\mathrm{AccBiv}_\varphi$ is sound and complete for the
certificate-grounded semantics with its own proofs as the
certificates, and it has the disjunction property
(Appendix~\ref{app:stage}, step 52).  The proof uses the classical
warrant of \S\ref{sec:one-rule}, the HOL well-foundedness of the
notation, once, through the pairs at which \iga does not refute the
order, which are the order of an accessibility induction for the
semantics.  At each universal of that
induction the certificate is the system's own proof of accessibility
at a free variable.  \iga's omega-soundness enters once, to carry the
notation's well-foundedness over to the pairs at which \iga does not
refute the order, and the omega-soundness of the bivalence rule is
never used.  Read from the grounded viewpoint, the bivalence rule
raises the ordinal while every universal that the system proves keeps
a proof of the system as its certificate.

Call \iga with bivalence of accessibility for one order a
\emph{rung}.  The argument uses three facts about the order and
nothing else about it, so the development proves it once, for every
rung whose order has those facts.

\begin{thm}[\isn{givR\_gsem\_sound\_of\_WF}, \isn{givbh\_gsem\_sound}]
\label{thm:rung}
Let $R$ be a closed, well-formed order with three properties: the
pairs of numbers at which $R$ is not false in the omega model form a
well-founded relation in HOL, the order's classical warrant; \iga
derives $R\cdot\bar m\cdot\bar n \vee \neg(R\cdot\bar m\cdot\bar n)$
for all numbers $m$ and $n$; and $\iga+\mathrm{AccBiv}_R$ derives
$\WF$ for $R$.  Then $\iga+\mathrm{AccBiv}_R$ is proof-grounded.  The
Veblen order and Buchholz's countable order $\prec_{\mathrm{BH}}$ have
the three properties.  So $\iga+\mathrm{AccBiv}_{\mathrm{BH}}$, \iga
with bivalence of accessibility for $\prec_{\mathrm{BH}}$, is
proof-grounded: it has the disjunction and witness properties, it is
consistent, and it derives transfinite induction along
$\prec_{\mathrm{BH}}$.
\end{thm}

Buchholz's notation for the Bachmann--Howard ordinal has two levels:
codes below a symbol $\Omega$, which stands for the first uncountable
ordinal, and codes at or above $\Omega$, which serve to build the
countable ones.  The order $\prec_{\mathrm{BH}}$ is the notation's
order on the codes below $\Omega$.  Its order type is the
Bachmann--Howard ordinal by Buchholz's identification, which the
development does not mechanize.  The Veblen instance of
\cref{thm:rung} re-derives the theorem of step 52, and a build-time
check confirms that the two statements are identical.  At both orders
the classical warrant and \iga's omega-soundness each enter once,
through the pairs at which \iga does not refute the order, and the
rule's own omega-soundness is never used.  Like
$\iga+\mathrm{AccBiv}_\varphi$, $\iga+\mathrm{AccBiv}_{\mathrm{BH}}$
is sound for the omega model and consistent
(\isn{givbh\_WF\_bh\_price}), so it belongs to the grounded family.
At $\prec_{\mathrm{BH}}$ the system's own proof of $\WF$, the third
hypothesis, proves the countable level of the notation well-founded;
\cref{thm:bh} adds the full order.

\begin{thm}[\isn{g\_bh\_ladder\_iff\_accbiv}, \isn{givbh\_WF\_full}]
\label{thm:bh}
In every calculus with ground induction, deriving $\WF$ for the
Veblen order is equivalent to admitting bivalence of accessibility for
it, and four statements are equivalent at Buchholz's notation:
admitting bivalence of accessibility for $\prec_{\mathrm{BH}}$;
deriving $\WF$ for $\prec_{\mathrm{BH}}$; deriving $\WF$ for the full
order of the notation, the codes at or above $\Omega$ included; and
deriving the notation's ladder sentence, that every code below each
term of a fixed fundamental sequence of the Bachmann--Howard ordinal
is accessible.
At \iga, deriving that the matching code in $\mathcal O$ is a
notation is a fifth, as in \cref{thm:triangle}.
$\iga+\mathrm{AccBiv}_{\mathrm{BH}}$ derives $\WF$ for the full order.
\end{thm}

The one new argument is for the full order.  A code that is not in
canonical form inherits accessibility from its canonical form, by
ground induction along the full order, and bivalence is used only at
the countable level.  That the full order's type exceeds the
Bachmann--Howard ordinal is classical and not mechanized.  A
build-time check confirms by name that the system's own proof of
$\WF$ for $\prec_{\mathrm{BH}}$ reaches no omega-model, classical
warrant or HOL-predicate fact: that proof is a grounded claim, a
derivation in the system (Appendix~\ref{app:stage}, step 56).  Bivalence at the one
family of accessibility atoms that the proof uses is already
equivalent to the rule (\isn{ObKm\_iff\_AccBivB}); whether its instances at the four motive
shapes of the proof are weaker is open, as it is for the instances at
finite first arguments at the Veblen order (\S\ref{sec:strength}).

\paragraph{The stratified rung.}
Above the Bachmann--Howard ordinal the rule form developed here uses
more than one order.  Take Buchholz's system $\mathrm{OT}(k)$ of \S\ref{sec:one-rule}
for a level bound $k \geq 1$.  The inductive-definitions proof of
well-foundedness builds one accessible part for each level, and from
level $2$ on, the field of each accessible part tests membership in
the accessible part of the level below, so the proof needs bivalence
at that lower accessible part as well; one rule at the countable order
does not supply that bivalence by this route.  The tower of Kripke
universes of \S\ref{sec:tower} supplies the same need in truth form, by a truth
term that reads the verdicts of the truth term below it negatively as
well as positively.  The \emph{components}
of a code $y$ at level $u$ are the principal codes $D_v a$ with $v \leq
u$ that are reached from $y$ through sums and through principal codes
of level above $u$ (the sets $K_u$ of the Bachmann--Howard
development, here at every level), and $\mathrm{km}_u(y)$ is
the largest of them, or $0$ when there is none.  The
\emph{stratified orders} are defined, for a fixed $k$, by recursion:
$R_0$ is the order
of the countable part of $\mathrm{OT}(k)$, and for $j \geq 0$ the order
$R_{j+1}$ holds between $y$ and $x$ when $y \prec x$ in $\mathrm{OT}(k)$
and $M_{j+1}(y)$, where $M_1(y)$ says that $\mathrm{km}_0(y)$ is
accessible along $R_0$, and $M_{j+2}(y)$ says $M_{j+1}(y)$ and that
$\mathrm{km}_{j+1}(y)$ is accessible along $R_{j+1}$.  Unlike $R_0$,
the orders $R_{j+1}$ contain accessibility atoms, atomic formulas
$\Acc_R\cdot d$, and the development
does not use their decidedness.  The \emph{stratified rung} at level $k$ is \iga with bivalence
of accessibility for each of $R_0, \ldots, R_{k-1}$ as primitive
rules, written $\iga+\mathrm{AccBiv}(R_0,\ldots,R_{k-1})$.

\begin{thm}[\isn{givRs\_gsem\_sound}, \isn{givotk\_gsem\_sound}]
\label{thm:otnu-rule}
Let $\mathcal R$ be a finite list of closed, well-formed orders such
that, for each $R$ in $\mathcal R$, the pairs of numbers at which \iga
does not refute $R$ form a well-founded relation in HOL, and \iga with
bivalence of accessibility for every order in $\mathcal R$ derives
$\Acc_R\cdot x$ from the hypothesis $x\,\mathbf N$.  Then that system
is proof-grounded.  For every $k \geq 1$,
$\iga+\mathrm{AccBiv}(R_0,\ldots,R_{k-1})$ derives $\WF$ for the
countable part of $\mathrm{OT}(k)$ and for every stratified order, and
it is proof-grounded: it is sound and complete for the
certificate-grounded semantics with its own proofs as the
certificates, and it does not derive $0=1$.
\end{thm}

The derivations of $\WF$ are grounded claims.  They are proved in an
abstract calculus that takes the $k$ rules as assumptions
(\isn{g\_WF\_otk\_of\_accbivk}, \isn{g\_WF\_oRk\_of\_accbivk}), as for
\cref{thm:price}, and are then instantiated in the system
(\isn{givotk\_WF0}).  The rules are used only to decide the formulas
$M_i$, at their accessibility atoms; from level $1$ on, the accessible
parts along the stratified orders increase with the level, with no
rule; and one lemma covers the collapses at every pair of levels from
$1$ on.  The first sentence is classical, and it generalizes the
one-order theorem of step 52 (\isn{givR\_gsem\_sound}), behind
\cref{thm:rung}, to a list of orders.
Each order's numbers are accessible in the semantics by that order's
own certificate, the system's proof of $\Acc_R\cdot x$, which comes
from the system's own proof of $\WF$; nothing is iterated from one
order to the next.  The well-foundedness hypothesis comes, for a
decided order, from its classical warrant and \iga's decidedness of the
order at numerals, as at \cref{thm:rung}.  For an order whose
decidedness is not available it comes from a decided
\emph{envelope}: a decided order
with a classical warrant such that every refutation in \iga of the
envelope at two numbers is a refutation of the order there
(\isn{gi\_unrefuted\_wf\_of\_env}).  At the stratified rung, $R_0$ is
decided, and the envelope of every $R_{j+1}$ is the order $\prec$ of
the whole notation, whose warrant is the HOL well-foundedness of
$\mathrm{OT}(\omega)$ of \S\ref{sec:one-rule}; the warrant used for
every $R_{j+1}$ is thus the well-foundedness of the whole notation,
stronger than the level's own.  So \iga's
omega-soundness and disjunction property are used only at decided
orders, as at the Veblen and Bachmann--Howard rungs, and the
omega-soundness of the rules is never used, as a build-time check
confirms at $k = 2$.  At length one the list theorem
gives the statement of the one-order theorem behind \cref{thm:rung}
(\isn{givR\_gsem\_sound\_of\_list}), and at $k = 2$ the list system
and the two-rule system of the development coincide (\isn{otRs\_2}), so
the theorems are the two-rule rung's, in both cases with statements
that a build-time check confirms identical.
At $k = 2$ the development also states the disjunction and witness
properties and transfinite induction along $R_0$
(Appendix~\ref{app:stage}, step 63); at every $k$ they follow, by the generic
consequences of the certificate semantics and by transfinite induction
from $\WF$, but they are stated at $k = 2$ only.  At $k = 1$
the system is the one-rule rung at the countable order of
$\mathrm{OT}(1)$, the second coding of the Bachmann--Howard notation;
re-deriving the system of \cref{thm:bh} as this instance is not done.
A build-time check confirms that the derivations of $\WF$ reach no
soundness theorem, no classical warrant, no certificate-semantics fact
and no fact of the claim route or of Kleene's $\mathcal O$, so the two
routes to Buchholz's notation are independent (Appendix~\ref{app:stage},
step 63).  Whether fewer than $k$ rules suffice at level $k$ is open,
and so is a rule form for $\mathrm{OT}(\omega)$, which would need a
rule at every level and is not attempted.

The rungs do not climb by themselves.  A closing-off rule licensed by
the system's own derivation of an order's well-foundedness sentence
adds nothing, since by \cref{thm:bh} that derivation already gives the
rule, at the Veblen order and at $\prec_{\mathrm{BH}}$.  So, by this
argument, which is not a mechanized statement, a progression of such
rungs, each licensed by its predecessor, stays where it starts.  The rungs that climb are those with a classical
warrant, one richer notation per rung, and \cref{thm:otnu-rule}
prices each of them, now with a list of orders, at three things: for
a decided order, its warrant and its decidedness at numerals, and for
an order whose decidedness is not available, a decided envelope with a
warrant; and the system's own proof of $\WF$.

So the two questions come apart.  The decidedness rule gives up
proof-groundedness and leaves the accessibility ordinal where it is
(\cref{thm:oiga}).  Bivalence of accessibility at the Veblen order
keeps proof-groundedness and reaches the whole notation
(\cref{thm:price}), and so does bivalence at $\prec_{\mathrm{BH}}$
(\cref{thm:rung,thm:bh}) and, at the countable part of each
$\mathrm{OT}(k)$, the stratified rung (\cref{thm:otnu-rule}).

Restrict a closing-off rule, one that concludes a disjunction or a
decidedness, by a side premise that the system must derive, as the
license does in the autonomous progression.  For three such premises
the restricted rule changes neither answer: a derivation of one disjunct; for
bivalence of $\Acc_R\cdot d$,
derivations of $\Acc_R\cdot e$ and $R\cdot d\cdot e$ for a bound $e$;
and, for the decidedness of a universal whose instances are decided, a
derivation that the universal holds from a bound on.  Each restricted
rule is a derived rule of every calculus with \rga's rules
(Appendix~\ref{app:stage}, step 53), so it adds no theorem.  Adoption
is not a closing-off rule: it adds a universal whose numeral instances
are theorems of a lower level, and in the calibration devices that
grant its side condition the autonomous progression rises to
$\varphi(\omega+1,0)$ (\cref{thm:auto}).

\paragraph{Closing off a whole class.}
The rungs close off one order at a time.  The other extreme closes
off a whole class of least fixed points at once.  Call
\emph{$\mu$-bivalence} at a positive body $\Phi$ the rule that infers
$(\fix\,\Phi\cdot x)\,\mathbf B$, that the fixed point is decided at
$x$, from $\Delta \ent x\,\mathbf N$, and call a body
\emph{leaf-decided} when its parts other than the recursive atoms have
a value in the omega model at every assignment of numbers to its
variables.  The \emph{class rung} is \iga with $\mu$-bivalence at
every closed, positive, leaf-decided body that uses its recursive
variable.  The truth body is leaf-decided, and so is the body of
accessibility along every closed order decided at numerals.  Let $L$
be the Liar of the truth term, a sentence $\neg(\Tr\cdot p)$ whose
point $p$ \rga proves equal to $\langle 0,\langle\lceil
L\rceil,1\rangle\rangle$.

\begin{thm}[\isn{g\_mubiv\_contra}, \isn{mubivD\_consistent}]
\label{thm:mubiv}
Every calculus with \rga's rules and $\mu$-bivalence at every positive
body derives $0=1$.  The class rung is consistent and sound for the
closed-off model.  It is conservative over the omega model: every
sentence it derives that has a value in the omega model is true
there.  It admits bivalence of accessibility for every closed order
decided at numerals in the omega model.  Neither \iga nor \iga with
bivalence of the truth atom derives $\Tr\cdot p$ from the hypothesis
$L$, and the latter system is consistent.
\end{thm}

The unrestricted schema fails at a body outside the class: the Liar
of \S\ref{sec:bg} itself, taken as a body with no recursive atom.
Such a body is positive vacuously, but it neither uses its recursive
variable nor is leaf-decided; its fixed point unfolds to the Liar, and
the Liar decided gives $0=1$.  The \emph{closed-off model} is a third
semantic home, beside the omega model and the certificate-grounded
semantics.  At each fixed-point atom its true points are the least
fixed point, as in the omega model, and its false points are the
greatest fixed point of the falsity operator, so at a determinate
body a point that nothing grounds, which the omega model leaves
without a value, is false.  At a leaf-decided body the model is
therefore two-valued at numerals, which is what $\mu$-bivalence
needs, and it extends the omega model, which gives the
conservativity.  The development proves once, for every oracle, an
assignment of true and false points to each body, that is consistent,
coherent with the bodies and closed under substitution and
congruence, that the model the oracle defines satisfies the interface
of \rga's rules.  A model for greatest fixed points supplies one
instance and the closed-off oracle the other.  The closed-off model is
the latter's model read through a translation of fixed-point formers,
by a simulation proof, and ground induction is proved for it
separately.  The class rung and \iga with bivalence of the truth atom
are sound for the closed-off model (\isn{imb\_closed\_off\_sound},
\isn{itb\_closed\_off\_sound}), so closing off all decided orders at
once is consistent.  The Liar's point is not in the least fixed point
of the truth body, so $\Tr\cdot p$ is false in the closed-off model
and $L$ is true there, and a derivation of $\Tr\cdot p$ from $L$
would be sound for the model.  \iga with bivalence of the truth atom
has no rule of bivalence of accessibility, but it admits that rule at
Buchholz's countable order, which it proves well-founded
(\cref{thm:itb-bh} in \S\ref{sec:measure}).  Read from the grounded viewpoint, this
is the pattern of Kripke--Feferman truth: the truth predicate closes
off classically, while the step from a sentence to its truth fails at
an ungrounded sentence.  The three former predictions are theorems.
First, \iga with bivalence of the truth atom derives $L$ outright
(\isn{itb\_liar}).  The step is a self-application clause: from the
truth atom at a point, the truth term reports value $1$ at the code
of that atom.  \iga derives it by one ground induction over the truth
term (\isn{tr\_tin\_pair}), whose closure is a derivation in \rga:
the truth term evaluates the code of every compiled application
(\isn{tr\_cmp}) and introduces every clause of its own body at
quoted codes (\isn{sa1\_tmpl\_gen}).  At the Liar's point this is
\isn{g\_psa}, and the conditional \isn{g\_trbiv\_liar\_untrue\_psa}
turns it into $L$.  So the system derives $L$, derives neither
$\Tr\cdot p$ nor $\neg L$ (\isn{itb\_kf\_triple}), and \iga alone
does not derive $L$ (\isn{gi\_not\_liar}).
The same system derives \rga's consistency sentence of the G\"odel
paper~\cite{goedel-paper}, that no number is accepted by \rga's proof
checker as a proof of $0=1$, from no hypothesis (\isn{itb\_ConRGA}):
an internal soundness fold through \rga's checker table whose motive
is the truth atom at the compiled code of each node's conclusion.
Every step of the fold is a derivation in \rga under four closed
formulas, which the system derives, three by ground induction over
the truth term and the fourth by the bivalence rule
(\isn{nd\_fold\_RGA}, \isn{fki\_g\_kit\_biv}), and a build-time
check confirms that the rule enters only there and that no soundness
theorem of any system and no omega-model reading of the truth term is
in the proof's cone.  The same fold at \iga's own proof checker,
whose table adds the ground-induction rule, derives \iga's own
consistency sentence $\mathit{Con}_{\mathrm{IGA}}$ of the G\"odel
paper from no hypothesis (\isn{itb\_ConIGA}); the arm for the
ground-induction rule is a derivation in the plain fixed-point
calculus by one ground induction with a positivity-indexed motive
(\isn{fc\_T\_gen}), and the build-time check extends to it.  Since
\iga itself does not derive $\mathit{Con}_{\mathrm{IGA}}$, by
G\"odel's second theorem at \iga~\cite{goedel-paper}
(\isn{gi\_goedel2\_ConIGA}), closing off the truth atom strictly
exceeds \iga in consistency strength, a classical corollary
(\isn{fc\_separation}).

\oiga derives \rga's consistency sentence from no hypothesis as well,
with no rule added (\isn{oiga\_ConRGA}).  The fold has the same shape, with a \emph{staged truth} as its motive
and its arms re-proved (\isn{ndT\_fold\_RGA}).  The staged truth is a
copy of the truth term in which
every point carries a stage, a Cantor normal form below $\eps$ as in
\S\ref{sec:conpa}, and every truth atom in a clause's body is read at
a lower stage (\isn{trmTS}).  \oiga derives four staged closed
formulas (the bivalence of the staged truth, the head step with
determinacy, the inclusion in the truth term, and the swap), the
bivalence by
one transfinite induction along the Cantor-normal-form order with
\oga's decidedness rule deciding each clause (\isn{ts\_biv}), so the
decidedness rule supplies the bivalence that \iga with bivalence of
the truth atom takes as a rule; the build-time check confirms that no
soundness theorem and no bivalence of the truth atom is in the proof's
cone.  With G\"odel's second theorem for \rga~\cite{goedel-paper},
\oiga strictly exceeds \rga in consistency strength, a classical
corollary (\isn{oiga\_separation}).  The same fold at \iga's own
checker is not claimed: \iga's theorems include the accessibility of
every notation below $\varphi(\omega,0)$ (\cref{thm:ordinal}), whose
truths sit at stages cofinal in $\varphi(\omega,0)$
(\isn{acc\_stage\_ge}, by an argument not mechanized), and the staged
bivalence at stages that high is predicted to need transfinite
induction at $\varphi(\omega,0)$, which \oiga, of the same ordinal
(\cref{thm:oiga}), lacks.

The class rung is not sound for the omega
model: at the successor-reversed order, which is closed, decided and
ill-founded, it derives excluded middle for accessibility at $0$, a
closed sentence with no value in the omega model
(\isn{imb\_not\_omega\_sound}, \isn{imb\_stepD\_lem},
\isn{stepD\_lem\_not\_true}).  It is not proof-grounded:
certificate evaluation refines the omega model (\isn{gwosr\_omv}), so
every proof-grounded calculus has omega-true theorems
(\isn{gs\_omega\_true}), and no certificate semantics containing the
class rung's step is sound (\isn{imb\_not\_proof\_grounded}).  A
lower bound is a theorem: \iga's decidedness of the Veblen order and
of $\prec_{\mathrm{BH}}$ at numerals, with its omega-soundness, puts
both orders in the rule's class (\isn{gi\_relDec\_vOp},
\isn{gi\_relDec\_bOpC}), so the class rung derives $\WF$ for both
(\isn{imb\_WF\_gamma0}, \isn{imb\_WF\_bh}); it is strictly stronger
than \iga, which does not derive the first (\isn{imb\_gamma0\_strict}).
Its exact strength is open.

\begin{center}
\footnotesize
\begin{tabular}{@{}p{0.25\linewidth}>{\raggedright\arraybackslash}p{0.22\linewidth}p{0.45\linewidth}@{}}
\toprule
system & kind & what it proves about notations: accessibility ordinal on the Veblen notation unless stated \\
\midrule
\iga & grounded family; proof-grounded (\cref{thm:pg}) & $\varphi(\omega,0)$ exactly (\cref{thm:ordinal}) \\
\oiga: \iga with \oga's decidedness rule & grounded family; omega-grounded (step 51) & $\varphi(\omega,0)$ exactly (\cref{thm:oiga}) \\
$\oiga+\mathrm{Adopt}$: adoption over \oiga & grounded family; omega-grounded (step 55) & every notation below $\varphi(\omega,0)$, none at or above $\varphi(\omega,1)$ (\cref{thm:adopt}); $\varphi(\omega,1)$ exactly if it derives one omega-true warrant (\cref{thm:adopt-if}) \\
$\oiga+\mathrm{Refl}$: one reflection, warrant granted & calibration & $\varphi(\omega,1)$ exactly (\cref{thm:adopt}) \\
level $a < \varphi(\omega+1,0)$ of the progression & calibration & $\varphi(\omega,a)$ exactly (\cref{thm:levels}) \\
the autonomous progression & calibration & $\varphi(\omega+1,0)$ exactly, in both forms (\cref{thm:auto}) \\
the control, without the license & calibration & every valid notation, at some level (\cref{thm:levels}) \\
$\iga+\mathrm{AccBiv}_\varphi$ & grounded family; proof-grounded, with a classical warrant (step 52) & proves the whole notation well-founded (\cref{thm:price}) \\
$\iga+\mathrm{AccBiv}_{\mathrm{BH}}$ & grounded family (\isn{givbh\_WF\_bh\_price}); proof-grounded, with a classical warrant (\cref{thm:rung}) & on Buchholz's notation: proves both levels well-founded (\cref{thm:bh}); the countable level's type is the Bachmann--Howard ordinal by Buchholz's identification, not mechanized \\
\iga with the claim that $\hat c$ is a notation & grounded family: adoption of one claim, omega-sound and consistent (\cref{thm:otnu-claim}) & on Buchholz's notation with every finite level: proves the countable part of $\mathrm{OT}(\omega)$ well-founded (\cref{thm:otnu-claim}); its type is $\psi_0(\Omega_\omega)$ by the cited ordinal analyses, not mechanized \\
the stratified rung: \iga with bivalence of accessibility at $R_0$, \ldots, $R_{k-1}$, for $k \geq 1$ & proof-grounded, with a classical warrant (\cref{thm:otnu-rule}); its theorems are omega-true, by \isn{gs\_omega\_true} applied to \isn{givotk\_gsem\_sound} & proves the countable part of $\mathrm{OT}(k)$ well-founded (\cref{thm:otnu-rule}); its type is $\psi_0(\Omega_{k+1})$ by the cited ordinal analyses, not mechanized \\
the class rung: \iga with $\mu$-bivalence at every closed, positive, leaf-decided body that uses its recursive variable & sound for the closed-off model (\cref{thm:mubiv}); not omega-sound (\isn{imb\_not\_omega\_sound}) and not proof-grounded (\isn{imb\_not\_proof\_grounded}) & admits bivalence of accessibility at every closed decided order; proves the Veblen notation and Buchholz's countable notation well-founded (\isn{imb\_WF\_gamma0}, \isn{imb\_WF\_bh}); exact strength open \\
$\iga+\mathrm{Biv}(\Tr)$: \iga with bivalence of the truth atom & sound for the closed-off model (\cref{thm:mubiv}); not omega-sound (\cref{thm:itb-nostage}) & proves Buchholz's countable notation well-founded (\cref{thm:itb-bh}); exactly the Bachmann--Howard ordinal is predicted, the upper bound open and not mechanized (\S\ref{sec:measure}) \\
the Kripke universe at level $k$ of the tower, $\iga+\mathrm{Biv}(\Tr_1,\ldots,\Tr_k)$, for $k \geq 1$ & sound for the closed-off model and consistent (\cref{thm:one-model}) & proves the countable part of $\mathrm{OT}(k)$ well-founded (\cref{thm:ladder-k}); its type is $\psi_0(\Omega_{k+1})$ by the cited ordinal analyses, not mechanized \\
\bottomrule
\end{tabular}
\end{center}

\subsection{Measuring one closing-off}
\label{sec:measure}

Write $\iga+\mathrm{Biv}(\Tr)$ for \iga with bivalence of the truth
atom, the rule that infers $(\Tr\cdot x)\,\mathbf B$, that the truth
term is decided at $x$, from $\Delta \ent x\,\mathbf N$.  The system is the
smallest closing-off of the truth term: one rule, at one predicate.
The paragraph \emph{Closing off a whole class} of \S\ref{sec:pricing}
showed the system consistent, sound for the closed-off model and strictly
stronger than \iga in consistency strength.  This subsection measures
the system on the ordinal scale.  From below, $\iga+\mathrm{Biv}(\Tr)$ reaches
Buchholz's notation for the Bachmann--Howard ordinal, by a theorem.
From above, the stage method of \S\ref{sec:strength} has no case for
the rule, also by a theorem, and the upper bound is open; the
Bachmann--Howard ordinal exactly is offered as a prediction.

\paragraph{The lower bound.}
The rung $\iga+\mathrm{AccBiv}_{\mathrm{BH}}$ uses its rule only at
accessibility atoms $\Acc_{\mathrm{BH}}\cdot d$, to split cases on
them, at five places of the derivation of \cref{thm:bh}.  $\iga+\mathrm{Biv}(\Tr)$ has no such rule, but it has a
decided stand-in for the atom.  For a closed order $R$ and a number
$d$, the \emph{proxy} $P_R(d)$ is the truth term's verdict on the
accessibility claim at $d$: the truth atom
$\Tr\cdot\langle 0,\langle\lceil \Acc_R\cdot\bar d\rceil,1\rangle\rangle$, whose code is computed
from $d$ inside the system.  The bivalence rule decides the proxy at
every number, for every $R$, since the proxy is one truth atom.  Two
derivations, neither of which uses a bivalence rule, tie the proxy to
the atom.  Call a closed, well-formed order a \emph{decided compiled
order} when it has the form $\lambda y.\lambda x.\,B$ with $B$ built
from equations between arithmetic terms by negation and disjunction,
and every operator node of the order has at most two immediate
subterms.  Buchholz's
countable order $\prec_{\mathrm{BH}}$ is one.  The Veblen order is not
of this form, since its body applies the fixed points that define the
Veblen notation's validity and order.

\begin{thm}[\isn{ut\_acc\_tout\_dc}, \isn{ut\_acc\_tin\_dc}]
\label{thm:acc-tinout}
In every calculus with \rga's rules and ground induction, for every
decided compiled order $R$, every well-formed context $\Delta$ and
every $d$ with $\Delta \ent d\,\mathbf N$: if $\Delta \ent P_R(d)$ then
$\Delta \ent \Acc_R\cdot d$ (T-out), and if $\Delta \ent \Acc_R\cdot d$
then $\Delta \ent P_R(d)$ (T-in).
\end{thm}

The development proves the two directions in general form
(\isn{ug\_acc\_tout\_gen}, \isn{ut\_acc\_tin\_gen}), each with one
premise about the truth term at the order's atoms in place of the
class, and discharges both premises for every decided compiled
order.  T-out is one ground
induction over the truth term's own fixed point.  Its motive gates the
conclusion, by a decided equality, at the four kinds of point that the
truth term passes through when it evaluates an accessibility claim: the
claim, its reduct after the beta step, the instances of the universal
in the reduct, and each instance's disjunction ``not below, or
accessible''.  The head steps of the fixed-point combinator never enter
the motive: they are handled by an \rga fact, the head form of the
code of a closed fixed point, together with the determinacy of the
truth term's values, which the motive carries.  T-in is one ground induction along accessibility,
whose step evaluates the truth term on the unfolded claim one clause
at a time.

The derivation of \cref{thm:bh} then moves to the proxy.  In general, a
\emph{proxy} for an accessibility predicate $A$ is a closed term $Q$
such that, at every number $d$, $Q\cdot d$ is decided and
interderivable with $A\cdot d$.  The development re-proves the Bachmann--Howard rung over an
abstract proxy (\isn{uq\_WF\_bh}): every calculus with ground induction
that has a proxy for the accessible part of $\prec_{\mathrm{BH}}$
derives $\WF$ for $\prec_{\mathrm{BH}}$.  Only one auxiliary order of the
derivation changes, the proxy standing where the atom stood inside it,
together with the motives in which the atom occurs negated.  The landed rung is the
instance at which the atom is its own proxy, with bivalence of
accessibility supplying the decidedness, and eleven of its statements, the headline
included, are re-derived through the proxy with statements that a
build-time check confirms identical.

\begin{thm}[\isn{itb\_WF\_bh}, \isn{itb\_accbiv\_bh}]
\label{thm:itb-bh}
$\iga+\mathrm{Biv}(\Tr)$ derives $\WF$ for Buchholz's countable order
$\prec_{\mathrm{BH}}$ from no hypothesis.  The system admits bivalence of
accessibility for $\prec_{\mathrm{BH}}$.
\end{thm}

The first sentence is a grounded claim, a derivation in
$\iga+\mathrm{Biv}(\Tr)$: the proxy is decided by the bivalence rule
and equivalent to the atom by \cref{thm:acc-tinout}, so the
proxy rung applies.  A build-time check confirms that no rule of
bivalence of accessibility is in the derivation's dependencies, by
name or by statement, and that the truth-atom bivalence rule enters
only through the proxy's decidedness.  The second sentence is
classical: in every calculus with ground induction, $\WF$ for
$\prec_{\mathrm{BH}}$ gives the rule, by the converse direction of
\cref{thm:bh}.  The same short argument gives the rule at any order of
the domain $\mathbb N$ that $\iga+\mathrm{Biv}(\Tr)$ proves
well-founded; the development states that step, in every calculus with
ground induction, at the Bachmann--Howard and Veblen orders, and the
general step is not stated separately.  At an order the system does not prove
well-founded, such as the successor-reversed order of
\S\ref{sec:pricing}, it is predicted not to admit the rule, because the
negative branch would need a negated fixed-point atom that has no
finite derivation; this prediction is not mechanized.  The same
proxies also give the second coding: $\iga+\mathrm{Biv}(\Tr)$ derives
$\WF$ for the countable part of $\mathrm{OT}(1)$
(\isn{uq2\_itb\_WF\_ot1}).  By Buchholz's identification, cited and
not mechanized, the order type is the Bachmann--Howard ordinal
$\psi_0(\Omega_2)$, the proof-theoretic ordinal of $\mathrm{ID}_1$.

\paragraph{The upper bound.}
The stage method bounds a system by showing that every closed theorem
is true at a stage of the omega model below the bound
(\S\ref{sec:strength}).  For $\iga+\mathrm{Biv}(\Tr)$ no such theorem
exists, at any bound.

\begin{thm}[\isn{up\_itb\_no\_stage\_sound}, \isn{up\_itb\_not\_omega\_sound}]
\label{thm:itb-nostage}
Some sentence that $\iga+\mathrm{Biv}(\Tr)$ derives from no
hypothesis is true at no stage of the omega model, and some such
sentence is not true in the omega model.
\end{thm}

Both parts are classical, and one sentence witnesses both: the system
derives the Liar $L$ of the truth term outright (\isn{itb\_liar},
\S\ref{sec:pricing}), and $L$ is not true in the omega model, since
the Liar's point is not in the least fixed point of the truth body
(\isn{up\_liar\_not\_true}).  So the obstacle is a failure of
soundness for the omega model, not of stage bookkeeping.  In the
method's terms, the rule's case of the step lemma would have to put $(\Tr\cdot x)\,\mathbf B$ at
some stage whenever $x$ is a number, and at the Liar's point it is at
none.  The remark that closes the autonomous progression's paragraph
in \S\ref{sec:pricing} has a companion here: the stage method cannot
price a rule whose system proves a sentence without a value in the
omega model.  Since every proof-grounded calculus has omega-true
theorems (\isn{gs\_omega\_true}), the same sentence shows that
$\iga+\mathrm{Biv}(\Tr)$ is not proof-grounded either; that
consequence is immediate and is not stated separately in the
development.

The natural repair, staging the closed-off model instead, is
predicted to meet an obstruction, and this paper does not attempt the repair.  At the truth body the false points of the closed-off model
are the numbers outside a least fixed point, and such a falsity has no
stage below any recursive bound.  T-out moves the truth term's stages
onto accessibility atoms, so a bound on the truths that the system
derives is at least its accessibility ordinal.  A bound would need
stages with a level standing for the undecided side and a collapsing
function back below it.  That is an ordinal analysis on the scale of Buchholz's
$\Omega$-rule and Pohlers's local
predicativity~\cite{buchholz81omega,pohlers81local}, redone for a grounded
calculus.  The bound offered instead is a prediction, and only that:
the accessibility ordinal of $\iga+\mathrm{Biv}(\Tr)$ on Buchholz's
notation is predicted to be exactly the Bachmann--Howard ordinal
$\psi_0(\Omega_2)$, with the lower bound mechanized
(\cref{thm:itb-bh}) and the upper bound open.  No argument for the
upper bound is offered here.  The closed-off model is an iterated
construction, a least fixed point taken over the closed-off model's
own reading of the other atoms, and whether it, and the system's
rules with it, formalize in a theory of one inductive definition is
the open question.  The predicted ordinal is the one that the classical landmark above assigns to
Kripke--Feferman truth with Cantini's minimality
principle~\cite{cantini89notes,burgess14friedman}; whether the
closing-off interprets into that theory, or that theory into the
closing-off, is likewise left open.

\paragraph{The contrast.}
The rung $\iga+\mathrm{AccBiv}_{\mathrm{BH}}$ reaches the same
notation and is proof-grounded (\cref{thm:rung}), and the stratified
rung reaches every $\mathrm{OT}(k)$ proof-groundedly
(\cref{thm:otnu-rule}).  $\iga+\mathrm{Biv}(\Tr)$ reaches the
Bachmann--Howard notation with theorems that are sound for the
closed-off model and not for the omega model.  That is the price of
closing off the truth predicate rather than one order, in one
sentence: the ordinal is the same, and every theorem of the rung has
its own proof as a certificate, while the closing-off
keeps only the closed-off model.  What the closing-off has that no
rung of this paper is shown to have is consistency strength, the
derivation of \iga's own consistency sentence (\isn{itb\_ConIGA},
\S\ref{sec:pricing}); and uniformity, since its one rule supplies a
decided proxy at every decided compiled order (\isn{uq\_proxy\_dc}),
where a rung adds one rule per order.

\subsection{The tower}
\label{sec:tower}

A closing-off can be iterated.  Call the system at each stage of the
iteration a \emph{Kripke universe}, indexed by its \emph{level}; the
step from one level to the next is a closing-off, which adds the
bivalence of the truth term of the level below.  The name records
Kripke's construction~\cite{kripke75outline}, in which closing off the
truth predicate makes it false of every sentence outside the fixed
point~\cite[p.~715]{kripke75outline}.  A Kripke universe is a system, not
a set: it is unrelated to a Grothendieck or Tarski--Grothendieck
universe, a set closed under the set-forming operations.  Kahle's
universes over Frege structures~\cite{kahle03universes} are the
closest precedent (\S\ref{sec:related}).
The next level's truth term must read the truth terms below it
negatively as well as positively, and the truth term of
\S\ref{sec:truth} cannot do that at an ungrounded point, by Kripke's
construction.  The construction below adds that reading as clauses.

\begin{defi}[\isn{ul\_TrL}, \isn{stratL}]
\label{def:level}
For a list $L = [T_1,\ldots,T_n]$ of closed terms, the \emph{level
truth term} $\Tr_L$ is the fixed point of a body generated from the
same table as the truth term's twenty-two clauses, with two changes.
The beta clause, which evaluates an application of a $\lambda$-term,
is guarded by a decided test that the function's code is not the code
of any $T_i$.  And for each $T_i$ there is an \emph{oracle pair} of
clauses: at the code of an application of $T_i$ to a term whose code
has value $v$ under $\Tr_L$, the value-$1$ point holds when $T_i\cdot
v$ holds, and the value-$0$ point when $\neg(T_i\cdot v)$ holds.  A list is
\emph{stratified} when it is $[\,]$ or $L$ followed by $\Tr_L$ for a
stratified $L$; write $\Tr_{j+1}$ for $\Tr_{[\Tr_1,\ldots,\Tr_j]}$, so
the stratified list of length $n$ is $[\Tr_1,\ldots,\Tr_n]$.
The \emph{Kripke universe at level} $n$ is \iga with bivalence of each
element of the stratified list $[\Tr_1,\ldots,\Tr_n]$ of length $n$,
written $\iga+\mathrm{Biv}(\Tr_1,\ldots,\Tr_n)$.
\end{defi}

The empty list gives the truth term exactly ($\Tr_{[\,]} = \Tr$,
\isn{ul\_TrL\_Nil}, an equation of HOL terms), so $\Tr_1 = \Tr$, the
Kripke universe at level $1$ is $\iga+\mathrm{Biv}(\Tr)$
(\isn{lvl\_itb\_derives}), and its
soundness for the closed-off model, its consistency and its
calculus facts are re-derived by the level route with statements that
a build-time check confirms identical.  $\Tr_2$ is the truth term with
an oracle pair reading $\Tr$, twenty-four clauses in all.  The body is
positive in its recursive variable, because the negation in the
oracle's false clause applies to the closed leaf $T_i\cdot v$ and not
to a recursive atom; so the fixed-point rule and ground induction
apply to $\Tr_L$.  A negated closed atom in a positive body lies outside
the clause grammar that the T-introduction laws of \S\ref{sec:pricing}
use, but it is exactly a leaf that the closed-off model decides.

\begin{thm}[\isn{MC\_TrL\_dec}, \isn{level\_sound}]
\label{thm:one-model}
The closed-off model decides every closed fixed-point atom, at every
argument with a value, whose body is positive and has leaves that the
model itself decides.  Consequently, if each element of $L$ is closed
and well-formed and the closed-off model decides it at every number,
it decides $\Tr_L$ at every number.
For every stratified list, the Kripke universe of that list is sound for the
closed-off model and consistent.
\end{thm}

All three sentences are classical.  The first is \isn{MC\_nu\_biv},
stated for the oracle model behind the closed-off model, at the atoms
into which the translation of fixed-point formers of
\S\ref{sec:pricing} carries the fixed points.  A lower level's truth atom is such a
leaf, positively and negatively, which gives the second sentence.  The
third follows by induction on stratified lists and the soundness of
\rga's rules and ground induction for the model, with one rule case
for bivalence, uniform in the element (\isn{level\_consistent} for
consistency).  So one model warrants the Kripke universe at every
finite level: the iteration lives in the systems and their truth terms,
not in the models.  Every Kripke universe from level $1$ on has the
bivalence rule of $\Tr$ among its rules, so every one of them derives
the Liar and, like the one at level $1$, is not sound for the omega model;
this consequence is immediate and is not stated separately in the
development.

\paragraph{The ladder.}
The kit of the lower bound carries to every Kripke universe.  The level truth
term is determinate: for every list of closed, well-formed fixed-point
terms, \iga derives that a code has at most one value under $\Tr_L$
(\isn{ul\_tr\_det}), by one ground induction along $\Tr_L$ with one
obligation per clause, the oracle clauses included, an \rga derivation
each.  T-out and T-in at the accessibility atom hold for $\Tr_L$, with
the proxy read by $\Tr_L$, for a list $L$ of closed, well-formed
fixed-point terms (in general form, \isn{ulv\_acc\_tout\_gen},
\isn{ulv\_acc\_tin\_gen}; for the class, \isn{ulv\_acc\_tout\_dc},
\isn{ulv\_acc\_tin\_dc}), for the decided compiled orders whose
leaves may also be elements of $L$ applied to arithmetic terms, and
whose accessibility predicate is not itself an element of $L$.  The
decidedness of those leaves enters as a premise, a closed formula
``every number decides $T_i$'' derivable in the context, so the
derivations use no bivalence
rule.  In a Kripke universe the rule enters only at two named points, the
derivation of those decidedness formulas and the proxy's own
decidedness, and a build-time check confirms that the rule enters
nowhere else.

\begin{thm}[\isn{ulv\_lvl\_WF\_bh}, \isn{lvlk\_WF\_otk}]
\label{thm:ladder-k}
For every $n \geq 1$, the Kripke universe at level $n$ derives $\WF$
for $\prec_{\mathrm{BH}}$ from no hypothesis.
For every $k \geq 1$, the Kripke universe at level $k$,
$\iga+\mathrm{Biv}(\Tr_1,\ldots,\Tr_k)$, derives $\WF$ for the
countable part of $\mathrm{OT}(k)$ from no hypothesis.
\end{thm}

Both sentences are grounded claims.  The second goes through a proxy
at every tier.  The derivation of the stratified rung
(\cref{thm:otnu-rule}) is re-proved over an abstract family of proxies
for the accessible parts of $R_0,\ldots,R_{k-1}$
(\isn{uq2\_g\_WF\_otk\_px}).  The landed derivation is the instance at
which the atoms are their own proxies, six of its statements re-derived
identically.  At level $k$ the tier-$j$ proxy is the verdict of
$\Tr_{j+1}$, the level truth term of the prefix of length $j$.  One
step needs a transfer.  The rung's literal order $R_{j+1}$ contains redexes,
the tier-$j$ proxy applied as a term among them, and lies outside the
class; its normal form, with the proxy's reduct as a literal leaf, lies
inside.  A refutation of either order at two numbers gives a
refutation of the other, by beta steps and with no bivalence rule, and
accessibility transfers between orders related in this way
(\isn{uq2\_acc\_transfer}).  A build-time check confirms that no rule
of bivalence of accessibility is in the dependencies, by name or by
statement.  The countable part of $\mathrm{OT}(k)$ has the order type
$\psi_0(\Omega_{k+1})$ of $\mathrm{ID}_k$, by the cited ordinal
analyses and not mechanized; so the accessibility ordinal of the
Kripke universe at level $k$ on Buchholz's notation is at least that of
$k$ iterated inductive definitions, for every finite $k$.  As at
level~1, the equality, with $\psi_0(\Omega_{k+1})$ as the accessibility
ordinal of the Kripke universe at level $k$, is
predicted and not mechanized; the upper bound is open.

The step from level $k$ to level $k+1$ is the oracle pair.  The
stratified rung needs bivalence at the accessible part of each lower
order because the field of the next order tests membership in it
(\S\ref{sec:pricing}, paragraph \emph{The stratified rung}).  In truth
form the same need is T-in at a negated lower truth atom: the tier-$j$
proxy appears inside the order of tier $j+1$, and the next proxy must
evaluate that lower proxy at value $0$ as well as $1$.  A truth
term cannot give its own atom the value $0$ at an ungrounded point, and the next level's oracle pair gives that value by
definition.  The same ordinals are thus reached twice: by $k$
accessibility rules, proof-groundedly, and by $k$ truth-term rules,
soundly for the closed-off model.  A Kripke universe at level $\omega$
is predicted to be out of
reach of this construction: a uniform family whose $n$-th term reads
the $m$-th negatively for every $m < n$ is not the body of one positive
fixed point, and would need a fixed-point former stratified by level.
Whether the Kripke universe at level $n+1$ derives the consistency of
the one at level $n$,
the iterated form of \isn{itb\_ConIGA}, is open at these levels: one
step of the fold, the congruence swap, fails at level~$2$ under one
named obligation.  \S\ref{sec:contower} on the consistency tower names
the obstruction and proves the iterated form for a constant-free
variant of these truth terms, which agrees with them at every argument
with a value (\cref{thm:cf-sep}).

\subsection{Conditional proof is inconsistent}
\label{sec:classical}

The deduction theorem of \rga holds exactly at decided hypotheses
(\S\ref{sec:bg}).  One might hope to extend \rga by a new conditional
with conditional proof at every hypothesis.  Curry's
paradox~\cite{curry42inconsistency} says no, and the fixed-point former
makes the argument five lines.

\begin{thm}[\isn{curry\_explosion}, \isn{ext\_curry\_explosion}]
\label{thm:curry}
Let $\Rightarrow$ be any term former over \rga's terms that preserves
well-formedness, respects substitution, and satisfies modus ponens
and conditional proof at arbitrary well-formed hypotheses, as derived
or admissible rules of \rga.  Then $\ent 0 = 1$, and since \rga is
consistent, no such former exists.  More generally, any proof system
over \rga's terms with hypothesis, head-step conversion in both
directions, modus ponens, and conditional proof for such a former
derives $0=1$, whatever else it validates.
\end{thm}

The Curry term $C = \fix(\lambda x.\, x \Rightarrow 0=1)$ unfolds to
$C \Rightarrow 0=1$.  From the hypothesis $C$, conversion gives $C
\Rightarrow 0=1$ and modus ponens gives $0=1$, using the hypothesis
twice; contraction is free because hypotheses form a set.  Conditional
proof discharges it: $\ent C \Rightarrow 0=1$; conversion backwards
gives $\ent C$; modus ponens gives $\ent 0=1$.  So the restriction of
the deduction theorem to decided hypotheses is forced by the presence
of the fixed-point former, and the exactness theorem of the
internal-reasoning paper tells how far it reaches.  A different argument, through the Liar with disjunction
elimination and double-negation elimination, shows the same for the definable implication $\neg a \vee b$.

The bivalence rule above is the deduction theorem at one particular
undecided hypothesis, and the Liar shows why it cannot be general:

\begin{thm}[\isn{accbiv\_shape\_explodes\_on\_liar}]
\label{thm:biv-liar}
\rga does not admit the bivalence rule with the Liar in place of
accessibility: if \rga's derivability were closed under it, \rga
would derive $0=1$.
\end{thm}

Accessibility survives where the Liar does not because accessibility
is never refutable (\cref{thm:acc-sem}), so the rule's conclusion
$\neg \Acc\cdot d \vee X$ is never a contradiction waiting to happen.
The rule is also harmless at any provably decided hypothesis, where
it is \rga's own deduction theorem.  The theorem covers the Liar; it says nothing about every other
undecided term.

\subsection{No internal soundness proof of \rga}

Could \iga prove \rga's consistency?  The obvious route is an
internal soundness proof: assign to each node of a proof the motive
``if the hypotheses are true then the conclusion is true'', with the
truth term of \S\ref{sec:truth} supplying truth, and prove it by
induction on the proof.  At a node whose hypothesis has no value, a
hypothesis such as a divergent term introduced by disjunction
elimination, that motive has no value, and no omega-sound system
proves it.  The development accompanying the internal-reasoning
paper~%
\cite{ford-conrga}
has made this a theorem, in an abstract calculus of omega-sound
extensions of \rga: inside it, the motive of the internal soundness
argument is not derivable at the hypothesis node of an explicit
five-node proof, and neither is any of the weaker motives that the
argument's cases could accept; the content of the theorem is
semantic, that the motive is not omega-true there.  \cref{thm:sound} is exactly the calculus' assumption for \iga, so its
theorems apply to \iga, and the instantiation is mechanized in a
session joining the two developments: the motive and the weaker
motives are underivable in \iga from omega-true hypotheses
(\isn{iga\_not\_motive}, \isn{iga\_not\_fold\_residual}).  The theorems cover one
argument and the weakenings of its shape, not every conceivable
internal argument.

\cref{thm:curry} sharpens the reason.  The motive needs conditional
proof at an ungrounded hypothesis, and conditional proof at
ungrounded hypotheses is inconsistent, not merely unsound.
Conditionals that validate ``$A$ implies $A$'' for a valueless $A$
do exist.  Field's conditional~\cite{field08saving} keeps structural
contraction and gives up conditional proof, which is the restriction \rga already
has.  The Łukasiewicz conditional keeps conditional proof and gives
up contraction, the rule that lets a hypothesis be used twice;
arithmetic with the unrestricted truth schema over Łukasiewicz logic
is consistent but $\omega$-inconsistent, proving every instance of
some universal together with the universal's
negation~\cite{restall92arithmetic,hajek00liar}.  But an induction over proof codes reuses its hypothesis at every
node, which is contraction; so, heuristically, a conditional that
survives Curry by dropping contraction cannot carry the induction
either.  The conjecture that
follows, that no consistent grounded extension of \rga proves \rga's
consistency by an internal soundness argument, is stated among the open questions of \S\ref{sec:concl}.  What remains open is the other classical route,
Gentzen's~\cite{gentzen36widerspruchsfreiheit}: transform proofs into
a fragment where every motive is decided, with an ordinal assignment,
and let \iga supply the descent principle.  That is where transfinite
induction along a notation is the right tool, and the ordinal
assignment for \rga's proofs is the open part.
For Peano Arithmetic the route is carried out in full (\cref{thm:pa}).

\subsection{The consistency tower}
\label{sec:contower}

\iga with bivalence of the truth atom derives \iga's own consistency
sentence by an internal soundness fold (\isn{itb\_ConIGA},
\S\ref{sec:pricing}).  The no-go of the previous subsection does not
reach that fold: the system is not omega-sound
(\cref{thm:itb-nostage}), and the fold's motive, the truth atom at the
code of each node's conclusion, is decided by the bivalence rule.
This subsection iterates the theorem along a tower parallel to that of
\S\ref{sec:tower}, so that each Kripke universe derives the consistency
sentence of the one below, which that one does not derive.  At the level
truth terms of Definition~\ref{def:level} the iteration meets an
obstruction at one step of the fold, under one named obligation.  At a
presentation of the level truth terms without the constant $0$, which
agrees with the original presentation at every argument with a value,
the obstruction cannot arise, and the iteration goes through.

\paragraph{The obstruction.}
The fold's \emph{congruence swap} replaces a subterm of a node's
conclusion by a subterm of the same value and carries the truth term's
verdict across the replacement.  At a level truth term $\Tr_L$ the swap
meets the oracle pair, which fires only at the exact code of an
element of $L$.  \rga's substitution rule, which replaces a term by a
provably equal one in any well-formed context, under binders included,
rewrites an element into variants with the same values.  A well-formed
term with a hole gives the truth term when the hole is filled by $0$,
and a different term when the hole is filled by a closed term that
\rga proves equal to $0$ (\isn{uc\_hole\_ex}).  The Kripke universe at
level~$1$ derives the
negation of the variant's Liar atom, by one substitution step from the
Liar (\isn{uc\_variant\_liar}), and the closed-off model reads that
atom as false (\isn{uc\_variant\_MC\_false}).  $\Tr_2$ has no oracle
clause for the variant, so it reads the variant only by unfolding the
variant's own clauses, and in the least model the variant's Liar atom
has no value, by the Kripke cycle of the Liar
(\isn{uc\_variant\_omv\_noval}).  The decisive half is one named
obligation, that $\Tr_2$ reads the variant's two Liar formulas by the
least model (\isn{ucSim}, stated in the development as a definition and
not proved).  Under that obligation, the Kripke universe at level~$2$
does not derive the level-$2$ truth of a sentence that the one at
level~$1$ derives
(\isn{uc\_no\_lvl2\_of\_sim}, \isn{uc\_OCONG\_of\_sim}), so the swap,
which would carry that truth from the truth term to its variant, fails
at level~$2$.  The derivation of the negated Liar atom is a grounded
claim, and the other facts are classical theorems, the last two
conditional on the obligation; that the same mechanism blocks the swap at every level
above the first is a prediction.  A repair that keys the oracle on
every code value-equal to an element's code, rather than on the one
code, makes two properties of the level term depend on each other: the
determinacy of its values, and their transport along value-equal codes.
No single induction is known that proves both properties, and the
partial development of that design, the session \texttt{IGA\_Univ\_RV},
is not used below.

\paragraph{The constant-free presentation.}
The variants exist because a truth term has closed subterms with
values: the constant $0$, and the numerals and codes built from it.
The substitution rule can replace such a subterm by any closed term
that \rga proves equal to it.  The \emph{constant-free presentation}
removes every closed subterm with a value.  For a body $\Phi$ whose
point is the variable $v$, the presentation $\Phi^\circ$ replaces every
occurrence of the constant $0$ in $\Phi$ by the term $\neg(v = v)$,
lifted under binders so that it names the same point at every depth.
A numeral $\bar k$, which is $k$ successors over $0$, becomes the tower
$S^k(\neg(v=v))$ of $k$ successors over the replacement, and the code
numerals and compiled terms inside the body are rewritten with
everything else.  The \emph{constant-free truth term} $\Trc$ is the
fixed point of the presentation of the truth term's body, and the
\emph{constant-free level term} $\Trc_L$ is the fixed point of the
presentation of the body of $\Tr_L$.  Nothing else changes: the
clauses, the beta guard and the oracle pair are those of
Definition~\ref{def:level}, read through the rewriting.  Three properties make
the presentation work.

\begin{itemize}
\item \emph{At a point with a value, the presentation changes no
  value.}  If $\Delta \ent v\,\mathbf N$, then
  $\Delta \ent \neg(v=v) = 0$ and $\Delta \ent S^k(\neg(v=v)) = \bar k$
  (\isn{cf\_zero\_eq}, \isn{cf\_num\_eq}).  So the beta clause's guard
  and the oracle's code test remain decided equalities between numbers
  (\isn{cf\_guard\_dec}), derivable exactly when the original tests
  are (\isn{cf\_guard\_iff}, \isn{cf\_key\_iff}), and every omega model
  gives the tower the value $k$ wherever the point has a value
  (\isn{cf\_M\_num}).  The two presentations therefore agree in the
  calculus at every point with a value, which is where the fold reads
  truth atoms.  In every
  calculus with \rga's rules and ground induction, for every list $L$
  of closed, well-formed terms, $\Trc_L\cdot x$ and $\Tr_L\cdot x$ are
  interderivable at every $x$ with $\Delta \ent x\,\mathbf N$
  (\isn{g\_cfL\_iff}; at the empty list, \isn{g\_trC\_iff}).  Each
  direction is one ground induction, and the passage between the two
  bodies is one use of the substitution rule at the point's equation
  $\neg(x = x) = 0$.
\item \emph{No closed subterm has a value.}  Call a term \emph{inert}
  when it is a $\lambda$-term, the recursor, $\cA$, the divergent
  constant $\bot$ or a fixed-point term with a well-formed body.  In a model that gives no
  value to $\lambda$-terms, the recursor and $\cA$, as the omega model
  and the closed-off model give none, no inert term has a value:
  $\bot$ diverges, and a fixed-point term reaches a $\lambda$-term in
  two head steps.  Every closed subterm of $\Trc$ is inert, and so is
  every closed subterm of $\Trc_L$ when the elements of $L$ are closed,
  well-formed and have only inert closed subterms
  (\isn{cf\_CF1\_trmTrC}, \isn{cf\_CF1\_cfTrL}).  The reason is
  syntactic.  A closed subterm of the presentation comes from a
  subterm of the original body that contains no $0$, since a subterm
  that contained $0$ now contains the point, which is free in it; and
  a scan of the truth term's clauses finds every closed zero-free
  subterm inert (\isn{cf\_scan\_trPhi}).
\item \emph{No variant exists.}  If filling the hole of a context, a
  variable that occurs in the context, with a term $a$ gives a closed
  term $T$, then $a$ is a closed
  subterm of $T$ (\isn{hole\_filler}).  When every closed subterm of
  $T$ is inert, $a$ has no value, so no equation $a = t$ is true in
  such a model, and refilling the hole with a closed term $t$ that the
  model makes equal to $a$ gives back $T$ itself
  (\isn{cf\_no\_hole\_M}).  The instances at $\Trc$ are
  \isn{cf\_no\_hole} for \rga's derivations, through omega-soundness,
  and \isn{cf\_no\_hole\_MC} for the closed-off model; at every element
  of a constant-free stratified list the fact follows with
  \isn{cf\_CF1\_cfTrL}.
\end{itemize}
The equations of the first property and the agreement are grounded
claims, rules in every calculus with \rga's rules (and ground induction,
for the agreement); the omega reading of the tower is classical.  The
second and third properties are classical: facts of HOL about the
syntax of the terms and about models.

The substitution rule thus cannot manufacture a value-equal variant of
an element, and the exact-code oracle is congruent by construction.
When the elements have only inert closed subterms, the swap's case at
the oracle pair is vacuous, already at the level term $\Tr_L$ of
Definition~\ref{def:level} (\isn{ul\_swap\_vac}), and the congruence
swap holds at $\Tr_L$ over such elements (\isn{ul\_cong\_swap}).  The
agreement carries the swap to $\Trc_L$, with the premises of the fold's
swap and syntactic facts about the elements (\isn{cf\_cong\_swap}).  The
swap thus needs only a property of the elements, that their closed
subterms be inert; the reader $\Trc_L$ is presented constant-free so that it can in turn be
an element one level up.  The presentation has a cost.  The tower is
over $\Trc$, not over $\Tr$, and the constant-free lists form a tower
of their own: $[\Trc]$ is not a stratified list in the sense of
Definition~\ref{def:level} (\isn{cf\_not\_stratL\_trmTrC}).  The agreement
carries positive facts, so every constant-free Kripke universe from
level~$1$ on derives $\WF$ for $\prec_{\mathrm{BH}}$
(\isn{cfv\_lvl\_WF\_bh}), and every constant-free Kripke universe is
sound for the closed-off model and consistent
(\isn{cfv\_lvl\_sound}, \isn{level\_consistentC}).  The bivalence
hypotheses differ, however.  The evident route to the bivalence of $\Tr$
from that of $\Trc$, by cases on the bivalence of $\Trc$, needs the
negative direction of the agreement, from
$\neg(\Trc\cdot x)$ to $\neg(\Tr\cdot x)$, and the agreement relates
derivable atoms only.  The development has the converse transfer of a
negation, from a decided $\Trc_L\cdot x$ and $\neg(\Tr_L\cdot x)$ to
$\neg(\Trc_L\cdot x)$ (\isn{g\_cfL\_neg\_of\_ul}), and no route to the
negative direction is known.  So the first step of the tower below is
$\iga+\mathrm{Biv}(\Trc) \ent \mathit{Con}_{\mathrm{IGA}}$, which stands
beside $\iga+\mathrm{Biv}(\Tr) \ent \mathit{Con}_{\mathrm{IGA}}$ of
\S\ref{sec:pricing}, and whether the two systems coincide is open.
Nor is the question narrowed by asking only for the sentences in
which neither truth term occurs: a truth term can be written as a
$\beta$-redex in which it does not occur, one head step from the term
itself, so the two systems prove the same such sentences exactly when
they are the same system (\isn{eqv\_bothfree\_iff\_equiv}; a
classical theorem about derivability).  Both systems are conservative
over the least model at the sentences that have a value there
(\isn{eqv\_sem\_conservative\_Tr}, \isn{eqv\_sem\_conservative\_TrC};
classical, through soundness for the closed-off model), which is not
agreement on theorems.

\begin{defi}[\isn{cfTrL}, \isn{stratC}]
\label{def:cf-level}
For a list $L$ of closed terms, the \emph{constant-free level truth
term} $\Trc_L$ is the fixed point of the constant-free presentation of
the body of $\Tr_L$, and $\Trc_{[\,]} = \Trc$.  A list is
\emph{constant-free stratified} when it is $[\,]$ or $L$ followed by
$\Trc_L$ for a constant-free stratified $L$.  Write $\Trc_{j+1}$ for
$\Trc_{[\Trc_1,\ldots,\Trc_j]}$, so that $\Trc_1 = \Trc$ and the
constant-free stratified list of length $n$ is
$[\Trc_1,\ldots,\Trc_n]$ (\isn{stratC\_iff}).  For a list $L$, write
$\iga+\mathrm{Biv}(L)$ for \iga with bivalence of each element of $L$.
The \emph{constant-free Kripke universe at level} $n$ is
$\iga+\mathrm{Biv}(\Trc_1,\ldots,\Trc_n)$, and at level $0$ it is
\iga.
\end{defi}

\begin{defi}[\isn{chkL}, \isn{ConL}]
\label{def:conl}
The \emph{level checker} $\mathit{chk}_L$ of a list $L$ is \iga's proof
checker of the G\"odel paper~\cite{goedel-paper} with one \emph{tag
rule} added for each element $T$ of $L$.  The tag rule accepts a node
whose conclusion is $(T\cdot x)\,\mathbf B$ and whose one premise
concludes $x\,\mathbf N$ in the same context.  The \emph{level
consistency sentence} $\mathit{Con}_L$ says that no number is accepted
by $\mathit{chk}_L$ as a proof of $0=1$.  At the empty list the checker
is \iga's, and $\mathit{Con}_{[\,]}$ is $\mathit{Con}_{\mathrm{IGA}}$
(\isn{ConL\_Nil}).
\end{defi}

\begin{thm}[\isn{cf\_lvl1\_ConIGA}, \isn{cf\_ConIGA\_L}]
\label{thm:cf-conIGA}
$\iga+\mathrm{Biv}(\Trc)$ derives $\mathit{Con}_{\mathrm{IGA}}$ from no
hypothesis.  For every constant-free stratified list $L$,
$\iga+\mathrm{Biv}(L)+\mathrm{Biv}(\Trc_L)$ derives
$\mathit{Con}_{\mathrm{IGA}}$ from no hypothesis.
\end{thm}

Both sentences are grounded claims.  The first sentence is the
constant-free counterpart of \isn{itb\_ConIGA}, and the development
states the two side by side (\isn{cf\_lvl1\_vs\_itb}).  A build-time
check confirms that the new derivation uses none of the one-level
fold's kit formulas, its arm for the ground-induction rule or the
bivalence rule of $\iga+\mathrm{Biv}(\Tr)$; the bivalence rule that the
new derivation uses is that of $\Trc$.  The first sentence is thus a
proof through the level construction and not a corollary of the
one-level theorem.  Since \iga does not derive
$\mathit{Con}_{\mathrm{IGA}}$, closing off $\Trc$ strictly exceeds \iga
in consistency strength, a classical corollary
(\isn{cf\_lvl1\_separation}).

\begin{thm}[\isn{cf\_lvl\_Con}, \isn{drvChkL\_fires\_iff}]
\label{thm:cf-lvlcon}
For every constant-free stratified list $L$,
$\iga+\mathrm{Biv}(L)+\mathrm{Biv}(\Trc_L)$ derives $\mathit{Con}_L$
from no hypothesis.  For every list $L$ of well-formed terms,
$\mathit{chk}_L$ accepts a node exactly when the node is a step of a
derivation in $\iga+\mathrm{Biv}(L)$.
\end{thm}

The first sentence is a grounded claim, and at the empty list it is
the first sentence of \cref{thm:cf-conIGA} again
(\isn{cf\_lvl\_Con\_Nil}).  The second sentence is
classical: the level checker is faithful, so $\mathit{Con}_L$ is the
consistency of $\iga+\mathrm{Biv}(L)$.  The faithfulness of one tag rule
is \isn{tag\_fires\_iff}; the faithfulness of the other rules is that of
\iga's checker, which the G\"odel paper's development proves
(\isn{drvChkI\_fires\_iff}).

\begin{thm}[\isn{cf\_lvl\_goedel2}, \isn{cf\_lvl\_goedel2\_internal}]
\label{thm:cf-goedel2}
For every constant-free stratified list $L$, $\iga+\mathrm{Biv}(L)$
does not derive $\mathit{Con}_L$.  For every list $L$ of well-formed
terms, $\iga+\mathrm{Biv}(L)$ derives, from the hypothesis
$\mathit{Con}_L$, that $\mathit{Con}_L$ has no proof accepted by
$\mathit{chk}_L$.
\end{thm}

The first sentence is classical: G\"odel's second theorem at the level
checker.  The G\"odel paper's development proves the theorem once for
any calculus that strips ill-formed hypotheses and any checker that
satisfies a list of laws.  Each Kripke universe strips hypotheses
(\isn{cf\_lvl\_strip}): a derivation remains a derivation when the
ill-formed hypotheses are dropped from its context.  The laws hold at
$\mathit{chk}_L$ by the development's table-generic transformers,
unchanged; the premises of the transformers are the faithfulness of
\cref{thm:cf-lvlcon}, a shift law of the level's table and the presence
of \rga's rules in that table.  The theorem's consistency premise is the
consistency of the Kripke universe itself, which comes from its soundness
for the closed-off model (\isn{level\_consistentC}).  At the empty list
the first sentence re-derives G\"odel's second theorem at \iga through
the level route (\isn{cf\_lvl\_goedel2\_Nil}), with a statement that a
build-time check confirms identical to that of
\isn{gi\_goedel2\_ConIGA}.  The second sentence has a grounded
statement, a derivation in the Kripke universe, but the proof that the derivation
exists uses the checker's faithfulness as a HOL fact, entering through
the necessitation step, as G\"odel's second theorem at \iga does in the
G\"odel paper.  A build-time check confirms that this proof uses no
soundness theorem, model reading or consistency fact.

\begin{thm}[\isn{cf\_separation}, \isn{cf\_tower\_strict}]
\label{thm:cf-sep}
For every constant-free stratified list $L$, $\iga+\mathrm{Biv}(L)$
does not derive $\mathit{Con}_L$, and
$\iga+\mathrm{Biv}(L)+\mathrm{Biv}(\Trc_L)$ derives $\mathit{Con}_L$
from no hypothesis.  Consequently, for every $n$ and all
$i < j \leq n$, the constant-free Kripke universe at level $j$ derives
everything that the one at level $i$ derives, from the same
hypotheses, and derives the level consistency sentence
$\mathit{Con}_{[\Trc_1,\ldots,\Trc_i]}$ of level $i$ from no
hypothesis, which the one at level $i$ does not derive.
\end{thm}

The non-derivations are classical (\cref{thm:cf-goedel2}), the
derivations of the consistency sentences grounded
(\cref{thm:cf-lvlcon}), and the inclusion classical: a Kripke universe's
entailments grow with the set of the list's elements
(\isn{cf\_lvl\_mono}).  The theorem is the consistency tower.  At the
constant-free presentation the closings-off of truth form a tower of
Kripke universes strictly increasing in consistency strength: each
proves the consistency of the one below and not its own.  The development also
states the step from level $i$ to level $i+1$ separately
(\isn{cf\_tower}).

\paragraph{What is not claimed.}
No equivalence of $\iga+\mathrm{Biv}(\Tr)$ and
$\iga+\mathrm{Biv}(\Trc)$ is claimed, and so no agreement on the sentences in which
neither truth term occurs, which is the same question
(\isn{eqv\_bothfree\_iff\_equiv}), and \cref{thm:cf-sep} is a
statement about the constant-free Kripke universes only; at the Kripke universes of
Definition~\ref{def:level} the iterated consistency theorem remains open.  Each
closing-off strictly increases consistency strength, by at least the
consistency of what it closes off.  No theorem says that a Kripke
universe adds exactly that consistency: the development proves no
equivalence of the Kripke universe at level $i+1$ with the one at
level $i$ together with that one's consistency sentence.  The ordinals
of the constant-free Kripke universes are not measured
here.  Their lower bound at $\prec_{\mathrm{BH}}$ carries over
(\isn{cfv\_lvl\_WF\_bh}), and the bounds at $\mathrm{OT}(k)$ of
\cref{thm:ladder-k} are stated for the Kripke universes of Definition~\ref{def:level}.

\paragraph{The fold at the level term.}
The derivations of \cref{thm:cf-conIGA,thm:cf-lvlcon} are the fold of
\S\ref{sec:pricing} with the constant-free level term $\Trc_L$ as its
reader, over \iga's checker table for the first theorem and over
$\mathit{chk}_L$ for the second.  As at one level, every arm of the
fold but the one for ground induction is a derivation in \rga under a
kit context of closed formulas, and the arm for ground induction is a
derivation in the calculus by ground induction.
At the level term the context holds four formulas, the bivalence of
$\Trc_L$, its determinacy, its head step and its congruence swap, and
the decidedness formula of each element of $L$.  The fold's interface,
a chain of locales whose assumptions are what the arms need, holds at
the level term from that context (\isn{cf\_fkitU}), so the arms
for \rga's rules of \S\ref{sec:pricing} apply unchanged; the arm for each tag rule is an \rga
derivation from the element's decidedness formula.  Each formula of
the context is derived in the Kripke universe from no hypothesis
(\isn{cfk\_g\_Gk}): the bivalence and the elements' decidedness by the
bivalence rule, the determinacy through the agreement with $\Tr_L$, the
swap from the congruence swap at $\Trc_L$, and the head step by one
ground induction over the level term, from the determinacy and swap
formulas and the elements' decidedness, with the bivalence rule
deciding the induction's antecedent.  Two steps
differ from the one-level fold.  First, the head steps of the fold's
interface, from a redex code to its head reduct and back, are guarded
by the well-formedness and genuineness of the code, that is, by the
condition that the code is the code of a well-formed term.  The guard
is needed at a level term with oracle clauses.  There, a function code
that is no element's code but head-reduces to one, applied to a
numeral that codes no term, is valued by the guarded beta clause, and
that valuation is a derivation in the Kripke universe (\isn{cfc\_tag2\_val},
\isn{cfc\_pt\_ne}).  In the least model nothing values the head reduct
of that application.  This reading of the model is not mechanized, so
the failure of the unguarded forward step at such a level term is a
prediction.  The
fold's use sites have the guard's data at every accepted node, and the
one-level instance satisfies the guarded steps by weakening.  Second,
the arm for the ground-induction rule meets a new case, in which the
fixed point at which a node applies ground induction is an element of
$L$, which the level term reads through its oracle pair rather than
through clauses.  That case is closed by one more ground induction, at the
element.  The induction needs the element's T-introduction at value $1$
with the node's motive in the recursive slot, and a quotation of the
motive's code.  The quotation is a \emph{lifting quotation}: a
quotation of a code that lifts the code's free variables under the
binders it passes, so that its substitution law needs no closedness of
the code.  That substitution law is composed of existing laws of the
code machinery, and the case closes without any argument about the
closedness of the motive's code (\isn{cfcRGT\_holds},
\isn{cf\_gfix\_arm\_final}).

Build-time checks confirm four properties of the dependencies.  No
oracle is under any of the theorems.  The dependencies of the
derivations of \cref{thm:cf-conIGA,thm:cf-lvlcon} contain no semantic
fact, such as a soundness theorem, a model's reading or a Kripke universe's
consistency, no faithfulness fact of the level checker and no part of
G\"odel's second theorem.  In those derivations the bivalence rule, the
swap and the agreement are reached only through the derivations of the
kit formulas and the final assembly.  No step of the fold proper, the
part that is a derivation in \rga and consists of every arm but the one
for ground induction, uses ground induction or the bivalence rule.  The presentation, its three properties, the agreement and the
congruence swap are the session \texttt{IGA\_Univ\_CF}, 32 theories and
10,100 lines; the kit, the fold at the level term, the level checkers
and the theorems of this subsection are the session
\texttt{IGA\_Univ\_CFCon}, 40 theories and 16,972 lines.

\section{Related Work}
\label{sec:related}

\paragraph{Theories of truth.}
Kripke's fixed-point construction~\cite{kripke75outline} is the
semantics; Feferman's Kripke--Feferman axioms~\cite{feferman91reflecting}
are its classical axiomatization, and Cantini's minimality
principle~\cite{cantini89notes,cantini96logical} is the induction
schema for the truth predicate that ground induction generalizes to
every positive body, and Burgess's analysis of its
strength~\cite{burgess14friedman} the classical calibration.  Halbach and Horsten's paracomplete
axiomatization~\cite{halbach06axiomatizing} and Halbach and Nicolai's
analysis of its cost~\cite{halbach18costs,nicolai18provably} are the
closest relatives of \iga's truth term and the source of the
comparison in \S\ref{sec:truth}.  Field's theory~\cite{field08saving}
and the Łukasiewicz-based theories~\cite{restall92arithmetic,hajek00liar}
are the conditionals that the inconsistency result of \S\ref{sec:classical} rules out for the present purpose; Burgess's $\Pi^1_1$-completeness of Kripke's least
fixed point~\cite{burgess86truth} is the classical measure of the class of inductive definitions of \S\ref{sec:induction}.  Halbach's survey~\cite{halbach11axiomatic} covers
the landscape.

\paragraph{Inductive definitions and fixed-point induction.}
Park's fixed-point induction~\cite{park69fixpoint} is the rule in its
original program-logic form, and Kozen's $\mu$-calculus~\cite{kozen83results}
its modal form; the least-fixed-point reading of recursive predicates
is the standard semantics of logic programs with negation, in
Fitting's Kripke--Kleene semantics~\cite{fitting85kripke} and the
well-founded semantics~\cite{vangelder91wellfounded}, both of which
are three-valued fixed points of the present kind.  Feferman's
theories of iterated inductive definitions~\cite{feferman70formal} and
the Buchholz--Feferman--Pohlers--Sieg volume~\cite{buchholz81iterated}
are the proof-theoretic setting; the Kleene--Spector
theorem~\cite{kleene55forms,spector61inductively} identifies the least
fixed points of positive arithmetic operators with the $\Pi^1_1$ sets.
Proof assistants' inductive definitions, with their positivity
requirement and generated induction principles, are the everyday form
of the same idea, and the closest analogy for a reader coming from
that side.  Moschovakis's monograph on positive inductive definitions over
abstract structures~\cite{moschovakis74elementary} is the natural
home of the oracle-reading operator and of the class question, and
Aczel's Frege structures~\cite{aczel80frege} and the applicative
theories of Feferman and J\"ager~\cite{feferman75language} are the
closest classical settings in which self-application, fixed points
and inductive generation coexist.  Goodstein's theorem, the hydra
game, and ordinal notations have been mechanized before:
Cast\'eran's hydra battles in Rocq~\cite{casteran22hydra} and the
ordinal libraries of Isabelle's Archive of Formal
Proofs~\cite{paulson22ordinal}.  Those developments prove the theorems in the
assistant's own classical or constructive logic; here the assistant
checks that a weaker object system proves them.

\paragraph{Ordinals and predicativity.}
The autonomous progressions of Feferman and
Sch\"utte~\cite{feferman62transfinite,feferman64systems,schutte65predikative},
which admit induction along an order only after proving it
well-founded, and Kreisel's analysis of ordinal
logics~\cite{kreisel58ordinal,turing39ordinals} place $\Gamma_0$ as
the limit of predicativity; Feferman's reflective closure of Peano
Arithmetic, $\mathrm{Ref}^*(\mathrm{PA})$, a truth-theoretic system
with a substitution rule rather than a progression of reflection
principles, has the same ordinal~\cite{feferman91reflecting}.
\S\ref{sec:gamma0}, on the price of the classical world, states the grounded form of the step past it as an equivalence.  Beyond $\Gamma_0$, Buchholz's collapsing functions and his notation system for them~\cite{buchholz86new}, from the ordinal analysis of iterated inductive definitions and of impredicative subsystems of analysis~\cite{buchholz81iterated,buchholz88proof}, supply the notations of the higher rungs of \S\ref{sec:gamma0}.  Reverse mathematics supplies the yardsticks for the
well-foundedness transfers of \S\ref{sec:wf}: ``$X$ well-ordered
implies $\omega^X$ well-ordered'' is equivalent to arithmetic
comprehension~\cite{girard87proof,hirst94reverse}, the $\varepsilon$
step to its $\omega$-jump strengthening, and the $\varphi(X,0)$ step
to arithmetic transfinite recursion, Friedman's theorem (see Marcone
and Montalb\'an~\cite{marcone11veblen}); which of these a
grounded system can state and prove is a natural next measurement.  Kreisel's warning that ordinal measures are meaningful only for
natural notations~\cite{kreisel58ordinal}, and Harrison's
pseudo-well-orderings~\cite{harrison68recursive}, bear on what a
grounded system can and cannot certify about a notation
(\S\ref{sec:wf}); Kreisel and L\'evy's equivalence of uniform
reflection with transfinite induction to
$\eps$~\cite{kreisel68reflection}, Schmerl's ordinals for its
iterations~\cite{schmerl79fine,schmerl82iterated,pakhomov21reflection},
and the folklore length $\varphi(2,0)$ of its autonomous
progression~\cite[Thm.~4.3.3]{horsten21truth}, are the classical
comparison for the reflection rule priced in
\S\ref{sec:pricing}.  Gentzen's consistency proof~\cite{gentzen36widerspruchsfreiheit}
is the route the no-go result of \S\ref{sec:classical} leaves open.
For Peano Arithmetic that route is carried out in full inside \iga (\cref{thm:pa}).

\paragraph{Iterated truth and universes.}
The tower of Kripke universes of \S\ref{sec:tower} has two precedents
in the type-free tradition.  Scott's truth definition for a model of
the $\lambda$-calculus is a monotone inductive
definition~\cite[p.~21]{scott75combinators}, and Scott iterates it into a
hierarchy of truth definitions, the $n$-th taking the earlier ones as
given and reading each of them both positively and negatively, which
he uses for a sequence of universes $V^{(n)}$ of classes that the
truth definition decides, each a member of the
next~\cite[pp.~24--25]{scott75combinators}.  Kahle's universes over Frege
structures~\cite{kahle03universes} axiomatize the idea over a total
applicative theory with a compositional truth predicate.  A universe
is a class closed under the truth conditions.  A universe $v$
reflects $u$ when it makes both the truth and the falsity of every
instance of $u$ true as membership~\cite[p.~202]{kahle03universes}, so
that negative statements about truth at one level
become positive statements at a higher one~\cite[p.~201]{kahle03universes}.
The limit axiom gives, for every class $f$, a universe that reflects
$f$~\cite[p.~203]{kahle03universes}, a class being a term each of whose
instances is true or false~\cite[p.~195]{kahle03universes}.  Reflection
is the oracle pair of Definition~\ref{def:level}, which reads $T_i
\cdot v$ at value $1$ and $\neg(T_i \cdot v)$ at value $0$, and the
premise that $f$ be a class is the ladder's decidedness premise,
``every number decides $T_i$'', which is Kahle's $C(T_i)$ and which a
Kripke universe supplies by its bivalence rule (\S\ref{sec:tower}).
No universe reflects itself, by a diagonal argument on the pattern of
the Liar~\cite[p.~204]{kahle03universes}; the tower's Liar-shaped
counterpart is that a truth term cannot give its own atom the value $0$
at an ungrounded point, while the next level's oracle pair gives that
value by definition (\S\ref{sec:tower}).  A separate fact, by
G\"odel's second theorem at the level checker and not by a diagonal on
truth, is that no constant-free Kripke universe derives its own
consistency sentence (\cref{thm:cf-goedel2}).  That the two constructions coincide at the core
is a prediction.  They differ in three ways.  Kahle's universes are
internal: one theory has the universe predicate and the term that
names a universe reflecting a class.  The Kripke universes are a
sequence of systems, each closing off the truth term of the one below,
and in this they are closer to Cantini's theory of truth with levels,
whose levels are indexed by external terms~\cite[pp.~192 and~201]{kahle03universes}.  Kahle's limit axiom is uniform in the class,
which builds transfinite hierarchies of
universes~\cite[pp.~192 and~201]{kahle03universes}, while the tower is a
finite list and its level $\omega$ is predicted to need a fixed-point
former stratified by level (\S\ref{sec:tower}).  And Kahle's outer
logic is classical, where \iga's is paracomplete.  Paracompleteness costs
nothing on the ordinal measure, as far as two different measures
compare: Kahle's theories have the proof-theoretic ordinals
$\varphi(1,0,0) = \Gamma_0$, $\varphi(1,\omega,0)$ and
$\varphi(1,\varepsilon_0,0)$ of the ternary Veblen function, by the
induction they admit~\cite[p.~221]{kahle03universes}, all below the
Bachmann--Howard ordinal, while every Kripke universe from level $1$ on
derives the well-foundedness of Buchholz's countable notation, of the
Bachmann--Howard order type (\cref{thm:ladder-k}; the order type is
cited), a lower bound on its accessibility ordinal.  The accessibility
ordinal, the measure of this paper, is distinct from Kahle's
proof-theoretic ordinals and from consistency strength
(\S\ref{sec:concl}), and its upper bound is open (\S\ref{sec:tower}).

\paragraph{The grounded family.}
\rga~\cite{rga-paper} is the base; the companion paper on internal
reasoning and consistency in \rga~%
\cite{ford-conrga} supplies the exact deduction
theorem, the internal-soundness no-go theorems cited in
\S\ref{sec:classical}, and the omega-model soundness machinery that the soundness proof of \S\ref{sec:induction} reuses; a further companion
paper~%
\cite{goedel-paper} proves G\"odel's theorems for \rga and its
extensions; Omega Grounded Arithmetic (\oga)~\cite{oga-paper} is the sibling
system: it adds a rule certifying that a universal over a decided
matrix is itself decided, which iterated up the quantifier hierarchy
decides every arithmetic sentence, so that \oga contains Peano
Arithmetic under the identity translation modulo the number
hypotheses on its variables, a theorem of the \oga development whose
home is the \oga paper.  Ground induction buys inductive definitions
and ordinals, and decidedness buys classical arithmetic.  \iga does
not contain \oga, since \oga decides every universal of the halting
family and \iga does not (\S\ref{sec:pricing}); whether \oga contains
\iga is open.  Grounded
deduction's
philosophy~\cite{ford24reasoning,ford25have} is the stance behind the
viewpoint labels.  Martin-L\"of type theory~\cite{martin-lof72intuitionistic,martin-lof80intuitionistic}
and Aczel's constructive set theory~\cite{aczelrathjen01czf} are the
constructive systems with which a strength comparison would be most
informative; the open questions of \S\ref{sec:concl} list it.

\section{Conclusion and Open Questions}
\label{sec:concl}

One rule, and nothing else, separates \iga from \rga: the elimination
rule for a positive self-referential term, which says that the
grounded meaning of a recursive definition is its least fixed point.
The rule is sound for the omega model and the system is consistent,
and \iga is proof-grounded: sound and complete for the
certificate-grounded semantics with its own proofs as the certificates
(\cref{thm:pg}).  With the rule, accessibility, well-foundedness, winning positions, and
grounded truth itself are single sentences with induction
principles; \iga proves Goodstein's theorem and the hydra game outright; certifies the Cantor and Veblen
notations as ordinals, on canonical codes, up to $\eps$ and every
$\varphi(n,0)$; and HOL's well-foundedness of the standard $\eps$
notation follows from the internal theorem.  The limits are as exact
as the results.  \iga's ordinal on the Veblen notation is $\varphi(\omega,0)$, the ordinal of PKF, by a stage-indexed refinement of the omega model; it proves the whole Veblen notation well-founded exactly when it admits bivalence of accessibility for that order, which it does not, and bivalence of accessibility is sound exactly when the order is well-founded, so that $\Gamma_0$ lies exactly one instance of excluded middle away, sound to add as a rule, with the rule, the sentence and
one membership claim in a
primitive-recursive form of Kleene's $\mathcal O$ interderivable;
conditional proof at arbitrary hypotheses is inconsistent; and the
internal soundness argument for \rga's consistency, in the form the
internal-reasoning paper analyzes, is unavailable to every omega-sound extension
of \rga.  Each limit names the same thing the classical logician has:
implication at a hypothesis that is not provably decided.  The
purchase climbs past $\Gamma_0$.  One membership claim in $\mathcal O$
buys the well-foundedness of the countable part of Buchholz's notation
with collapsing functions at every finite level, of order type
$\psi_0(\Omega_\omega)$ by the cited ordinal analyses, not mechanized, and \iga with
bivalence of accessibility at $k$ stratified orders is proof-grounded
and proves the countable part of the notation's levels up to $k$
well-founded, for every finite $k \geq 1$ (\cref{thm:otnu-claim,thm:otnu-rule}).
An adopted claim is tested only through its consequences: autonomy
over bare claims adds nothing to \iga, unrefutedness is no adequacy
rule for claims or for well-foundedness sentences, consistency with
the system is none for well-foundedness sentences (for claims, a
prediction), and G\"odel's second theorem holds at every sound adoption,
whose consistency sentence the sound claim built from its own proof
checker derives (\cref{thm:claim-autonomy,thm:claim-tests,thm:claim-g2}).
Closing off the truth predicate instead of one order reaches the same
notations.  \iga with bivalence of the truth atom proves Buchholz's
countable notation well-founded through a decided proxy for
accessibility, and the tower of its iterations, the Kripke universes,
each closing off the truth term of the one below, proves the countable
part of $\mathrm{OT}(k)$ well-founded at level $k$, every Kripke
universe sound for the one closed-off model; the price is the omega
model, which no Kripke universe from level $1$ on is sound for
(\cref{thm:itb-bh,thm:itb-nostage,thm:one-model,thm:ladder-k} and
\S\ref{sec:tower}).  At a constant-free presentation of the truth term
the Kripke universes also form a consistency tower: each derives the
consistency of the one below, which that one does not derive
(\cref{thm:cf-sep}).

The open questions fall into four groups.

\paragraph{Strength.}
\iga's accessibility ordinal on the Veblen notation is $\varphi(\omega,0)$ (\cref{thm:ordinal}), the ordinal of the paracomplete Kripke--Feferman theory PKF, and the whole distance from there to \iga's classical twin, Kripke--Feferman truth with Cantini's minimality, is excluded middle at the inductive predicates.  Two questions remain on the other measure: whether \iga's consistency strength equals PKF's, the soundness proof giving one inductive definition as the bound, and whether PKF interprets into \iga.  The stage method also prices the other rules of the grounded family (\S\ref{sec:pricing}): \oga's decidedness rule leaves the accessibility ordinal at $\varphi(\omega,0)$ (\cref{thm:oiga}); one reflection over \oiga costs at most one step, to $\varphi(\omega,1)$ (\cref{thm:adopt}), and exactly that step in the calibration device that grants the warrant; and the autonomous progression of reflection, with its warrants granted, closes at $\varphi(\omega+1,0)$ (\cref{thm:auto}).  Three questions remain there: whether \oiga with adoption derives the one omega-true warrant that would make its ordinal exactly $\varphi(\omega,1)$ (\cref{thm:adopt-if}), which needs reasoning about the system's proof search at a free variable; what a progression whose adoptions are keyed to a proof checker inside the system reaches; and whether transfinite induction at decided motives in \oiga reaches beyond $\varphi(\omega,0)$.  \oiga does derive sentences that \iga does not: \oiga decides every universal of the halting family, and \iga does not (\S\ref{sec:pricing}).  The grounded form of Kreisel and L\'evy's question, what the reflection rule is worth on the ordinal measure, now has an answer over \oiga: one reflection is worth one step, from $\varphi(\omega,0)$ to $\varphi(\omega,1)$, as one reflection over Peano Arithmetic is worth one step from $\eps$ to $\varepsilon_1$.  The answer is exact for the calibration device and an upper bound for the grounded system; over \rga the question is open.  Above $\Gamma_0$, on Buchholz's notation, three questions remain (\S\ref{sec:pricing}): whether fewer than $k$ bivalence rules suffice for the countable part of $\mathrm{OT}(k)$; whether a rule form reaches $\mathrm{OT}(\omega)$, which would need a rule at every level; and the mechanization of the order types $\psi_0(\Omega_{k+1})$ and $\psi_0(\Omega_\omega)$, which the paper cites.  For the closings-off of the truth term (\S\ref{sec:measure}, \S\ref{sec:tower}) the upper bounds are open: the stage method has no case for the bivalence rule (\cref{thm:itb-nostage}); the accessibility ordinal of the Kripke universe at level $k$ is predicted to be exactly $\psi_0(\Omega_{k+1})$, the Bachmann--Howard ordinal at level~1, with the lower bound mechanized up to the cited order type and no argument offered for the upper bound; a stage bound for the closed-off model would need an ordinal analysis with a collapsing function.  A Kripke universe at level $\omega$, which would reach toward $\psi_0(\Omega_\omega)$, is predicted to need a fixed-point former stratified by level.

\paragraph{Generality.}
Ground induction here is for unary bodies with numeral-expression
atoms; parameters are allowed, except inside the arguments of
recursive atoms.  Binary and mutual recursive variables, parameters
inside atom arguments, and bodies whose atoms take computed arguments
without the equality-guard routing, are natural extensions, each with
its own soundness case; for the last two, whether the routing is an
elaboration step that a different soundness proof makes unnecessary
is open.    The syntactic $\Pi^1_1$ class, and
Kleene's normal form in HOL, would make the Kleene--Spector
identification a theorem of the development rather than an analogy.

\paragraph{Consistency.}
The internal-soundness route to \rga's consistency is closed in the
form analyzed.  Two things would settle the rest, a theorem and a construction.  The theorem is the conjecture of \S\ref{sec:classical}, which sharpens
the no-go from omega-sound to all consistent systems: that no
consistent grounded extension of \rga proves \rga's consistency by
an internal soundness argument.  \cref{thm:curry} settles the case of
a conditional with conditional proof at every hypothesis and
contraction; what remains are conditionals that keep conditional
proof and drop contraction, for which the argument above is a
heuristic, so that the obstacle would be Curry's paradox rather than
omega-soundness.  The construction is Gentzen's route:
an ordinal assignment to \rga's proofs under which a cut-free
fragment carries the argument, with \iga supplying the descent.
\iga's half of that route is a theorem: if \iga proves a primitive
recursive sequence descending along a closed order from a point it
proves accessible, then \iga derives $0=1$; so, \iga being consistent,
on the Veblen notation no such sequence provably starts below
$\varphi(\omega,0)$ (Appendix~\ref{app:stage}, step 54).  The conjecture is a theorem to prove; Gentzen's route is a construction to find.

For Peano Arithmetic, Gentzen's route is carried out in full in \iga (\S\ref{sec:conpa}).

\paragraph{The consistency tower.}
\iga with bivalence of the truth atom derives \iga's own consistency
sentence (\isn{itb\_ConIGA}, \S\ref{sec:pricing}).  At the level
truth terms of \S\ref{sec:tower} the iterated form meets an
obstruction, under one named obligation.  The oracle pair fires only at the exact code of a lower
truth term, while \rga's substitution rule rewrites that term into
variants with the same values, which the next level reads only by
unfolding; under that obligation the fold's congruence swap fails
at level~$2$, and its failure at the higher levels is predicted
(\S\ref{sec:contower}).  The constant-free presentation
of the truth terms replaces the constant $0$ inside a body by the
negated self-equation of the body's point, so that no closed subterm
has a value and no variant exists.  At that presentation the fold goes
through, and each constant-free Kripke universe derives the
consistency sentence of the one below, which that one does not derive
(\cref{thm:cf-sep}).  Three questions remain open: whether \iga with
bivalence of the truth atom and \iga with bivalence of the
constant-free truth atom are the same system, which by the evident
route needs the negative direction of the agreement between the two truth terms;
whether the iterated form holds at the level truth terms of
\S\ref{sec:tower} themselves, by a repair of the oracle that does not
make the level term's determinacy and transport depend on each other;
and whether a Kripke universe adds exactly the consistency of what it closes off,
or more.

\paragraph{Artifact.}
Every theorem in this paper is checked in Isabelle/HOL, in sessions
that reject unproved assumptions at build time, with no axioms beyond
HOL's.  The two remarks of \S\ref{sec:induction} on the rule, that
dropping its number premise makes it unsound and that its strong form
is derivable, are checked as well.
Gentzen's route for Peano Arithmetic in \cref{thm:pa} is a theorem in full: the checker, its faithfulness, the reduction, the descent and the transfer.
The development, the sessions
\texttt{IGA}, \texttt{OIGA}, \texttt{OIGA\_Stage}, \texttt{IGA\_PGSem},
\texttt{IGA\_Closing}, \texttt{IGA\_RE}, \texttt{IGA\_NoGo},
\texttt{IGA\_Eval}, \texttt{IGA\_Gentzen}, \texttt{IGA\_CNF\_Kit},
\texttt{IGA\_ConPA\_Checker}, \texttt{IGA\_ConPA\_Buchholz}, \texttt{IGA\_ConPA\_Main}, \texttt{IGA\_Fold}, \texttt{IGA\_Fold\_Con}, \texttt{OIGA\_Fold},
\texttt{IGA\_\allowbreak OT\_\allowbreak Kit}, \texttt{IGA\_\allowbreak OT\_\allowbreak Claim}, \texttt{IGA\_\allowbreak OT\_\allowbreak Rule}, \texttt{IGA\_\allowbreak OT\_\allowbreak PG},
\texttt{IGA\_\allowbreak Univ}, \texttt{IGA\_\allowbreak Univ\_\allowbreak Lvl}, \texttt{IGA\_\allowbreak Univ\_\allowbreak OT}, \texttt{IGA\_\allowbreak Univ\_\allowbreak RV}, \texttt{IGA\_\allowbreak Univ\_\allowbreak CF}, \texttt{IGA\_\allowbreak Univ\_\allowbreak CFCon} and
\texttt{IGA\_\allowbreak Fuel} (five theories, 3,970 lines, over the session \texttt{PG\_\allowbreak Sets}) over the
sessions of \rga, of \oga and of the internal-reasoning and G\"odel
papers~%
\cite{ford-conrga,goedel-paper},
is recorded at the tag \artifacttag of the project repository.  The
development was produced by the author directing Claude models from
Anthropic, which mechanized the rule and its soundness proof, the
ordinal certifications, the certificate-grounded semantics of \iga,
the truth term, the equivalences of \S\ref{sec:gamma0}, the rungs above $\Gamma_0$, the closings-off of the truth term and their tower, the constant-free presentation and the consistency tower, the limits of adopted claims, and the consistency of Peano Arithmetic, with the author reviewing the definitions and
the theorem statements.  The
theorem names and theory files that this paper gives are those of the
development at that tag, and the build command, whose sessions pull in
every other session named above, is
\texttt{isabelle build -d hol OIGA\_Stage IGA\_Closing IGA\_RE
IGA\_NoGo IGA\_Eval IGA\_ConPA\_Main IGA\_OT\_PG IGA\_Fold\_Con
OIGA\_Fold IGA\_Univ\_OT IGA\_Univ\_RV IGA\_Univ\_CFCon
IGA\_Fuel}.
Appendix~\ref{sec:table} lists the theorems with their names.

\bibliographystyle{plainurl}
\bibliography{logic}

\clearpage
\appendix
\sloppy
\section*{Appendices: the headline proofs in detail}
\addcontentsline{toc}{section}{Appendices}

The eight appendices below give, for each headline theorem, the
statement as it stands in the Isabelle sources with every hypothesis,
the definitions it rests on, the proof as numbered steps that follow
the mechanized proof, with the inductions and their motives written
out, and a list of the lemmas a reader re-mechanizing the proof in
another assistant would need.  They are written for a competent
Masters or early PhD student without logic expertise.  Isabelle names
appear throughout, in \isn{sans serif}, so that each step can be found
in the sources.

\paragraph{Common vocabulary.}
Variables are de~Bruijn indices in the sources; the appendices write
named variables and say which index is meant where it matters.  A
term is \emph{well-formed} (\isn{wft}) when it is built correctly from
the constructors; \emph{closed} (\isn{closed}) when no index occurs
free; and index $k$ is \emph{free} in $t$ (\isn{free}) when it occurs
unbound.  Lifting (\isn{lift}) shifts the free indices of a term to
move it under a binder.  $\mu\Phi$ in these appendices is the main
text's $\fix\,\Phi$.  $b[u]$ substitutes $u$ for the binder's variable
of a body $b$ (the result of a beta step; \isn{bsub}).  $\sigma
\downarrow t$ applies a substitution $\sigma$ to every free variable
of $t$ at once (\isn{subf}); $\sigma'$ is $\sigma$ pushed under one
binder, leaving index $0$ alone and lifting the images of the others;
$\bar m :: \sigma$ sends index $0$ to $\bar m$ and index $v+1$ to
$\sigma\,v$.  An open term $c$ is \emph{omega-true under $\sigma$}
when $\sigma\downarrow c$ has value $1$; soundness is stated for every
substitution $\sigma$ of well-formed terms at once (\isn{oomjt}),
because a derivation's hypotheses and conclusion may be open, and the
induction has to apply its hypothesis at a substitution that
instantiates an eigenvariable.  Two terms are \emph{value-congruent}
(\isn{ovcr}) when they are equal except that, at some subterm
positions, they hold well-formed terms with the same omega value;
value congruence preserves omega values, which is what lets a
substituted argument be replaced by its numeral.  A \emph{head step} is one step of head reduction
(\isn{hstep}).  The omega model is a pair of mutually inductive
relations: $t$ reaches the head form $w$, written $\mathrm{hd}(t, w)$
(\isn{omh}), where $w$ is an abstraction with a body, the recursor
awaiting three, two or one arguments, or the universal combinator; and
$t$ has the value $r$, written $\mathrm{val}(t, r)$ (\isn{omv}), with
$1$ for true and $0$ for false.  There are twenty-one rules, eight for
head forms and thirteen for values (Appendix~\ref{app:truth} lists
them), and a term has at most one value (\isn{omv\_unique}).  A
sequent $\Gamma \ent c$ is derivable in \rga or, inside the locale
\isn{gi}, in \iga; $\Gamma$ is well-formed (\isn{wfc}) when its
members are; $c\,\mathbf N$ says $c$ is a number and $c\,\mathbf B$
that it is decided.  Conversion in both directions is \isn{ir\_conv}
and \isn{ir\_convE}.

\section{The soundness case of ground induction}
\label{app:gfix}

Files: \isn{IGA\_Positive.thy} (the transport lemma and the headline
lemma), \isn{IGA\_Proof.thy} (the rule), \isn{IGA\_Sound.thy} (the
soundness case and \isn{omega\_sound}).

\paragraph{Definitions.}
A body $\Phi$ has the recursive variable $a$ at index $1$ and the
point $x$ at index $0$; $\mu\Phi = \fix(\lambda a.\lambda x.\Phi)$
(\isn{trmGfix}).  $\Phi[Q/a]$ substitutes $Q$ for the recursive
variable under the point binder (\isn{asub}), and $\Phi[Q/a][t/x]$
substitutes both (\isn{gsub}).  $\mu\Phi\cdot x$ unfolds in three head
steps to $\Phi[\mu\Phi/a][x/x]$ (\isn{gfix\_step1} to
\isn{gfix\_step3}), and in the omega model $\mu\Phi$ reaches the head
form ``abstraction with body $\Phi[\mu\Phi/a]$''
(\isn{omh\_trmGfix}).  A numeral expression (\isn{nexp}) is a variable
other than $a$, zero, or a successor of one; the parametric grammar
(\isn{pnexp}) requires the variable to be bound inside the body.
Positivity (\isn{pos}, \isn{neg}; parametric \isn{posP}, \isn{negP}) is
as in \S\ref{sec:positive}; \isn{posP\_subf2} says that substituting
anything for the parameters preserves it.

The \emph{oracle reading} (\isn{osem}) is an inductive relation
$\mathrm{read}_S(k, \sigma, t, b)$: the term $t$, with recursive index
$k$ and environment $\sigma$, reads as the truth value $b$ when its
recursive atoms are read as membership in the set $S$.  Its clauses:
a subterm in which $k$ does not occur reads as the omega value of
$\sigma\downarrow t$ (\isn{os\_nf}); an atom $a\cdot u$ reads true when
$\sigma\downarrow u$ is literally a numeral $\bar m$ with $m \in S$
(\isn{os\_app}; the argument may not mention the recursive variable),
and no clause makes an atom read false; disjunction
and negation have their usual clauses; a universal reads true when
every numeral instance of its body does and false when some instance
reads false (\isn{os\_allT}, \isn{os\_allF}).  An environment is
\emph{admissible} for $\Phi$ at $k$ (\isn{gfenv}) when $\sigma k = T$
and every other free variable of $\Phi$ is sent to a numeral; the
parametric form (\isn{gfenvP}) asks for numerals only at the bound
variables, the indices below $k$.

The induction's motive has five parts (\isn{gfInv}), stated for a
value judgment $\mathrm{val}(t, r)$ with the set $S$ and the term $T =
\mu\Phi_0$ fixed:
\begin{enumerate}
\item (\isn{gfA}) if $t = T\cdot\bar m$ and $r$ is true, then $m \in
  S$;
\item (\isn{gfP}) for every positive $\Phi$, admissible $\sigma$ and
  $t = \sigma\downarrow\Phi$ with $r$ true, $\Phi$ reads true;
\item (\isn{gfN}) the same for negative $\Phi$, $r$ false, reads
  false;
\item (\isn{gfRP}) for every positive quantifier body $d$ and
  admissible $\sigma$ for $\cA\cdot\lambda d$, if $t$ is the instance
  redex $(\lambda(\sigma'\downarrow d))\cdot\bar m$ and $r$ is true,
  then $d$ reads true at the environment extended by $\bar m$;
\item (\isn{gfRN}) the same for a negative $d$ and $r$ false.
\end{enumerate}
The parametric motive (\isn{gfInvP}) replaces \isn{pos}, \isn{neg} and
\isn{gfenv} by their parametric forms.

\paragraph{Statements.}
\emph{Transport} (\isn{gfix\_transport}).  Let $T = \mu\Phi_0$ with
$\Phi_0$ positive and $\lambda a.\lambda x.\Phi_0$ closed, and let $S$
satisfy the closure condition: for every $n$, if $\Phi_0$ reads true
at the point $\bar n$ with recursive variable $T$ and atoms in $S$,
then $n \in S$.  Then every head judgment satisfies the trivial motive
and every value judgment $\mathrm{val}(t, r)$ satisfies the five-part
motive.  \emph{Set form} (\isn{gfix\_omv\_ind\_set}): under the same
hypotheses, if $\mu\Phi_0\cdot\bar n$ is true in the omega model then
$n \in S$.  \emph{Term form} (\isn{gfix\_omv\_ind},
\cref{lem:gfix-ind}): for an arbitrary term $P$, neither closed nor
well-formed, if for every $n$ the truth of $\Phi_0[P/a][\bar n/x]$
implies the truth of $P\cdot\bar n$, then the truth of
$\mu\Phi_0\cdot\bar n$ implies the truth of $P\cdot\bar n$.  The
parametric forms (\isn{gfix\_transport\_par},
\isn{gfix\_omv\_ind\_set\_par}, \isn{gfix\_omv\_ind\_par}) are the same
three statements with the parametric grammar and no closedness
hypothesis at all.

\emph{The rule} (\isn{pr\_gfix}).  From $\Phi$ parametrically positive,
$\Gamma \ent \mu\Phi\cdot x$, the closure premise $i\,\mathbf N,\
\Phi[\lambda c/a][i/x],\ \Gamma \ent c[i]$, and $\Gamma \ent x\,\mathbf
N$, with $\Phi$ and $c$ well-formed and $i$ not free in $\Gamma$, in
$\lambda c$, or in $\lambda a.\lambda x.\Phi$, conclude $\Gamma \ent
c[x]$.  \emph{Soundness} (\isn{omega\_sound}): if $\Gamma \ent c$ in
\iga, $\Gamma$ is well-formed, and every member of $\Gamma$ is
omega-true, then $c$ is omega-true.  \emph{Consistency}
(\isn{gi\_consistent}): $\emptyset \ent 0 = 1$ is not derivable.

\paragraph{Proof of the transport lemma.}
A simultaneous rule induction on the two omega relations
(\isn{omh\_omv.inducts}).  The head motive is trivial, so the eight
head cases are immediate.  In each of the thirteen value cases, prove
the five parts separately.  Two devices recur.  First, \emph{shape
analysis} (\isn{pos\_substE}, \isn{neg\_substE}): if $\Phi$ is positive
and $t = \sigma\downarrow\Phi$, then at least one of the following
holds: $k$ is not free in $\Phi$; $\Phi$ is an atom $a\cdot u$ with $u$
a numeral expression, and $t = \sigma k\cdot(\sigma\downarrow u)$;
$\Phi$ is a disjunction; $\Phi$ is a negation of a negative body;
$\Phi$ is a universal $\cA\cdot\lambda d$ with $d$ positive one index
up, and $t = \cA\cdot\lambda(\sigma'\downarrow d)$.  For a negative
body the atom case is absent.  Second, \emph{the non-occurrence case
is always immediate}: if $k$ is not free in $\Phi$, the value
judgment on $t$ is exactly the premise of \isn{os\_nf}.
\begin{enumerate}
\item \emph{Zero, successor, equality.}  $t$ is not an application, so
  part 1 is vacuous; in parts 2 and 3 only the non-occurrence shape
  matches; $t$ is not a beta redex, so parts 4 and 5 are vacuous.
\item \emph{Negation true} (from $\mathrm{val}(b, 0)$, conclude
  $\mathrm{val}(\neg b, 1)$).  Part 2: either non-occurrence, or $\Phi
  = \neg\Psi$ with $\Psi$ negative; admissibility passes to $\Psi$
  (\isn{gfenv\_mono}), the induction hypothesis' part 3 at $b =
  \sigma\downarrow\Psi$ gives that $\Psi$ reads false, and
  \isn{os\_notT} finishes.  The other parts are
  vacuous by the value or the shape.
\item \emph{Negation false}: symmetric, with part 2 of the hypothesis
  and \isn{os\_notF}.
\item \emph{Disjunction}: for part 2, $\Phi = \Phi_1 \vee \Phi_2$ and
  the hypothesis' part 2 on the true disjunct, then \isn{os\_orT1} or
  \isn{os\_orT2}; for part 3, both hypotheses' part 3 and
  \isn{os\_orF}.
\item \emph{Beta} (from $\mathrm{hd}(t_0, \lambda b)$ and
  $\mathrm{val}(b[u], r)$, conclude $\mathrm{val}(t_0\cdot u, r)$).  This
  is the central case.
  \begin{enumerate}
  \item Part 1.  Suppose $t_0\cdot u = T\cdot\bar m$ and $r$ is true.
    Then $t_0 = T$; determinacy of head forms (\isn{omh\_det}) with
    \isn{omh\_trmGfix} gives $b = \Phi_0[T/a]$, so $b[u] =
    \sigma_0\downarrow\Phi_0$ for the environment $\sigma_0$ sending
    the point to $\bar m$, the recursive variable to $T$, and index $v +
    2$ to the variable $v$ (\isn{bsub\_asub}).  This $\sigma_0$ is
    admissible: $\sigma_0 1 = T$, and every free variable of $\Phi_0$
    other than $1$ is $0$, by closedness of $\lambda a.\lambda x.\Phi_0$
    (\isn{closed\_body\_free}), and $\sigma_0 0$ is a numeral.  \emph{This
    is the only use of closedness.}  The hypothesis' part 2 at $b[u]$
    gives that $\Phi_0$ reads true, and the closure condition gives
    $m \in S$.
  \item Part 2.  Non-occurrence: immediate.  The negation and
    disjunction shapes are impossible, since $t$ is an application.
    Atom $\Phi = a\cdot u'$:
    $\sigma k = T$ by admissibility and $\sigma\downarrow u'$ is a
    numeral $\bar m$ (\isn{gfenv\_atom}, through
    \isn{nexp\_subf\_num}); so $t_0\cdot u = T\cdot\bar m$, part (a)
    gives $m \in S$, and \isn{os\_app} finishes.  Universal: $t_0$ would
    be $\cA$, which contradicts $\mathrm{hd}(t_0, \lambda b)$
    (\isn{omh\_A\_inv}).
  \item Part 3: non-occurrence, or the universal shape, which is
    excluded as in (b).
  \item Parts 4 and 5.  Suppose $t_0 = \lambda(\sigma'\downarrow d)$ and
    $u = \bar m$.  Then $b = \sigma'\downarrow d$ (\isn{omh\_lam\_inv}),
    so $b[u] = (\bar m :: \sigma)\downarrow d$
    (\isn{bsub\_subf\_lift0}); the extended environment is admissible
    one index up (\isn{gfenv\_bind}), and the hypothesis' part 2 or 3
    at $b[u]$ gives the result.
  \end{enumerate}
\item \emph{Recursor at zero and at successor.}  The head $t_0$ has
  the head form of the recursor awaiting one argument, so it is not $T$, not $\cA$ and
  not an abstraction; parts 1, 4 and 5 are vacuous, the atom and
  universal shapes are impossible, and only non-occurrence remains.
\item \emph{Universal true} (from $\mathrm{hd}(t_0, \cA)$ and every
  $u\cdot\bar n$ true).  Parts 1, 4 and 5 are vacuous as in the
  previous case.  For part 2: non-occurrence is immediate; the atom
  shape is impossible since $t_0 \ne T$; in the universal shape $\Phi =
  \cA\cdot\lambda d$ and $u = \lambda(\sigma'\downarrow d)$, and for each
  $n$ the hypothesis at $u\cdot\bar n$, which is a true value judgment,
  includes part 4, which says $d$ reads true at the environment
  extended by $\bar n$; then \isn{os\_allT}.
\item \emph{Universal false}: symmetric, with part 5 at the refuted
  instance and \isn{os\_allF}.
\end{enumerate}

\paragraph{Proof of the set form and the term form.}
\begin{enumerate}\setcounter{enumi}{8}
\item The set form: apply the transport lemma to the value judgment
  $\mathrm{val}(\mu\Phi\cdot\bar n, 1)$ and read off part 1 at $m = n$.
\item The term form: take $S = \{m \mid P\cdot\bar m \text{ is
  true}\}$ and apply the set form.  Its closure obligation, that an
  oracle-true reading at the point $\bar m$ puts $m$ in $S$, is
  discharged in three moves: \isn{osem\_term}, an induction on the
  oracle reading, converts a reading at the truth set of $P$ into the
  term reading, the omega-truth of $\Phi$ with $P$ substituted for the
  recursive variable; that term is $\Phi[P/a][\bar m/x]$; and the
  hypothesis gives $P\cdot\bar m$ true.
\end{enumerate}

\paragraph{The parametric forms.}
\begin{enumerate}\setcounter{enumi}{10}
\item Steps 1 to 8 repeat verbatim with the parametric grammar, the
  parametric admissibility and the parametric motive.  The one step
  that changes is 5(a): parametric admissibility asks for a numeral
  only at the bound variables, here the point alone, so no closedness
  is needed.  In 5(b) the atom's argument is a parametric numeral
  expression, whose free variables are all bound
  (\isn{pnexp\_free\_lt}), so admissibility still makes it a numeral
  (\isn{gfenvP\_atom}).  This is why a parameter may not appear in an
  atom's argument.
\item The set form and the term form follow as in steps 9 and 10.
\end{enumerate}

\paragraph{The soundness case.}
Soundness is proved as: for every derivable $H \ent c$ with $H$
well-formed and every substitution $\sigma$ of well-formed terms, if
every member of $\sigma\downarrow H$ is omega-true then
$\sigma\downarrow c$ is omega-true (\isn{oomjt}).  The proof is an
induction on the proof with a case analysis of the last rule
(\isn{gi\_step\_oomjt}): the forty-two \rga cases replay from the \rga
development, and the case of ground induction is new.  Fix $\sigma$
and write $\Phi_\sigma$ for the body with its parameters instantiated
(the substitution lifted under two binders), $c_\sigma$ for the motive
likewise under one binder, and $P = \lambda c_\sigma$.
\begin{enumerate}\setcounter{enumi}{12}
\item From the premise $H \ent x\,\mathbf N$, $\sigma\downarrow x$ has
  a value $v$ (\isn{omtrue\_NE}); the premise's validity also gives
  that $x$ is well-formed (\isn{valid\_wff\_exp}), which value
  congruence needs below.
\item From the premise $H \ent \mu\Phi\cdot x$, and since
  $\sigma\downarrow\mu\Phi = \mu\Phi_\sigma$ (\isn{subf\_trmGfix}),
  $\mu\Phi_\sigma\cdot(\sigma\downarrow x)$ is true.
\item Value congruence replaces $\sigma\downarrow x$ by the numeral
  $\bar v$: two well-formed terms with the same value are congruent
  (\isn{ovcr\_swap}), congruence is closed under application
  (\isn{ovcr\_appI}, \isn{ovcr\_refl}), and it preserves values
  (\isn{omv\_vcr\_ex}); so $\mu\Phi_\sigma\cdot\bar v$ is true.
\item $\Phi_\sigma$ is parametrically positive (\isn{posP\_subf2}).
\item \emph{Closure for $P$.}  For each $n$, suppose
  $\Phi_\sigma[P/a][\bar n/x]$ is true, and let $\tau$ be $\sigma$ with
  $i$ sent to $\bar n$.  Then $\tau$ is a substitution of well-formed
  terms; $\tau$ leaves $c_\sigma$ and $\Phi_\sigma$ unchanged, since
  $i$ is not free in $\lambda c$ or in $\lambda a.\lambda x.\Phi$; and
  $\tau$ makes every hypothesis of the closure premise's context true:
  $i\,\mathbf N$ because $\bar n$ has value $n$, the body hypothesis
  because $\tau\downarrow\Phi[\lambda c/a][i/x] =
  \Phi_\sigma[P/a][\bar n/x]$ (\isn{bsub\_asub},
  \isn{subf\_gsub\_par}), and the members of $H$ because $i$ is not
  free in them.  The induction hypothesis for the closure premise
  gives that $\tau\downarrow c[i]$, which is $c_\sigma[\bar n]$, is true
  (\isn{bsub\_subf\_commute}); one beta step gives $P\cdot\bar n$ true.
\item The parametric term form with steps 15 to 17 gives $P\cdot\bar
  v$ true; one beta step gives $c_\sigma[\bar v]$ true
  (\isn{om\_hstep\_rev\_omv} with \isn{hs\_beta}).
\item Value congruence back (\isn{ovcr\_bsub}) gives
  $c_\sigma[\sigma\downarrow x]$ true, which is $\sigma\downarrow c[x]$.
\item \isn{omega\_sound} is the case of the identity substitution;
  \isn{gi\_consistent} follows because $0 = 1$ is never omega-true.
\end{enumerate}

\paragraph{Lemmas a re-mechanizer needs.}
The omega model as a simultaneous inductive definition, with
determinacy of head forms and the two shape inversions; the
substitution algebra (\isn{bsub\_asub}, \isn{bsub\_subf\_lift0},
\isn{subf\_trmGfix}, \isn{subf\_gsub\_par}, \isn{bsub\_subf\_commute},
\isn{closed\_body\_free}); the grammar facts (shape analysis,
\isn{nexp\_subf\_num}, \isn{pnexp\_free\_lt}, \isn{posP\_subf2}); the
admissibility facts; \isn{osem\_term}; value congruence
(\isn{ovcr\_swap}, \isn{ovcr\_appI}, \isn{ovcr\_refl}, \isn{ovcr\_bsub},
\isn{omv\_vcr\_ex}); head-step invariance of values in the direction from redex to reduct (\isn{om\_hstep\_rev\_omv}); \isn{omtrue\_NE}
and the validity of premises; and the forty-two \rga cases of omega
soundness.

\paragraph{The case for the certificate-grounded semantics.}
\cref{thm:pg} needs the same case with certificates, proved in
session \textsf{IGA\_PGSem} (\isn{GSem\_IGA.thy}).  The lemma is
\isn{gfix\_gwosr\_ind\_par}, the term form of \cref{lem:gfix-ind} for
the evaluation of any certificate system containing \rga's rules.  A
true instance $\mu\Phi\cdot\bar n$ makes $P\cdot\bar n$ true when $P$
is closed under the body in the semantics and the fixed point is
provably included in $P$: every well-formed context that derives
$\mu\Phi\cdot v$ and that $v$ is a number derives $P\cdot v$, at
$\mu\Phi$ and $P$ and again at their lifts.  The two inclusion
hypotheses take the place of the oracle set.  The proof is the same
simultaneous rule induction with a five-part motive, the motive going
straight to $P$.  Only the two clauses of a universal differ.  A true
universal's certificate is moved from $\mu\Phi$ to $P$ by positive
replacement inside the calculus, one binder down, and a false
universal's refutation is moved the same way in the empty context.
The inclusions come from the rule itself (\isn{gwgf\_rep}): the
closure premise, under the substitution, is applied at a variable
fresh for the context, and its substituted hypotheses are discharged
by the completeness of the semantics.  The case holds for every system
whose steps contain \iga's (locale \isn{gsem\_gfix},
\isn{gfix\_step\_gowsjt}).  The other forty-two cases are \rga's,
proved once for every certificate system (locale \isn{gsem}, in
\isn{RGA\_GSem.thy} of session \textsf{RGA}).

\section{The characterization theorem}
\label{app:lfp}

File: \isn{IGA\_Lfp.thy}.  The proof uses the oracle reading, the
lemma \isn{osem\_term}, and the set form of the induction lemma from
Appendix~\ref{app:gfix}.

\paragraph{Definitions.}
The operator $F_\Phi$ (\isn{gfixOp}) sends a set $S$ to the set of $n$
at which the body reads true, in the oracle reading, at the point
$\bar n$ with the recursive variable sent to $\mu\Phi$ and atoms read
in $S$.  The value placed at the recursive variable is never
consulted, because the oracle reading reads that index only through
the atom clause, which tests membership.  $\lfp$ is HOL's least fixed
point of a monotone set operator, the intersection of all pre-fixed
points.

\paragraph{Statement.}
For $\Phi$ positive with $\lambda a.\lambda x.\Phi$ closed, the set of
$n$ at which $\mu\Phi\cdot\bar n$ is omega-true equals
$\lfp\,F_\Phi$ (\isn{gfix\_omv\_eq\_lfp}).

\paragraph{Proof.}
\begin{enumerate}
\item \emph{Monotonicity, for every body} (\isn{osem\_mono},
  \isn{gfixOp\_mono}).  If $S \subseteq S'$ and $t$ reads $b$ at $S$,
  then it reads $b$ at $S'$: an induction on the oracle reading, in
  which only the atom clause consults $S$, and it consults membership,
  which transfers.  No clause consults non-membership.  Hence
  $F_\Phi$ is monotone and $F_\Phi(\lfp\,F_\Phi) = \lfp\,F_\Phi$
  (Knaster--Tarski, \isn{lfp\_fixpoint}).
\item \emph{Oracle-to-term transport at the truth set}
  (\isn{osem\_gfix\_omv}).  Let $\mathit{Tr}$ be the truth set of
  $\mu\Phi$.  If $\Phi$ reads true at the point $\bar n$ with atoms in
  $\mathit{Tr}$, then $\mu\Phi\cdot\bar n$ is true.  \isn{osem\_term}
  with $P = \mu\Phi$ gives the omega-truth of $\Phi[\mu\Phi/a][\bar
  n/x]$; writing it as $(\Phi[\mu\Phi/a])[\bar n]$ (\isn{bsub\_asub}),
  the beta rule of the omega model with the head form of $\mu\Phi$
  (\isn{omh\_trmGfix}) folds it back to $\mu\Phi\cdot\bar n$.  This
  direction uses neither positivity nor closedness.
\item \emph{The truth set is a pre-fixed point} (\isn{gfix\_omv\_pre}):
  $F_\Phi(\mathit{Tr}) \subseteq \mathit{Tr}$, which is step 2
  restated.
\item $\lfp\,F_\Phi \subseteq \mathit{Tr}$ (\isn{lfp\_gfix\_omv}), by
  \isn{lfp\_lowerbound} on step 3.  This inclusion holds for every
  body.
\item $\mathit{Tr} \subseteq \lfp\,F_\Phi$ (\isn{gfix\_omv\_lfp}).  Let
  $n \in \mathit{Tr}$ and apply the set form of the induction lemma
  with $S = \lfp\,F_\Phi$: its closure obligation, that a true reading
  at the point $\bar m$ with atoms in $\lfp\,F_\Phi$ puts $m$ in
  $\lfp\,F_\Phi$, is $m \in F_\Phi(\lfp\,F_\Phi)$ by the definition of
  $F_\Phi$, and that set equals $\lfp\,F_\Phi$ by step 1.  This is the
  only step that uses positivity and closedness.
\item Antisymmetry of inclusion.  As a corollary, $F_\Phi(\mathit{Tr})
  = \mathit{Tr}$ (\isn{gfix\_omv\_fixpoint}).
\end{enumerate}

\paragraph{The anchors.}
For closed $R$, the truth set of $\Acc_R$ is $\lfp$ of the operator of
the accessibility body (\isn{acc\_omv\_eq\_lfp}), and that least fixed
point is HOL's accessible part of the not-refuted relation
(\isn{acc\_lfp\_eq}): one inclusion is the soundness of accessibility
from Appendix~\ref{app:gfix}'s lemma, the other an accessible-part
induction that splits on whether $R\cdot\bar y\cdot\bar x$ is
omega-false.  The winning-position term is handled the same way: its truth set
is the least fixed point (\isn{win\_omv\_eq\_lfp}), and that least
fixed point is HOL's inductively defined set of winning positions, by
an induction on that inductive definition (\isn{win\_lfp\_eq}).

\paragraph{Lemmas a re-mechanizer needs.}
\isn{osem\_mono} and Knaster--Tarski; \isn{osem\_term},
\isn{omh\_trmGfix}, the beta rule of the omega model,
\isn{bsub\_asub}; and the set form of the induction lemma.

\section{Well-foundedness of the Cantor-normal-form order}
\label{app:cnf}

File: \isn{IGA\_Ordinal.thy}, section ``G4 in GI: epsilon-0''; the
theorem is \isn{gi\_WF\_cnf} and its helpers are \isn{cnf\_zero\_acc},
\isn{cnf\_inv\_acc}, \isn{cnf\_Q} and \isn{cnf\_acc\_ranks}.
\isn{IGA\_Epsilon.thy} instantiates it at the compiled projections
(\isn{gi\_WF\_epsilon0}) and derives HOL's \isn{epsilon0\_wf}.

\paragraph{Definitions.}
$\Acc_R$, $\WF(D,R)$ and the domain $\mathrm{Nat} = \lambda x.\,(x = x)$,
true at every number, are as in \S\ref{sec:wf}; $R\cdot y\cdot x$ reads
``$y$ is below $x$''.  The order on the naturals (\isn{ltOp}) is a
closed recursor term computing $<$; \rga proves it well-founded by
number induction (\isn{rga\_WF\_omega}).  Two closed, well-formed,
provably total terms $\mathit{hi}$ and $\mathit{lo}$ read a code: its
\emph{rank} is $\mathit{hi}\,x$, its \emph{head} (the code of the
leading exponent) is $\mathit{hi}(\mathit{lo}\,x)$, and its
\emph{tail} is $\mathit{lo}(\mathit{lo}\,x)$; a code of rank $0$
denotes $0$ and otherwise $\omega^{\mathrm{head}} + \mathrm{tail}$.
The \emph{comparison} $L$ (\isn{cnfL}) is a fixed point that unfolds to
(\isn{ir\_cnfL\_unfold}): $y$ has rank $0$ and $x$ has not; or both
satisfy the rank guard $G$ (head and tail have smaller rank than the
code) and either $L(\mathrm{head}\,y, \mathrm{head}\,x)$, or the heads
are equal and $L(\mathrm{tail}\,y, \mathrm{tail}\,x)$.  \emph{Validity}
$V$ (\isn{cnfV}) is a fixed point that unfolds to
(\isn{ir\_cnfV\_unfold}): the rank is $0$; or $G$ holds, head and tail
are valid, and the tail has rank $0$, or the tail's head is below the
head, or the two heads are equal (the exponents weakly descend along
the spine).  The \emph{order} $E$ (\isn{cnfOp}) is $\lambda y.\lambda
x.\,V y \wedge V x \wedge L\,y\,x$, with unfolding and folding laws
(\isn{ir\_cnfC\_unfold}, \isn{ir\_cnfC\_fold}).  \emph{Relative
transfinite induction} along a term $R$ (\isn{giRelTI}) is a HOL
predicate on $R$: for every well-formed context and motive, if in
every extension $\Delta$ and at every number $z$ the motive at $z$
follows from the motive at every $R$-predecessor of $z$, then the
universal closure of the motive is derivable; its instance at
\isn{ltOp} (\isn{gi\_relTI\_ltOp}) is strong induction on the naturals
inside \iga.  \emph{Ground induction at a proved accessibility}
(\isn{gi\_TI}): from $\Gamma \ent \Acc_R\cdot x$, $x$ a number, and the
same progressiveness premise, the motive holds at $x$; it is one
application of the rule at the accessibility body.

\paragraph{Statement.}
In a context fixing $\mathit{hi}$ and $\mathit{lo}$ closed and
well-formed, with the hypotheses that \rga proves $\mathit{hi}\cdot x$
and $\mathit{lo}\cdot x$ numbers whenever it proves $x$ a number:
$\emptyset \ent \WF(\mathrm{Nat}, E)$ in \iga.

\paragraph{Auxiliary facts.}
\begin{itemize}
\item[(D1)] \emph{Decidedness, in \rga.}  $L\,y\,x$ and $V\,x$ are decided
  at numbers (\isn{cnfL\_B}, \isn{cnfV\_B}), each by \rga's strong
  induction on the rank of $x$, with $y$ universally quantified in the
  motive ``for every code of rank $z$ the matrix is decided'': one
  unfolding, the decidedness of \isn{ltOp}, and the hypothesis at the
  smaller ranks the guard guarantees; a conjunction whose first
  conjunct is refuted is decided.  $E\,y\,x$ is decided as the
  conjunction of the three (\isn{cnfOp\_B}).
\item[(D2)] \emph{Case split on a decided term} (\isn{g\_caseB}); an
  equality of numbers is decided (\isn{g\_eqTI}).
\item[(D3)] \emph{Accessibility introduction} (\isn{g\_acc\_intro}): if
  in every extension, for every number $j$, $\neg(E\,j\,x) \vee \Acc
  \cdot j$ is derivable, then $x$ is accessible; universal introduction
  followed by folding.  Every use below first splits on the decided
  $E\,j\,x$ (D1, D2) to obtain the disjunction.
\item[(D4)] \emph{Rank-$0$ codes are accessible} (\isn{cnf\_zero\_acc}):
  by D3, any $E\,y\,x$ unfolds to $L\,y\,x$, whose first disjunct needs
  $x$ of nonzero rank and whose second needs something below rank $0$,
  which \isn{ir\_ltOp\_zero} refutes.
\item[(D5)] \emph{Invalid codes are accessible} (\isn{cnf\_inv\_acc}),
  since $E\,y\,x$ asserts $V\,x$.
\item[(D6)] Logical glue: transport of an equation under a function
  (\isn{g\_app\_transport}), disjunctive syllogism, universal
  elimination, ex falso, conjunction projections.
\end{itemize}

\paragraph{The nest.}
Four inductions: (I) strong induction on the rank; (II) inside it,
ground induction on the accessibility of the head; (III) inside that,
strong induction on the rank again; (IV) inside that, ground
induction on the accessibility of the tail.
\begin{enumerate}
\item \emph{The theorem from the rank lemma.}  Universal introduction
  at $j$; instantiate the rank lemma (step 2) at $\mathit{hi}\,j$ and
  at $j$; the disjunct $\neg(\mathit{hi}\,j = \mathit{hi}\,j)$ is
  refuted by reflexivity, leaving $\Acc\cdot j$; weaken to the matrix
  of $\WF$.
\item \emph{The rank lemma} (\isn{cnf\_acc\_ranks}, induction I): $\ent
  \cA\cdot\lambda r.\,\cA\cdot\lambda x.\,\neg(\mathit{hi}\,x = r) \vee
  \Acc\cdot x$.  Relative transfinite induction along \isn{ltOp} with
  the motive $M(r)$ = ``every code of rank $r$ is accessible'' and
  hypothesis $M(r')$ for all $r' < r$.  Fix $x$ and split: on
  $\mathit{hi}\,x = r$ (the unequal case is immediate); on $V\,x$ (the
  invalid case is D5); on $\mathit{hi}\,x = 0$ (the zero case is D4).
  In the main case, $x$ is valid, nonzero, of rank $r$: unfolding
  $V\,x$ gives the guard, so the head has smaller rank, and the
  hypothesis at that rank gives $\Acc(\mathrm{head}\,x)$; apply step 3
  at $e_0 = \mathrm{head}\,x$ and instantiate the result at $n =
  \mathit{hi}\,x$ and $x$; the two equality disjuncts are refuted by reflexivity and the
  other two by the case hypotheses, leaving $\Acc\cdot x$.
\item \emph{The head lemma} (\isn{cnf\_Q}, induction II).  From
  $\Gamma \ent e_0\,\mathbf N$ and $\Gamma \ent \Acc\cdot e_0$,
  conclude $Q(e_0)$, where $Q(e)$ = ``for every $n$ and $x$: $x$ is
  invalid, or has rank $0$, or has head other than $e$, or has rank
  other than $n$, or is accessible''.  This is ground induction along
  $E$ at $e_0$ with motive $Q$; in the closure at $e$ the hypothesis
  is $Q(e')$ for every $e'$ with $E\,e'\,e$.
  \begin{enumerate}
  \item \emph{Use of the hypothesis}: if $u$ is a valid nonzero number
    with $E(\mathrm{head}\,u, e)$, then $\Acc\cdot u$: instantiate the
    hypothesis at $\mathrm{head}\,u$, then at $\mathit{hi}\,u$ and $u$,
    and refute the side disjuncts.
  \item \emph{Building $E$}: from $V\,a$, $V\,b$ and $L\,a\,b$, $E\,a\,b$
    by folding.
  \item \emph{Induction III}: relative transfinite induction along
    \isn{ltOp} with motive $P(n)$ = ``every valid nonzero $x$ with head
    $e$ and rank $n$ is accessible'' and hypothesis $P(n')$ for $n' <
    n$.  Fix $x$ and split, all decided, on $V\,x$, $\mathit{hi}\,x =
    0$, $\mathrm{head}\,x = e$ and $\mathit{hi}\,x = n$.  In the main
    case let $t = \mathrm{tail}\,x$; unfolding $V\,x$ gives the guard,
    so $t$ has rank below $n$, validity of the head, hence of $e$ by
    transport, validity of $t$, and the spine condition.
  \item \emph{Accessibility of the tail.}  If $t$ has rank $0$: D4.
    Otherwise the spine condition leaves two cases.  If
    $L(\mathrm{head}\,t, \mathrm{head}\,x)$: transport to
    $L(\mathrm{head}\,t, e)$, get $V(\mathrm{head}\,t)$ by unfolding
    $V\,t$, build $E(\mathrm{head}\,t, e)$ by (b), and conclude
    $\Acc\cdot t$ by (a).  If $\mathrm{head}\,t = \mathrm{head}\,x = e$:
    the guard gives $t$ of rank below $n$, and the hypothesis of
    induction III at that rank, instantiated at $t$, gives $\Acc\cdot
    t$.
  \item \emph{Induction IV, the jump}: ground induction along $E$ at
    $t$, from $\Acc\cdot t$, with motive $J(z)$ = ``every valid nonzero
    $x'$ with head $e$ and tail $z$ is accessible'' and hypothesis
    $J(z')$ for every $z'$ with $E\,z'\,z$.  Fix $x'$ and split, all
    decided, on $V\,x'$, $\mathit{hi}\,x' = 0$, $\mathrm{head}\,x' = e$
    and $\mathrm{tail}\,x' = z$; in the main case get $V\,e$ and $V\,z$
    by unfolding and transport.  Prove $\Acc\cdot x'$ by D3: for $y$
    with $E\,y\,x'$, split on $\mathit{hi}\,y = 0$ (then D4); otherwise
    unfold $L\,y\,x'$, whose first disjunct contradicts the case and
    whose second gives either $L(\mathrm{head}\,y, \mathrm{head}\,x')$,
    hence $L(\mathrm{head}\,y, e)$, $V(\mathrm{head}\,y)$ by unfolding
    $V\,y$, $E(\mathrm{head}\,y, e)$ by (b) and $\Acc\cdot y$ by (a); or
    $\mathrm{head}\,y = e$ and $L(\mathrm{tail}\,y, z)$, hence, with
    $V(\mathrm{tail}\,y)$ from unfolding $V\,y$, $E(\mathrm{tail}\,y, z)$
    by (b), $J(\mathrm{tail}\,y)$ by the
    hypothesis, and its instance at $y$ gives $\Acc\cdot y$ after
    refuting $\neg(\mathrm{tail}\,y = \mathrm{tail}\,y)$.
  \item Instantiate $J(t)$ at $x$, with $\mathrm{tail}\,x = t$ by
    reflexivity, to get $\Acc\cdot x$, which closes induction III and
    with it the motive $Q(e)$.
  \end{enumerate}
\end{enumerate}

\paragraph{Lemmas a re-mechanizer needs.}
The two induction engines, \isn{gi\_TI} and relative transfinite
induction at \isn{ltOp}; the unfold and fold laws of the three fixed
points; decidedness of $L$, $V$ and $E$ and of \isn{ltOp} and of
equality; \isn{ir\_ltOp\_zero}; the derived rules D2, D3 and D6; the
substitution bookkeeping for the nested motives, which in another
assistant is just the instantiation of a nested universal; and the
totality of $\mathit{hi}$ and $\mathit{lo}$ lifted into \iga.

\section{The $\Gamma_0$ equivalence, the one-claim theorem, and the triangle}
\label{app:gamma0}

Files: \isn{IGA\_Veblen.thy} (the equivalence), \isn{IGA\_VeblenAcc.thy}
(the one-claim theorem and the triangle), \isn{IGA\_VeblenO.thy} (the
$\Gamma_0$ code and the converse), \isn{IGA\_KleeneO.thy} (the
$\mathcal O$ term and its induction), \isn{IGA\_VeblenFS.thy} and
\isn{IGA\_VeblenFS2.thy} (the fundamental sequences and their
cofinality).  Every derivation is in \iga.

\paragraph{Statements.}
\emph{The equivalence} (\isn{gi\_WF\_gamma0\_iff\_bivalence}):
$\emptyset \ent \WF(\mathrm{Nat}, \mathrm{vOp})$ if and only if
$\mathrm{AccBiv}$, where $\mathrm{AccBiv}$ is the HOL proposition: for
every well-formed context $\Delta$, every $d$ with $\Delta \ent
d\,\mathbf N$ and every well-formed $X$, if $\Acc\cdot d, \Delta \ent
X$ then $\Delta \ent \neg(\Acc\cdot d) \vee X$, with $\Acc$ the
accessibility term of the Veblen order $\mathrm{vOp}$.  \emph{The
one-claim theorem} (\isn{o\_wf\_of\_gamma0}): $\gamma_0 \in \mathcal O,
\emptyset \ent \WF(\mathrm{Nat}, \mathrm{vOp})$; in context form, from
$\Gamma$ well-formed and $\Gamma \ent \gamma_0 \in \mathcal O$, $\Gamma
\ent \WF(\mathrm{Nat}, \mathrm{vOp})$.  \emph{The triangle}
(\isn{gamma0\_triangle}): $\WF$ derivable iff $\mathrm{AccBiv}$, and
$\WF$ derivable iff $\gamma_0 \in \mathcal O$ derivable.

\paragraph{Definitions.}
A Veblen code is a number with four compiled, provably total
accessors: rank, tag, and two fields.  Three shapes: zero (rank $0$),
principal $\varphi(F_1, F_2)$ (nonzero rank, tag $0$), and sum (nonzero
rank, nonzero tag), read as $p + t$ with head $p = F_1$ and rest $t =
F_2$; that the head is principal and the rest nonzero is part of
validity, not of the shape.  The guard
says both fields have smaller rank.  The comparison $y < x$ ($\mathrm{vLt}$) is one closed fixed point:
$y$ is zero and $x$ is not; or both satisfy the guard and one of
Sch\"utte's four shape cases holds.  Principal below principal,
$\varphi(a, b) < \varphi(c, d)$: $a < c$ and $b < \varphi(c, d)$; or $a =
c$ and $b < d$; or $\varphi(a, b) < d$.  Principal $y$ below sum $x$: $y
< F_1 x$, or $y = F_1 x$ and $F_2 x$ is nonzero.  Sum $y$ below
principal $x$: $F_1 y < x$.  Sum below sum: $F_1 y < F_1 x$, or $F_1 y
= F_1 x$ and $F_2 y < F_2 x$.  Validity $\mathrm{vV}$ is a closed fixed
point: a code is valid when it is zero, or it satisfies the guard,
both fields are valid, and either it is principal with $F_2 x < x$,
or it is a sum whose head is principal, whose rest is nonzero, and
which satisfies the normal-form condition: $F_2 x < F_1 x$, or $F_2 x
= F_1 x$, or $F_2 x$ is a sum with the same head as $x$.  The two
fixed points are defined one after the other, validity referring to
the already defined comparison.
$\mathrm{vOp} = \lambda y.\lambda x.\,\mathrm{vV}\,y \wedge
\mathrm{vV}\,x \wedge \mathrm{vLt}\,y\,x$ relates only valid codes.
Motives: $\mathrm{RC}(c, d)$ = ``every valid principal code with
fields $c$ and $d$ is accessible''; $\mathrm{RCall}(c) =
\cA\cdot\lambda e.\,\neg(\Acc\cdot e) \vee \mathrm{RC}(c, e)$, ``$\varphi(c,
\cdot)$ carries accessible arguments to accessible values'', whose
antecedent appears as a negated disjunct, not an implication;
$\mathrm{rkAll}(r)$ = ``every code of rank $r$ is accessible'';
$\mathrm{mL}(c, d) = \Acc\cdot d \wedge \mathrm{RC}(c, d)$;
$\mathrm{PA}(x, s)$ = ``every valid principal code below $x$ of rank
$s$ is accessible''; $\mathrm{kAll}(p)$ = ``every valid sum with head
$p$ is accessible'', with the variants $\mathrm{kTAll}(p, t)$ (rest
$t$) and $\mathrm{kRk}(p, r)$ (rank $r$).  For the one-claim theorem:
$\mathcal O$ is the $\mu$-atom of a body over tagged points, $\langle
0, a\rangle$ for ``$a$ is a notation'' and $\langle 1, \langle a,
b\rangle\rangle$ for ``$a$ is below $b$'', with seven clauses: the
zero code is a notation; the successor of a notation is a notation;
a limit code, which carries an index tree, is a notation when, for
every $n$, the $n$-th value of the decoded sequence exists, is a
notation, and is below the $(n+1)$-th; and four clauses for the
order, that $a$ is below its successor, that below is transitive
through a successor, that every value of a limit's sequence is below
the limit, and that below is transitive through a limit.  The
decoded sequence is the index tree compiled and applied, so
membership of a limit code says that a primitive recursive sequence
of notations increases; $\gamma_0$ (\isn{gamma0C}) is the limit
code of the sequence of partial sums $g_0 = 0$, $g_{n+1} = g_n +
\omega^{\gamma_n}$ with $\gamma_n$ the iterates of
$\varphi(\cdot, 0)$ from $0$; the \emph{fundamental sequence} of a limit code $a$ is a compiled
tree $k \mapsto a[k]$ of codes below $a$, defined by the shape of $a$
(sums on the last term, $\varphi(0, b)$ through $\omega^{b[k]}\cdot
k$, $\varphi(a', b)$ through iterates of $\varphi(a'[k], \cdot)$), with
the \emph{internal cofinality} lemma that every canonical $y < a$ lies
below some $a[k]$ at an explicitly computed index; the
\emph{canonicalizer} is a compiled tree sending a valid code to the
canonical code of its value, the identity on canonical codes;
$\mathrm{DenV}(a)$ = ``every valid $b$ whose canonical form lies below
$a$ is accessible''; and the motive
$\mathrm{moGam}$ of the $\mathcal O$-induction conjoins ``whenever the
point is the code of $x + \omega^a$ with $a$ canonical and valid,
$\mathrm{DenV}(a)$'' with ``the point is not $\gamma_0$, or
$\mathrm{DenV}(\gamma_n)$ for every $n$''.

\paragraph{Kit.}
Ground induction at a proved accessibility (\isn{gi\_TI}) and relative
transfinite induction along the number order, as in
Appendix~\ref{app:cnf}; universal introduction at a fresh number and
elimination; disjunctive syllogism; case analysis on a decided term;
the shape trichotomy of Veblen codes (\isn{vk\_shape\_cases});
decidedness of validity and comparison at numbers; the laws taking
validity apart for a principal and for a sum (\isn{vk\_V\_P},
\isn{vk\_V\_S}); the unfolding of the comparison for each pair of
shapes; introduction and elimination of accessibility at a point
(\isn{vk\_acc\_I}, \isn{vk\_acc\_E}), of $\mathrm{vOp}$, and of each
motive; a zero code is accessible; accessibility transports along an
equation.  Two lemmas carry the content: \isn{vk\_K}, that every sum
over an accessible, valid, principal head is accessible; and
\isn{vk\_Lgen}, the principal step relative to a schema: from $c$ and
$d$ numbers with $\Acc\cdot d$, and the premise that in every
extension, for every $c'$ with $\mathrm{vOp}\,c'\,c$ and every
accessible $e'$, $\mathrm{RC}(c', e')$ holds, conclude $\mathrm{RC}(c,
d)$.

\paragraph{The equivalence, left to right.}
\begin{enumerate}
\item Assume $\WF$ derivable; fix $\Delta$, $d$, $X$ and the premise
  $\Acc\cdot d, \Delta \ent X$.
\item Weaken $\WF$ to $\Delta$ and instantiate at $d$: $\Delta \ent
  \neg(\mathrm{Nat}\cdot d) \vee \Acc\cdot d$.
\item $\mathrm{Nat}\cdot d$ holds at a number, so disjunctive
  syllogism gives $\Delta \ent \Acc\cdot d$.
\item Cut against the premise: $\Delta \ent X$; right-disjunct
  introduction gives the conclusion.
\end{enumerate}

\paragraph{The equivalence, right to left.}
Assume $\mathrm{AccBiv}$.  \emph{Lemma} (\isn{vk\_RCall\_main}): from
$\Gamma \ent c\,\mathbf N$ and $\Gamma \ent \Acc\cdot c$, $\Gamma \ent
\mathrm{RCall}(c)$.  Proof by ground induction along $\mathrm{vOp}$ at
$c$ with motive $\mathrm{RCall}$:
\begin{enumerate}
\item In the step at $z$, with the hypothesis $\mathrm{RCall}(c')$ for
  every predecessor $c'$ of $z$, introduce the universal at a fresh
  number $d$; the goal is $\neg(\Acc\cdot d) \vee \mathrm{RC}(z, d)$.
\item Under the hypothesis $\Acc\cdot d$, prove $\mathrm{RC}(z, d)$ by
  \isn{vk\_Lgen} at $z$ and $d$; its premise is discharged from the
  induction hypothesis: for $c'$ a predecessor of $z$ and $e'$
  accessible, $\mathrm{RCall}(c')$ instantiated at $e'$ and disjunctive
  syllogism give $\mathrm{RC}(c', e')$.
\item \emph{The single use of $\mathrm{AccBiv}$}: apply it to $\Delta$,
  $d$ and $X = \mathrm{RC}(z, d)$, with the hypothetical derivation of
  step 2, to obtain the disjunction and close the universal.
\end{enumerate}
\emph{Main proof.}
\begin{enumerate}
\item Strong induction on the rank with motive $\mathrm{rkAll}(r)$; in
  the step at $r$ the hypothesis gives, for any number $u$ of rank
  below $r$, $\Acc\cdot u$.
\item Introduce the universal at a number $x$; the goal $\neg(\mathrm{rank}
  \,x = r) \vee \Acc\cdot x$ reduces, the guard being decided, to
  $\Acc\cdot x$ under the hypothesis $\mathrm{rank}\,x = r$.
\item Split on validity.  If $x$ is invalid it has no predecessor, so
  it is accessible vacuously.  If valid, split by shape: zero is
  accessible; for a principal $x = \varphi(c, d)$ the guard gives $c$
  and $d$ of rank below $r$, so both are accessible by step 1, the
  lemma at $c$ gives $\mathrm{RCall}(c)$, its instance at $d$ and
  disjunctive syllogism give $\mathrm{RC}(c, d)$, and its elimination
  at $x$ gives $\Acc\cdot x$; for a sum $x = p + t$ the guard gives $p$
  of rank below $r$, hence accessible, and valid and principal, and
the instance \isn{vk\_K\_at} of \isn{vk\_K} at $p$ gives $\Acc\cdot x$.
\item Conclude $\WF$: at a number $x$, instantiate step 1 at
  $\mathrm{rank}\,x$ and at $x$, discharge the guard by reflexivity,
  and introduce the right disjunct.
\end{enumerate}
\emph{Proof of \isn{vk\_K}} for a fixed accessible, valid, principal
$p$.  (i) The tail step: if $z$ is a valid sum with head $p$ and
accessible rest, then $z$ is accessible, by ground induction on the
rest with motive $\mathrm{kTAll}(p, \cdot)$; in the step, a predecessor
$w$ of such a $z'$ is split by shape, and the comparison laws for a
principal or a sum against a sum reduce each case to the
accessibility of $p$, to an equation, or to the hypothesis at a
smaller rest.  (ii) Strong induction on the rank of the sum with
motive $\mathrm{kRk}(p, \cdot)$: the rest of a valid sum with head $p$
is accessible by the normal-form condition, which places it below
$p$, equal to $p$, or a sum with head $p$ of smaller rank, and the
tail step gives the sum.  (iii) Instantiate at the sum's rank.
\emph{Proof of \isn{vk\_Lgen}}: ground induction on $\Acc\cdot d$ with
motive $\mathrm{mL}(c, \cdot)$.  In the step at $z$, the first conjunct
holds because every predecessor satisfies the motive's first conjunct;
for the second, take a valid principal $x = \varphi(c, z)$ and show
$\Acc\cdot x$ by a side strong induction on a rank $s$ with motive
$\mathrm{PA}(x, s)$, whose step unfolds ``principal below principal''
into three disjuncts: $F_1 p' < c$ with $F_2 p' < x$, where $F_2 p'$ is
accessible by a helper (every valid $e < x$ of smaller rank is
accessible: zero outright, principal by the side hypothesis, a sum
through its head and \isn{vk\_K\_at}) and the schema premise then
applies; $F_1 p' = c$ with $F_2 p' < z$, by the outer hypothesis at
$F_2 p'$; or $p' < z$, by accessibility of $z$; the predecessors of
$x$ are then zero, principal (covered by the side induction) or sums
whose head is below $x$ (covered by the side induction and
\isn{vk\_K}).

\paragraph{The one-claim theorem.}
$\mathcal O$-induction (\isn{o\_ind}, the rule at the body of $\mathcal
O$ with a two-part motive) at $\gamma_0$ with motive $\mathrm{moGam}$.
\begin{enumerate}
\item Zero and successor codes: the first conjunct holds by the laws
  for those shapes, and they differ from $\gamma_0$, so the second
  conjunct's left disjunct holds.
\item Limit code with the induction hypothesis for the sequence: the
  first conjunct passes to the limit (\isn{g\_moAccV\_lim}), using that
  at a point coding $x + \omega^a$ the sequence is the chain whose
  $(k+1)$-th value codes $x_k + \omega^{a[k]}$ (graph functionality),
  so the hypothesis gives $\mathrm{DenV}(a[k])$ for every $k$, and
  \isn{vden\_limit} gives $\mathrm{DenV}(a)$: for a valid $b$ whose
  canonical form lies below $a$, the internal cofinality of the
  fundamental sequence (\isn{vfs\_cofinal}, with the descent and
  canonicity laws of the sequence) places that canonical form below
  some $a[k]$, and $\mathrm{DenV}(a[k])$ applies; a zero exponent is the
  case \isn{vden\_zero}.  For the second conjunct, case on the
  decided equation with $\gamma_0$: if equal, the sequence is
  $\gamma_0$'s, its $(k+1)$-th value codes $g_k + \omega^{\gamma_k}$, and
  the hypothesis gives $\mathrm{DenV}(\gamma_k)$ for every $k$.
\item Extract, at $\gamma_0$, $\mathrm{DenV}(\gamma_n)$ for every $n$.
\item $\WF$: at a number $x$, if $x$ is valid let $c$ be its canonical
  form; the bound $c < \gamma_{c+1}$ (\isn{gm\_bound}, by rank
  induction) and $\mathrm{DenV}(\gamma_{c+1})$ give $\Acc\cdot x$; if
  invalid, $x$ is accessible vacuously.
\end{enumerate}
The inductions are $\mathcal O$-induction once, at the claimed code,
and rank and number inductions along the number order in the
canonicalizer, fundamental-sequence and bound lemmas; no induction
along $\mathrm{vOp}$ occurs.  The build-time check computes, through
the proof terms, the transitive dependencies of the theorem and of
the accessibility lemma behind it, exempting the current theory, and
fails if any dependency is named by a forbidden family (the
well-foundedness theorems, relative transfinite induction except at
the number order, the classical warrant, bivalence, the $\mathcal
O$-membership lemmas that consume well-foundedness) or states anything
about $\mathrm{vOp}$ together with well-foundedness, transfinite
induction or bivalence; it also requires $\mathcal O$-induction and
the cofinality lemma to be present; and a control theorem that uses a
forbidden lemma is checked to be flagged by both tests.  The test is
by names and statement shapes; it is evidence that no such induction
is used, and a reading of the proofs confirms it, but it is not a
proof-term search for instances of the rule at the Veblen order.  The
same theorem shape, one membership claim in $\mathcal O$ giving $\WF$,
is proved at Buchholz's notation system with every finite level, once
for a generic fundamental-sequence kit and with no canonical form
(Appendix~\ref{app:stage}, step 62).

\paragraph{The one-rule system and its soundness.}
\cref{thm:price} is proved in an abstract calculus whose derivations
are those of \iga together with one more rule, bivalence of
accessibility at the Veblen order: from $\Delta \ent d\,\mathbf N$ and
$\Acc\cdot d, \Delta \ent X$ conclude $\Delta \ent \neg(\Acc\cdot d)
\vee X$ (\isn{accbivS}, in the calculus \isn{giv\_calc}).  Its
soundness for the omega model is \iga's soundness with one new case,
and the new case is exactly the statement that every instance of the
rule is sound for the Veblen order, which the classical warrant
supplies through \cref{thm:biv-wf}; consistency follows as for \iga.
The well-foundedness theorem is the right-to-left direction of the
equivalence replayed in the calculus with the rule present rather
than admitted, which the abstract calculus makes one instantiation.

\paragraph{The triangle.}
The first conjunct is the equivalence.  For the second, the forward
half (\isn{o\_gamma0}) derives $\gamma_0 \in \mathcal O$ from $\WF$: every
partial sum $g_n$ is a notation, by number induction whose step
consumes $\WF$ through transfinite induction along $\mathrm{vOp}$; the
sums increase; and the limit rule of $\mathcal O$ gives the limit code.
The backward half is the one-claim theorem at the empty context.

\paragraph{Lemmas a re-mechanizer needs.}
The \iga calculus with universal introduction at fresh variables,
elimination, weakening, cut, disjunction introduction, disjunctive
syllogism, case analysis on decided terms, ex falso and existential
elimination; \isn{gi\_TI}, relative transfinite induction at the
number order, rank and number induction, and $\mathcal O$-induction;
the Veblen kit named above; for the one-claim theorem, the code
builders of $\mathcal O$ with their injectivity and distinctness, graph
functionality of the decoded sequence, the canonicalizer facts, the
fundamental sequences with their internal cofinality, the bound, and
the algebra of $\mathrm{DenV}$.

\section{Goodstein's theorem}
\label{app:goodstein}

Files: \isn{IGA\_Goodstein.thy} (the HOL machinery, the index trees,
the internal terms, and the theorem under the decrease hypothesis),
\isn{IGA\_GoodsteinDec.thy} (the derivation of the decrease lemma),
\isn{RGA\_Kit/Cov\_Induct.thy} (the course-of-values induction
pattern that the decrease theory re-proves with a well-formedness
hypothesis).

\paragraph{Statements.}
\isn{gDec\_derivable}: $\emptyset \ent \mathrm{gDec}$ in \rga.
\isn{iga\_goodstein}: $\emptyset \ent \mathrm{gGood}$ in \iga.
Companions: \isn{iga\_goodstein\_of\_decrease}, $\mathrm{gDec}, \emptyset
\ent \mathrm{gGood}$; and \isn{goodstein\_hol\_unconditional}, the HOL
statement that every Goodstein sequence reaches zero, read off
through omega soundness.  Truncated subtraction $a \dotminus b$ is
computed by the recursor (\isn{gapSub}), so $a \dotminus b = 0$ says $a
\le b$ and $(a+1) \dotminus b = 0$ says $a < b$.

\paragraph{Definitions, HOL side.}
The leading exponent $\mathrm{lgc}(b, m)$ is the unique $L$ with $b^L
\le m < b^{L+1}$ for $b \ge 2$ and $m > 0$.  The hereditary
representation $\mathrm{hord}(b, m)$ is the normal form $Z$ when $m =
0$ or $b < 2$, and otherwise $W(\mathrm{hord}(b, e), \mathrm{hord}(b, m
- b^e))$ with $e = \mathrm{lgc}(b, m)$: the leading power is split off
one copy at a time.  Evaluation at a base reads a normal form with
$\omega$ replaced by the base.  The Goodstein step from step index $k$
and value $m$ evaluates $\mathrm{hord}(k+2, m)$ at base $k+3$ and
subtracts one; a state is the pair $\langle k, m\rangle$; the run's
value after $j$ steps is $\mathrm{grun}(k, m, j)$.  The HOL decrease
(\isn{hord\_gstep\_lt}): for $m > 0$ the normal form after the step is
below the one before.

\paragraph{Definitions, index trees and internal terms.}
A generic course-of-values tree (\isn{prCov}) computes, at $(x, n)$, a
history list whose entry $i$ is $V(i)$, where $V(0) = 0$ and $V(m)$ is
obtained by a combining tree from $\langle x, m\rangle$ and the two
earlier values $V(E(x, m))$ and $V(T(x, m))$, both children being
below $m$.  Entry $i$ is read by $n - i$ second projections and one
first projection.  Its instances: the canonical code
(\isn{prHordC}), with $E$ the leading exponent and $T$ the tail $m -
b^{\mathrm{lgc}(b, m)}$, whose first projection is the canonical code
of $\mathrm{hord}(b, m)$; and the base bump (\isn{prBumpC}), whose
first projection is $\mathrm{hord}(b, m)$ evaluated at a base $c$.  The
Goodstein trees (\isn{prGord}, \isn{prGstep}, \isn{prRun},
\isn{prGrun}) compute the ordinal, the stepped state, the $j$-fold run
and the run's value.  Internal terms are compiled trees: $\mathrm{gordT}$,
$\mathrm{gstepT}$, $\mathrm{mT}$ (the value component), $\mathrm{runT}$,
$\mathrm{grunT}$, $\mathrm{pairT}$; $\mathrm{hiT}$, $\mathrm{loT}$ the
projections; the $\eps$ order $\mathrm{gOp}$ is the order of
Appendix~\ref{app:cnf} at these projections.  Termination from $p$,
$\mathrm{gTerm}(p)$, is ``some step of the run from $p$ has value $0$''
(a negated universal); $\mathrm{gDec}$ is $\cA\cdot\lambda q.\,
\mathrm{mT}\cdot q = 0 \vee \mathrm{gOp}\cdot(\mathrm{gordT}\cdot
(\mathrm{gstepT}\cdot q))\cdot(\mathrm{gordT}\cdot q)$; $\mathrm{gGood}$
is $\cA\cdot\lambda n.\,\mathrm{gTerm}(\mathrm{pairT}\cdot 0\cdot n)$; and
the motive $\mathrm{MB}(z)$ is ``for every start state $p$ and step
count $i$, if the state reached from $p$ after $i$ steps has ordinal
code $z$, then the run from $p$ terminates''.  At open arguments the
proof abbreviates: $\mathrm{Pr}$, $\mathrm{Px}$, $\mathrm{Py}$ for
pairing and projections; $\mathrm{Gc}(b, m)$ for the canonical code of
$m$ in base $b$; $\psi(b, c, m)$ for the bump from base $b$ to base
$c$, and $\psi_S(b, m)$ for $c = b+1$; $\mathrm{lgcT}$, $\mathrm{tailT}$,
powers and products; $L$ and $V$ for the comparison and validity of
Appendix~\ref{app:cnf} at the projections.

\paragraph{Proof of the decrease lemma, in \rga.}
Steps 1 to 3 evaluate the compiled trees at open arguments by number
induction; step 4 is the digit law; steps 5 to 9 are the four strong
inductions and the bound; steps 10 and 11 assemble the lemma.  Two
engines.  \emph{Engine A} (\isn{ind\_gen}): for a term family $P$
stable under substitution, from $\Gamma \ent P(0)$ and, in every
extension and at every number $k$, $P(k)$ implies $P(k+1)$, conclude
$P(n)$ for every number $n$; it is \rga's number induction at a fresh
variable.  \emph{Engine B} (\isn{par\_strong\_univ}, through \isn{cov\_parW}):
for a family $P(b, n)$ with a base parameter, from a step that, in
every well-formed context in which the base is a number at least $2$,
derives $P(b, x)$ from $P(b, y)$ for every $y < x$, conclude
$\emptyset \ent \cA\cdot\lambda b.\,(2 \le b) \rightarrow
\cA\cdot\lambda n.\,P(b, n)$, where $a \rightarrow b$ at a decided $a$
abbreviates $\neg a \vee b$; it is universal introduction on $b$,
implication introduction at the decided guard, universal introduction
on $n$, and course-of-values induction, which is number induction on
the bounded motive ``$P$ at every $k \le N$'' and which the decrease
theory proves in a form whose step carries the well-formedness of
its context (\isn{cov\_parW}, after the kit's
\isn{ir\_cov\_ind\_par}).  Its instances at a base $b \ge 2$ and a
number $n$ are read off by \isn{par\_inst}.
\begin{enumerate}
\item \emph{Arithmetic at open arguments} (Engine A): the compiled
  addition and truncated subtraction agree with the recursor forms;
  the product laws (zero, successor, commutativity, monotonicity); the
  power laws (positivity for base at least $1$, monotonicity,
  strictness, reflection, and $t < b^t$ for $b \ge 2$); zero test and
  conditional; the order test; and the Cantor projections at open
  arguments, $\mathrm{Py}(\mathrm{Pr}\,a\,b) = b$ and
  $\mathrm{Px}(\mathrm{Pr}\,a\,b) = a$, from the row search
  (\isn{Dy\_inv}, by induction with the motive that the triangular
  number of the row bounds the argument on both sides).
\item \emph{Powers and the leading exponent}: the invariant of the
  partial count by Engine A, hence the bounds $b^{\mathrm{lgcT}} \le m
  < b^{\mathrm{lgcT}+1}$, uniqueness of the exponent satisfying them,
  $\mathrm{lgcT}\,b\,m < m$, $\mathrm{tailT}\,b\,m < m$, and
  monotonicity of the exponent in the value, all for $b \ge 2$ and $m
  \ge 1$.
\item \emph{The course-of-values tree at open arguments}: the zero and
  successor equations of the history, the peeling law, the shifted
  lookup by Engine A, and the \emph{unfolding law} (\isn{cov\_unfold}):
  if both children of $n+1$ are at most $n$, the first projection at
  $n+1$ is the combining tree applied to $\langle x, n+1\rangle$ and the
  two children's values.  Its instances at a nonzero $m$ with $b \ge
  2$: $\mathrm{Gc}(b, m) = \mathrm{Pr}(1 + \max(\mathrm{rank}\,e,
  \mathrm{rank}\,t), \mathrm{Pr}(\mathrm{Gc}(b, e), \mathrm{Gc}(b, t)))$
  with $e$ the exponent and $t$ the tail (\isn{G\_pos}), and
  $\psi(b, c, m) = c^{\psi(b, c, e)} + \psi(b, c, t)$ (\isn{psi\_pos});
  both are $0$ at $m = 0$.
\item \emph{The digit law} (\isn{digits}, Engine A on $k$): for $b \ge
  2$, $r < b^e$ and $k + 2 \le b$, $\psi(b, c, b^e(k+1) + r) =
  c^{\psi(b, c, e)}(k+1) + \psi(b, c, r)$.  The motive carries the
  decided guard $k + 2 \le b$; the step unfolds $\psi$ once, whose
  leading exponent is $e$ by uniqueness and whose tail is $b^e k + r$.
\item \emph{Strict monotonicity of the bump} (\isn{M\_univ},
  \isn{M\_inst}; Engine B).  Motive: for every $m < n$, $\psi_S(b, m) <
  \psi_S(b, n)$; in the step at $x$, the hypothesis holds at every $y <
  x$.  Write $e_m, e_x$ for the exponents and $t_m, t_x$ for the tails
  of $m$ and $x$, and $C_x = (b+1)^{\psi_S(b, e_x)}$.  Split on $m = 0$:
  then $\psi_S(b, m) = 0$ and $\psi_S(b, x) \ge C_x \ge 1$.  Otherwise
  $e_m \le e_x$; split on equality.  If $e_m = e_x$, the leading terms
  agree, $t_m < t_x$, the hypothesis at $y = t_x$ gives $\psi_S(b, t_m)
  < \psi_S(b, t_x)$, and adding the leading term finishes.  If $e_m <
  e_x$, let $N = b^{e_m}(b-1) + (b^{e_m} - 1)$, so $N + 1 = b^{e_m+1}
  \le b^{e_x} \le x$ and $m \le N$; the hypothesis gives $\psi_S(b, b^{e_m} - 1) < \psi_S(b, b^{e_m})$
  and, by cases $m = N$ or $m < N$, $\psi_S(b, m) \le \psi_S(b, N)$; the digit law
  at $N$ and at $b^{e_m}$ then gives the bound $\psi_S(b, m) <
  (b+1)^{\psi_S(b, e_m) + 1}$ (\isn{psi\_UB}); the hypothesis at $y =
  e_x$ gives $\psi_S(b, e_m) < \psi_S(b, e_x)$, so the bound is at most $C_x
  \le \psi_S(b, x)$.
\item \emph{The upper bound} (\isn{psi\_UB\_full}, no induction): for
  $b \ge 2$ and $m \ge 1$, $\psi_S(b, m) < (b+1)^{\psi_S(b,
  \mathrm{lgcT}(b, m)) + 1}$, by the second case of step 5 with the
  instance of step 5 in place of the hypothesis.
\item \emph{Base change} (\isn{BC\_univ}, \isn{BC\_inst}; Engine B).
  Motive: $\mathrm{Gc}(b+1, \psi_S(b, n)) = \mathrm{Gc}(b, n)$.  At $x
  = 0$ both sides are $0$.  Otherwise let $y = \psi_S(b, x)$, $e$ and
  $t$ the exponent and tail of $x$, $c = b + 1$: the unfolding gives
  $c^{\psi_S(e)} \le y$, the bound gives $y < c^{\psi_S(e)+1}$, so
  uniqueness gives $\mathrm{lgcT}(c, y) = \psi_S(e)$ and the tail of
  $y$ is $\psi_S(t)$; the hypothesis at $e$ and at $t$ and the
  unfolding of $\mathrm{Gc}$ at $(c, y)$ and at $(b, x)$ give equal
  right-hand sides.
\item \emph{Monotonicity of the canonical code} (\isn{MONO\_univ},
  \isn{MONO\_inst}; Engine B).  Motive: for every $m < n$, $L(\mathrm{Gc}
  (b, m), \mathrm{Gc}(b, n))$.  Fold the goal, so that the unfolded
  clause of $L$ suffices; the code reader gives rank, head and tail of
  a canonical code.  At $m = 0$ the first disjunct holds (rank $0$
  against nonzero rank).  Otherwise the rank guards hold for both
  codes; if the exponents agree, the hypothesis at the tails gives the
  tail comparison; if $e_m < e_x$, the hypothesis at the exponents gives
  the head comparison.
\item \emph{Validity of the canonical code} (\isn{VALID\_univ},
  \isn{VALID\_inst}; Engine B).  Motive: $V(\mathrm{Gc}(b, n))$.  At $x
  = 0$ the rank is $0$.  Otherwise the guards hold, head and tail are
  valid by the hypothesis, and the spine condition is by cases on the
  tail: zero rank, or its exponent equals $e_x$ (equal heads), or is
  below it (the head comparison by step 8).
\item \emph{Decrease at an open state} (\isn{dec\_open}): from $q$ a
  number with $\mathrm{Py}\,q \ne 0$, $\mathrm{gOp}\cdot(\mathrm{gordT}
  \cdot(\mathrm{gstepT}\cdot q))\cdot(\mathrm{gordT}\cdot q)$.  Evaluate:
  the ordinal of $q$ is $\mathrm{Gc}(b, m)$ with $b = \mathrm{Px}\,q + 2$
  and $m = \mathrm{Py}\,q$; the next state is $\mathrm{Pr}(\mathrm{Px}\,q
  + 1, p - 1)$ with $p = \psi_S(b, m)$; so its ordinal is $\mathrm{Gc}(b
  + 1, p - 1)$.  Since $p \ge (b+1)^{\psi_S(b, \mathrm{lgcT}(b, m))} \ge 1$ by the
  unfolding and the positivity of powers, $p - 1 < p$; step 8 at base $b + 1$
  gives $L(\mathrm{Gc}(b+1, p-1), \mathrm{Gc}(b+1, p))$ and step 7
  rewrites the second code to $\mathrm{Gc}(b, m)$; step 9 at both bases
  gives validity; fold into $\mathrm{gOp}$ and transport back through
  the evaluations.
\item \emph{The universal}: introduce $q$ with $q\,\mathbf N$; $\mathrm{mT}
  \cdot q = \mathrm{Py}\,q$ is a number, so the equation with $0$ is
  decided; in the zero case take the left disjunct, otherwise the
  right by step 10.
\end{enumerate}

\paragraph{Proof of Goodstein's theorem, in \iga.}
\isn{iga\_goodstein} is the hypothetical theorem cut against the
embedding of the \rga derivation of the decrease lemma
(\isn{entails\_embed}).  The hypothetical theorem, with $\Gamma =
\{\mathrm{gDec}\}$:
\begin{enumerate}
\item $\Gamma \ent \WF(\mathrm{Nat}, \mathrm{gOp})$: the instance of
  Appendix~\ref{app:cnf}'s theorem at the compiled projections
  (\isn{gi\_WF\_epsilon0}), weakened.
\item Transfinite induction (\isn{gi\_WF\_TI}) along $\mathrm{gOp}$ on
  the domain $\mathrm{Nat}$ with motive $\mathrm{MB}$: it suffices to
  show progressiveness, that in every well-formed extension $\Delta$
  and at every number $z$, $\mathrm{MB}(z)$ follows from $\mathrm{MB}(y)$
  for every $y$ with $\mathrm{gOp}\cdot y\cdot z$.
\item \emph{Closure} (\isn{goodstein\_closure}).  Introduce $p$ and $i$;
  let $q$ be the state after $i$ steps from $p$, a number.  Split, decided,
  on whether its ordinal is $z$; if not, the left disjunct of
  $\mathrm{MB}$'s matrix.  If so, split on $\mathrm{mT}\cdot q = 0$: if
  zero, the run from $p$ terminates at step $i$ (the run's value is the
  state's value component, and the existential is introduced as a
  negated universal); if nonzero, instantiate $\mathrm{gDec}$ at $q$,
  disjunctive syllogism gives $\mathrm{gOp}\cdot y\cdot(\mathrm{gordT}
  \cdot q)$ for $y$ the ordinal of the next state, transported along
  the equation to $\mathrm{gOp}\cdot y\cdot z$ (\isn{g\_app\_transport}), the hypothesis gives $\mathrm{MB}(y)$,
  and its instance at $p$ and $i + 1$, whose guard the run equation
  discharges, gives termination.
\item Steps 2 and 3 give $\Gamma \ent \cA\cdot\lambda x.\,\neg(\mathrm{Nat}
  \cdot x) \vee \mathrm{MB}(x)$.
\item At a number $n$ let $p = \mathrm{pairT}\cdot 0\cdot n$ and $a$ the
  ordinal of $p$ after $0$ steps; instantiate step 4 at $a$, whose
  domain guard $\mathrm{Nat}\cdot a$ holds at a number
  (\isn{gi\_natD\_inst}), and $\mathrm{MB}(a)$ at $p$ and $0$, whose
  guard $\neg(a = a)$ reflexivity refutes; termination converts to the
  body of $\mathrm{gGood}$ at $n$.
\end{enumerate}
The HOL reading: omega soundness with the empty context gives the
omega-truth of $\mathrm{gGood}$, which \isn{omv\_gGood\_iff} reads as
``for every $n$ some step of the run from $\langle 0, n\rangle$ is
$0$''.

\paragraph{Lemmas a re-mechanizer needs.}
\rga's number induction, universal introduction at an eigenvariable
and elimination, case split on a decided term with decidedness of
equality, implication introduction at a decided antecedent and modus
ponens, equality transport and conversion; the verified compiler's
totality and compositional evaluation at open arguments; the two engines, with \isn{cov\_parW}, \isn{par\_inst} and the
bounded-motive helpers of the kit (\isn{acc\_step\_p},
\isn{acc\_descent\_p}); the internal arithmetic of steps 1 and 2; the
course-of-values laws of step 3; the digit law, the bound and the four inductions; the
evaluations of the Goodstein terms; and, on the \iga side,
\isn{gi\_WF\_epsilon0}, \isn{gi\_WF\_TI}, \isn{gi\_natD\_inst},
\isn{g\_app\_transport}, the run equations, and the embedding of \rga
into \iga.

\section{Truth exactness}
\label{app:truth}

File: \isn{IGA\_Truth.thy}; supporting machinery in
\isn{IGA\_Positive.thy} (the oracle reading, the set form of the
induction lemma, \isn{osem\_term}), \isn{IGA\_Lfp.thy}
(\isn{osem\_gfix\_omv}) and \isn{RGA\_Omega\_Model.thy}.

\paragraph{Statement.}
For every term $t$, every result $r$ and every head form $w$, with no
hypothesis on $t$: $\Tr\cdot\langle 0, \langle\lceil t\rceil, r\rangle
\rangle$ is omega-true if and only if $t$ has value $r$
(\isn{tr\_omv\_iff}), and $\Tr\cdot\langle 1, \langle\lceil t\rceil,
\mathrm{hdc}(w)\rangle\rangle$ is omega-true if and only if $t$ reaches
the head form $w$ (\isn{tr\_omv\_iff\_hd}).

\paragraph{Definitions.}
\emph{The omega model's twenty-one rules.}  Head forms: an
abstraction, the recursor awaiting three, two or one arguments, and
the universal combinator.  Eight head rules: an abstraction, the
recursor and the universal combinator are their own head forms; applying a partial recursor adds an argument
(two rules); if $t$ reaches $\lambda b$ and $b[u]$ reaches $w$ then
$t\cdot u$ reaches $w$; if $t$ reaches the two-argument recursor with
step terms $g$, $h$ and $m$ has value $0$ then $t\cdot m$ reaches what
$g$ reaches; and likewise at value $k+1$ with the reduct $h\cdot\bar
k\cdot(R\cdot g\cdot h\cdot\bar k)$.  Thirteen value rules: zero and
successor; equality, with value ``the two values are equal''; negation
true and false; disjunction true-left, true-right and false; the value
forms of the three reduction rules; the universal true when every
instance has value $1$ and false when some instance has value $0$.
\emph{Head-form codes} $\mathrm{hdc}$ tag the five head forms and
pair them with the codes of their components; they are injective and
invertible by tag.  \emph{Code shapes}: the code of an operator node
with zero, one or two children, with inversion lemmas saying that a
number of that shape coding a term codes a node of that operator with
the children's codes.  \emph{Code builders}: closed \rga terms
computing codes from codes, one per constructor, together with
compiled substitution on codes (\isn{bsubC}), the numeral coder, and
the recursor-reduct coder; their omega values at codes and at
arbitrary numbers follow from the value of a compiled binary primitive
recursive function at terms with values.  \emph{The point} $\langle
\mathrm{tag}, \langle\mathrm{code}, \mathrm{result}\rangle\rangle$ is
built by compiled pairing.  \emph{The body} is a list disjunction over
twenty-two clauses, one per omega rule with the equality rule split
into an equal-values clause and a distinct-values clause; each clause
is a nest of existentials over the components, with equality guards
``the point is the builder applied to the bound codes'' and recursive
atoms $a\cdot m$ at bound variables $m$ equated to premise points.
The clauses, with the rule each transcribes, the number of
existential binders, and the number of recursive atoms:
\begin{center}
\footnotesize
\begin{tabular}{@{}llcc@{}}
\toprule
clause & rule & binders & atoms \\
\midrule
H1--H3 & an abstraction, the recursor, the universal combinator reach themselves & 1, 0, 0 & 0 \\
H4, H5 & the recursor applied to one, two arguments & 3, 4 & 1 \\
H6 & beta, head form & 6 & 2 \\
H7, H8 & recursor at zero, at successor, head form & 8, 9 & 3 \\
V1, V2 & zero, successor & 0, 3 & 0, 1 \\
V3, V4 & equality, equal values; distinct values & 5, 6 & 2 \\
V5, V6 & negation true, false & 2 & 1 \\
V7--V9 & disjunction true-left, true-right, false & 3, 3, 4 & 1, 1, 2 \\
V10 & beta, value & 6 & 2 \\
V11, V12 & recursor at zero, at successor, value & 8, 9 & 3 \\
V13 & universal true (an inner universal over an existential) & 3 + 1 & 2 \\
V14 & universal false & 5 & 2 \\
\bottomrule
\end{tabular}
\end{center}
The clause for disjunction false (V9), for instance, reads: there are $k_p$,
$k_q$, $m_1$, $m_2$ with the point equal to the code of $p \vee q$ at
value $0$, $m_1$ the point of $p$ at value $0$, $a\cdot m_1$, $m_2$ the
point of $q$ at value $0$, and $a\cdot m_2$.  The universal-true clause
has, under its existentials, a universal whose body is itself an
existential with a recursive atom.  The body is positive, and the abstraction of the body over its
two variables is closed (\isn{pos\_trPhi}, \isn{closed\_trPhi}).  \emph{The
two target sets}: $\mathrm{trS}$, for soundness, contains a point when
for every term $t$ its code part codes, the tag-$0$ reading gives
$t$ the point's result as value, and the tag-$1$ reading gives $t$ a
head form with the point's result as code; points whose code part
codes no term, or with tag at least $2$, are unconstrained.
$\mathrm{trL}$, for completeness, is the truth set of $\Tr$.

\paragraph{Proof.}
\begin{enumerate}
\item \emph{Soundness at every point} (\isn{tr\_omv\_point}): if
  $\Tr\cdot\bar p$ is true then $p \in \mathrm{trS}$.  Apply the set
  form of Appendix~\ref{app:gfix}'s induction lemma with $S =
  \mathrm{trS}$, using positivity and closedness of the body; the set
  form is the one needed because no predicate term for $\mathrm{trS}$
  is available before the theorem.  The obligation is closure, step 2.
\item \emph{Closure of $\mathrm{trS}$ under the oracle reading}
  (\isn{tr\_body\_closed}).  An oracle-true list disjunction has an
  oracle-true member, since the empty disjunction $0 = 1$ is never true
  (\isn{omv\_zero\_one\_never}), so some clause reads true; twenty-two cases, one
  lemma each (\isn{sound\_clH1} to \isn{sound\_clV14}), each by the same
  five moves, written out for disjunction false: open the
  existentials, obtaining numbers $k_p, k_q, m_1, m_2$ at which every
  conjunct reads true; read the guards, which are equalities between a
  numeral and a builder term and so say the builder has the numeral's
  value; compute the builders' values from the value lemmas and
  identify, by determinacy of values, $p$, $m_1$, $m_2$ with the
  explicit points; read the atoms, which say $m_1, m_2 \in \mathrm{trS}$; decode
  and apply the omega rule: for any $t$ whose code is the disjunction
  node, the inversion gives $t = p \vee q$ with the bound codes, the
  memberships give $p$ and $q$ false, and the rule gives $t$ false.
  Clause-specific features: the clauses with no atoms decode a
  constant or an abstraction; head-form premises decode the component
  through the head-form inversion; the beta clauses use that compiled
  substitution on codes computes the code of the substitution
  (\isn{bsub\_code\_correct}); the recursor clauses decode the reduct
  code; the distinct-values equality clause reads its guard $\neg(v_1
  = v_2)$ as inequality of numbers; the universal-true clause
  instantiates the inner universal at every $j$ and reads, with the
  numeral coder's value, that every instance point is in
  $\mathrm{trS}$, then applies the universal rule; the universal-false
  clause uses one instance.
\item \emph{Soundness}: step 1 at the two point shapes, read through
  the two characterizations of $\mathrm{trS}$ (\isn{trS\_V},
  \isn{trS\_H}).
\item \emph{The oracle-to-term step} (\isn{tr\_compl\_clause}): if a
  clause reads true at the environment of the point $\bar p$ with
  atoms in $\mathrm{trL}$, then $p \in \mathrm{trL}$.  Lift the clause
  to the whole disjunction; the hypothesis is then the premise of
  \isn{osem\_gfix\_omv} of Appendix~\ref{app:lfp}, which needs neither
  positivity nor closedness.
\item \emph{Twenty-two completeness lemmas} (\isn{compl\_clH1} to
  \isn{compl\_clV14}): if the premise points of a rule are in
  $\mathrm{trL}$ then its conclusion point is.  Each lemma is step 4 with the
  matching clause, supplying the witnesses for the existentials,
  proving the guards by the builders' value lemmas and the atoms by
  the atom clause from the memberships; the universal-true clause
  proves its inner universal by the oracle reading's universal clause.
\item \emph{Completeness} (\isn{tr\_omv\_complete\_all}): simultaneous
  rule induction on the omega relations with the motives ``the
  tag-$1$ point of $t$ and $w$ is in $\mathrm{trL}$'' and ``the tag-$0$
  point of $t$ and $r$ is in $\mathrm{trL}$''; the twenty-one cases are
  the lemmas of step 5 applied to the hypotheses, with one case split,
  the equality rule, by whether the two values agree; the
  universal-true case uses the hypothesis at every instance.
\item \emph{The equivalences}: step 3 and step 6; for head forms,
  injectivity of the head-form coding identifies the witness.
\end{enumerate}

\paragraph{Lemmas a re-mechanizer needs.}
The omega model with its determinacy and inversion lemmas; a term
coding with explicit node shapes and inversions, injectivity of
pairing, and the head-form coding with its inversion and injectivity;
the code builders as closed terms with their values, and the
code-level substitution law; the oracle reading with its inversions
(list disjunction and conjunction, existential, universal at any
instance, atom, equality and inequality guards) and introductions;
the set form of the induction lemma; \isn{osem\_term} and
\isn{osem\_gfix\_omv}; positivity and closedness of the body, by
simplification over the twenty-two clauses.

\subsection{Determinacy of values}
\label{app:truth-det}

Files: \isn{IGA\_TruthInv.thy} (the motive and the reduction),
\isn{IGA\_TruthDet.thy} (the closure obligation).  Every closure step
is an \rga derivation; ground induction enters once, in the reduction.

\paragraph{Statement} (\isn{gi\_tr\_det}, \cref{thm:tr-det}).
For a well-formed context $\Gamma$ and $\Gamma \ent c\,\mathbf N$,
$\Gamma \ent r\,\mathbf N$, $\Gamma \ent r'\,\mathbf N$: if $\Gamma \ent
\Tr(c, r)$ and $\Gamma \ent \Tr(c, r')$, with $\Tr(c,r) =
\Tr\cdot\langle 0, \langle c, r\rangle\rangle$ (\isn{trV}), then $\Gamma
\ent r' = r$; at numerals, \isn{gi\_tr\_det\_num}.

\paragraph{Definitions.}
The motive (\isn{trDetM}) at a point $z = \langle g, \langle c,
r\rangle\rangle$ is $\cA\cdot\lambda j.\,\neg\Tr\cdot\langle g, \langle
c, j\rangle\rangle \vee j = r$, the components read by the compiled
projections: every result of the same tag and code is $r$.  The
\emph{variant points} of $z$ are the points $\langle g, \langle c,
r'\rangle\rangle$ with $r' \neq r$.  The closure obligation
(\isn{trDetClosure}): in every well-formed context, at every number
$z$, the truth body with the motive for the recursive variable gives
the motive at $z$.

\paragraph{Proof.}
\begin{enumerate}
\item \emph{The reduction} (\isn{gi\_tr\_consistent\_values\_of\_closure}).
  Ground induction at the truth term (\isn{tr\_ind}) gives the motive
  at $\langle 0, \langle c, r\rangle\rangle$ from the closure obligation;
  the motive's instance at $r'$ and disjunctive syllogism against
  $\Tr(c, r')$ give $r' = r$ (\isn{trDetM\_E}).
\item \emph{The frame} (\isn{trDet\_close}): the motive holds at $z$ once
  every variant point is refuted in every extension; introduce the
  universal at $j$ and split on the decided equation $j = r$.
\item \emph{The clause eliminations at a general predicate}
  (\isn{trp\_clH1\_E} to \isn{trp\_clV14\_E}) invert the twenty-two
  clauses with the recursive variable read as any predicate, here the
  motive, so every sub-point carries the motive rather than truth.
\item \emph{Refuting a variant point.}  The negative unfolding of the
  truth term (\isn{tr\_negI}) asks for one refutation per clause at a
  decided head (\isn{trd\_clH1\_negD} to \isn{trd\_clV14\_negD}): the
  existentials refuted by a universal, the head guard and the point
  guards decided, a recursive atom refuted by a premise.  Ten lemmas
  group these refutations, one per tag and head operator
  (\isn{trd\_shape\_lam}, \isn{trd\_shape\_R}, \isn{trd\_shape\_A},
  \isn{trd\_shape\_zero}, \isn{trd\_shape\_happ}, \isn{trd\_shape\_suc},
  \isn{trd\_shape\_eq}, \isn{trd\_shape\_neg}, \isn{trd\_shape\_or},
  \isn{trd\_shape\_vapp}); the clauses of another shape fall to the
  mismatch kit of tags and operators, and each shape lemma keeps one
  premise per clause of its own shape.
\item \emph{The twenty-two closure cases} (\isn{trd\_close\_clH1} to
  \isn{trd\_close\_clV14}) discharge those premises from the motive at
  the sub-points (\isn{trDetM\_neg}): head determinacy of the function
  code, by distinctness of the kinds of head codes (\isn{trd\_ne\_lam\_R0}
  and its siblings) and injectivity within a kind, for the application
  and recursor clauses; value determinacy of the scrutinee and of the
  arguments; and, for the universal clauses, the instance family.  The
  motive at every instance point refutes the false universal clause at
  its witness, and the motive at the false instance gives the witness
  that refutes the true clause (\isn{trd\_V13\_univ\_neg}).
\item \emph{Assembly} (\isn{trd\_closure}, \isn{gi\_trDetClosure}): the
  body at the motive is the disjunction of the twenty-two clauses, each
  clause closed by step 5; the \rga derivation lifts to \iga.
\end{enumerate}
A dependency check at the end of \isn{IGA\_TruthDet.thy} follows the
proof terms of \isn{gi\_tr\_det}, \isn{gi\_tr\_det\_num} and
\isn{gi\_trDetClosure} and fails on an oracle or on any dependency
that states something about the omega model, its head evaluation, the
oracle reading, the fixed-point operator on sets of points, or the
target sets above; the check requires the reduction, the fixed-point
rule and \isn{trd\_closure} to be present, and its control, the
soundness theorem \isn{gi\_tr\_sound}, must be flagged.

\section{Curry's paradox}
\label{app:curry}

File: \isn{IGA\_Curry.thy}.  The statements are about \rga's own
entailment (locale \isn{curry\_cond}) or about an arbitrary entailment
over \rga's terms (locale \isn{curry\_ext}); neither is about \iga.

\paragraph{Definitions.}
The locale \isn{curry\_cond} extends \rga with a term former
$\Rightarrow$ and four assumptions: the former preserves
well-formedness; it commutes with substitution at the binder's
variable, $(a \Rightarrow b)[u] = a[u] \Rightarrow b[u]$; modus ponens,
from $\Gamma \ent a \Rightarrow b$ and $\Gamma \ent a$ conclude $\Gamma
\ent b$; and conditional proof at every well-formed antecedent, from
$a, \Gamma \ent b$ conclude $\Gamma \ent a \Rightarrow b$.  The locale
\isn{curry\_ext} replaces \rga's entailment by an arbitrary relation
$\mathrm{ent}$ on contexts and terms, and assumes of it only the
hypothesis rule, head-step conversion in both directions, and the
same four properties of the former; contexts are finite sets, so a
hypothesis may be used twice.  The Curry term is $C = \fix(\lambda
x.\,x \Rightarrow 0 = 1)$ in either locale.

\paragraph{Statements.}
\isn{curry\_explosion}: $\emptyset \ent 0 = 1$.  \isn{curry\_inconsistent}:
the locale's assumptions are contradictory, from the previous
statement, the consistency of \rga (\isn{rga\_consistent}, by the
soundness of the certificate semantics) and $\ent \neg(0 = 1)$.
\isn{ext\_curry\_explosion}: $\mathrm{ent}(\emptyset, 0 = 1)$.

\paragraph{Proof, identical in both locales.}
\begin{enumerate}
\item $\lambda x.\,x \Rightarrow 0 = 1$ is well-formed, so $C$ is.
\item $C$ reaches $C \Rightarrow 0 = 1$ in two head steps: the fixed-point
  step, then a beta step after rewriting the result as a substitution
  into the former (the closed term $0 = 1$ is unchanged).  So from
  $\Gamma \ent C$ follows $\Gamma \ent C \Rightarrow 0 = 1$, and
  conversely, by conversion in the two directions.
\item $C, \emptyset \ent C$, by the hypothesis rule.
\item $C, \emptyset \ent C \Rightarrow 0 = 1$, from 3 by step 2.
\item $C, \emptyset \ent 0 = 1$, by modus ponens on 4 and 3; the
  hypothesis is used twice, as one member of the set.
\item $\emptyset \ent C \Rightarrow 0 = 1$, by conditional proof on 5.
\item $\emptyset \ent C$, from 6 by step 2.
\item $\emptyset \ent 0 = 1$, by modus ponens on 6 and 7.
\item For the inconsistency of the locale: step 8 and $\ent \neg(0 =
  1)$ contradict the consistency of \rga.
\end{enumerate}

\paragraph{Companions.}
\isn{liar\_explosion\_cp}: if conditional proof holds for the definable
implication $\neg a \vee b$ at every well-formed antecedent, then
$\emptyset \ent 0 = 1$.  With $L$ the Liar, $L \ent \neg L$ and $\neg L
\ent L$ by conversion; the assumption gives $\ent \neg L \vee \neg L$
and $\ent \neg\neg L \vee L$; disjunction elimination, with
double-negation elimination in one branch, gives $\ent \neg L$ and
$\ent L$, and negation elimination closes.  This argument uses
disjunction elimination and double-negation elimination and no modus
ponens for a former.  \isn{accbiv\_shape\_fails\_on\_liar}: the
bivalence shape with the Liar in place of accessibility fails, by the
instance at the empty context and $X = \neg L$.

\paragraph{Lemmas a re-mechanizer needs.}
The fixed-point step and the beta step, well-formedness of the
fixed-point former, conversion in both directions, the hypothesis
rule; and for the inconsistency statement only, the consistency of
\rga (\isn{rga\_consistent}) and $\ent \neg(0 = 1)$
(\isn{ir\_zero\_one\_ne}).  The requirement that the former commutes with substitution is,
in a named-variable presentation, that $\Rightarrow$ is a term former.

\section{The ordinal: stage-indexed soundness}
\label{app:stage}

Files: \isn{IGA\_Stage.thy} (the stage model, persistence, inversions,
the cost of unfolding, the staged oracle reading and transport, the
headline at an abstract bound, the rank step),
\isn{IGA\_StageArith.thy} (the bound family and its order facts),
\isn{IGA\_StageCong.thy} (value congruence at stages),
\isn{IGA\_StageStep.thy} (the stage judgment, replay, head steps, and
the hard rule cases), \isn{IGA\_StageSound.thy} (the remaining rule
cases, the step lemma, the induction over proofs and stage soundness),
\isn{IGA\_StageOrd.thy} (the rank lemma, the ordinal and the five
independence theorems).  Steps 33 to 63 at the end give the
statements of \S\ref{sec:pricing} and, step 62, of
\S\ref{sec:one-rule}, from session \textsf{OIGA\_Stage},
except step 51, from session \textsf{IGA\_PGSem}, steps 52 to 54, from
session \textsf{IGA\_Closing}, step 55, from session \textsf{IGA\_RE},
and steps 56 to 58, from
session \textsf{IGA\_Closing}, step 59, from session \textsf{IGA\_Eval}, step 60, from session \textsf{IGA\_ConPA\_Checker}, and step 61, from sessions \textsf{IGA\_CNF\_Kit}, \textsf{IGA\_ConPA\_Checker}, \textsf{IGA\_ConPA\_Buchholz} and \textsf{IGA\_ConPA\_Main},
step 62, from sessions \textsf{IGA\_OT\_Kit} and \textsf{IGA\_OT\_Claim}, and step 63, from sessions \textsf{IGA\_OT\_Rule} and \textsf{IGA\_OT\_PG}, with facts from \textsf{OIGA}
and \textsf{IGA} where named;
and the last paragraph the stage algebra that they share with \iga's
proof.  Every statement below is a classical theorem about the
systems' derivations and the omega model or the closed-off model,
unless it says otherwise.
The measure is accessibility, equivalently well-foundedness as one
sentence; transfinite induction at decided motives is not bounded by
this argument, and with \oga's decidedness rule it is derivable as a
sentence below $\varphi(\omega,0)$ (step 38).

\paragraph{Statements.}
\emph{Stage soundness} (\isn{gi\_stage\_sound}): if $\Gamma \ent c$,
$\alpha$ is a valid notation, and every member of $\Gamma$ is true at
the stage $\alpha$, then for some $n$ and $k$, $c$ is true at the stage
$\mathrm{vst}\,n\,k\,\alpha$; no well-formedness of $\Gamma$ is assumed.
With $\alpha < \varphi(\omega,0)$, $c$ is true at some valid stage below
$\varphi(\omega,0)$ (\isn{gi\_stage\_sound\_pw}).  \emph{The closed
form} (\isn{gi\_stage\_sound\_below\_pw}): if $\emptyset \ent c$ then
for some $n$ and $k$, $c$ is true at $\mathrm{vst}\,n\,k\,\omega$, which
lies below $\varphi(n+1,0)$, which lies below $\varphi(\omega,0)$.
\emph{The ordinal}: for a valid notation $x$, $\emptyset \ent
\Acc\cdot\lceil x\rceil$ implies $x < \varphi(n+1,0)$ for some $n$
(\isn{gi\_acc\_below\_level}), and $\emptyset \ent \Acc\cdot\lceil
x\rceil$ holds if and only if $x < \varphi(\omega,0)$
(\isn{gi\_acc\_iff\_below\_pw}).  \emph{Independence}, five theorems:
if $D$ is closed, $x$ is a valid notation with $\varphi(\omega,0) \leq
x$, and $D\cdot\lceil x\rceil$ is omega-true, then $\WF(D,
\mathrm{vOp})$ is not derivable (\isn{gi\_not\_WF\_above\_pw}); in
particular $\WF(\mathrm{Nat}, \mathrm{vOp})$ is not derivable
(\isn{gi\_not\_WF\_gamma0}); $\mathrm{AccBiv}$ of
Appendix~\ref{app:gamma0} fails (\isn{gi\_not\_AccBiv}), and so does its
restriction $\mathrm{ObRC}$ to the closure-shaped instances
(\isn{gi\_not\_ObRC}); and the ladder sentence $\cA\cdot\lambda c.\,
\neg(\mathrm{vOp}\cdot c\cdot\lceil\omega\rceil) \vee
\mathrm{RC}(c, 0)$ of Theorem~\ref{thm:ladder} is not derivable
(\isn{gi\_not\_LADDER}).

\paragraph{Notation.}
Stages are \emph{notations}, terms of the datatype \isn{veb}: $0$
(\isn{VZ}) and $\varphi(a, b) + c$ (\isn{VPhi}), not their numeric
codes; a code $\lceil x\rceil$ (\isn{veb\_code}) appears only as the
argument of $\Acc$ or of a domain.  A notation is \emph{valid}
(\isn{veb\_valid}) when it is in Sch\"utte's normal form; $<$ is the
notation's order (\isn{veb\_lt}) and $a \leq b$ means $a = b$ or $a <
b$.  $x + 1$ (\isn{veb\_p1}) appends $1 = \varphi(0,0)$ to the sum $x$,
and $x + j$ (\isn{veb\_pn}) does so $j$ times; the finite stage $n$
(\isn{vfin}) is $1 + \dots + 1$; $\omega = \varphi(0,1)$
(\isn{veb\_w}); $\max$ is \isn{veb\_max}.  The normalized
$\varphi(a, y)$ (\isn{veb\_nphi}) returns $y$ when $y$ is already a
fixed point of $\varphi(a,\cdot)$ and the raw $\varphi(a,y)+0$
otherwise.  $\mathrm{val}_\alpha(t, r)$ and $\mathrm{hd}_\alpha(t, w)$
say that $t$ has the value $r$, or reaches the head form $w$, at stage
$\alpha$; ``true at $\alpha$'' is $\mathrm{val}_\alpha(t, 1)$.

\paragraph{The stage model.}
The relations \isn{omhS} and \isn{omvS} form one mutual inductive
definition, clause for clause the twenty-one rules of the omega model
(Appendix~\ref{app:truth}), with one extra argument, the stage.  Every
clause's conclusion is at a valid stage $\alpha$, and every premise at
its own valid stage strictly below $\alpha$; the atoms hold at every
valid stage, $0$ included.  The universal-true clause asks, for every
$n$ separately, for a valid stage below $\alpha$ at which the instance
at $\bar n$ is true; no supremum of these stages is part of the
definition.  A fact at a stage is a fact of the omega model
(\isn{omvS\_omv}), so a term has at most one value at any two stages
(\isn{omvS\_unique}, from \isn{omv\_unique}).  Each inversion
(\isn{omvS\_app\_inv}, \isn{omvS\_or\_inv}, \isn{omvS\_AT\_inv},
\isn{omvS\_AF\_inv}, \isn{omvS\_N\_inv} and their siblings) returns the
premises at valid stages strictly below the conclusion's.
\begin{enumerate}
\item \emph{Persistence} (\isn{omhS\_omvS\_up\_all}): a fact at
  $\alpha$ holds at every valid $\alpha' > \alpha$, for every value, so
  truth and falsity both persist.  Rule induction
  (\isn{omhS\_omvS.inducts}): reapply each clause at $\alpha'$ with the
  same premises at the same stages, each $\beta < \alpha$ giving $\beta
  < \alpha'$ by transitivity (\isn{vlt\_trans}); in the universal-true
  clause every per-instance stage moves the same way.  Corollaries:
  \isn{omvS\_up} ($\beta < \alpha$) and \isn{omvS\_le} ($\beta \leq
  \alpha$).
\item \emph{The cost of unfolding.}  Every fixed point reaches its head
  form at the fixed stage $2$ (\isn{omhS\_trmGfix}: $\mu\Phi$ is a
  self-application, and two beta clauses build the head).  Downward, if
  $\mu\Phi\cdot x$ has value $r$ at $\alpha$, the body instance
  $\Phi[\mu\Phi/a][x/x]$ has value $r$ at some valid $\beta < \alpha$
  (\isn{omvS\_gfix\_inv}, the beta branch of the inversion with the
  body identified by head determinacy, \isn{omh\_det}).  Upward, a body
  instance at $\beta$ gives the fixed-point instance at $\max(2, \beta)
  + 1$ (\isn{omvS\_gfix\_succ}).
\end{enumerate}

\paragraph{The bound family.}
$\mathrm{vst}\,n\,k\,\alpha = \varphi_n^k(\max(\alpha,\omega)+1)$
(\isn{vst}), the $k$-fold iterate of the raw $\varphi(n,\cdot)$ from
$\max(\alpha,\omega)+1$, so that $\mathrm{vst}\,n\,(k+1)\,\alpha =
\varphi(n, \mathrm{vst}\,n\,k\,\alpha)$; raising $k$ by one is one
\emph{step} of the family, $n$ is the \emph{level} and $k$ the
\emph{count}.  The principal sum $P + \beta$ (\isn{veb\_padd}) of a
principal notation $P$ and a notation $\beta$ is $\beta$ when the
leading term of $\beta$ exceeds $P$, and $P$ prepended to $\beta$
otherwise.  With $A = \mathrm{vst}\,0\,1\,\alpha$: the ground-induction
bound $B\,\gamma = \varphi(n+1, (A+\gamma)+1)$ (\isn{vst\_gB}), its
segment bound $L\,\gamma$, the normalized $\varphi(n+1, A+\gamma)$
(\isn{vst\_gL}), and the number-induction bound $\Lambda = \varphi(m+1,
\mathrm{vst}\,N\,K\,\alpha + 1)$ (\isn{vst\_nB}), $m$ the level of the
step premise.  The stage judgment's bounds are all closed notations
$\mathrm{vst}\,n\,k\,\alpha$; the intermediate bounds are closed
notations built from them.
\begin{enumerate}\setcounter{enumi}{2}
\item \emph{The successor property} (\isn{veb\_succ\_le}): for valid
  $x < y$, $x+1 \leq y$.  By trichotomy, the proof need only exclude $y <
  x+1$; cofinality of the fundamental sequence (\isn{veb\_fs\_cofinal})
  at $x+1$, where the sequence is constantly $x$, gives $y \leq x$,
  against $x < y$.  The content is that nothing lies strictly between
  $x$ and $x+1$.
\item \emph{Closure under lower levels} (\isn{vst\_below}): if $n < a$
  and $x < \varphi(a, z)$ then $\mathrm{vst}\,n\,k\,x < \varphi(a,z)$,
  from $\omega < \varphi(a,z)$ for $a \neq 0$, a head comparison for the
  successor (\isn{veb\_p1\_below\_prin}), and the first clause of the
  order, $c < a$ and $y < \varphi(a,b)$ giving $\varphi(c,y) <
  \varphi(a,b)$, for the iterations (\isn{veb\_phi\_lower}).  The same
  argument gives $\mathrm{vst}\,n\,k\,\alpha < \varphi(\omega,0)$ for
  $\alpha < \varphi(\omega,0)$ (\isn{vst\_below\_pw}), $x < B\,\gamma
  \Rightarrow \mathrm{vst}\,n\,k\,x + 1 < B\,\gamma$ (\isn{vst\_gB\_clo})
  and $x < \Lambda \Rightarrow \mathrm{vst}\,m\,k\,x < \Lambda$
  (\isn{vst\_nB\_clo}).  These closure lemmas use no fundamental
  sequence.
\item \emph{The family.}  $\alpha$ and $\omega$ lie below
  $\mathrm{vst}\,n\,k\,\alpha$ (\isn{vst\_gt}, \isn{vst\_gt\_w}); the
  family is monotone in $n$, $k$ and $\alpha$ (\isn{vst\_mono},
  \isn{vst\_mono\_arg}); \emph{successor absorption}: $s \leq
  \mathrm{vst}\,n\,k\,\alpha$ implies $s + j < \mathrm{vst}\,n\,(k+1)\,
  \alpha$ for every $j$ (\isn{vst\_pn\_lt\_le}); \emph{composition}: $x
  \leq \mathrm{vst}\,n_1\,k_1\,\alpha$ implies $\mathrm{vst}\,n_2\,k_2\,x
  \leq \mathrm{vst}\,(\max(n_1,n_2))\,(k_2+k_1+1)\,\alpha$
  (\isn{vst\_comp\_max}; $\max(x,\omega)+1 \leq
  \mathrm{vst}\,n\,(k_1+1)\,\alpha$ by the successor property, then the
  iterates compose); and $\mathrm{vst}\,n\,k\,(\max(\alpha, \omega)) =
  \mathrm{vst}\,n\,k\,\alpha$ (\isn{vst\_max\_w}).  The principal sum is
  valid, strictly monotone in $\beta$, at least $P$ and $\beta$, and
  below every principal above $P$ and $\beta$ (\isn{veb\_padd\_valid},
  \isn{veb\_padd\_strict}, \isn{veb\_padd\_ge1}, \isn{veb\_padd\_ge2},
  \isn{veb\_padd\_below}).
\item \emph{The ground-induction bounds}, for valid arguments:
  $\delta < \gamma$ implies $B\,\delta \leq L\,\gamma$
  (\isn{vst\_gB\_le\_gL}: $(A+\delta)+1 \leq A+\gamma$ by the successor
  property); $L\,\gamma < B\,\gamma$ (\isn{vst\_gL\_lt\_gB}); for $\beta <
  \gamma$, $\mathrm{vst}\,n\,k\,(\max(\max(\alpha,\omega), \max(\beta,
  L\,\gamma) + j)) + 1 < B\,\gamma$ (\isn{vst\_gB\_clo\_h}, by
  \isn{vst\_gB\_clo}, since every argument of the maxima lies below
  $B\,\gamma$); and the class bound: $\gamma \leq \mathrm{vst}\,M\,k\,
  \alpha$, $n+1 \leq M'$ and $M \leq M'$ imply $B\,\gamma <
  \mathrm{vst}\,M'\,(k+2)\,\alpha$ (\isn{vst\_gB\_below}).  The stage
  moves \isn{omvS\_gB\_seg} (from $B\,\delta$ to $L\,\gamma$) and
  \isn{omvS\_gB\_clo} (from the stage of \isn{vst\_gB\_clo\_h} to
  $B\,\gamma$) apply these bounds by persistence.
\item \emph{The number-induction bound.}  $\alpha$, $\omega$ and
  $\mathrm{vst}\,N\,K\,\alpha$ lie below $\Lambda$ (\isn{vst\_nB\_gt\_hyp},
  \isn{vst\_nB\_gt\_w}, \isn{vst\_nB\_gt}),
  and $m+1 \leq N$ implies $\Lambda \leq \mathrm{vst}\,N\,(K+2)\,\alpha$
  (\isn{vst\_nB\_le}, by the successor property).
\item \emph{Numerals and the uniform universal.}  $\bar m\,\mathbf N$
  is true at $m+1$, hence at $\omega$ and above
  (\isn{omvS\_num\_N\_at}); if $u\cdot\bar n$ is true at one valid $s$
  for every $n$, then $\cA\cdot u$ is true at $s+1$
  (\isn{omvS\_allT\_unif}).
\end{enumerate}

\paragraph{Value congruence at stages.}
The relation $\mathrm{cong}_\tau(X, Y)$ (\isn{ovcrdS}) relates two
terms equal except at some \emph{sites}, where $X$ holds a well-formed
term with omega value $n$ and $Y$ a well-formed term with value $n$ at
the stage $\tau$.  The relation is closed under the constructors and under
substitution (\isn{ovcrS\_appI}, \isn{ovcrS\_subf1}, \isn{ovcrS\_bsub}),
relates a term to the numeral of its value and back
(\isn{ovcrS\_swap\_num\_trm}, \isn{ovcrS\_swap\_trm\_num\_w}), and each
inversion says that a related term is a site or has the same
constructor with related parts.
\begin{enumerate}\setcounter{enumi}{8}
\item \emph{The transport} (\isn{omh\_vcrS}).  Let $P$ be a valid
  principal notation with $\tau \leq P$.  If $\mathrm{val}_\beta(X, r)$
  and $\mathrm{cong}_\tau(X, Y)$, then $\mathrm{val}_{P+\beta}(Y, r)$;
  likewise for heads.  Rule induction, case for case \rga's value
  congruence (\isn{omh\_vcr}), carrying the stage.  At a site the value
  is identified by determinacy, and the fact for $Y$ moves from $\tau$
  to $P+\beta$ by persistence (\isn{vcrS\_site}); a head clause cannot
  sit at a site, since a term with a head form has no value
  (\isn{omv\_omh\_excl}).  Otherwise the induction hypotheses apply to
  the related parts, a premise at $\beta' < \beta$ lands at $P + \beta'
  < P + \beta$, and the same clause applies at $P+\beta$.
\item \emph{The use sites.}  With $P = \mathrm{vst}\,n\,(k+1)\,\alpha$,
  for a term at a site: $\tau \leq P$ and $\beta <
  \mathrm{vst}\,n\,(k+2)\,\alpha$ give the value at
  $\mathrm{vst}\,n\,(k+2)\,\alpha$ (\isn{omvS\_cong\_vst}).  With $P =
  \omega$, for a numeral at a site: $\tau \leq \omega$ and $\beta <
  \mathrm{vst}\,n\,(k+1)\,\alpha$ give the value at
  $\mathrm{vst}\,n\,(k+1)\,\alpha$ (\isn{omvS\_cong\_num}).
\end{enumerate}

\paragraph{The staged transport.}
The staged oracle reading $\mathrm{read}^\alpha_S(k, \sigma, t, b)$
(\isn{osemS}) is the oracle reading of Appendix~\ref{app:gfix} with a
subterm without $k$ read by its value at the stage $\alpha$; the
reading is monotone in the stage and in the atom set (\isn{osemS\_mono}).  The
\emph{entry set} $\mathrm{En}_\alpha(T) = \{m \mid T\cdot\bar m \text{
true at } \alpha\}$ (\isn{EnS}) grows with $\alpha$.  The invariant
(\isn{gfInvS}) is the parametric motive of Appendix~\ref{app:gfix}
without its atom part, at the conclusion's stage $\alpha$ with atoms in
$\mathrm{En}_\alpha(T)$: parts \isn{gfPS}, \isn{gfNS} for positive and
negative bodies, and \isn{gfRPS}, \isn{gfRNS} for the instance redexes.
The atom part is not needed, because membership in
$\mathrm{En}_\alpha(T)$ is the truth of the instance at $\alpha$; so the
invariant carries no motive, no set and no closure hypothesis.  The
invariant persists upward (\isn{gfInvS\_up}).
\begin{enumerate}\setcounter{enumi}{10}
\item \emph{The transport} (\isn{gfix\_transport\_stage}): every value
  fact $\mathrm{val}_\alpha(t, r)$ satisfies the invariant at $\alpha$,
  for every term $T$.  Rule induction.  In each value case, first lift
  each premise's invariant to $\alpha$ by \isn{gfInvS\_up}; the parts
  then follow by shape analysis (\isn{posP\_substE},
  \isn{negP\_substE}) as in steps 1 to 8 of Appendix~\ref{app:gfix}.  In
  the beta and recursor cases the atom shape $T\cdot\bar m$ closes
  directly, since the conclusion itself says $m \in
  \mathrm{En}_\alpha(T)$.  In the universal-true case the invariant of
  each instance is lifted from its own stage separately; no supremum is
  formed.
\item \emph{The headline at an abstract bound}
  (\isn{gfix\_stage\_ind\_set}).  Let $\Phi$ be parametrically positive,
  $T = \mu\Phi$, $B$ any map on notations and $Q(s, m)$ any predicate.
  Closure premise: for valid $\beta < \gamma$, if $\Phi$ reads true at
  stage $\beta$ at the point $\bar m$, with atoms in the set of those
  $m'$ whose instance is true at some valid $\delta < \gamma$ with
  $Q(B\,\delta, m')$, then $Q(B\,\gamma, m)$.  Conclusion: if
  $T\cdot\bar n$ is true at $\gamma$ then $Q(B\,\gamma, n)$.  Proof by
  well-founded induction on $\gamma$ (\isn{veb\_wf}), with motive ``for
  every $m$, if $T\cdot\bar m$ is true at $\gamma$ then
  $Q(B\,\gamma,m)$'': unfold downward to the body true at some $\beta <
  \gamma$; the transport reads $\Phi$ true at $\beta$ with atoms in
  $\mathrm{En}_\beta(T)$; the induction hypothesis at $\beta$ gives
  $Q(B\,\beta, m')$ for each atom $m'$; \isn{osemS\_mono} and the
  closure premise finish.  The induction runs over any stage at which
  the instance is true, not a least one, and $B$ enters only through
  the closure premise.
\item \emph{Replay} (\isn{osemS\_term}).  If
  $\mathrm{read}^\beta_S(k,\sigma,t,r)$ with $\beta$, $\tau$ valid and
  $P\cdot\bar m$ true at $\tau$ for every $m \in S$, then $\sigma[k :=
  P]\downarrow t$ has the same value $r$ at $\max(\beta,\tau) + |t|$,
  $|t|$ the size of the body pattern $t$.  Rule induction on the
  reading: a subterm without $k$ moves from $\beta$ by persistence, an
  atom becomes $P\cdot\bar m$, true at $\tau$, and each connective adds
  one clause over a strictly smaller pattern.  Every instance of a
  universal is read at the same body $d$, so all instance redexes land
  at one stage and the universal two above it; the depth is uniform
  across the instances, though in ground induction (step 20) the
  pattern is the substituted body, so the replay depth there depends on
  the substitution.
\item \emph{Head steps.}  If $c$ head-steps to $c'$, a value of $c$ at
  $s$ is a value of $c'$ at $s$ (\isn{omvS\_hstep\_rev}); conversely
  there is a spine depth $d$ such that a value of $c'$ at $s$ gives the
  value of $c$ at $\max(\omega, s) + d$ (\isn{om\_hstep\_simS}), with
  $d = 1$ for a beta or recursor step (the recursor head is reached at
  $2$) and one more per spine position.
\end{enumerate}

\paragraph{The stage judgment and the rule costs.}
$H \ent_{n,k} c$ (\isn{sjtS}) means: for every substitution $\sigma$
of well-formed terms and every valid $\alpha$, if every member of
$\sigma\downarrow H$ is true at $\alpha$, then $\sigma\downarrow c$ is
true at $\mathrm{vst}\,n\,k\,\alpha$.  The bound depends on $\alpha$
only, never on $\sigma$, and universal introduction, number induction
and the closure premise of ground induction consume that uniformity.
A fact at any $s \leq \mathrm{vst}\,n\,k\,\alpha$ suffices
(\isn{sjtS\_leI}).  The step lemma (\isn{gi\_step\_sjtS}) has one case
per rule, forty-three in all; its hypothesis gives some bound for
each premise, and the validity of the proof gives each premise
well-formed.  The table names the case lemmas (\isn{sjtS\_cut} is the
case of \isn{pr\_cut}); $M$ and $K$ are the maxima of the premises'
levels and counts.
\begin{center}
\footnotesize
\begin{tabular}{@{}p{0.2\linewidth}p{0.74\linewidth}@{}}
\toprule
bound & rules \\
\midrule
the premise's & \isn{sjtS\_sub1}, \isn{sjtS\_wk1}; \isn{sjtS\_true\_eqE},
\isn{sjtS\_false\_eqI}, \isn{sjtS\_false\_eqE}, \isn{sjtS\_dnegE}, \isn{sjtS\_disjE2},
\isn{sjtS\_disjE3}, \isn{sjtS\_eq\_refl1}, \isn{sjtS\_eq\_refl2}, \isn{sjtS\_eq\_sym},
\isn{sjtS\_ne\_sym}, \isn{sjtS\_suc\_eqE}, \isn{sjtS\_suc\_neE}, \isn{sjtS\_suc\_TE};
\isn{sjtS\_convE}, \isn{sjtS\_conv\_negE} \\
$(0,0)$ & \isn{sjtS\_hyp}, \isn{sjtS\_0}; vacuous: \isn{sjtS\_negE}, \isn{sjtS\_bot\_TE} \\
$(M, K+1)$ & \isn{sjtS\_true\_eqI}, \isn{sjtS\_disjI1}, \isn{sjtS\_disjI2},
\isn{sjtS\_disjI3}, \isn{sjtS\_suc\_neI}, \isn{sjtS\_suc\_TI}, \isn{sjtS\_conv},
\isn{sjtS\_conv\_eq}, \isn{sjtS\_allI} \\
$(M, K+2)$ & \isn{sjtS\_dnegI}, \isn{sjtS\_eqTI}, \isn{sjtS\_suc\_ne0};
\isn{sjtS\_eq\_subs}, \isn{sjtS\_R\_suc}, \isn{sjtS\_allE}, \isn{sjtS\_not\_allI} \\
$(n, k+3)$ & \isn{sjtS\_eqTE} \\
composed & \isn{sjtS\_cut}, \isn{sjtS\_not\_allE}: $(\max(n_1,n_2),
k_2+k_1+1)$; \isn{sjtS\_disjE1}: $(\max(n_1,n_2,n_3), \max(k_2,k_3)+k_1+1)$ \\
iterated & \isn{sjtS\_nat\_ind}: $(\max(n_0, n_2, n_1+1), K+3)$;
\isn{sjtS\_gfix}: $(\max(n_1, n_3, n_2+1), K+3)$ \\
\bottomrule
\end{tabular}
\end{center}
Seventeen rules keep the premise's bound: the two structural rules,
thirteen truth rules (by inversion, by definition, or by clauses at the
premise's stage), and the two reverse conversions.  Hypothesis and zero
sit at the base bound.  Two rules are vacuous: negation elimination,
since a term and its negation are never both true (\isn{omvS\_unique}),
and the elimination of $\bot\,\mathbf N$, since $\bot = \bot$ has no
value in the omega model (\isn{omv\_eq\_inv}).  The introductions cost
one or two steps of the family, each clause adding a successor stage
that successor absorption takes into the next step; the elimination of
a decided equation costs three.  A forward conversion costs its spine
depth in successor stages, which one step absorbs; reverse conversion
costs nothing.  The four congruence sites (substitution of equals, the
recursor at a successor, universal elimination, the refuted universal)
cost two steps.  The level rises by one only at the step premise of
number induction and the closure premise of ground induction.
\begin{enumerate}\setcounter{enumi}{14}
\item \emph{Cut} (\isn{sjtS\_cut}).  Fix $\sigma$ and $\alpha$ with the
  hypotheses true at $\alpha$, and let $s_1 = \mathrm{vst}\,n_1\,k_1\,
  \alpha$.  The first premise makes $\sigma\downarrow a$ true at $s_1$,
  and the hypotheses persist to $s_1$; the second premise at the stage
  $s_1$ makes $\sigma\downarrow c$ true at $\mathrm{vst}\,n_2\,k_2\,s_1$,
  which composition bounds.
\item \emph{Universal introduction} (\isn{sjtS\_allI}).  From
  $i\,\mathbf N, H \ent_{n,k} f\cdot i$ at a fresh $i$, conclude $H
  \ent_{n,k+1} \cA\cdot f$.  Let $\alpha' = \max(\alpha, \omega)$.  For
  each $m$ the premise at $\bar m :: \sigma$ and $\alpha'$ has its
  hypotheses true there ($\bar m\,\mathbf N$ since $m+1 < \omega$, the
  members of $H$ by persistence), so $(\sigma\downarrow f)\cdot\bar m$ is
  true at $\mathrm{vst}\,n\,k\,\alpha' = \mathrm{vst}\,n\,k\,\alpha$, the
  same stage for every $m$.  \isn{omvS\_allT\_unif} makes the universal
  true one stage above, and successor absorption puts the universal at
  $\mathrm{vst}\,n\,(k+1)\,\alpha$.
\item \emph{Number induction} (\isn{sjtS\_nat\_ind}).  Premises $H
  \ent_{n_0,k_0} c[0/i]$, $i\,\mathbf N, c, H \ent_{n_1,k_1} c[S\,i/i]$
  and $H \ent_{n_2,k_2} a\,\mathbf N$, with $i$ not free in $H$ and $a$
  well-formed.  Write $N = \max(n_0, n_2, n_1+1)$, $K = \max(k_0,k_1,k_2)$
  and $\Lambda = \varphi(n_1+1, \mathrm{vst}\,N\,K\,\alpha +
  1)$.  By number induction in HOL on $m$, with motive
  ``$\sigma[i := \bar m]\downarrow c$ is true at some valid $s <
  \Lambda$'': at $0$, the base premise gives the stage
  $\mathrm{vst}\,n_0\,k_0\,\alpha < \Lambda$; at $m+1$, the step premise
  at the hypothesis stage $h = \max(\max(\alpha, \omega), s) < \Lambda$
  gives the next instance at $\mathrm{vst}\,n_1\,k_1\,h$, below $\Lambda$
  by \isn{vst\_nB\_clo}.  Then the third premise gives a value $v$ of
  $\sigma\downarrow a$ below $\mathrm{vst}\,n_2\,k_2\,\alpha$; the
  instance at $v$ holds at $\Lambda \leq \mathrm{vst}\,N\,(K+2)\,\alpha$;
  congruence from $\bar v$ back to $\sigma\downarrow a$
  (\isn{omvS\_cong\_vst}) lands $\sigma\downarrow c[a/i]$ at
  $\mathrm{vst}\,N\,(K+3)\,\alpha$.  The instances sit at growing stages
  below one closed bound; no supremum over them is formed.
\end{enumerate}

\paragraph{Ground induction} (\isn{sjtS\_gfix}).
Premises: $H \ent_{n_1,k_1} \mu\Phi\cdot x$; the closure premise $i\,
\mathbf N, \Phi[\lambda c/a][i/x], H \ent_{n_2,k_2} c[i]$; $H
\ent_{n_3,k_3} x\,\mathbf N$; $\Phi$ parametrically positive, $x$
well-formed, $i$ not free in $H$, $\lambda c$ or $\lambda a.\lambda
x.\Phi$.  Conclusion: $H \ent_{M,K+3} c[x]$ with $M = \max(n_1, n_3,
n_2+1)$ and $K = \max(k_1,k_2,k_3)$.  Fix $\sigma$ and $\alpha$ with
the hypotheses true at $\alpha$, and write $\Phi_\sigma$, $c_\sigma$
and $P = \lambda c_\sigma$ as in Appendix~\ref{app:gfix}.
\begin{enumerate}\setcounter{enumi}{17}
\item \emph{The point's value.}  The third premise gives a value $v$
  of $\sigma\downarrow x$ at some $\beta_x < \mathrm{vst}\,n_3\,k_3\,
  \alpha$.
\item \emph{The instance at the numeral.}  The first premise and
  congruence at a numeral site (\isn{ovcrS\_swap\_trm\_num\_w},
  \isn{omvS\_cong\_num}) make $\mu\Phi_\sigma\cdot\bar v$ true at
  $\gamma_0 = \mathrm{vst}\,M\,(K+1)\,\alpha$.
\item \emph{The headline at the closed bound.}  Apply step 12 with
  $B\,\gamma = \varphi(n_2+1, (A+\gamma)+1)$ and $Q(s, m)$ = ``$P\cdot\bar
  m$ is true at $s$''.  The headline's closure premise, at $\beta < \gamma$ and
  $m$, is discharged in five moves.  (a) \emph{Segment}: every atom $m'$
  has $P\cdot\bar m'$ true at some $B\,\delta$ with $\delta < \gamma$,
  hence at $L\,\gamma$ (\isn{omvS\_gB\_seg}).  (b) \emph{Replay} with
  $\tau = L\,\gamma$: $\Phi_\sigma[P/a][\bar m/x]$ is true at $h_0 =
  \max(\beta, L\,\gamma) + |\Phi_\sigma|$.  (c) \emph{The rule's closure
  premise} at $\sigma[i := \bar m]$ and the stage
  $\max(\max(\alpha,\omega), h_0)$, where $\bar m\,\mathbf N$, the body
  hypothesis (from move b, by \isn{bsub\_asub} and \isn{subf\_gsub\_par})
  and the members of $H$ are true, the freshness conditions leaving
  $c_\sigma$ and $\Phi_\sigma$ unchanged: $c_\sigma[\bar m]$ is true at
  $\mathrm{vst}\,n_2\,k_2$ of that stage.  (d) \emph{The motive}: one
  beta clause makes $P\cdot\bar m$ true one stage above.  (e)
  \emph{Below the bound}: \isn{omvS\_gB\_clo} with $j = |\Phi_\sigma|$
  puts $P\cdot\bar m$ at $B\,\gamma$.  So $P\cdot\bar v$ is true at
  $B\,\gamma_0$.
\item \emph{The reduct.}  The head step keeps the stage:
  $c_\sigma[\bar v]$ is true at $B\,\gamma_0$ (\isn{omvS\_hstep\_rev}).
\item \emph{Congruence back and the class bound.}  \isn{ovcrS\_bsub}
  relates $c_\sigma[\bar v]$ to $c_\sigma[\sigma\downarrow x]$, with the
  site at $\beta_x$; $B\,\gamma_0 < \mathrm{vst}\,M\,(K+3)\,\alpha$ by
  \isn{vst\_gB\_below} with $n_2+1 \leq M$; \isn{omvS\_cong\_vst} makes
  $\sigma\downarrow c[x]$ true at $\mathrm{vst}\,M\,(K+3)\,\alpha$.
\end{enumerate}
The level rises by one at the closure premise because $B\,\gamma$, a
value of $\varphi(n_2+1,\cdot)$, is closed under the level-$n_2$ bound
of the closure derivation.  The replay depth and the hypothesis stage
depend on $\sigma$, and the closure of $B\,\gamma$ under finite
successors and under $\mathrm{vst}\,n_2\,k_2$ absorbs the two
dependencies.

\paragraph{The induction over proofs and stage soundness.}
\begin{enumerate}\setcounter{enumi}{22}
\item \isn{gi\_prf\_sjtS}: if a valid proof list contains $H \ent c$,
  then $H \ent_{n,k} c$ for some $n$ and $k$, by induction on the list,
  the step lemma closing the new judgment at the head with
  \isn{gi.valid\_wff\_exp} for the premises' well-formedness.
\item \isn{gi\_stage\_sound}: a derivation gives a valid proof
  (\isn{entails\_containsE}); step 23 and the identity substitution
  finish.  \isn{gi\_stage\_sound\_pw} adds \isn{vst\_below\_pw}; the
  closed form takes $\alpha = \omega$, with $\mathrm{vst}\,n\,k\,\omega <
  \varphi(n+1,0)$ by \isn{vst\_below} and $\varphi(n+1,0) <
  \varphi(\omega,0)$ by \isn{veb\_phi\_lower}.
\end{enumerate}
Stage soundness refines omega soundness, since a stage fact is an
omega fact (\isn{omvS\_omv}; for closed theorems
\isn{gi\_thm\_omv\_stage}), and needs no well-formed context.  Read off
the case table rather than stated in HOL: a derivation whose number
and ground inductions nest at most $n$ deep through their step and
closure premises has level at most $n$; the HOL statements quantify
the level existentially.

\paragraph{Rank and the ordinal.}
$(y,x)$ is in the not-refuted relation $\mathrm{accRelN}(R)$
(\isn{accRelN}) when $R\cdot\bar y\cdot\bar x$ is not omega-false.
\begin{enumerate}\setcounter{enumi}{24}
\item \emph{The rank step} (\isn{acc\_stage\_rank}).  If $R$ is closed,
  $\Acc_R\cdot\bar x$ is true at $s$, and $(y,x) \in
  \mathrm{accRelN}(R)$, then $\Acc_R\cdot\bar y$ is true at some valid
  $s' < s$.  Unfold $\Acc_R = \mu(\cA\cdot\lambda v.\, \neg(R\cdot v\cdot
  x) \vee a\cdot v)$ (\isn{acc\_as\_gfix}) downward, and invert the
  universal at $y$, the instance redex and the disjunction, each
  inversion strictly below the last; a true left disjunct would make
  $R\cdot\bar y\cdot\bar x$ omega-false.
\item \emph{The bridge.}  $\mathrm{accRelN}(\mathrm{vOp})$ is the
  relation of $\mathrm{vOp}$ on numbers, since $\mathrm{vOp}$ is decided
  at numerals (\isn{accRelN\_vOp\_eq}), and $(\lceil a\rceil, \lceil
  b\rceil)$ is in the relation exactly when $a$ and $b$ are valid and $a < b$
  (\isn{veb\_denote}).
\item \emph{The rank lemma} (\isn{acc\_stage\_ge}).  For valid $x$, if
  $\Acc\cdot\lceil x\rceil$ is true at $s$ then $x \leq s$.  Well-founded
  induction on $x$ (\isn{veb\_wf}), with motive ``for every $s$, if
  $\Acc\cdot\lceil x\rceil$ is true at $s$ then $x \leq s$''.  Suppose $s
  < x$, by trichotomy.  The stage $s$ is itself a notation, so by the
  bridge and the rank step $\Acc\cdot\lceil s\rceil$ is true at some $s'
  < s$; the induction hypothesis at $s$, at the stage $s'$, gives $s \leq
  s'$, so $s < s$, against irreflexivity.
\item \emph{The level bound} (\isn{gi\_acc\_below\_level}).  By the
  closed form, $\Acc\cdot\lceil x\rceil$ is true at $S =
  \mathrm{vst}\,n\,k\,\omega < \varphi(n+1,0)$, and the rank lemma gives
  $x \leq S$.
\item \emph{The ordinal} (\isn{gi\_acc\_iff\_below\_pw}).  Left to
  right: step 28.  Right to left: cofinality of the fundamental sequence
  of $\varphi(\omega,0)$ places $x$ at or below its $k$-th term
  $\varphi_k^k(0)$ (\isn{veb\_fs\_pw}), which lies below $\varphi(k+1,0)$;
  every valid notation below $\varphi(k+1,0)$ is derivably accessible
  (\isn{gi\_acc\_below\_phi\_fin}, from Theorem~\ref{thm:veblen}).
\end{enumerate}

\paragraph{The independence results.}
Two facts recur: if $\cA\cdot\lambda b$ is true at $\beta$, then
$b[\bar n]$ is true below $\beta$ (\isn{omvS\_all\_inst}); and if $\neg g
\vee b$ is true at $\beta$ and $g$ is omega-true, then $b$ is true below
$\beta$ (\isn{omvS\_guard\_or}).
\begin{enumerate}\setcounter{enumi}{29}
\item \emph{Above $\varphi(\omega,0)$} (\isn{gi\_not\_WF\_above\_pw}).
  $\WF(D, R) = \cA\cdot\lambda v.\, \neg(D\cdot v) \vee \Acc_R\cdot v$
  (\isn{trmWF}).  If $\emptyset \ent \WF(D, \mathrm{vOp})$, the closed
  form makes it true at some $S < \varphi(\omega,0)$; the instance at
  $\lceil x\rceil$ and the omega-true guard $D\cdot\lceil x\rceil$ make
  $\Acc\cdot\lceil x\rceil$ true below $S$; the rank lemma gives $x < S <
  \varphi(\omega,0) \leq x$.  \isn{gi\_not\_WF\_gamma0} is the instance
  $D = \mathrm{Nat}$ (\isn{trmNatD}) at $x = \varphi(\omega,0)$.
\item \emph{Bivalence} (\isn{gi\_not\_AccBiv}, \isn{gi\_not\_ObRC}):
  $\mathrm{AccBiv}$ is equivalent to $\emptyset \ent \WF(\mathrm{Nat},
  \mathrm{vOp})$ (\isn{gi\_WF\_gamma0\_iff\_bivalence}) and $\mathrm{ObRC}$
  to $\mathrm{AccBiv}$ (\isn{ObRC\_iff\_AccBiv}), and step 30 refutes the
  well-foundedness sentence.
\item \emph{The ladder} (\isn{gi\_not\_LADDER}).  $\mathrm{RC}(c,d) =
  \cA\cdot\lambda v.\,\neg\mathrm{rP}(c,d,v) \vee \Acc\cdot v$ says that
  every valid principal code with fields $c$ and $d$ is accessible
  (\isn{RC}).  If the ladder is derivable, the closed form makes it true
  at $S = \mathrm{vst}\,n\,k\,\omega < \varphi(n+1,0)$.  Instantiate at
  the code $c$ of the finite notation $n+1$ (\isn{vfinc}): the guard
  $\mathrm{vOp}\cdot c\cdot\lceil\omega\rceil$ is omega-true by evaluation
  (\isn{vOp\_omv}), so $\mathrm{RC}(c,0)$ is true below $S$.  Instantiate
  $\mathrm{RC}(c,0)$ at the code $x$ of $\varphi(n+1,0)$: the guard
  $\mathrm{rP}(c, 0, x)$ is derivable (\isn{vk\_rP\_I}), hence omega-true,
  so $\Acc\cdot x$ is true below $S$.  The rank lemma gives
  $\varphi(n+1,0) < S < \varphi(n+1,0)$.  The instance chosen is one
  level above the stage bound of the derivation, so the same derivation
  cannot reach it.
\end{enumerate}

\paragraph{Lemmas a re-mechanizer needs.}
The omega model with determinacy, exclusivity of heads and values, and
inversions; the Veblen notation as a datatype with validity, the order,
codes, the successor, the normalized $\varphi$, the iterates and the
fundamental sequences.  The order facts on notations, irreflexivity
and transitivity (\isn{veb\_lt\_irrefl}, \isn{veb\_lt\_trans}), are
pulled back from \iga derivations through soundness; trichotomy
(\isn{veb\_lt\_trichotomy}) is proved by structural induction from them;
well-foundedness (\isn{veb\_wf}) is proved directly in HOL, with no \iga
derivation.  Cofinality of the fundamental sequence is used twice: at
a successor for the successor property, on which the segment bound,
composition, successor absorption and monotonicity in the hypotheses'
stage depend, and at $\varphi(\omega,0)$ for the converse of the
ordinal.  Then steps 3 to 14; the forty-three rule lemmas; the
substitution algebra and positivity facts of Appendix~\ref{app:gfix};
and, for the converse and the independence results,
Theorem~\ref{thm:veblen}, the equivalence of Appendix~\ref{app:gamma0}
and \isn{ObRC\_iff\_AccBiv}.  Every stage theory ends with a check that
its theorems use no oracle.  Persistence, and through it every theorem
above, uses the pulled-back transitivity, so the stage theorems depend
on omega soundness by construction, and no dependency check against it
is possible.

\paragraph{Pricing the other rules: the systems.}
The theorems of \S\ref{sec:pricing} are in session
\textsf{OIGA\_Stage}: \isn{GStage.thy} (the decidedness rule at
stages and \oiga), \isn{GAdoptStage.thy} (adoption over \oiga),
\isn{GAutoStage.thy} (the progression with the license in the empty
context, and the control) and \isn{GAutoCtx.thy} (the license in
context).  The systems \oiga and $\oiga+\mathrm{Adopt}$ are defined in
\isn{Union.thy} and \isn{GAdopt.thy} of session \textsf{OIGA}.  The
sources name the systems as follows: $\oiga+\mathrm{Adopt}$ is
C0, the locale \isn{oiaga}; $\oiga+\mathrm{Refl}$ is C\_HOL
(\isn{chol\_derives}, entailment \isn{chol\_ent}); the progression is
\isn{cauto\_vd} with its license flag set (C\_auto, entailment
\isn{cauto\_aent}) or cleared (the control), and \isn{cauto2\_vd} in
the in-context form.  In this appendix $\varphi(\omega,1)$ is
\isn{veb\_pw1}, $F = \varphi(\omega+1,0)$ is \isn{cauto\_F}, and $R(a)$
(\isn{cauto\_R}) is the normalized $\varphi(\omega,a)$.  No valid
notation below $F$ is a fixed point of $\varphi(\omega,\cdot)$
(\isn{cauto\_notfixd\_below\_F}), so below $F$, $R(a)$ is
$\varphi(\omega,a)$ itself.  ``\oiga derives $c$ from $\Gamma$'' is
\isn{oiga\_ent}, in the locale \isn{oiga}, which assumes nothing beyond
the syntax.  \oiga's step relation is the union of \oga's and \iga's
(\isn{oiga\_derives}), and \oiga contains both systems
(\isn{gi\_oiga\_embed}, \isn{oga\_oiga\_embed}).

\paragraph{The decidedness rule and \oiga.}
\begin{enumerate}\setcounter{enumi}{32}
\item \emph{The rule at stages} (\isn{omvS\_ati\_inst}).  If
  $(\lambda x.\,(u\cdot x)\,\mathbf B)\cdot\bar m$ is true at $\beta$,
  then $u\cdot\bar m$ is true or false at some valid $\gamma < \beta$.
  The case lemma (\isn{sja\_allTI}, at the bound family
  \isn{sjtS\_allTI}): from $H \ent_{n,k} \cA\cdot\lambda x.\,(u\cdot
  x)\,\mathbf B$ conclude $H \ent_{n,k+1} (\cA\cdot u)\,\mathbf B$.  The
  proof splits on whether every instance has a true stage below the
  premise's stage, so it uses neither the uniqueness of values nor
  persistence, and the at most three successor stages of the
  universal's clause, the negation and the disjunction are absorbed by
  one step (\isn{vst\_pn\_lt}).
\item \emph{The step lemma} (\isn{oiga\_step\_sjtS}).  An \oga step is
  an \rga step or a decidedness step (\isn{oga\_step\_split}); an \rga
  step is an \iga step (\isn{gi.derives\_embed}); \iga's forty-three
  cases (\isn{gi\_step\_sjtS}) and step 33 finish.  The induction over
  proof lists (\isn{step\_prf\_sjtS}) is stated once for any step
  relation of the generic calculus whose steps satisfy the stage
  judgment; the adoption systems use its generalization to an
  arbitrary invariant (\isn{adst\_prf\_inv}).
\item \emph{Stage soundness} (\isn{oiga\_stage\_sound}): if \oiga
  derives $c$ from $\Gamma$, $\alpha$ is valid and every member of
  $\Gamma$ is true at $\alpha$, then $c$ is true at
  $\mathrm{vst}\,n\,k\,\alpha$ for some $n$ and $k$.  The closed form
  (\isn{oiga\_stage\_sound\_below\_pw}): if \oiga derives $c$ from
  $\emptyset$, then $c$ is true at $\mathrm{vst}\,n\,k\,\omega <
  \varphi(n+1,0) < \varphi(\omega,0)$ for some $n$ and $k$.
\item \emph{The ordinal} (\isn{oiga\_acc\_iff\_below\_pw}): for valid
  $x$, \oiga derives $\Acc\cdot\lceil x\rceil$ from $\emptyset$ if and
  only if $x < \varphi(\omega,0)$.  Left to right by the rank lemma
  (\isn{oiga\_acc\_below\_level}); right to left by \isn{gi\_oiga\_embed}
  applied to \isn{gi\_acc\_below\_phi\_fin}, after the cofinality of the
  fundamental sequence of $\varphi(\omega,0)$ (\isn{veb\_fs\_cofinal}).
\item \emph{Independence.}  If $D$ is closed, $x$ is valid with
  $\varphi(\omega,0) \leq x$ and $D\cdot\lceil x\rceil$ is omega-true,
  \oiga does not derive $\WF(D,\mathrm{vOp})$
  (\isn{oiga\_not\_WF\_above\_pw}); in particular not
  $\WF(\mathrm{Nat},\mathrm{vOp})$ (\isn{oiga\_not\_WF\_gamma0}).
  Bivalence of accessibility for $\mathrm{vOp}$ fails
  (\isn{oiga\_not\_accbiv}), stated in the generic calculus's form
  \isn{gcalc\_accbiv}, which is equivalent to the well-foundedness
  sentence (\isn{g\_WF\_gamma0\_iff\_accbiv}).  The ladder sentence is not
  derivable (\isn{oiga\_not\_LADDER}).  Accessibility is not decided as a
  rule: it is not the case that for every well-formed $\Delta$ and
  every $d$ with $\Delta \ent d\,\mathbf N$, $\Delta \ent
  (\Acc\cdot d)\,\mathbf B$ (\isn{oiga\_not\_dec\_acc}).  Such a
  decision would give bivalence by disjunction elimination on $\Acc\cdot
  d \vee \neg(\Acc\cdot d)$, the premise's derivation in one branch and
  the hypothesis rule in the other.  Nor is it decided as the sentence
  $\cA\cdot\lambda y.\,(\Acc\cdot y)\,\mathbf B$
  (\isn{oiga\_not\_dec\_acc\_sentence}).
\item \emph{Transfinite induction as a sentence}
  (\isn{oiga\_TI\_decided\_below\_pw}).  If $x$ is valid, $x <
  \varphi(\omega,0)$, $M$ is closed, and for every $\Delta$ and $a$,
  $\Delta \ent a\,\mathbf N$ implies $\Delta \ent (M\cdot a)\,\mathbf
  B$, then \oiga derives $\mathrm{Prog}(M) \to M\cdot\lceil x\rceil$
  from $\emptyset$, with $\mathrm{Prog}$ as in \cref{thm:oiga-ti}.  The
  derivation: accessibility of $\lceil x\rceil$ by step 36; \iga's
  transfinite induction (\isn{gi\_TI}) in the context
  $\{\mathrm{Prog}(M)\}$; the decidedness rule twice
  (\isn{ir\_allTI\_open}) to decide $\mathrm{Prog}(M)$, the order being
  decided on numbers (\isn{gi.vk\_Op\_B}); and the deduction theorem at
  a decided hypothesis (\isn{ir\_impI\_bool}).
\end{enumerate}

\paragraph{Adoption.}
\begin{enumerate}\setcounter{enumi}{38}
\item \emph{The systems.}  The locale \isn{oiaga} extends \isn{oiga}
  by an index $h$ with the law \isn{hO\_law}: $h$ computes \oiga's
  empty-context proof search at a paired instance code; such an index
  exists (\isn{oiaga\_head\_exists}), and every result below holds for
  every $h$ that satisfies the law.  The warrant of $u$ is
  \isn{warrant\_oiga}~$u$; the rule (\isn{ir\_adopt\_oiga}): if $\Gamma
  \ent \mathrm{warrant}(u)$, $u$ is well-formed and $u$ is closed, then
  $\Gamma \ent \cA\cdot u$.  For closed $u$, the warrant is omega-true
  if and only if \oiga derives $u\cdot\bar n$ from $\emptyset$ for every
  $n$ (\isn{warrant\_oiga\_omv\_iff}).  The system is sound for the
  omega model and consistent (\isn{oiaga\_omega\_sound},
  \isn{oiaga\_consistent}).  $\oiga+\mathrm{Refl}$ adds to \oiga's steps
  the premise-free steps \isn{chol\_adoptS}: a conclusion $\cA\cdot u$
  with $u$ well-formed and closed and every $u\cdot\bar m$ derivable in
  \oiga from $\emptyset$, at any context; the derived rule is
  \isn{ir\_chol\_adopt}.
\item \emph{The adoption case} (\isn{adst\_adopt\_case}).  Let $M$ be
  closed and such that $M$ true at $\alpha$ implies $\varphi(\omega,0)
  \leq \alpha$, let $u$ be closed, and suppose that for every
  substitution $\sigma$ of well-formed terms and every valid $\alpha$
  at which the members of $\sigma\downarrow(M \triangleright H)$ are
  true, every $u\cdot\bar m$ is derivable in \oiga from $\emptyset$.
  Then $M \triangleright H \ent_{0,0} \cA\cdot u$.  Each instance is true
  below $\varphi(\omega,0)$ by step 35, hence below
  $\mathrm{vst}\,0\,0\,\alpha$, and the universal-true clause applies;
  the warrant's own stage is never used.  The marker
  $\Acc\cdot\lceil\varphi(\omega,0)\rceil$ (\isn{adst\_mark}) satisfies
  the hypothesis on $M$ by the rank lemma (\isn{adst\_mark\_ge}), and it
  is true at some valid stage below $\varphi(\omega,1)$
  (\isn{gi.adst\_mark\_stage}): $\cA\cdot u_A$ is true at
  $\varphi(\omega,0)+1$ (\isn{gi.adst\_uA\_stage}), \iga derives the marker
  from $\cA\cdot u_A$ (\isn{gi.adst\_acc\_pw}), and \iga's stage soundness
  puts it at $\mathrm{vst}\,n\,k\,(\varphi(\omega,0)+1) <
  \varphi(\omega,1)$ (\isn{vst\_below\_pw1}).  The closed marker is
  added to every context of a proof list (\isn{adst\_aug},
  \isn{gi.adst\_gi\_aug} for \iga's rules, \isn{adst\_oiga\_step\_marked}
  for \oiga's step), which gives stage soundness under the marker
  (\isn{oiaga\_stage\_sound\_marked}).  At an adoption step of
  $\oiga+\mathrm{Adopt}$ the warrant's judgment is used only for the
  warrant's omega-truth (\isn{omvS\_omv}, then
  \isn{warrant\_oiga\_omv\_iff}).  Adopted sentences cannot be treated
  instead as global hypotheses: an adoption under unsatisfiable
  hypotheses can adopt a false sentence.
\item \emph{The bound of $\oiga+\mathrm{Adopt}$.}  If
  $\oiga+\mathrm{Adopt}$ derives $c$ from $\emptyset$, then $c$ is true at
  some valid $s < \varphi(\omega,1)$ (\isn{oiaga\_stage\_sound\_closed}).
  Hence a derived $\Acc\cdot\lceil x\rceil$ has $x < \varphi(\omega,1)$
  (\isn{oiaga\_acc\_below\_pw1}), $\Acc\cdot\lceil\varphi(\omega,1)\rceil$
  is not derivable (\isn{oiaga\_not\_acc\_pw1}), the well-foundedness
  results of step 37 hold at $\varphi(\omega,1)$
  (\isn{oiaga\_not\_WF\_above\_pw1}, \isn{oiaga\_not\_WF\_gamma0}), and
  bivalence fails (\isn{oiaga\_not\_accbiv}).  Every valid notation
  below $\varphi(\omega,0)$ is derivably accessible by step 36 and the
  embedding \isn{oiga\_oiaga\_embed}.
\item \emph{The two bodies and the lower bound.}  $u_A$ is
  \isn{adst\_uA} $= \lambda y.\,\neg(\mathrm{vOp}\cdot y\cdot
  \lceil\varphi(\omega,0)\rceil) \vee \Acc\cdot y$, and $u_C$ is
  \isn{adst\_uC} $= \lambda z.\,\neg(\mathrm{vOp}\cdot z\cdot
  \lceil\varphi(\omega,0)+1\rceil) \vee \mathrm{vOp}\cdot
  \mathrm{CanVT}(z)\cdot\lceil\varphi(\omega,0)\rceil \vee
  \mathrm{CanVT}(z) = \lceil\varphi(\omega,0)\rceil$, where
  $\mathrm{CanVT}$ is the canonicalizer.  For every $\Gamma$ and $m$,
  \iga derives $u_A\cdot\bar m$ and $u_C\cdot\bar m$ from $\Gamma$
  (\isn{gi.adst\_uA\_inst}, \isn{gi.adst\_uC\_inst}).  The validity test
  accepts non-canonical codes, while \cref{thm:ordinal} is at canonical
  codes, so accessibility is pulled back along the canonicalizer, a
  grounded claim (\isn{gi.adst\_acc\_pull}): if $\Gamma$ is
  well-formed, $\Gamma \ent b\,\mathbf N$ and $\Gamma \ent
  \Acc\cdot\mathrm{CanVT}(b)$, then $\Gamma \ent \Acc\cdot b$, by ground
  induction at $\mathrm{CanVT}(b)$.  From the two universals \iga
  derives accessibility at $\varphi(\omega,0)+1$ (\isn{gi.adst\_acc\_p1}),
  then at every $\varphi_k^j(\varphi(\omega,0)+1)$ (\isn{gi\_L\_fin},
  \isn{adst\_acc\_phi\_step}, \isn{adst\_acc\_it}); the $k$-th term of the
  fundamental sequence of $\varphi(\omega,1)$ is
  $\varphi_k^k(\varphi(\omega,0)+1)$ (\isn{veb\_fs\_pw1}).  So, a
  grounded claim (\isn{gi.adst\_acc\_below\_pw1}): if $\Gamma$ is
  well-formed, $\Gamma \ent \cA\cdot u_A$, $\Gamma \ent \cA\cdot u_C$,
  $x$ is valid and $x < \varphi(\omega,1)$, then $\Gamma \ent
  \Acc\cdot\lceil x\rceil$; \isn{gi.adst\_acc\_below\_pw1\_ctx} is the
  instance at the context $\{\cA\cdot u_A, \cA\cdot u_C\}$.
\item \emph{$\oiga+\mathrm{Refl}$.}  The two adoptions \isn{chol\_uA}
  and \isn{chol\_uC} discharge their side conditions by step 42 and
  \isn{gi\_oiga\_embed}.  If $\oiga+\mathrm{Refl}$ derives $c$ from
  $\emptyset$, then $c$ is true at some valid $s < \varphi(\omega,1)$
  (\isn{chol\_stage\_sound\_closed}); $0 = 1$ is not derivable
  (\isn{chol\_consistent}); for valid $x$, $\Acc\cdot\lceil x\rceil$ is
  derivable from $\emptyset$ if and only if $x < \varphi(\omega,1)$
  (\isn{chol\_acc\_iff\_below\_pw1}); and
  $\Acc\cdot\lceil\varphi(\omega,1)\rceil$ is not
  (\isn{chol\_not\_acc\_pw1}).
\item \emph{The conditional exact ordinal.}  $u_{AC}$ (\isn{adst\_uAC})
  is the abstraction of the conjunction of the two bodies.  Its warrant
  is omega-true (\isn{oiaga\_warrant\_uAC\_true}).  If
  $\oiga+\mathrm{Adopt}$ derives that warrant from $\emptyset$, then for
  valid $x$ it derives $\Acc\cdot\lceil x\rceil$ from $\emptyset$ if and
  only if $x < \varphi(\omega,1)$ (\isn{oiaga\_acc\_iff\_below\_pw1\_if}).
\end{enumerate}

\paragraph{The progression.}
\begin{enumerate}\setcounter{enumi}{44}
\item \emph{The system} (\isn{cauto\_vd}, an inductive definition with
  a flag $f$ for the license and a level $a$).  The empty list is valid
  at level $a$; a valid list extends by an \oiga step; and a valid list
  $P$ extends by $(H, \cA\cdot u)$ when $b$ is valid, $b < a$, the flag
  implies that $P$ contains $\emptyset \ent \Acc\cdot\lceil b\rceil$
  (\isn{cauto\_mk}), $u$ is well-formed and closed, and for every $m$
  some list valid at level $b$ with the same flag contains $\emptyset
  \ent u\cdot\bar m$.  \isn{cauto\_ent}~$f\,a\,\Gamma\,c$ holds when a
  list valid at level $a$ contains $\Gamma \ent c$; \isn{cauto\_aent}
  $\Gamma\,c$ holds when \isn{cauto\_ent}~True~$a\,\Gamma\,c$ for some
  valid $a$; the control is the flag False.  In the in-context form
  (\isn{cauto2\_vd}) the license is that $P$ contains $H' \ent
  \Acc\cdot\lceil b\rceil$ for some sub-context $H'$ of $H$, and the side
  condition is read in the in-context form.  The license at exactly
  $H$ is a derived rule (\isn{cauto2\_vd\_adopt\_ctx}); every list valid
  in the empty-context form is valid in the in-context form
  (\isn{cauto2\_of\_cauto}); dropping the license lands in the control
  (\isn{cauto2\_drop}).
\item \emph{The adoption case with parameters}
  (\isn{adst\_adopt\_case\_gen}).  Step 40 with the bound $\beta$, the
  base's closed stage theorem (every $c$ with $E\,c$ is true at some
  valid stage below $\beta$) and the marker $M$ (true at $\alpha$ implies
  $\beta \leq \alpha$) as parameters, concluding $M \triangleright H
  \ent_{0,0} \cA\cdot u$.  Step 41's stage theorem is re-derived from it
  (\isn{adst\_oiaga\_closed\_regr}).
\item \emph{The levels.}  Lower bound (\isn{cauto\_LB}): if $e$ is valid,
  the flag implies $e < F$, $x$ is valid and $x < R(e)$, then level $e$
  derives $\Acc\cdot\lceil x\rceil$ from $\emptyset$.  Upper bound
  (\isn{cauto\_UB}): if $a$ is valid, the flag implies $a < F$, and level
  $a$ derives $c$ from $\emptyset$, then $c$ is true at some valid $s <
  R(a)$.  Together (\isn{cauto\_acc\_iff\_level}): for valid $a < F$ and
  valid $x$, level $a$ with the license derives $\Acc\cdot\lceil
  x\rceil$ from $\emptyset$ if and only if $x < R(a)$.  The lower bound
  is a well-founded induction on $e$ (\isn{veb\_wf}).  Below
  $\varphi(\omega,0)$, \oiga suffices.  Otherwise
  \isn{cauto\_below\_R\_split} gives a valid $d < e$ with $x <
  \varphi(\omega,d+1)$; level $e$ adopts the bodies \isn{cauto\_uA} and
  \isn{cauto\_uC} at the point $\varphi(\omega,d)$, whose instances are
  level-$d$ theorems (\isn{cauto\_uA\_inst\_lev}, \isn{gi.cauto\_uC\_inst}),
  and climbs (\isn{gi.cauto\_acc\_below\_W}); the license
  $\Acc\cdot\lceil d\rceil$ is a level-$d$ theorem because $d < R(d)$
  for $d < F$ (\isn{cauto\_lt\_R\_below\_F}).  The upper bound is a
  well-founded induction on $a$: a list valid at level $a$ is valid at
  level $0$ or at some $t+1$ with $t < a$ (\isn{cauto\_vd\_top}), and the
  marker $\Acc\cdot\lceil\varphi(\omega,t)\rceil$ is true below
  $\varphi(\omega,t+1)$ (\isn{cauto\_mark\_stage}).
\item \emph{Autonomy.}  If $L$ is valid, $L < F$, $b$ is valid and
  level $L$ with the license derives $\Acc\cdot\lceil b\rceil$ from
  $\emptyset$, then $b < F$ (\isn{cauto\_licence\_below\_F}), by step 47
  at $L$, the rank lemma and $R(L) < F$ (\isn{cauto\_R\_below\_F}).  By
  induction on the list, a list valid at some valid level with the
  license is valid at some level $L < F$ (\isn{cauto\_vd\_aut}).  Hence
  every closed theorem of \isn{cauto\_aent} is true at some valid stage
  below $F$ (\isn{cauto\_stage\_sound\_closed}), $0 = 1$ is not
  derivable (\isn{cauto\_consistent}), for valid $x$
  $\Acc\cdot\lceil x\rceil$ is derivable from $\emptyset$ if and only
  if $x < F$ (\isn{cauto\_acc\_iff\_below\_pw\_omega}), and
  $\Acc\cdot\lceil F\rceil$ is not (\isn{cauto\_not\_acc\_pw\_omega}).
\item \emph{The license in context.}  The set $A_F$ of valid notations
  below $F$ (\isn{cauto\_AF}) is closed under the bounds
  (\isn{stA\_cauto\_AF}), and the stage judgment restricted to
  hypothesis stages in $A_F$ is \isn{sjtA}~\isn{cauto\_AF}.  There an
  adoption at $b \geq F$ is vacuous (\isn{cauto2\_adopt\_high}): its
  license's judgment would make $\Acc\cdot\lceil b\rceil$ true at a
  stage below $F$.  An adoption at $b \leq t < F$ costs a constant under
  the marker $\Acc\cdot\lceil\varphi(\omega,t)\rceil$
  (\isn{cauto2\_adopt\_low}), with the instance bound taken from the
  control.  Every list has a valid $t < F$ that bounds every $b < F$ for
  which the list contains a judgment concluding $\Acc\cdot\lceil
  b\rceil$ (\isn{cauto2\_lic\_bound}), and the invariant over the list
  is \isn{cauto2\_prf\_marked}.  So every closed theorem is true at some
  valid stage below $F$ (\isn{cauto2\_stage\_sound\_closed}), $0 = 1$ is
  not derivable (\isn{cauto2\_consistent}), and for valid $x$,
  $\Acc\cdot\lceil x\rceil$ is derivable from $\emptyset$ if and only
  if $x < F$ (\isn{cauto2\_acc\_iff\_below\_pw\_omega}), the lower bound
  by \isn{cauto2\_of\_cauto}.  At the levels: for valid $a < F$, level
  $a$ is exact at $R(a)$ (\isn{cauto2\_acc\_iff\_level}); for every valid
  $a$, level $a$ derives $\Acc\cdot\lceil x\rceil$ from $\emptyset$ if
  and only if $x < R(a)$ and $x < F$
  (\isn{cauto2\_acc\_iff\_level\_all}).
\item \emph{The control.}  For valid $a$ and $x$, level $a$ of the
  control derives $\Acc\cdot\lceil x\rceil$ from $\emptyset$ if and
  only if $x < R(a)$ (\isn{cctl\_acc\_iff\_level}); level $x+1$ derives
  it (\isn{cctl\_acc\_all}); in particular level $F+1$ derives
  $\Acc\cdot\lceil F\rceil$ (\isn{cctl\_acc\_pw\_omega}).  Every closed
  theorem of level $a$ is true below $R(a)$
  (\isn{cctl\_stage\_sound\_closed}), and every level is consistent
  (\isn{cctl\_consistent}).
\end{enumerate}
Each of the four theories ends with a check that its results use no
oracle and with dependency checks by name.  \isn{oiga\_step\_sjtS} and
\isn{oiga\_stage\_sound\_below\_pw} reach \isn{sja\_allTI},
\isn{sja\_gfix} and the instance witness \isn{stage\_alg\_gfix\_vst},
while the negative case, \iga's \isn{gi\_stage\_sound\_below\_pw},
reaches \isn{sja\_gfix} and not \isn{sja\_allTI}.  The ordinal of
\isn{cauto\_aent} reaches the generic adoption case, the rank lemma,
the license bound and the progression's adoption clause
(\isn{cauto\_vd\_adopt}), and \isn{cctl\_acc\_iff\_level} reaches the
generic adoption case and the progression's adoption clause
(\isn{cauto\_vd\_adopt}), while the closed stage theorem of
$\oiga+\mathrm{Adopt}$ reaches the adoption case of step 40 and not
the progression's adoption clause (\isn{cauto\_vd\_adopt}).
In the adoption theory, the closed stage theorems of
$\oiga+\mathrm{Adopt}$ and $\oiga+\mathrm{Refl}$
(\isn{oiaga\_stage\_sound\_closed}, \isn{chol\_stage\_sound\_closed})
reach \isn{sja\_allTI}, \isn{sja\_gfix}, \isn{adst\_adopt\_case} and
\isn{adst\_gi\_aug}, \isn{chol\_acc\_iff\_below\_pw1} reaches
\isn{adst\_acc\_pull} as well, and the control
\isn{oiga\_stage\_sound\_below\_pw} reaches \isn{sja\_gfix} and not
\isn{adst\_adopt\_case}.  In the in-context theory, the restricted step
lemmas \isn{gi\_step\_sjtA} and \isn{oiga\_step\_sjtA} do not reach
\iga's unrestricted step lemma \isn{gi\_step\_sjtS}.  The
in-context theory's checks match names exactly, up to the locale
qualifier, the others by substring.  As for \iga, no dependency check against omega-soundness
is possible, since the order facts on notations are pulled back
through soundness.

\paragraph{No certificate semantics for \oiga.}
\begin{enumerate}\setcounter{enumi}{50}
\item \emph{The instance} (\isn{oiga\_not\_cert\_sound}, in
  \isn{ATI\_NoCert.thy} of session \textsf{IGA\_PGSem}).  Let
  $\mathrm{Tr}$ and $\mathrm{Fa}$ be predicates on terms such that
  $\mathrm{Tr}(\cA\cdot w)$ implies that \oiga derives the instance of
  $w$ at a fresh variable $x$ from the hypothesis that $x$ is a number;
  $\mathrm{Fa}(\cA\cdot w)$ implies $\mathrm{Fa}(w\cdot\bar n)$ for
  some $n$; $\mathrm{Tr}(a \vee b)$ implies $\mathrm{Tr}(a)$ or
  $\mathrm{Tr}(b)$; $\mathrm{Tr}(\neg a)$ implies $\mathrm{Fa}(a)$; and
  no term satisfies both $\mathrm{Fa}$ and $\mathrm{Tr}$.  Then some $c$
  that \oiga derives from $\emptyset$ does not satisfy $\mathrm{Tr}$.
  The proof instantiates a theorem for every calculus with the
  decidedness rule that is $\Pi_1$-sound and recursively enumerable
  (\isn{gcalc\_ati.ati\_cert\_unsound}): \oiga's derivability is
  recursively enumerable through its proof checker
  (\isn{oiga\_entails\_re}) and $\Pi_1$-sound by omega-soundness.  Some
  universal of the halting family at a non-halting point then gets
  neither value, so the bivalence that the decidedness rule derives for
  it (\isn{ofam\_decided}) does not satisfy $\mathrm{Tr}$.  The control
  (\isn{omega\_reading\_nonre}): the omega model meets every
  hypothesis of the general argument except recursive enumerability.
  The theory ends with a check that these results use no oracle, that
  \isn{oiga\_not\_cert\_sound} reaches the general argument
  (\isn{cert\_sem\_gap}), the non-enumerability of the non-halting set,
  \oiga's omega-soundness and its recursive enumerability, and that the
  general argument reaches no soundness theorem, with a live control.
\end{enumerate}

\paragraph{Proof-groundedness, licensed closing-off and descending sequences.}
\begin{enumerate}\setcounter{enumi}{51}
\item \emph{$\iga+\mathrm{AccBiv}_\varphi$ is proof-grounded}
  (\isn{giv\_gsem\_sound}, \isn{giv\_disj\_prop}, in
  \isn{CO\_PGBiv.thy}).  For a closed order $R$, \iga with the
  bivalence rule for $\Acc_R$ is sound and complete for the
  certificate-grounded semantics with its own proofs as the
  certificates (\isn{givR\_gsem\_sound}) under two hypotheses: the
  pairs at which \iga does not refute $R$ form a well-founded relation
  in HOL, and the extended system derives $\Acc_R\cdot x$ from the
  hypothesis that $x$ is a number.  The proof is an accessibility
  induction for the semantics (\isn{gwosr\_acc\_wf}): the certificate
  of the universal in the unfolding of $\Acc_R$ is the system's own
  derivation at the free variable, and the numeral instances come from
  the HOL induction.  At the Veblen order both hypotheses are
  discharged.  The first comes from the notation's HOL
  well-foundedness (\isn{gamma0\_pedigree}) with \iga's decidedness of
  the order at numerals, \iga's disjunction property and \iga's
  omega-soundness (\isn{gi\_unrefuted\_vOp\_wf}); the second comes
  from the system's proof that the notation is well-founded.
  Conversely (\isn{gs\_acc\_total}), a sound certificate-grounded
  semantics for a calculus that has the rule at a closed order and is
  omega-sound at sentences forces $\Acc_R\cdot\bar n$ to be derivable
  at every $n$.  A build-time check confirms that the generic theorem
  (\isn{givR\_gsem\_sound}) reaches no omega-model soundness theorem,
  and that the Veblen instance reaches none for the bivalence rule;
  \iga's omega-soundness enters only through
  \isn{gi\_unrefuted\_vOp\_wf}, to obtain the HOL well-foundedness of
  the unrefuted pairs from the pedigree.
\item \emph{Licensed closing-off is derived} (\isn{g\_close\_L0},
  \isn{g\_accbiv\_below}, \isn{g\_ati\_L1}, in \isn{CO\_Links.thy}),
  in every calculus with \rga's rules.  A disjunction $a \vee b$ is
  derived whenever the context derives $a$ or derives $b$, so excluded
  middle at $a$ whenever the context derives $a$ or $\neg a$.  The
  bivalence conclusion $\neg(\Acc_R\cdot d) \vee X$ is derived from
  $\Acc_R\cdot d, \Gamma \ent X$ when $\Gamma$ derives that $d$ is a
  number, $\Acc_R\cdot e$ and $R\cdot d\cdot e$.  The decidedness of
  $\cA\cdot u$ is derived when a well-formed $\Gamma$ derives the
  decidedness of every instance and, for a number $b$, that every
  instance at or above $b$ holds.  All three are proved in the plain
  calculus locale, which has none of the decidedness, ground-induction
  and bivalence rules; a check confirms that \isn{g\_accbiv\_below}
  and \isn{g\_ati\_L1} reach neither the decidedness nor the bivalence
  rule, and that \isn{g\_ati\_L1} does not reach ground induction.
\item \emph{Descending sequences} (\isn{g\_desc\_contra},
  \isn{gi\_desc\_pr\_contra}, \isn{gi\_no\_desc\_below\_pw}).  In
  every calculus with ground induction, let $F$ be a closed sequence
  term that is total as a rule (the calculus derives $F\cdot t$ a
  number whenever it derives $t$ one), and let a well-formed context
  derive that $F$ descends along a closed order $R$, as the one
  sentence that $F$ at the successor of every number lies below $F$ at
  that number, and that $F\cdot 0$ is accessible along $R$.  Then the
  context derives $0 = 1$; the proof is accessibility induction at the
  motive ``the sequence never hits $x$''.  At \iga the sequence is
  $n \mapsto f(0, n)$ for a primitive recursive index $f$ of two
  arguments, in the empty context.  On the Veblen notation, \iga
  proves no such sequence descending when \iga proves its starting
  point equal to a notation below $\varphi(\omega,0)$, since \iga
  proves that notation accessible (\cref{thm:ordinal}) and is
  consistent by \cref{thm:pg}.
\end{enumerate}

\paragraph{\iga and the halting family.}
\begin{enumerate}\setcounter{enumi}{54}
\item (\isn{gi\_entails\_re}, \isn{gi\_halting\_undecided},
  \isn{gi\_not\_halting\_biv}, in \isn{RE\_IGA.thy}).  \iga's
  derivations are recursively enumerable: its step is \rga's or the
  ground-induction detector, both primitive recursive.  Some
  non-halting $x$ has $\emptyset \nvdash (\cA\cdot\mathit{ofam}(x))\,\mathbf
  B$, an instance of the general argument of step 51
  (\isn{cert\_sem\_gap}) with truth read as \iga-derivability.  Nor
  does \iga derive the decidedness of the family uniformly over codes,
  stated through the value record of the internal-reasoning
  paper~%
  \cite{ford-conrga}.  The proofs use only \iga's consistency,
  disjunction property and witness property from \cref{thm:pg}; a
  build-time check confirms that they reach no omega-model soundness
  theorem of \iga.  So \iga contains neither \oga nor \oiga
  (\isn{iga\_not\_contains\_oga}, \isn{iga\_not\_contains\_oiga}),
  since \oga derives the decidedness of every universal of the family
  (\isn{ofam\_decided}) and \oiga contains \oga
  (\isn{oga\_oiga\_embed}).  In the same theory, $\oiga+\mathrm{Adopt}$
  has no certificate-grounded semantics (\isn{oiaga\_not\_cert\_sound}),
  the instance of the general theorem of step 51 through its
  omega-soundness and the recursive enumerability of its derivations
  (\isn{oiaga\_entails\_re}).
\end{enumerate}

\paragraph{The Bachmann--Howard rung and the closed-off model.}
\begin{enumerate}\setcounter{enumi}{55}
\item \emph{The generic rung and the Bachmann--Howard order}
  (\isn{givR\_gsem\_sound\_of\_WF}, \isn{givbh\_gsem\_sound}, in
  \isn{CO\_PGBiv\_BH.thy}).  The generic theorem of step 52 has its
  first hypothesis discharged from the classical warrant and \iga's
  decidedness of the order at numerals (\isn{gi\_unrefuted\_wf\_of\_dec}),
  and its second from the extended system's own proof of $\WF$.  The
  Veblen instance (\isn{giv\_gsem\_sound\_of\_WF}) is checked at build
  time to have the statement of \isn{giv\_gsem\_sound}.  At Buchholz's
  countable order (\isn{bOpC}) the hypotheses come from
  \isn{bOpC\_pedigree}, \isn{gi\_bOpC\_num\_dec} and the system's
  derivation of $\WF$; the corollaries are \isn{givbh\_disj\_prop},
  \isn{givbh\_not\_all\_witness}, \isn{givbh\_consistent\_gsem} and
  \isn{givbh\_TI}.  The equivalences of \cref{thm:bh} are proved in
  the locale of calculi with ground induction
  (\isn{g\_WF\_gamma0\_iff\_accbiv}, \isn{g\_WF\_bh\_iff\_accbiv},
  \isn{g\_bh\_ladder\_iff\_WF}, \isn{g\_bh\_ladder\_iff\_accbiv},
  \isn{g\_WF\_bhF\_iff\_accbiv}), with \isn{bh\_square} at \iga, and the
  full order (\isn{bOp}) is derived well-founded in the system
  (\isn{givbh\_WF\_full}) through the transport to canonical forms
  (\isn{bf\_can}).  Bivalence at the one family of accessibility
  atoms that the proof uses is equivalent to the rule
  (\isn{g\_obkm\_iff\_accbiv}, at \iga \isn{ObKm\_iff\_AccBivB}).  A
  build-time check confirms that no Isabelle oracle lies under 35
  headline theorems.  By name, it confirms that the system's
  derivations (\isn{givbh\_cert}, \isn{givbh\_TI},
  \isn{givbh\_WF\_full}) reach no omega-model, pedigree or
  HOL-predicate fact, and that the proof-groundedness theorem reaches
  the pedigree and \iga's omega-soundness only through
  \isn{gi\_unrefuted\_wf\_of\_dec}.  It has four live controls.
\item \emph{$\mu$-bivalence and the closed-off model}
  (\isn{g\_mubiv\_contra}, \isn{mubivD\_consistent}, in
  \isn{CO\_Mu.thy} and \isn{CM\_Class.thy}).  The unrestricted schema
  derives $0=1$ in the plain calculus locale with the rule
  (\isn{gcalc\_mubiv.g\_mubiv\_contra}).  The closed-off oracle
  (\isn{OstarC}, in \isn{CM\_Closed.thy}) takes the least fixed point
  for truth and the greatest for falsity.  The oracle model at it is an omega model by the generic interface
  theorem for coherent oracles (\isn{cohO3.nu\_model\_omega\_model},
  in session \textsf{IGA}), giving \isn{closed\_off\_omega\_model}
  through the closures \isn{OstarC\_le\_Osub} and
  \isn{OstarC\_le\_Ocg}.  The closed-off model (\isn{MC}, in
  \isn{CM\_Model.thy}) reads that model through a translation of
  fixed-point formers (\isn{tauI}, in \isn{CM\_Tau.thy}) and
  inherits the interface by a simulation (\isn{MC\_omega\_model});
  ground induction holds in it separately (\isn{MC\_gfix\_ind\_par},
  \isn{MC\_omega\_model\_gfix}).  \iga with bivalence of the truth atom
  (\isn{itb}, in \isn{CM\_Trbiv.thy}) is sound for the model
  (\isn{itb\_closed\_off\_sound}) and consistent
  (\isn{trbiv\_consistent}); the class rung (\isn{imb}), with
  $\mu$-bivalence at the closed, positive, leaf-decided bodies that use
  their recursive variable, is sound for the model
  (\isn{imb\_closed\_off\_sound}), consistent, conservative
  (\isn{imb\_conservative}, through \isn{omv\_MC}) and admits
  bivalence of accessibility at every closed decided order
  (\isn{accbiv\_of\_class}).  The Liar's point is not in the least
  fixed point (\isn{liar\_pt\_not\_lfp}), which gives
  \isn{itb\_not\_tin} and \isn{gi\_not\_tin}.  Two build-time checks
  confirm that the consistency proofs use no Isabelle oracle and no
  omega-soundness or consistency theorem of \iga, with live controls.
\item \emph{The class rung's strength} (\isn{imb\_WF\_gamma0},
  \isn{imb\_not\_omega\_sound}, \isn{imb\_not\_proof\_grounded}, in
  \isn{CM\_Strength.thy}).  \iga's decidedness of the Veblen order and
  of Buchholz's countable order at numerals, with its omega-soundness,
  puts both orders in the class rule's form
  (\isn{gi\_relDec\_of\_num\_dec}, \isn{gi\_relDec\_vOp},
  \isn{gi\_relDec\_bOpC}), so the class rung derives $\WF$ for both
  (\isn{imb\_WF\_gamma0}, \isn{imb\_WF\_bh}) and is strictly stronger
  than \iga, which does not derive the first
  (\isn{imb\_gamma0\_strict}, from \isn{gi\_not\_WF\_gamma0}).  The
  successor-reversed order \isn{stepD} is closed, decided
  (\isn{relDec\_stepD}) and ill-founded (\isn{stepD\_not\_acc}); the
  class rung derives excluded middle for accessibility at $0$ along
  it (\isn{imb\_stepD\_lem}), a closed sentence with no value in the
  omega model (\isn{stepD\_lem\_not\_true}), so the rung is not
  omega-sound (\isn{imb\_not\_omega\_sound}, the object-level
  negation of the shape of \isn{omega\_sound}).  Certificate
  evaluation refines the omega model (\isn{gwosr\_omv}), so every
  proof-grounded calculus has omega-true theorems
  (\isn{gs\_omega\_true}), and no certificate semantics over a step
  containing the class rung's is sound
  (\isn{imb\_not\_proof\_grounded}); the route through the
  disjunction property does not apply, since it presupposes
  omega-soundness at closed theorems.  A build-time check confirms
  no Isabelle oracle under the headlines, that the lower bound
  reaches \iga's decidedness and omega-soundness and not the rungs'
  own well-foundedness proofs, and that the two negative results
  avoid the closed-off model, with four live controls.
\item \emph{The truth term's self-application clause and the Liar}
  (\isn{tr\_cmp}, \isn{tr\_tin\_pair}, \isn{itb\_liar}, in session
  \textsf{IGA\_Eval}).  In \rga, the truth
  term gives the code of every compiled application, at template
  argument codes, its compiled value (\isn{tr\_cmp}, a HOL induction
  over the primitive-recursive index; the recursor case an internal
  induction), its coded beta substitution commutes with quotation at
  open codes (\isn{bsub\_quote}), and every clause of its own body is
  introduced into truth at quoted codes, generically over a clause
  grammar with any template point code (\isn{sa1\_tmpl\_gen}; the
  quotation of the body at a closed substitution never unfolds it,
  \isn{qC\_asub}).  One ground induction with a conjunctive motive
  (\isn{tr\_tin\_pair}) gives T-introduction for the truth atom in
  rule form (\isn{tr\_tin\_open}) and the self-application clause at
  every closed arithmetic point (\isn{tr\_tin\_closed}), hence the
  premise of the conditional at the Liar's point in every calculus
  with ground induction (\isn{g\_psa}).  With truth-atom bivalence
  the Liar is derived (\isn{itb\_liar}), while its truth is not
  derived, outright or from the Liar, and the Liar is not refuted
  (\isn{itb\_kf\_triple}); \iga alone does not derive the Liar
  (\isn{gi\_not\_liar}).  Build-time checks confirm no Isabelle
  oracle, cones of at least a thousand nodes, no decider constant in
  any statement, and that the \rga theorems use neither ground
  induction, determinacy nor the truth term's HOL adequacy, with live
  controls.
\item \emph{The PA checker's terminal fragment is safe}
  (\isn{pa\_safe}, \isn{gi\_ConPA\_of\_gentzen}, in session
  \textsf{IGA\_ConPA\_Checker}; \cref{thm:pa}).  A decided evaluator of
  closed quantifier-free arithmetic codes, one instance of the
  primitive-recursive recursor with a class tag and a value
  (\isn{peIdx}; its node laws hold at open argument codes,
  \isn{pe\_add}, and the tags are exact, \isn{pe\_decode}), and an
  explicit index-tree checker for Peano arithmetic on the generic
  checker (\isn{ChkPA}; the terminal table is closed propositional
  natural deduction with hypotheses, reflexivity, true closed atomic
  equations and negations of false ones as axioms; the full table adds
  the axioms, substitution, induction and the quantifier rules, with
  the side conditions unchecked, so the checker is liberal and
  contains the rules of Peano arithmetic).  One course-of-values
  induction with a decided motive, for any table whose detectors
  preserve truth (\isn{pe\_fold}), gives that every accepted terminal
  code has a true head, so \rga derives that no terminal-fragment
  code is accepted as a proof of $0=1$ from no hypotheses
  (\isn{pa\_safe}), and \iga derives the checker's consistency
  sentence from an internal reduction with an iteration, an ordinal
  assignment with a decided comparison, and a bound below
  $\varphi(\omega,0)$ (\isn{gi\_ConPA\_of\_gentzen}, the abstract
  conditional of the G\"odel paper at this checker and this safety
  theorem).  The true-equation axioms are theorems of Peano arithmetic
  (\isn{pa\_compAx\_sound}); that every proof of Peano arithmetic is
  accepted is not stated for this checker; step 61 states it for the guarded one.  Build-time checks confirm no Isabelle
  oracle, cones of at least a thousand nodes, no decider constant in
  any statement, and that the \rga theorems use no ground induction,
  determinacy or truth term, with live controls.
\item \emph{The consistency of Peano arithmetic} (\isn{gi\_ConPA}, \isn{pa\_ChkPAr\_complete}, \isn{gi\_ConPA\_relative}, in sessions \textsf{IGA\_CNF\_Kit}, \textsf{IGA\_ConPA\_Checker}, \textsf{IGA\_ConPA\_Buchholz} and \textsf{IGA\_ConPA\_Main}; \cref{thm:pa}).
  The guarded checker (\isn{ChkPAr}) has the 22 rows of \isn{ChkPA}
  with guards added: each row tests the formula or term at the positions
  where Peano arithmetic has a side condition, and at the bodies of the
  quantifier rules and the conclusions of the two closed-equation rows,
  for being arithmetic, well-formed and genuine (in the encoding's
  range); universal introduction also tests that the node's hypotheses
  are well-formed and genuine; every list that the guarded checker
  accepts, the liberal one accepts (\isn{paAr\_valid\_le}).
  If Peano arithmetic derives a sentence from no hypotheses, \rga derives
  that the guarded checker accepts some proof code of it
  (\isn{pa\_ChkPAr\_complete}, a HOL induction over the derivation).
  Buchholz's notation system for derivations with the omega rule~\cite{buchholz91notation,buchholz97explaining} is coded as compiled recursions with its information function, and its checker (\isn{chkN}) accepts a notation that its information
  function marks as correct, of cut degree $0$, deriving a sequent of
  false literals at the code of $0=1$.
  The reduction law is one course-of-values induction in \rga for the local correctness of every notation, with the 22 rows of the embedding among its cases; the ordinal assignment takes Cantor normal forms on the Veblen codes, with a decided comparison inside the Veblen order and a bound below $\varphi(1,0)$; the iteration, the totality, the comparison and the bound are \rga theorems.
  The Gentzen schema at its empty fragment gives a derivation in \iga of the consistency sentence of \isn{chkN}, a grounded claim; an \rga embedding of accepted proof lists as notations, with the cuts removed by iteration, transfers it to the guarded checker (\isn{gi\_ConPA}).
  If Peano arithmetic derives $0 = 1$, then \rga derives $0 = 1$ from the consistency sentence by instantiation and explosion, so \iga derives $0 = 1$ (\isn{gi\_ConPA\_relative}); \iga's soundness is not used, and that Peano arithmetic does not derive $0 = 1$ is not stated.
  The safety theorem of step 60 is not used.
  Build-time checks confirm no Isabelle oracle, cones of at least a thousand nodes, no decider constant in any statement, and that the full dependency cone of \isn{gi\_ConPA}, of 94,010 nodes,
  contains no soundness theorem of \iga (neither the omega-model
  soundness of \cref{thm:sound} nor the stage-indexed soundness of this
  appendix), nor the proof-grounded soundness of \cref{thm:pg}, nor the
  truth term's adequacy, induction or determinacy, nor the safety route
  of step 60, with live controls.
\end{enumerate}

\paragraph{Buchholz's notation with natural-number levels.}
\begin{enumerate}\setcounter{enumi}{61}
\item \emph{The $\omega$-level notation system and the one-claim theorem}
  (\isn{o\_wf\_of\_comega}, \isn{o\_wf\_of\_ck\_K}, \isn{ot\_iff\_claim},
  \isn{otC\_adopt\_consistent}, in sessions \textsf{IGA\_OT\_Kit} and
  \textsf{IGA\_OT\_Claim}; \cref{thm:otnu-claim}).  The HOL notation
  (\isn{OT\_Notation.thy}) is the datatype \isn{ot} of codes $0$ and
  $D_u a + t$ with natural-number levels, the order \isn{ot\_lt},
  a strict linear order on all raw codes (\isn{ot\_lt\_linear}),
  Buchholz's sets $G_u$ and validity (\isn{ot\_valid}), the level bound
  \isn{ot\_lvb} and the countable codes \isn{ot\_cnt}.  The classical
  warrant (\isn{OT\_Pedigree.thy}) is the HOL well-foundedness of the
  order on all valid codes (\isn{ot\_wf\_all}), by the accessible-part
  argument of the Bachmann--Howard development with its two levels
  replaced by a recursion on the level and an outer induction on a base
  level; the countable part follows (\isn{ot\_wf\_cnt}), and each
  $\mathrm{OT}(k)$ and its countable part are well-founded as well
  (\isn{ot\_wf\_lvb}, \isn{ot\_wf\_cnt\_lvb}).  The HOL
  fundamental sequences (\isn{ot\_fs}, in \isn{OT\_FS.thy}) use the
  diagonal chain \isn{ot\_gam}; Buchholz's Lemmas 3.2 to
  3.6~\cite{buchholz86new} are proved, 3.3 and 3.6 for every element of
  the chain, and cofinality (\isn{ot\_diag\_cofinal}, \isn{ot\_fs\_cofinal\_at}),
  the bounds (\isn{ot\_bound\_omega}, \isn{ot\_bound\_k}) and the
  description of $\mathrm{OT}(k)$ as the valid countable codes below
  $c_k$ (\isn{ot\_OTk\_cnt\_iff}) are HOL theorems.  The codes are
  compiled recursions in the layout of the Veblen codes with the level
  as tag, with a decided comparison (\isn{o3}, transitive on raw codes,
  \isn{o3\_trans}), a validity test (\isn{isOT}) and the order terms
  \isn{oOp}, \isn{oOpC}, \isn{oOpK} and \isn{oOpCK} (all valid codes; the
  countable ones; levels at most $k$; both), whose omega-model relations
  are exactly the HOL orders pulled back along decoding
  (\isn{accRelO\_oOp}), so that the warrants transfer
  (\isn{oOp\_pedigree}, \isn{oOpC\_pedigree}, \isn{oOpCK\_pedigree}).
  The internal fundamental-sequence theorems are \rga derivations at
  open codes: validity of the elements (\isn{oFs\_valid}), descent
  (\isn{oFs\_desc}), cofinality (\isn{oFs\_cofinal}) and the bounds
  (\isn{oC\_bound}, \isn{oA\_bound}), by inductions on the code with
  decided motives.  The hardest case, validity of the elements at codes
  of cofinality $\omega$, the diagonal among them, uses the device of the Bachmann--Howard development that
  rebuilds a code along its rightmost path and replaces its tip
  (\isn{oMd}), lifted to natural-number levels.  The map from a code
  $a$ to the code of $x + \omega^a$ in $\mathcal O$ is built once for
  every sequence tree (\isn{prFOC}); the guard is validity,
  countability and a bound (\isn{oGdB}), with no canonical form; the
  claim codes $\hat c$ and $\hat c_k$ are \isn{otC} and \isn{otCk}.
  From $\WF$ for the countable order, the code of $\omega^a$ is in
  $\mathcal O$ for every valid countable code $a$
  (\isn{wf\_ot\_imp\_oc\_valid\_cnt}).  The
  locale \isn{fsClaim} (in \isn{OT\_AccLocale.thy}) fixes a sequence
  tree, a guard, the comparisons below and at most, an order term, a
  cofinality witness, an approximant tree and a bound index, assumes
  the order term decided with its introduction and elimination, the
  congruence that a code below one at most another is below the other,
  the guard's closure under the sequence, cofinality and the bound in
  at-most form, and that the approximants pass the guard, and proves
  that membership of the claim code in $\mathcal O$ gives $\WF$ for the
  order (\isn{wf\_of\_claim}); the limit step reduces a code at most an
  element of a sequence to the codes strictly below it
  (\isn{fc\_den\_limit}).  Its instances are \isn{fsClaim\_otC}, at the
  countable order, and \isn{fsClaim\_otCk}, at the countable order
  below $c_k$ (\isn{o\_wf\_of\_ck}); $\WF$ for the countable part of
  $\mathrm{OT}(k)$ (\isn{o\_wf\_of\_ck\_K}) follows because every countable code of
  level at most $k$ lies below $c_k$ (\isn{oOpCK\_sub\_B}).  The claim
  and the sentence are interderivable at the empty context
  (\isn{ot\_iff\_claim}, \isn{ot\_iff\_claim\_k}) and as extensions
  (\isn{wf\_ext\_otC}, \isn{otC\_ext\_wf}).  The price, through
  \isn{IGA\_KleeneO\_Adopt}, is \isn{otC\_omega\_true},
  \isn{otC\_adopt\_sound}, \isn{otC\_in\_lfp} and
  \isn{otC\_adopt\_consistent}, with their forms at each $k$.  Build-time
  checks over the transitive dependencies confirm no Isabelle oracle
  and cones of at least a thousand nodes; that no internal
  fundamental-sequence theorem routes through the HOL theorems, the
  HOL device or the HOL readings of the trees; that under the one-claim
  theorems there is no soundness theorem, no classical warrant or
  omega-truth fact of the order, no lemma deriving membership from
  $\WF$, and, by statement, no dependency stating HOL well-foundedness
  or accessibility together with the order; that the generic theorem
  uses no fact of the notation; and that the price uses the warrant
  and \iga's omega-soundness; with live controls.
\item \emph{The stratified rung at every finite level} (\isn{givRs\_gsem\_sound},
  \isn{g\_WF\_otk\_of\_accbivk}, \isn{givotk\_gsem\_sound}, in the
  sessions \textsf{IGA\_OT\_Rule} and \textsf{IGA\_OT\_PG};
  \cref{thm:otnu-rule}).  The component tree (\isn{oKm}, in
  \isn{OT\_Km.thy}) returns, for a code and a level $u$, the summand
  whose principal code is the largest of the level-$u$ components, as
  a sub-code, so that inductions on the code number apply; its head is
  $\mathrm{km}_u$ (\isn{oCm}), read exactly in HOL (\isn{pr2fi\_oCm}),
  where it is the largest component (\isn{ot\_kmw\_max}).  The law
  that connects the component to the notation is the G-condition law
  (\isn{oCm\_below}): if the arguments that $G_w$ collects from a code
  $x$ lie below $c$, and $c$ is at most $z$, then the level-$w$
  component of $x$ lies below $D_w z$.
  The two-rule proof at $k = 2$ (\isn{OT\_Biv2.thy}) derives $\WF$ for
  the countable part of $\mathrm{OT}(2)$ and for $R_1$
  (\isn{g\_WF\_ot2\_of\_accbiv2}, \isn{g\_WF\_oR1\_of\_accbiv2}); its
  orders $R_1$ and $R_2$ are restrictions of the order of the codes of
  level at most $2$.  The $k$-rule proof
  (\isn{OT\_Bivk.thy}) defines the family by primitive recursion
  (\isn{mKf}, with \isn{mKf}~$k$~$i$ the formula $M_{i+1}$; \isn{oRk},
  \isn{oWk}), takes the rules for $j < k$ as hypotheses with $k > 0$,
  and derives $\WF$ for the countable part of $\mathrm{OT}(k)$ and for
  every stratified order (\isn{g\_WF\_otk\_of\_accbivk},
  \isn{g\_WF\_oRk\_of\_accbivk}).  From level $1$ on, the accessible
  parts increase with the level, with no rule (\isn{kwK\_mono}); one
  lemma covers every collapse and every top closure at levels $1 \leq p
  \leq u \leq k$ (\isn{tc\_all}, with its inner step \isn{tc\_step},
  where a head of level $u < k$ uses the G-condition law); the closing chain is a downward
  induction on the inner level with an upward one inside it
  (\isn{tc\_omw}); and the rules enter only through the bivalence of the
  formulas $M_{i}$ (\isn{mBiv\_mKf}).  The $k = 2$ theorems are
  instances with statements checked identical
  (\isn{g\_WF\_ot2\_of\_accbiv2\_K}), and $k = 1$ is the one-rule derivation in the
  abstract calculus (\isn{g\_WF\_ot1\_of\_accbiv1}).  The list systems (\isn{givRs}, with
  \isn{accbivL}, in \isn{OT\_GIvL.thy}) add the rule at each order of
  a list to \iga, with \isn{givRs\_single} identifying the list of one
  order with the one-rule system at that order (\isn{givR}, step 52).  The list theorem
  (\isn{givRs\_gsem\_sound}) takes the well-foundedness hypothesis and
  the certificate of each order as hypotheses, and obtains each order's
  numeral accessibility in the semantics from that order's certificate
  (\isn{gwosr\_acc\_wf}).  At the stratified rung the certificate comes
  from the system's proof of $\WF$ (\isn{givRs\_acc\_free\_of\_WF}),
  and the hypothesis, for every $R_{j+1}$, from the envelope, with no
  decidedness of the order itself (\isn{gi\_unrefuted\_wf\_of\_env},
  through \isn{gi\_unrefuted\_wf\_of\_dec} of step 56).  At length one it gives
  the statement of \isn{givR\_gsem\_sound}
  (\isn{givR\_gsem\_sound\_of\_list}), and the Veblen and
  Bachmann--Howard rungs read as lists of one order
  (\isn{giv\_gsem\_sound\_L}, \isn{givbh\_gsem\_sound\_L}).  At level $k$
  the list is $R_0, \ldots, R_{k-1}$ (\isn{otRs}); $R_0$ has
  \isn{oOpCK\_pedigree} and \iga's decidedness at numerals
  (\isn{gi\_oOpCK\_num\_dec}), and every $R_{j+1}$ has the envelope
  \isn{oOp}, with \isn{oOp\_pedigree} and \isn{gi\_oOp\_num\_dec}, since
  a refutation of \isn{oOp} refutes $R_{j+1}$ with no rule
  (\isn{kor\_env}, \isn{gi\_rk\_env}).  The headlines are
  \isn{givotk\_WF0} and \isn{givotk\_WFi}, grounded, and
  \isn{givotk\_gsem\_sound}, \isn{givotk\_consistent} and
  \isn{givotk\_consistent\_gsem}, classical; at $k = 2$ the corollaries
  \isn{givot2\_disj\_prop}, \isn{givot2\_not\_all\_witness},
  \isn{givot2\_gwosr\_open\_iff}, \isn{givot2\_consistent\_gsem} and
  \isn{givot2\_TI}, with \isn{givot2\_WF0\_K} and
  \isn{givot2\_gsem\_sound\_K} checked identical in statement to the
  two-rule theorems.  Build-time checks over the transitive dependencies
  confirm no Isabelle oracle and cones of at least a thousand nodes;
  that the derivations of $\WF$ reach no soundness theorem, no
  classical warrant, no certificate-semantics fact, no fact of the claim
  route of step 62 or of Kleene's $\mathcal O$, and, by statement, no
  HOL well-foundedness of the notation's order, so the two routes are
  independent; which lemmas take which rules, by statement; that the
  proof-groundedness theorem needs the two warrants, \iga's disjunction
  property and omega-soundness, the envelope and the derivations of
  $\WF$, and no fact of the claim route; and the statement identities;
  with live controls.
\end{enumerate}

\paragraph{The stage algebra.}
The rule lemmas of the stage judgment are proved once, in a locale
(\isn{stage\_alg}, in \isn{IGA\_StageStep.thy}) that fixes a stage
function $\mathrm{st}\,n\,k\,\alpha$ and a set $A$ of admissible
hypothesis stages, with twelve assumptions, each named after the
property of $\mathrm{vst}$ it replaces (\isn{st\_valid},
\isn{st\_strict\_k}, \isn{st\_mono}, \isn{st\_gt}, \isn{st\_gt\_w},
\isn{st\_pn\_lt}, \isn{st\_comp\_max}, \isn{st\_max\_w}, \isn{st\_nB},
\isn{omvS\_cong\_st}, \isn{omvS\_cong\_num\_st}, \isn{st\_A\_clo}).  The
judgment there (\isn{sja}) reads hypotheses only at stages in $A$.  A
second locale (\isn{stage\_alg\_gfix}) adds the ground-induction bounds
with five assumptions (\isn{gB\_valid}, \isn{gL\_valid},
\isn{gB\_le\_gL}, \isn{gB\_clo\_h}, \isn{gB\_below}).  The bound family
is an instance: \isn{stage\_alg\_gfix\_vst} interprets the second locale
at $\mathrm{vst}$ and its ground-induction bounds for every set closed
under the bounds (\isn{stA}), with $A$ the set of all notations giving
\isn{sjtS} (\isn{sja\_vst\_UNIV}) and a closed set $A$ the restricted
judgment \isn{sjtA} (\isn{sja\_vst}).  The landed lemmas of this
appendix, the forty-three rule cases (\isn{gi\_step\_sja}) included, are
the instances at $\mathrm{vst}$ with every stage admissible, with their
statements unchanged; a build-time check compares their statements with recorded
fingerprints (\isn{Stage\_Fingerprint}, \isn{IGA\_StageSound.thy}).  The
decidedness rule's case and \oiga's step lemma are proved in the same
locales, as is the marked step of $\oiga+\mathrm{Adopt}$, and the
restricted judgment of the progression's in-context form is the
instance at the set $A_F$.  The algebra adds no theorem about any
system; it is what lets the three extensions reuse the forty-three
cases unchanged.

\clearpage
\section{The Theorems and Their Isabelle Names}
\label{sec:table}

Every statement cited in the text appears below with its name in the
development, its session (\textsf{RGA} and \textsf{RGA\_Kit}, the
base; \textsf{Ordinal} on \rga alone; \textsf{IGA}; \textsf{OIGA}
for the combined systems; \textsf{OIGA\_Stage} for their stage
theorems and the calibration devices of \S\ref{sec:pricing};
\textsf{IGA\_PGSem} for the absence of a certificate semantics for \oiga
and for \iga's certificate-grounded semantics;
\textsf{IGA\_Closing} for the bivalence rule's certificate-grounded
semantics, the licensed rules and descending sequences,
for the generic rung and the
Bachmann--Howard rung, and, in six of its theories, for
$\mu$-bivalence and the closed-off model, whose interface proof is
\textsf{IGA}'s, stated once for every coherent oracle and shared with
a model for greatest fixed points;
\textsf{IGA\_RE} for the enumerability of \iga's derivations, the
halting family and the separations;
\textsf{IGA\_Eval} for the truth term's self-application clause and
the Liar;
\textsf{IGA\_Fold} for the fold of the truth atom through \rga's
checker table and \rga's consistency in \iga with bivalence of the
truth atom; \textsf{IGA\_Fold\_Con} for the same fold at \iga's own
checker and \iga's consistency; \textsf{OIGA\_Fold} for the staged
truth, its bivalence in \oiga, and the fold with the staged truth that
derives \rga's consistency in \oiga; \textsf{IGA\_CNF\_Kit} for the internal Cantor normal forms
of the Peano Arithmetic ordinal assignment;
\textsf{IGA\_ConPA\_Checker} for the Peano Arithmetic checkers, the guarded one's completeness and the
terminal fragment's safety; \textsf{IGA\_ConPA\_Buchholz} for
Buchholz's notations for derivations with the omega rule and their
reduction; \textsf{IGA\_ConPA\_Main} for the consistency of Peano
Arithmetic;
\textsf{IGA\_OT\_Kit} for Buchholz's notation with natural-number
levels, its classical warrant and its internal fundamental sequences;
\textsf{IGA\_OT\_Claim} for the one-claim theorem at that notation;
\textsf{IGA\_OT\_Rule} for the derivations of $\WF$ from the
stratified rules; \textsf{IGA\_OT\_PG} for the stratified rung;
\textsf{IGA\_Univ} for the measurement of \iga with bivalence of the
truth atom (theories \isn{UV\_Obstr} and \isn{UV\_NuBiv} for the
obstruction to a stage bound and the closed-off model's decidedness,
\isn{UV\_Tout} and \isn{UV\_Tin} for T-out and T-in at accessibility,
\isn{UV\_Proxy}, \isn{UV\_ProxyBH} and \isn{UV\_ProxyInst} for the
proxy rung); \textsf{IGA\_Univ\_Lvl} for the tower (\isn{UL\_Term},
\isn{UL\_Model} and \isn{UL\_Level} for the level truth term, the one
model and the Kripke universes, \isn{UL\_DetKit}, \isn{UL\_DetShape} and
\isn{UL\_Det} for the determinacy of the level truth term,
\isn{UL\_Kit}, \isn{UL\_Tout}, \isn{UL\_Tin}, \isn{UL\_Class},
\isn{UL\_Proxy} and \isn{UL\_ProxyWF} for the ladder's kit,
\isn{UL\_Obstr} for the obstruction to the consistency tower); \textsf{IGA\_Univ\_OT} for the ladder at $\mathrm{OT}(k)$
(\isn{UO\_Class}, \isn{UO\_Family}, \isn{UO\_Rung}, \isn{UO\_Inst},
\isn{UO\_Lvlk});
\textsf{IGA\_Univ\_CF} for the constant-free presentation, its
inertness, the absence of variants, the agreement with the level truth
term and the congruence swap at the constant-free level term;
\textsf{IGA\_Univ\_CFCon} for the fold at the constant-free level
term, the level checkers and the consistency tower;
\textsf{IGA\_Fuel} for the adopted claims of \S\ref{sec:claims};
\textsf{ConOGA} and \textsf{OGA} for the two imports; \textsf{IGA\_NoGo}
for the instantiation of the no-go result), and its
viewpoint: grounded, a derivation in \rga, \iga, \oiga, one of the
extensions priced in \S\ref{sec:pricing} such as \iga with
bivalence of accessibility at an order, or every calculus of a locale;
classical, a HOL theorem about the system, the omega model or the
closed-off model; or def., a definition.

{\footnotesize
\setlength{\tabcolsep}{4pt}
\begin{longtable}{@{}p{0.30\textwidth}p{0.42\textwidth}p{0.08\textwidth}p{0.13\textwidth}@{}}
\toprule
Statement & Isabelle name & Session & View \\
\midrule
\endhead
\multicolumn{4}{@{}l}{\emph{The rule} (\S\ref{sec:positive}, \S\ref{sec:induction})}\\
positivity grammar; parametric grammar & \isn{pos}, \isn{neg}; \isn{posP}, \isn{pnexp} & IGA & def. \\
closed body: the grammars agree; stability & \isn{posP\_of\_pos\_closed}, \isn{posP\_subf2} & IGA & classical \\
fixed point of a body; unfolding & \isn{trmGfix}, \isn{gfix\_step1..3} & IGA & def., grounded \\
the Liars are not positive & \isn{liar\_not\_pos}, \isn{liar1\_not\_pos} & IGA & classical \\
the rule; closed form as an instance & \isn{pr\_gfix}, \isn{pr\_gfix\_closed} & IGA & def., grounded \\
induction schema & \isn{IGA\_Proof.gfix\_TI} & IGA & grounded \\
strong form of the rule & \isn{ir\_gfix\_strong\_par}, \isn{ir\_gfix\_strong}, \isn{gi\_gfix\_strong} & OIGA & grounded \\
the number premise is not decorative & \isn{gfix\_nonum\_rule}, \isn{gfix\_nonum\_derives\_bot\_eq}, \isn{gfix\_nonum\_unsound} & IGA & def., classical \\
\rga embeds in \iga & \isn{entails\_embed} & IGA & grounded \\
odd positions winning & \isn{gi\_win\_odd} & IGA & grounded \\
winning positions nonzero & \isn{gi\_win\_nonzero} & IGA & grounded \\
provably winning iff odd & \isn{gi\_win\_iff} & IGA & classical \\
transport; oracle reading; monotone & \isn{gfix\_transport}, \isn{osem},\newline \isn{osem\_mono} & IGA & classical \\
parametric transport and induction & \isn{gfix\_transport\_par}, \isn{gfix\_omv\_ind\_par} & IGA & classical \\
omega-model induction, set and term forms (Lem.~\ref{lem:gfix-ind}) & \isn{gfix\_omv\_ind\_set}, \isn{gfix\_omv\_ind} & IGA & classical \\
soundness; consistency (Thm.~\ref{thm:sound}) & \isn{omega\_sound}, \isn{gi\_consistent} & IGA & classical \\
certificate-parametric evaluation; soundness reduced to the extra steps; completeness & \isn{gsem}, \isn{gsem\_sound}, \isn{prf\_gowsjt},\newline \isn{gs\_complete\_open} & RGA & def., classical \\
ground induction for the certificate semantics; the case for every system containing \iga's step & \isn{gfix\_gwosr\_ind\_par}, \isn{gwgf\_rep},\newline \isn{gsem\_gfix}, \isn{gfix\_step\_gowsjt} & IGA\_\allowbreak{}PGSem & def., classical \\
\iga is proof-grounded: sound, complete, consistent (Thm.~\ref{thm:pg}) & \isn{gi\_prf\_gowsjt}, \isn{gi\_complete\_open},\newline \isn{gi\_gwosr\_open\_iff}, \isn{gi\_consistent\_gsem} & IGA\_\allowbreak{}PGSem & classical \\
disjunction and witness properties, generic and at \iga (Thm.~\ref{thm:pg}) & \isn{gs\_disj\_prop}, \isn{gs\_not\_all\_witness},\newline \isn{gi\_disj\_prop}, \isn{gi\_not\_all\_witness} & IGA\_\allowbreak{}PGSem & classical \\
truth set is the lfp (Thm.~\ref{thm:lfp}) & \isn{gfix\_omv\_eq\_lfp} & IGA & classical \\
accessibility denotes HOL's acc & \isn{acc\_lfp\_eq} & IGA & classical \\
games complete and sound & \isn{win\_lfp\_eq}, \isn{win\_omv\_iff} & IGA & classical \\
the class of truth sets & \isn{gfix\_class\_eq\_lfp\_class},\newline \isn{gfix\_class\_eq\_omv\_class} & IGA & classical \\
closure properties & \isn{fixClass\_Un} and companions & IGA & classical \\
\midrule
\multicolumn{4}{@{}l}{\emph{Truth} (\S\ref{sec:truth})}\\
the body and the term; positive & \isn{trPhi}, \isn{trmTr}, \isn{pos\_trPhi} & IGA & def., classical \\
exactness (Thm.~\ref{thm:tr-exact}) & \isn{tr\_omv\_iff}, \isn{tr\_omv\_iff\_hd} & IGA & classical \\
the lfp; false atom not true & \isn{tr\_eq\_lfp}, \isn{tr\_bot\_false} & IGA & classical \\
\iga-truth sound; $0=1$ not provably true & \isn{gi\_tr\_sound}, \isn{gi\_not\_tr\_bot} & IGA & classical \\
clause laws; induction principle & \isn{tr\_orT1}, \isn{tr\_allT}, \isn{tr\_ind} & IGA & grounded \\
inversions; case split & \isn{trV\_or\_inv} and companions, \isn{tr\_cases} & IGA & grounded \\
refutation at $0=1$ (Thm.~\ref{thm:tr-bot}) & \isn{tr\_bot\_refuted}, \isn{gi\_tr\_bot\_refuted} & IGA & grounded \\
determinacy of values (Thm.~\ref{thm:tr-det}) & \isn{gi\_trDetClosure}, \isn{gi\_tr\_det} & IGA & grounded \\
\midrule
\multicolumn{4}{@{}l}{\emph{Well-foundedness} (\S\ref{sec:wf})}\\
accessibility, WF, Ord as terms & \isn{trmAcc}, \isn{trmWF}, \isn{trmOrd} & Ordinal & def. \\
unfolding rules & \isn{acc\_unfoldI}, \isn{acc\_unfoldE} & Ordinal & grounded \\
denotation; never false (Thm.~\ref{thm:acc-sem}) & \isn{acc\_omv\_iff}, \isn{acc\_omv\_never\_false} & IGA & classical \\
decided case & \isn{acc\_omv\_iff\_dec} & IGA & classical \\
WF never refutable & \isn{WF\_not\_refutable} & Ordinal & classical \\
soundness bridge (Thm.~\ref{thm:wf-sound}) & \isn{gi\_wf\_sound}, \isn{gi\_wf\_sound\_true} & IGA & classical \\
\rga: $\omega$, Ord($\omega$), towers & \isn{rga\_WF\_omega}, \isn{rga\_Ord\_omega},\newline \isn{rga\_WF\_tower} & Ordinal & grounded \\
relative TI schema & \isn{relTI} & Ordinal & classical \\
TI as a theorem of \iga & \isn{gi\_TI}, \isn{gi\_WF\_TI} & IGA & grounded \\
TI at an open order & \isn{parWF}, \isn{gi\_WF\_TI\_par} & IGA & grounded \\
comparability of well-orderings as one sentence & \isn{oiga\_cwo\_sentence} & OIGA & grounded \\
$\omega^\omega$ order; transfer; CNF order (Thm.~\ref{thm:iga-wf}) & \isn{gi\_WF\_omega\_omega}, \isn{gi\_WF\_exp\_hyp},\newline \isn{gi\_WF\_cnf} & IGA & grounded \\
the $\eps$ instance & \isn{gi\_WF\_epsilon0} & IGA & grounded \\
CNF datatype, coding, exactness & \isn{cnf}, \isn{cnf\_code}, \isn{cnf\_denote},\newline \isn{decode\_surj} & IGA & def., classical \\
codes of one form incomparable & \isn{cnf\_same\_decode\_incomparable} & IGA & classical \\
HOL's $\eps$ (Thm.~\ref{thm:eps-wf}) & \isn{epsilon0\_wf} & IGA & classical \\
Veblen order and denotation & \isn{vOp}, \isn{veb\_denote}, \isn{vCn} & IGA & def. \\
$\varphi(n,0)$; certificate (Thm.~\ref{thm:veblen}) & \isn{gi\_acc\_phi\_fin}, \isn{vOrd\_canon\_ordinal} & IGA & grounded \\
segment below $\varphi(n,0)$ in HOL & \isn{veblen\_fin\_wf} & IGA & classical \\
transitivity, irreflexivity; trichotomy fails on codes & \isn{vk\_lt\_trans}, \isn{vk\_lt\_irrefl},\newline \isn{vOrd\_tri\_fails} & IGA & grounded / classical \\
\midrule
\multicolumn{4}{@{}l}{\emph{Goodstein and the hydra} (\S\ref{sec:goodstein})}\\
index trees; run; ordinal map & \isn{prGstep}, \isn{prGrun}, \isn{prGord}, \isn{prCov} & IGA & def. \\
Goodstein's sentence; decrease lemma & \isn{gGood}, \isn{gDec} & IGA & def. \\
termination from decrease & \isn{iga\_goodstein\_of\_decrease} & IGA & grounded \\
decrease lemma derived; Goodstein (Thm.~\ref{thm:goodstein}) & \isn{gDec\_derivable}, \isn{iga\_goodstein} & IGA & grounded \\
the four strong inductions and the bound & \isn{M\_inst}, \isn{BC\_inst}, \isn{MONO\_inst},\newline \isn{VALID\_inst}; \isn{psi\_UB\_full} & IGA & grounded \\
HOL's Goodstein through soundness & \isn{goodstein\_hol\_unconditional} & IGA & classical \\
hydra coding; chop; ordinal map & \isn{hychop}, \isn{hyord}, \isn{prHyStep} & IGA & def. \\
every chop decreases (HOL) & \isn{hyord\_chop\_lt} & IGA & classical \\
generic engine & \isn{iga\_termg\_of\_decrease} & IGA & grounded \\
hydra from decrease; the decrease lemma derived; the hydra (Thm.~\ref{thm:hydra}) & \isn{iga\_hydra\_of\_decrease}, \isn{hyd\_hDec\_derivable}, \isn{iga\_hydra} & IGA & grounded \\
the internal decrease and the order theory & \isn{hyd\_Chop\_lt}, \isn{hyd\_Lt\_cnfL}, \isn{hyd\_Ins\_mono2} & IGA & grounded \\
decrease true in the omega model & \isn{omv\_hDec} & IGA & classical \\
three HOL readings & \isn{hydra\_hol}, \isn{hydra\_hol\_direct},\newline \isn{hydra\_hol\_classical} & IGA & classical \\
\midrule
\multicolumn{4}{@{}l}{\emph{The consistency of Peano Arithmetic} (\S\ref{sec:conpa})}\\
the PA checker's terminal fragment safe in \rga; \iga derives the checker's consistency from a reduction and a bound (step 60) & \isn{peIdx}, \isn{ChkPA}, \isn{pe\_fold}, \isn{pa\_safe},\newline \isn{gi\_ConPA\_of\_gentzen}, \isn{pa\_compAx\_sound} & IGA\_\allowbreak{}ConPA\_\allowbreak{}Checker & classical, grounded \\
the guarded PA checker; it accepts a proof of every PA theorem; \iga derives its consistency; an inconsistency in PA gives one in \iga (Thm.~\ref{thm:pa}, step 61) & \isn{ChkPAr}, \isn{chkN}, \isn{pa\_ChkPAr\_complete},\newline \isn{gi\_ConPA}, \isn{gi\_ConPA\_relative} & IGA\_\allowbreak{}ConPA\_\allowbreak{}Checker,\newline IGA\_\allowbreak{}ConPA\_\allowbreak{}Buchholz,\newline IGA\_\allowbreak{}ConPA\_\allowbreak{}Main & def., classical, grounded \\
\midrule
\multicolumn{4}{@{}l}{\emph{The price of the classical world} (\S\ref{sec:gamma0})}\\
bivalence of accessibility (admissibility) & \isn{AccBivR}, \isn{AccBiv} & IGA & def. \\
WF($\Gamma_0$) iff bivalence (Thm.~\ref{thm:gamma0}) & \isn{gi\_WF\_gamma0\_iff\_bivalence}, \isn{gi\_WF\_gamma0\_of\_bivalence} & IGA & classical \\
bivalence sound iff WF (Thm.~\ref{thm:biv-wf}) & \isn{accbiv\_omega\_sound\_iff\_wf} & IGA & classical \\
the ladder from the closure-shaped instances; those iff bivalence (Thm.~\ref{thm:ladder}) & \isn{ladder\_gen}, \isn{LADDER\_of\_ObRC}, \isn{ObRC\_iff\_AccBiv} & IGA & classical \\
the stage model; transport and rank at stages & \isn{omvS}, \isn{gfix\_transport\_stage}, \isn{gfix\_stage\_ind\_set},\newline \isn{acc\_stage\_rank} & IGA & def., classical \\
the bounds; congruence at stages; the step lemma & \isn{vst}, \isn{vst\_gB}, \isn{omh\_vcrS}, \isn{sjtS}, \isn{gi\_step\_sjtS} & IGA & def., classical \\
stage soundness; the ordinal (Thm.~\ref{thm:ordinal}) & \isn{gi\_stage\_sound\_below\_pw}, \isn{acc\_stage\_ge},\newline \isn{gi\_acc\_iff\_below\_pw} & IGA & classical \\
not derivable: WF, bivalence, the ladder (Thm.~\ref{thm:ordinal}) & \isn{gi\_not\_WF\_gamma0}, \isn{gi\_not\_WF\_above\_pw},\newline \isn{gi\_not\_AccBiv}, \isn{gi\_not\_ObRC}, \isn{gi\_not\_LADDER} & IGA & classical \\
Veblen notation well-founded in HOL; pedigree & \isn{veb\_wf}, \isn{gamma0\_pedigree},\newline \isn{gamma0\_accbiv\_sound} & IGA & classical \\
one rule: the system and its rule (Thm.~\ref{thm:price}) & \isn{giv\_calc}, \isn{accbivS}, \isn{giv\_WF\_gamma0\_price} & OIGA & def., classical \\
one claim (Thm.~\ref{thm:triangle}) & \isn{o\_wf\_of\_gamma0}, \isn{o\_wf\_of\_gamma0\_ctx} & IGA & grounded \\
the claim from WF & \isn{o\_gamma0} & IGA & grounded \\
the triangle (Thm.~\ref{thm:triangle}) & \isn{gamma0\_triangle}, \isn{gamma0\_iff\_bivalence} & IGA & classical \\
adoption: interderivable; sound; consistent & \isn{wf\_ext\_gamma0}, \isn{gamma0\_adopt\_wf},\newline \isn{gamma0\_adopt\_sound}, \isn{gamma0\_adopt\_consistent} & IGA & classical \\
Bachmann--Howard: triangle; pedigree & \isn{bh\_triangle}, \isn{bh\_pedigree} & IGA & classical \\
Buchholz's notation with natural-number levels; the warrant (\S\ref{sec:one-rule}, step 62) & \isn{ot}, \isn{ot\_valid}, \isn{ot\_wf\_all},\newline \isn{oOpC\_pedigree} & IGA\_\allowbreak{}OT\_\allowbreak{}Kit & def., classical \\
the internal fundamental sequences: validity, cofinality, bound (step 62) & \isn{oFs\_valid}, \isn{oFs\_cofinal}, \isn{oC\_bound} & IGA\_\allowbreak{}OT\_\allowbreak{}Kit & grounded \\
one claim at $\omega$ and at each $k$ (Thm.~\ref{thm:otnu-claim}) & \isn{o\_wf\_of\_comega}, \isn{o\_wf\_of\_ck},\newline \isn{o\_wf\_of\_ck\_K} & IGA\_\allowbreak{}OT\_\allowbreak{}Claim & grounded \\
the claim from WF; the two interderivable (Thm.~\ref{thm:otnu-claim}) & \isn{o\_otC\_oOpC}, \isn{ot\_iff\_claim} & IGA\_\allowbreak{}OT\_\allowbreak{}Claim & grounded, classical \\
$\omega^a$ in $\mathcal O$ for every valid countable $a$, from WF (step 62) & \isn{wf\_ot\_imp\_oc\_valid\_cnt} & IGA\_\allowbreak{}OT\_\allowbreak{}Claim & grounded \\
adoption of the claim: sound; consistent (Thm.~\ref{thm:otnu-claim}) & \isn{otC\_adopt\_sound}, \isn{otC\_adopt\_consistent} & IGA\_\allowbreak{}OT\_\allowbreak{}Claim & classical \\
one claim buys WF, for every sequence kit (step 62) & \isn{fsClaim}, \isn{wf\_of\_claim} & IGA\_\allowbreak{}OT\_\allowbreak{}Claim & def., grounded \\
the adoption of a claim; the two presentations; sound iff in $\mathcal O$; consistent (\S\ref{sec:claims}) & \isn{adoptDer}, \isn{adopt\_ext\_iff},\newline \isn{adopt\_sound\_iff\_lfp}, \isn{adopt\_consistent} & IGA & def., classical \\
autonomy over bare claims: admitted iff derivable; the autonomous union is \iga; bivalence from a derived WF (Thm.~\ref{thm:claim-autonomy}) & \isn{fuAdm}, \isn{fuAdm\_iff}, \isn{fuAut\_eq\_iga},\newline \isn{wf\_gives\_accbiv} & IGA\_\allowbreak{}Fuel & def., classical \\
the finite union of sound adoptions: sound; consistent; cut; complete for test universals (Thm.~\ref{thm:claim-autonomy}) & \isn{fuU}, \isn{fuU\_omega\_sound}, \isn{fuU\_consistent},\newline \isn{fuU\_cut}, \isn{fuU\_pi01\_iff} & IGA\_\allowbreak{}Fuel & def., classical \\
the code of a test: the claim and the test universal interderivable; in $\mathcal O$ iff the test never fires (Thm.~\ref{thm:claim-autonomy}) & \isn{pwcT}, \isn{pwcT\_rule\_I}, \isn{pwcT\_rule\_E},\newline \isn{fu\_pi01\_iff} & IGA\_\allowbreak{}Fuel & def., grounded, classical \\
no WF sentence refuted; WF of the successor-reversed order explodes; the direct rule inconsistent (Thm.~\ref{thm:claim-tests}) & \isn{gi\_never\_refutes\_WF}, \isn{wf\_stepD\_explodes},\newline \isn{dirAC\_inconsistent}, \isn{dirAC\_rule\_inconsistent} & IGA\_\allowbreak{}Fuel & classical, grounded \\
the consistency-filtered rule fails at the Liar order; at \iga; at every sound adoption (Thm.~\ref{thm:claim-tests}) & \isn{fu\_AC\_inconsistent}, \isn{fu\_ACdec\_inconsistent},\newline \isn{gi\_AC\_inconsistent}, \isn{fu\_AC\_inconsistent\_fuel},\newline \isn{fuRL\_WF\_explodes\_seq} & IGA\_\allowbreak{}Fuel & classical, grounded \\
the Rosser form: no negation introduction; every consistent system and adoption (Thm.~\ref{thm:claim-tests}) & \isn{fu\_WF\_no\_neg\_intro}, \isn{fu\_ACR\_inconsistent},\newline \isn{gi\_ACR\_inconsistent}, \isn{fu\_ACR\_inconsistent\_fuel} & IGA\_\allowbreak{}Fuel & classical \\
falsification; the explosion at a refuted consequence (Thm.~\ref{thm:claim-tests}) & \isn{fu\_falsified}, \isn{fu\_falsified\_by\_test},\newline \isn{fu\_falsified\_by\_refuted}, \isn{fu\_falsified\_by\_instance},\newline \isn{fu\_refuted\_explodes} & IGA\_\allowbreak{}Fuel & classical, grounded \\
the quine code: unrefuted, valueless, not in $\mathcal O$, explodes; the code-form rule (Thm.~\ref{thm:claim-tests}) & \isn{fuQ\_falsified\_by\_consequence}, \isn{fuQ\_omv\_gap},\newline \isn{adopt\_fuQ\_explodes}, \isn{dirAC\_code\_inconsistent} & IGA\_\allowbreak{}Fuel & classical, grounded \\
Turing's code: the claim and $\mathit{Con}_{\mathrm{IGA}}$ interderivable; in $\mathcal O$; strictly beyond \iga (Thm.~\ref{thm:claim-g2}) & \isn{tc}, \isn{tc\_ConIGA\_iff}, \isn{tc\_in\_lfp},\newline \isn{fu\_tc\_strict} & IGA\_\allowbreak{}Fuel & def., grounded, classical \\
the checker of an adoption: the axiom cell, faithfulness, the consistency sentence; G\"odel's second theorem, internal and at a sound adoption (Thm.~\ref{thm:claim-g2}) & \isn{chkF}, \isn{drvChkF\_fires\_iff}, \isn{ConF},\newline \isn{fu\_goedel2\_internal}, \isn{fu\_goedel2} & IGA\_\allowbreak{}Fuel & def., classical, grounded \\
the larger claim from the adoption's checker (Thm.~\ref{thm:claim-g2}) & \isn{tcF}, \isn{tcF\_ConF\_iff}, \isn{fu\_needs\_larger\_fuel},\newline \isn{fu\_not\_tcF} & IGA\_\allowbreak{}Fuel & def., grounded, classical \\
Curry (Thm.~\ref{thm:curry}) & \isn{curry\_explosion}, \isn{curry\_inconsistent}, \isn{ext\_curry\_explosion} & IGA & classical \\
Curry through the Liar & \isn{liar\_explosion\_cp} & IGA & classical \\
bivalence shape not admitted at the Liar (Thm.~\ref{thm:biv-liar}) & \isn{accbiv\_shape\_explodes\_on\_liar}, \isn{accbiv\_shape\_fails\_on\_liar} & IGA & classical \\
exact deduction theorem & \isn{deduction\_theorem\_exact} & RGA & classical \\
no-go: motive, residuals & \isn{xder\_not\_motive}, \isn{xder\_not\_fold\_residual} & ConOGA & classical \\
the no-go instantiated by \iga & \isn{iga\_not\_motive}, \isn{iga\_not\_fold\_residual}, \isn{iga\_not\_iResGuard} & IGA\_NoGo & classical \\
\oga contains PA & \isn{oga\_contains\_pa} & OGA & grounded \\
\midrule
\multicolumn{4}{@{}l}{\emph{Pricing the other rules} (\S\ref{sec:pricing})}\\
\oiga: the union; embeddings & \isn{oiga\_derives}, \isn{gi\_oiga\_embed}, \isn{oga\_oiga\_embed} & OIGA & def., classical \\
\oiga has no certificate semantics; the control (step 51) & \isn{oiga\_not\_cert\_sound}, \isn{omega\_reading\_nonre} & IGA\_\allowbreak{}PGSem & classical \\
the general argument (step 51) & \isn{cert\_sem\_gap}, \isn{gcalc\_ati.ati\_cert\_unsound} & IGA\_\allowbreak{}PGSem & classical \\
$\iga+\mathrm{AccBiv}_\varphi$ is proof-grounded; generic form; accessibility induction (step 52) & \isn{giv\_gsem\_sound}, \isn{giv\_disj\_prop},\newline \isn{givR\_gsem\_sound}, \isn{gwosr\_acc\_wf} & IGA\_\allowbreak{}Closing & classical \\
the unrefuted pairs well-founded at the Veblen order (step 52) & \isn{gi\_unrefuted\_vOp\_wf} & IGA\_\allowbreak{}Closing & classical \\
a certificate semantics with the bivalence rule forces accessibility (step 52) & \isn{gs\_acc\_total}, \isn{g\_accbiv\_dp\_iff\_total} & IGA\_\allowbreak{}Closing & classical \\
licensed closing-off is derived (step 53) & \isn{g\_close\_L0}, \isn{g\_accbiv\_below}, \isn{g\_ati\_L1} & IGA\_\allowbreak{}Closing & grounded \\
descending sequences (step 54) & \isn{g\_desc\_contra}, \isn{gi\_desc\_pr\_contra},\newline \isn{gi\_no\_desc\_below\_pw} & IGA\_\allowbreak{}Closing & grounded, classical \\
the generic rung; its Veblen instance (Thm.~\ref{thm:rung}, step 56) & \isn{givR\_gsem\_sound\_of\_WF}, \isn{gi\_unrefuted\_wf\_of\_dec},\newline \isn{giv\_gsem\_sound\_of\_WF} & IGA\_\allowbreak{}Closing & classical \\
$\iga+\mathrm{AccBiv}_{\mathrm{BH}}$ proof-grounded; disjunction, witness, consistency; TI (Thm.~\ref{thm:rung}) & \isn{givbh\_gsem\_sound}, \isn{givbh\_disj\_prop},\newline \isn{givbh\_not\_all\_witness}, \isn{givbh\_consistent\_gsem}, \isn{givbh\_TI} & IGA\_\allowbreak{}Closing & classical, grounded \\
$\iga+\mathrm{AccBiv}_{\mathrm{BH}}$ omega-sound, consistent; the countable order & \isn{givbh\_WF\_bh\_price}, \isn{bOpC}, \isn{bOpC\_pedigree} & OIGA, IGA & def., classical \\
WF iff the rule at both orders; ladder; full order; $\mathcal O$ (Thm.~\ref{thm:bh}) & \isn{g\_WF\_gamma0\_iff\_accbiv}, \isn{g\_WF\_bh\_iff\_accbiv},\newline \isn{g\_bh\_ladder\_iff\_accbiv}, \isn{g\_WF\_bhF\_iff\_accbiv}, \isn{bh\_square} & OIGA, IGA\_\allowbreak{}Closing & classical \\
bivalence at one family of atoms equivalent to the rule & \isn{g\_obkm\_iff\_accbiv}, \isn{ObKm\_iff\_AccBivB} & IGA\_\allowbreak{}Closing & classical \\
the full order derived well-founded; transport to canonical forms (Thm.~\ref{thm:bh}) & \isn{givbh\_WF\_full}, \isn{bf\_can} & IGA\_\allowbreak{}Closing & grounded \\
WF from the stratified rules, at $k = 2$ and at every $k$ (Thm.~\ref{thm:otnu-rule}, step 63) & \isn{g\_WF\_ot2\_of\_accbiv2}, \isn{g\_WF\_otk\_of\_accbivk},\newline \isn{g\_WF\_oRk\_of\_accbivk} & IGA\_\allowbreak{}OT\_\allowbreak{}Rule & grounded \\
the stratified rung theorem for a list of orders; length one; the envelope (Thm.~\ref{thm:otnu-rule}, step 63) & \isn{givRs\_gsem\_sound}, \isn{givR\_gsem\_sound\_of\_list},\newline \isn{gi\_unrefuted\_wf\_of\_env} & IGA\_\allowbreak{}OT\_\allowbreak{}PG & classical \\
the stratified rung derives WF; proof-grounded; does not derive $0=1$ (Thm.~\ref{thm:otnu-rule}, step 63) & \isn{givotk\_WF0}, \isn{givotk\_gsem\_sound},\newline \isn{givotk\_consistent\_gsem} & IGA\_\allowbreak{}OT\_\allowbreak{}PG & grounded, classical \\
the unrestricted schema derives $0=1$ (Thm.~\ref{thm:mubiv}, step 57) & \isn{gcalc\_mubiv.g\_mubiv\_contra} & IGA\_\allowbreak{}Closing & classical, grounded \\
the closed-off oracle and model; the interface & \isn{OstarC}, \isn{MC}, \isn{MC\_omega\_model},\newline \isn{cohO3.nu\_model\_omega\_model} & IGA\_\allowbreak{}Closing, IGA & def., classical \\
truth-atom bivalence and the class rung sound for the closed-off model, consistent; conservative, all decided orders (Thm.~\ref{thm:mubiv}) & \isn{itb\_closed\_off\_sound}, \isn{imb\_closed\_off\_sound},\newline \isn{trbiv\_consistent}, \isn{mubivD\_consistent},\newline \isn{imb\_conservative}, \isn{accbiv\_of\_class} & IGA\_\allowbreak{}Closing & classical \\
the Liar's step to its truth not derivable (Thm.~\ref{thm:mubiv}); the Liar derived from self-application (step 59) & \isn{liar\_pt\_not\_lfp}, \isn{itb\_not\_tin}, \isn{gi\_not\_tin},\newline \isn{g\_trbiv\_liar\_untrue\_psa}, \isn{itb\_liar}, \isn{itb\_kf\_triple}, \isn{gi\_not\_liar} & IGA\_\allowbreak{}Closing, IGA\_\allowbreak{}Eval & classical, grounded \\
\iga with bivalence of the truth atom derives \rga's consistency sentence; the fold through \rga's checker table is \rga reasoning under four closed formulas, the rule entering only there (\S\ref{sec:gamma0}) & \isn{itb\_ConRGA}, \isn{fki\_ConRGA\_Gk},\newline \isn{nd\_fold\_RGA}, \isn{fki\_g\_kit\_biv} & IGA\_\allowbreak{}Fold & grounded \\
\iga with bivalence of the truth atom derives \iga's own consistency sentence, which \iga alone does not (G\"odel's second theorem); the arm for the ground-induction rule is one ground induction in the plain fixed-point calculus (\S\ref{sec:gamma0}) & \isn{itb\_ConIGA}, \isn{fc\_separation},\newline \isn{fc\_gfix\_arm}, \isn{fc\_T\_gen} & IGA\_\allowbreak{}Fold\_\allowbreak{}Con & grounded, classical \\
\oiga derives \rga's consistency sentence, which \rga does not (G\"odel's second theorem); the fold with a staged truth, its bivalence derived by transfinite induction below $\eps$ with the decidedness rule (\S\ref{sec:gamma0}) & \isn{oiga\_ConRGA}, \isn{oiga\_separation},\newline \isn{ts\_biv}, \isn{ndT\_fold\_RGA} & OIGA\_\allowbreak{}Fold & grounded, classical \\
the truth term evaluates compiled codes and introduces its own atom (step 59) & \isn{tr\_cmp}, \isn{sa1\_tmpl\_gen},\newline \isn{tr\_tin\_open}, \isn{tr\_tin\_closed}, \isn{g\_psa} & IGA\_\allowbreak{}Eval & classical \\
the class rung derives WF at the Veblen and Buchholz orders, strictly above \iga (step 58) & \isn{gi\_relDec\_vOp}, \isn{gi\_relDec\_bOpC},\newline \isn{imb\_WF\_gamma0}, \isn{imb\_WF\_bh}, \isn{imb\_gamma0\_strict} & IGA\_\allowbreak{}Closing & grounded, classical \\
the class rung not omega-sound and not proof-grounded; proof-grounded calculi are omega-true (step 58) & \isn{imb\_not\_omega\_sound}, \isn{imb\_stepD\_lem},\newline \isn{gs\_omega\_true}, \isn{imb\_not\_proof\_grounded} & IGA\_\allowbreak{}Closing & classical \\
\iga's derivations r.e.; the halting family undecided (step 55) & \isn{gi\_entails\_re}, \isn{gi\_halting\_undecided},\newline \isn{gi\_not\_halting\_biv} & IGA\_\allowbreak{}RE & classical \\
\iga contains neither \oga nor \oiga (step 55) & \isn{iga\_not\_contains\_oga},\newline \isn{iga\_not\_contains\_oiga} & IGA\_\allowbreak{}RE & classical \\
$\oiga+\mathrm{Adopt}$ has no certificate semantics; its derivations r.e. (step 55) & \isn{oiaga\_not\_cert\_sound},\newline \isn{oiaga\_entails\_re} & IGA\_\allowbreak{}RE,\newline OIGA & classical \\
\oga decides the halting family & \isn{ofam\_decided} & OGA & grounded \\
the decidedness rule at stages; \oiga's step lemma; list induction & \isn{omvS\_ati\_inst}, \isn{sjtS\_allTI}, \isn{oiga\_step\_sjtS},\newline \isn{step\_prf\_sjtS} & OIGA\_\allowbreak{}Stage & classical \\
\oiga's stage soundness & \isn{oiga\_stage\_sound}, \isn{oiga\_stage\_sound\_below\_pw} & OIGA\_\allowbreak{}Stage & classical \\
\oiga's ordinal (Thm.~\ref{thm:oiga}) & \isn{oiga\_acc\_iff\_below\_pw} & OIGA\_\allowbreak{}Stage & classical \\
not derivable in \oiga: WF, bivalence, the ladder; accessibility not decided (Thm.~\ref{thm:oiga}) & \isn{oiga\_not\_WF\_gamma0}, \isn{oiga\_not\_WF\_above\_pw},\newline \isn{oiga\_not\_accbiv}, \isn{oiga\_not\_LADDER}, \isn{oiga\_not\_dec\_acc},\newline \isn{oiga\_not\_dec\_acc\_sentence} & OIGA\_\allowbreak{}Stage & classical \\
transfinite induction as a sentence at decided motives (Thm.~\ref{thm:oiga-ti}) & \isn{oiga\_TI\_decided\_below\_pw} & OIGA\_\allowbreak{}Stage & grounded \\
$\oiga+\mathrm{Adopt}$: the system, warrant and rule; omega reading; sound; consistent & \isn{oiaga}, \isn{warrant\_oiga}, \isn{ir\_adopt\_oiga},\newline \isn{warrant\_oiga\_omv\_iff}, \isn{oiaga\_omega\_sound}, \isn{oiaga\_consistent} & OIGA & def., classical \\
adoption at a constant cost; the marker & \isn{adst\_adopt\_case}, \isn{adst\_mark}, \isn{adst\_mark\_ge},\newline \isn{adst\_mark\_stage} & OIGA\_\allowbreak{}Stage & def., classical \\
the bound of $\oiga+\mathrm{Adopt}$ (Thm.~\ref{thm:adopt}) & \isn{oiaga\_stage\_sound\_closed}, \isn{oiaga\_acc\_below\_pw1},\newline \isn{oiaga\_not\_acc\_pw1}, \isn{oiaga\_not\_WF\_gamma0}, \isn{oiaga\_not\_accbiv} & OIGA\_\allowbreak{}Stage & classical \\
accessibility pulls back along the canonicalizer & \isn{adst\_acc\_pull} & OIGA\_\allowbreak{}Stage & grounded \\
the two bodies; their instances; the lower bound (Thm.~\ref{thm:adopt-lb}) & \isn{adst\_uA}, \isn{adst\_uC}, \isn{adst\_uA\_inst}, \isn{adst\_uC\_inst},\newline \isn{adst\_acc\_below\_pw1}, \isn{adst\_acc\_below\_pw1\_ctx} & OIGA\_\allowbreak{}Stage & def., grounded \\
$\oiga+\mathrm{Refl}$: the system; consistent; ordinal (Thm.~\ref{thm:adopt}) & \isn{chol\_derives}, \isn{chol\_stage\_sound\_closed},\newline \isn{chol\_consistent}, \isn{chol\_acc\_iff\_below\_pw1} & OIGA\_\allowbreak{}Stage & def., classical \\
the conditional exact ordinal (Thm.~\ref{thm:adopt-if}) & \isn{oiaga\_warrant\_uAC\_true}, \isn{oiaga\_acc\_iff\_below\_pw1\_if} & OIGA\_\allowbreak{}Stage & classical \\
the adoption case with bound, base and marker as parameters & \isn{adst\_adopt\_case\_gen} & OIGA\_\allowbreak{}Stage & classical \\
the progression; levels exact (Thm.~\ref{thm:levels}) & \isn{cauto\_vd}, \isn{cauto\_LB}, \isn{cauto\_UB}, \isn{cauto\_acc\_iff\_level} & OIGA\_\allowbreak{}Stage & def., classical \\
autonomy: the licenses stay below $\varphi(\omega+1,0)$ & \isn{cauto\_licence\_below\_F}, \isn{cauto\_vd\_aut} & OIGA\_\allowbreak{}Stage & classical \\
the empty-context form: stage soundness, consistency, ordinal (Thm.~\ref{thm:auto}) & \isn{cauto\_aent}, \isn{cauto\_stage\_sound\_closed},\newline \isn{cauto\_consistent}, \isn{cauto\_acc\_iff\_below\_pw\_omega} & OIGA\_\allowbreak{}Stage & def., classical \\
the in-context form: the rule, its two adoption cases, levels, ordinal (Thms.~\ref{thm:levels}, \ref{thm:auto}) & \isn{cauto2\_vd}, \isn{cauto2\_vd\_adopt\_ctx}, \isn{cauto2\_of\_cauto},\newline \isn{cauto2\_adopt\_high}, \isn{cauto2\_adopt\_low}, \isn{cauto2\_acc\_iff\_level},\newline \isn{cauto2\_stage\_sound\_closed}, \isn{cauto2\_consistent},\newline \isn{cauto2\_acc\_iff\_below\_pw\_omega}, \isn{cauto2\_acc\_iff\_level\_all} & OIGA\_\allowbreak{}Stage & def., classical \\
the control: every valid notation (Thm.~\ref{thm:levels}) & \isn{cctl\_acc\_iff\_level}, \isn{cctl\_acc\_all}, \isn{cctl\_acc\_pw\_omega} & OIGA\_\allowbreak{}Stage & classical \\
the stage algebra; its instance at the bound family & \isn{stage\_alg}, \isn{stage\_alg\_gfix}, \isn{stage\_alg\_gfix\_vst}, \isn{sjtA} & IGA & def., classical \\
\midrule
\multicolumn{4}{@{}l}{\emph{Measuring one closing-off; the tower; the consistency tower} (\S\ref{sec:measure}, \S\ref{sec:tower}, \S\ref{sec:contower})}\\
no stage bound and no omega-soundness for \iga with bivalence of the truth atom (Thm.~\ref{thm:itb-nostage}) & \isn{up\_liar\_not\_true}, \isn{up\_itb\_no\_stage\_sound},\newline \isn{up\_itb\_not\_omega\_sound} & IGA\_\allowbreak{}Univ & classical \\
the closed-off model decides positive fixed points with decided leaves (Thm.~\ref{thm:one-model}) & \isn{MC\_nu\_biv} & IGA\_\allowbreak{}Univ & classical \\
the decided compiled orders; T-out and T-in at accessibility (Thm.~\ref{thm:acc-tinout}) & \isn{ut\_DC}, \isn{ug\_acc\_tout\_gen}, \isn{ut\_acc\_tin\_gen},\newline \isn{ut\_acc\_tout\_dc}, \isn{ut\_acc\_tin\_dc} & IGA\_\allowbreak{}Univ & def., grounded \\
the proxy; the Bachmann--Howard rung over a proxy; the landed rung as an instance; a proxy at every decided compiled order & \isn{uvP}, \isn{gcalc\_proxy}, \isn{uq\_WF\_bh},\newline \isn{uq\_proxy\_of\_accbiv}, \isn{uq\_proxy\_dc} & IGA\_\allowbreak{}Univ & def., grounded \\
\iga with bivalence of the truth atom derives $\WF$ for $\prec_{\mathrm{BH}}$ and for $\mathrm{OT}(1)$; admits the rule at $\prec_{\mathrm{BH}}$ (Thm.~\ref{thm:itb-bh}) & \isn{itb\_WF\_bh}, \isn{uq2\_itb\_WF\_ot1},\newline \isn{itb\_accbiv\_bh} & IGA\_\allowbreak{}Univ, IGA\_\allowbreak{}Univ\_\allowbreak{}OT & grounded, classical \\
the level truth term; the empty list; stratified lists; the Kripke universes (Def.~\ref{def:level}) & \isn{ul\_TrL}, \isn{ul\_TrL\_Nil}, \isn{stratL},\newline \isn{trbivLS}, \isn{lvl\_derives} & IGA\_\allowbreak{}Univ\_\allowbreak{}Lvl & def., classical \\
one model for every Kripke universe; soundness and consistency of each; the Kripke universe at level $1$ is \iga with bivalence of the truth atom (Thm.~\ref{thm:one-model}) & \isn{MC\_TrL\_dec}, \isn{stratL\_MCdec}, \isn{level\_sound},\newline \isn{level\_consistent}, \isn{trbivLS\_trmTr}, \isn{lvl\_itb\_derives} & IGA\_\allowbreak{}Univ\_\allowbreak{}Lvl & classical \\
determinacy of the level truth term; T-out and T-in at a level (\S\ref{sec:tower}) & \isn{ul\_tr\_det},\newline \isn{ulv\_acc\_tout\_gen}, \isn{ulv\_acc\_tin\_gen} & IGA\_\allowbreak{}Univ\_\allowbreak{}Lvl & grounded \\
the Kripke universe at every level $n \geq 1$ derives $\WF$ for $\prec_{\mathrm{BH}}$ (Thm.~\ref{thm:ladder-k}) & \isn{ulv\_lvl\_WF\_bh}, \isn{ulv\_itb\_WF\_bh} & IGA\_\allowbreak{}Univ\_\allowbreak{}Lvl & grounded \\
the transfer between the rung's literal orders and their normal forms (\S\ref{sec:tower}) & \isn{uq2\_acc\_transfer}, \isn{uq2\_proxy\_conv} & IGA\_\allowbreak{}Univ\_\allowbreak{}OT & grounded \\
the stratified rung over a family of proxies; the Kripke universes at level $2$ and at level $k$ (Thm.~\ref{thm:ladder-k}) & \isn{uq2\_g\_WF\_otk\_px}, \isn{lvl2\_WF\_ot2},\newline \isn{lvlk\_WF\_otk} & IGA\_\allowbreak{}Univ\_\allowbreak{}OT & grounded \\
the obstruction to the consistency tower: the hole; the Kripke universe at level $1$ derives the variant's Liar negation; the closed-off model reads it; the named obligation; the one at level $2$ does not derive the node's level-$2$ truth, given the obligation (\S\ref{sec:contower}) & \isn{uc\_hole\_ex}, \isn{uc\_variant\_liar},\newline \isn{uc\_variant\_MC\_false}, \isn{uc\_variant\_omv\_noval}, \isn{ucSim},\newline \isn{uc\_no\_lvl2\_of\_sim}, \isn{uc\_OCONG\_of\_sim} & IGA\_\allowbreak{}Univ\_\allowbreak{}Lvl & classical, grounded, def. \\
the constant-free presentation: the zero replacement, the presentation of a body, the tower over the replacement, the constant-free truth and level terms (\S\ref{sec:contower}) & \isn{cfZ}, \isn{czd}, \isn{cfBd}, \isn{cfNumT},\newline \isn{trmTrC}, \isn{cfTrL} & IGA\_\allowbreak{}Univ\_\allowbreak{}CF & def. \\
at a point with a value: the replacement is $0$, the tower is $k$, the guard and the key test are decided and agree with the original tests; the omega reading (\S\ref{sec:contower}) & \isn{cf\_zero\_eq}, \isn{cf\_num\_eq}, \isn{cf\_guard\_dec},\newline \isn{cf\_guard\_iff}, \isn{cf\_key\_iff}, \isn{cf\_M\_num} & IGA\_\allowbreak{}Univ\_\allowbreak{}CF & grounded, classical \\
inert terms; every closed subterm inert, by the scan of the clauses (\S\ref{sec:contower}) & \isn{cfInert}, \isn{cfCF1}, \isn{cf\_scan\_trPhi},\newline \isn{cf\_CF1\_trmTrC}, \isn{cf\_CF1\_cfTrL} & IGA\_\allowbreak{}Univ\_\allowbreak{}CF & def., classical \\
no variant: the hole's filler; no value-equal variant in a model, in \rga, in the closed-off model (\S\ref{sec:contower}) & \isn{hole\_filler}, \isn{cf\_no\_hole\_M},\newline \isn{cf\_no\_hole}, \isn{cf\_no\_hole\_MC} & IGA\_\allowbreak{}Univ\_\allowbreak{}CF & classical \\
the agreement of the two presentations; the transfer of a negation (\S\ref{sec:contower}) & \isn{g\_trC\_iff}, \isn{g\_cfL\_iff}, \isn{g\_cfL\_neg\_of\_ul} & IGA\_\allowbreak{}Univ\_\allowbreak{}CF & grounded \\
the oracle case vacuous; the congruence swap at the level term over constant-free elements and at the constant-free level term (\S\ref{sec:contower}) & \isn{ul\_swap\_vac}, \isn{ul\_cong\_swap},\newline \isn{cf\_cong\_swap} & IGA\_\allowbreak{}Univ\_\allowbreak{}CF & grounded \\
constant-free stratified lists; not stratified in the earlier sense; every constant-free Kripke universe derives $\WF$ for $\prec_{\mathrm{BH}}$, is sound for the closed-off model and consistent (Def.~\ref{def:cf-level}) & \isn{stratC}, \isn{stratC\_iff}, \isn{cf\_not\_stratL\_trmTrC},\newline \isn{cfv\_lvl\_WF\_bh}, \isn{cfv\_lvl\_sound}, \isn{level\_consistentC} & IGA\_\allowbreak{}Univ\_\allowbreak{}CF & def., classical, grounded \\
the kit context; the fold's interface at the constant-free level term; the kit formulas derived; the gate of the ground-induction arm and the arm (\S\ref{sec:contower}) & \isn{cfkGk}, \isn{cf\_fkitU}, \isn{cfk\_g\_Gk},\newline \isn{cfcRGT\_holds}, \isn{cf\_gfix\_arm\_final} & IGA\_\allowbreak{}Univ\_\allowbreak{}CFCon & def., classical, grounded \\
\iga with bivalence of the constant-free truth atom derives $\mathit{Con}_{\mathrm{IGA}}$, beside \isn{itb\_ConIGA}; so does every constant-free Kripke universe from level $1$ on; the first separation (Thm.~\ref{thm:cf-conIGA}) & \isn{cf\_lvl1\_ConIGA}, \isn{cf\_lvl1\_vs\_itb},\newline \isn{cf\_ConIGA\_L}, \isn{cf\_lvl1\_separation} & IGA\_\allowbreak{}Univ\_\allowbreak{}CFCon & grounded, classical \\
the tag rule; the level checker; the level consistency sentence; at the empty list (Def.~\ref{def:conl}) & \isn{chkTag}, \isn{chkL}, \isn{ConL}, \isn{ConL\_Nil} & IGA\_\allowbreak{}Univ\_\allowbreak{}CFCon & def., classical \\
the next Kripke universe derives the level consistency sentence; the checker is faithful, from the faithfulness of a tag rule and of \iga's checker (Thm.~\ref{thm:cf-lvlcon}) & \isn{cf\_lvl\_Con}, \isn{cf\_lvl\_Con\_Nil},\newline \isn{tag\_fires\_iff}, \isn{drvChkL\_fires\_iff},\newline \isn{drvChkI\_fires\_iff} & IGA\_\allowbreak{}Univ\_\allowbreak{}CFCon,\newline IGA\_\allowbreak{}Checker & grounded, classical \\
G\"odel's second theorem at the level checker; the internal form; at the empty list; the strip (Thm.~\ref{thm:cf-goedel2}) & \isn{cf\_lvl\_goedel2}, \isn{cf\_lvl\_goedel2\_internal},\newline \isn{cf\_lvl\_goedel2\_Nil}, \isn{cf\_lvl\_strip} & IGA\_\allowbreak{}Univ\_\allowbreak{}CFCon & classical, grounded \\
the separation; the tower; the strict tower; monotonicity along lists (Thm.~\ref{thm:cf-sep}) & \isn{cf\_separation}, \isn{cf\_tower}, \isn{cf\_tower\_strict},\newline \isn{cf\_lvl\_mono} & IGA\_\allowbreak{}Univ\_\allowbreak{}CFCon & classical, grounded \\
\bottomrule
\end{longtable}}

\end{document}